\documentclass{dis}
\usepackage{makeidx}
\usepackage{hyperref}
\usepackage{amsmath,amssymb,amsthm}
\usepackage{color, bm, amscd, tikz-cd}
\usepackage[normalem]{ulem}
\theoremstyle{plain}

\newtheorem{theorem}{Theorem}[section]
\newtheorem{lemma}[theorem]{Lemma}
\newtheorem{proposition}[theorem]{Proposition}

\newtheorem{corollary}[theorem]{Corollary}

\numberwithin{theorem}{chapter}
\numberwithin{equation}{chapter}

\allowdisplaybreaks[1]

\theoremstyle{definition}

\newtheorem{definition}[theorem]{Definition}
\newtheorem{remark}[theorem]{Remark}
\newtheorem{example}[theorem]{Example}
\newtheorem{obs}[theorem]{Observation}

\theoremstyle{remark}
\newtheorem{notation}[theorem]{Notation}

\newcommand{\cB}{{\mathcal B}}

\newcommand{\cD}{{\mathcal D}}
\newcommand{\cE}{{\mathcal E}}
\newcommand{\cF}{{\mathcal F}}
\newcommand{\cG}{{\mathcal G}}
\newcommand{\cH}{{\mathcal H}}
\newcommand{\cI}{{\mathcal I}}

\newcommand{\cK}{{\mathcal K}}
\newcommand{\cL}{{\mathcal L}}
\newcommand{\cM}{{\mathcal M}}

\newcommand{\cO}{{\mathcal O}}
\newcommand{\cP}{{\mathcal P}}
\newcommand{\cQ}{{\mathcal Q}}
\newcommand{\cR}{{\mathcal R}}
\newcommand{\cS}{{\mathcal S}}

\newcommand{\cU}{{\mathcal U}}

\newcommand{\cW}{{\mathcal W}}

\newcommand{\cY}{{\mathcal Y}}

\newcommand{\be}{{\mathbf e}}

\newcommand{\biota}{{\boldsymbol \iota}}

\newcommand{\bTheta}{{\boldsymbol \Theta}}

\newcommand{\bT}{{\mathbf T}}
\newcommand{\boldW}{{\mathbf W}}

\newcommand{\bH}{{\mathbb{H}}}
\newcommand{\bW}{{\mathbb W}}
\newcommand{\bk}{{\mathbf k}}
\newcommand{\bS}{{\mathbf S}}
\newcommand{\bU}{{\mathbf{U}}}
\newcommand{\bV}{{\mathbf{V}}}
\newcommand{\bPi}{{\boldsymbol{\Pi}}}

\newcommand{\bcK}{{\boldsymbol{\mathcal K}}}
\newcommand{\btau}{{\boldsymbol{\tau}}}

\newcommand{\fH}{{\mathfrak H}}

\newcommand{\bD}{{\mathbb D}}

\newcommand{\bev}{\mathbf{ev}}

\newcommand{\sbm}[1]{\left[\begin{smallmatrix} #1
		\end{smallmatrix}\right]}

\newcommand{\ff}{{\mathfrak f}}

\newcommand{\boldm}{{\mathbf m}}
\newcommand{\boldn}{{\mathbf n}}

\newcommand{\blam}{{\boldsymbol{\lambda}}}
\newcommand{\bmu}{\boldsymbol{\mu}}

\newcommand{\boldz}{{\mathbf z}}

\newcommand{\bbW}{{\mathbb W}}

\newcommand{\bell}{{\boldsymbol \ell}}

\usepackage{graphicx}

\makeindex

\begin{document}
\keywords{Sch\"affer dilation, Douglas dilation, And\^o dilation, characteristic function, functional model, invariant subspace}

\mathclass{Primary: 47A13; Secondary: 47A20, 47A25, 47A56, 47A68, 30H10}


\abbrevauthors{J. A. Ball and H. Sau}
\abbrevtitle{Dilations and Models for Commuting Contractions}

\title{Dilation and Model Theory for Pairs of Commuting Contractions}
\author{Joseph A. Ball}
\address{Department of Mathematics, Virginia Tech, Blacksburg, VA 24061-0123, USA\\ joball@math.vt.edu}
\author{Haripada Sau}
\address{Department of Mathematics, Indian Institute of Science Education and Research, Pashan,
Pune, Maharashta 411008, India \\
hsau@iiserpune. ac.in, haripadasau215@gmail.com}

\maketitledis

\tableofcontents

\begin{abstract}  The Sz.-Nagy dilation theorem is a seminal result in operator theory. In short, any contraction operator $T$ on $\cH$ has a minimal isometric lift $V$ on 
$\cK \supset \cH$ which is unique up to a unitary change of coordinates in $\cK$ and correspondingly in $\cH$.  The Sz.-Nagy--Foias functional-model identifies 
the change of coordinates which leads to a functional-model representation for $V$ on a
functional-model Hilbert space $\cK_\Theta$ and for $T$ on $\cH_\Theta \subset \cK_\Theta$ defined solely in terms of the Sz.-Nagy--Foias characteristic function
$\Theta = \Theta_T$ of $T$. This, combined with spectral theory for the unitary part of $T$ if $T$ has a unitary part, reduces the study of a general contraction operator
$T$ to the study of a contractive analytic function $\Theta$ on the unit disk, in principle a much simpler object than $T$ (at least in the case when $\Theta$ is matrix-valued). The purpose  of this manuscript is to obtain the analogue of these results for the case of a commuting contractive pair $(T_1, T_2)$ in place of a single contraction operator $T$.

The first step has already appeared in the 1963 result of And\^o:  {\em any commuting pair of Hilbert-space contraction can be lifted to a commuting isometric pair.}  We provide two more constructive new proofs of And\^o's result, each of which leads to a new functional-model representation for such a lift. The construction leads to the
identification of a set of additional free parameters which serves to classify the distinct unitary-equivalence classes of minimal And\^o lifts. However this lack of uniqueness limits the utility of such minimal And\^o lifts for the construction of a functional model for a commuting contractive pair $(T_1, T_2)$. We identify an intermediate type of lift, called {\em pseudo-commuting contractive lift} $(\bbW_1,$ $ \bbW_2)$ of $(T_1, T_2)$. The operators $\bbW_1, \bbW_2$ are no longer commuting isometries, but are characterized by a slight weakening of
the commutativity condition, which still guarantees
that $\bbW_1$, $\bbW_2$ are multiplication operators of a simple form acting on the Sz.-Nagy--Foias minimal isometric lift space of the product contraction
$T = T_1 T_2$. In the Sz.-Nagy--Foias-like model form, the characteristic function $\Theta_T$
is augmented by what is called the {\em Fundamental-Operator pair} $(G_1, G_2)$, together with a \textit{canonical pair of commuting unitary operators} $(W_{\sharp 1}, W_{\sharp 2})$,  so that the combined collection
$((G_1, G_2), (W_{\sharp 1}, W_{\sharp 2}), \Theta_{T_1 T_2})$ (called the {\em characteristic triple} for $(T_1, T_2)$) is a complete unitary invariant for $(T_1, T_2)$. There is also a notion of
{\em admissible triple} $\Xi:= ((G_1, G_2), (W_1, W_2), \Theta)$ as the substitute for a {\em purely contractive analytic function $\Theta$} in the Sz.-Nagy--Foias theory,
from which one can construct a functional-model commuting contractive operator-pair $(T_{\Xi, 1}, T_{\Xi, 2})$ having its characteristic triple
\textit{coinciding} with the original admissible triple  in an appropriate sense.
 \end{abstract}
 \makeabstract

 \chapter{Introduction}   \label{C:intro}

The starting point for many future developments in nonselfadjoint operator theory was the Sz.-Nagy dilation theorem
from 1953 \cite{sz-nagy}:   {\em if $T$ is a contraction operator on a Hilbert space $\cH$, then there is a
unitary operator $\cU$ on a  larger Hilbert space  $\widetilde \cK \supset \cH$ such that $T^n = P_\cH \cU^n |_\cH$
for all $n=0,1,2,\dots$.}
While the original proofs were more existential than constructive, there followed more concrete constructive proofs
(e.g., the Sch\"affer-matrix construction from \cite{Schaffer} to be discussed below)  which evolved into a detailed geometric picture
of the dilation space (see \cite[Chapter II]{Nagy-Foias}).  Analysis of how the original Hilbert space $\cH$ fits into
the dilation space $\widetilde \cK$ and the appropriate implementations of the discrete Fourier transform convert the abstract spaces to
spaces of functions (holomorphic or measurable as the case may be) led to the discovery of the
characteristic function $\Theta_T$ of any completely nonunitary (c.n.u.) contraction operator $T$ and how the c.n.u.\  contraction
operator $T$ can be represented  (up to unitary equivalence) as a compressed multiplication operator on a
functional-model Hilbert space constructed directly from $\Theta_T$.   Here we say that the contraction operator $T$ is a {\em c.n.u. contraction} if $T$ has no non-trivial reducing subspace on which $T$ is unitary.
We prefer to work with the equivalent notion of minimal isometric lifts $V$ rather than minimal unitary dilations $\cU$ of $T$; here we say that
an operator $V$ on $\cK \supset \cH$ is a {\em lift}  of $T$ on $\cH$ if $\cH$ is invariant for $V^*$ and $V^*|_\cH = T^*$. More generally, if
$\Pi \colon \cH \to \cK$ is an isometric embedding of $\cH$ into $\cK$ and $V^* \Pi = \Pi T^*$, we shall also say that $(\Pi, V)$ on $\cK$ is a lift of
$T$ on $\cH$.

To describe our results it is convenient to describe the Sz.-Nagy--Foias functional model for the c.n.u.~contraction operator and the associated minimal isometric
lift in some detail as follows.
We define the defect operators 
$$
   D_T = (I- T^*T)^{\frac{1}{2}}, \quad D_{T^*} = (I - T T^*)^{\frac{1}{2}}
$$
and defect spaces
$$
 \cD_T = \overline{\operatorname{Ran}}\, D_T, \quad \cD_{T^*} = \overline{\operatorname{Ran}}\, D_{T^*},
 $$
introduce the characteristic function of $T$
 $$
  \Theta_T(z)  = \left( -T^* + z D_{T^*} (I - z T^*)^{-1} D_{T}\right)|_{\cD_{T}} \colon \cD_{T} \to \cD_{T^*},
 $$
 and the pointwise-defect operator of $\Theta_T$:
 $$
 \Delta_{\Theta_T}(\zeta) = (I - \Theta_T(\zeta)^* \Theta_T(\zeta))^{\frac{1}{2}} \text{ for } \zeta \in {\mathbb T}.
 $$
 Define functional Hilbert spaces 
 $$
 \cK_{\Theta_T} := \begin{bmatrix} H^2(\cD_{T^*}) \\\overline{ \Delta_{\Theta_T} L^2(\cD_T)} \end{bmatrix}, \quad
 \cH_{\Theta_T} := \cK_{\Theta_T} \ominus \begin{bmatrix} \Theta_T \\ \Delta_{\Theta_T} \end{bmatrix} H^2(\cD_T)
 $$
 and define operators $V_{\Theta_T}$ on $\cK_{\Theta_T}$ and $T_{\Theta_T}$ on $\cH_{\Theta_T}$ by
 \begin{align}\label{VThetaT}
  V_{\Theta_T} = \begin{bmatrix} M_z^{\cD_{T^*}} & 0 \\ 0 & M_\zeta|_{\overline{ \Delta_{\Theta_T} L^2(\cD_T)}} \end{bmatrix} 
  \mbox{ and } T_{\Theta_T} = P_{\cH_{\Theta_T}} V_{\Theta_T} |_{\cH_{\Theta_T}}. 
 \end{align}
 Then we have:
 \smallskip
 
 \noindent
 \textbf{Theorem A.} {\sl $\cH_{\Theta_T}$ is invariant for $V_{\Theta_T}^*$.  If $V$ on $\cK$ is any minimal isometric lift of $T$ on $\cH$,  then there is a unitary transformation $\tau \colon \cK \to \cK_{\Theta_T}$ such that}
 $$
 \tau \cH = \cH_{\Theta_T}, \quad  \tau V = V_{\Theta_T} \tau, \quad (\tau|_\cH) T = T_{\Theta_T} (\tau|_\cH)
  $$
  (so $T$ is unitarily equivalent to $T_{\Theta_T}$ via the unitary operator $\tau|_\cH \colon \cH \to \cH_{\Theta_T}$).
 Conversely, if $(\cD, \cD_*, \Theta)$ is a purely contractive analytic function on ${\mathbb D}$ (meaning that $\Theta(z) \in \cB(\cD, \cD_*)$ for $z \in {\mathbb D}$
 and that $\| \Theta(0) d \| < \| d \|$ for $0 \ne d \in \cD$), and if we define $\Delta_\Theta(\zeta)$, $\cK_\Theta$, $\cH_\Theta$, $V_\Theta$, $T_\Theta$ as above with
 $(\cD, \cD_*, \Theta)$ in place of $(\cD_T, \cD_{T^*}, \Theta_T)$, then {\em $V_\Theta$ is the minimal isometric lift of the c.n.u.~contraction $T_\Theta$ and
 the characteristic function $(\cD_{T_\Theta}, \cD_{T^*_\Theta}, \Theta_{T_\Theta})$  coincides with $(\cD, \cD_*, \Theta)$}, 
  i.e., there are unitary operators $u \colon \cD \to \cD_{T_\Theta}$, 
 $u_* \colon \cD_* \to \cD_{T_\Theta^*}$ so that $\Theta_{T_\Theta}(z) u = u_* \Theta(z)$ for $z \in {\mathbb D}$.

\smallskip

The And\^o dilation theorem \cite{ando}, coming ten years later, provides a 2-variable analogue of the
Sz.-Nagy dilation theorem:  {\em given a commuting pair of contraction operators $(T_1, T_2)$ on a
Hilbert space $\cH$, there is a commuting pair of unitary operators $(\cU_1, \cU_2)$ on a larger Hilbert
space $\widetilde \cK \supset \cH$ so that, for all $n,m \ge 0$, $T_1^n T_2^m = P_\cH \cU_1^n \cU_2^m |_\cH$.}
The proof there is an expanded version of the Sch\"affer-matrix construction for the single-operator case
which failed to shed much light on the geometry of the dilation space (a consequence of the lack of uniqueness up to a
notion of unitary equivalence for And\^o dilations).  Consequently there has been essentially no follow-up
to the And\^o result in the direction of a Sz.-Nagy--Foias-type model theory for a commuting pair of contraction
operators as there was in the single-operator setting, although there have now  been some preliminary results in this direction
(see \cite{D-S-S, BV, A-M-Dist-Var}).

In an independent development, Berger-Coburn-Lebow \cite{B-C-L} obtained a model for a commuting-tuple
of isometries $(V_1, \dots, V_d)$ by considering the Wold decomposition for the product $V= V_1 \cdots V_d$ and
understanding what form the factors $V_1, \dots, V_d$ must take so as (i) to be themselves
commuting isometries, and (ii) to have product equal to $V$.  The conditions required to guarantee commutativity of the model isometries
$V_1, \dots, V_d$ are rather involved for the case $d \ge 3$ but are immediately transparent and succinct for the case $d=2$.  For $d=2$
the model is determined by a collection of objects $(\cF, P, U, W_1, W_2)$ which we call a BCL-{\em tuple} consisting of
\begin{itemize}
\item[(i)] a coefficient Hilbert space $\cF$,
\item[(ii)] a projection $P$ and a unitary operator $U$ on $\cF$, and
\item[(iii)] a commuting pair of unitary operators $W_1$, $W_2$ on a common Hilbert space $\cH_u$.
\end{itemize}
Given such a BCL-tuple, we associate a pair of operators in two distinct ways:
\begin{equation}   \label{Intro:BCL1model}
 V_1 = \begin{bmatrix} M_{P^\perp U + z P U} & 0 \\ 0 & W_1 \end{bmatrix}, \quad V_2 = \begin{bmatrix} M_{U^* P + z U^* P^\perp} & 0 \\ 0 & W_2 \end{bmatrix}\quad\mbox{on}\quad\begin{bmatrix}
 H^2(\cF)\\ \cH_u
 \end{bmatrix},
\end{equation}
or
\begin{equation}   \label{Intro:BCL2model}
 V_1 = \begin{bmatrix} M_{U^*P^\perp  + z U^*P } & 0 \\ 0 & W_1 \end{bmatrix}, \quad V_2 = \begin{bmatrix} M_{PU + z  P^\perp U} & 0 \\ 0 & W_2 \end{bmatrix}\quad\mbox{on}\quad\begin{bmatrix}
 H^2(\cF)\\ \cH_u
 \end{bmatrix}.
\end{equation}
Here, for a Hilbert space $\cF$, $H^2(\cF)$ denotes the $\cF$-valued \emph{Hardy space}
$$
H^2(\cF):=\{f:\bD\to\cF: f(z)=\sum_{n=0}^\infty z^nf_n\;\mbox{ and }\; \sum_{n=0}^\infty \|f_n\|_\cF^2<\infty\}.
$$
In the first case \eqref{Intro:BCL1model} we say that $(\cF, P, U, W_1, W_2)$ is a {\em BCL1-model for $(V_1, V_2)$}, while in the second case
\eqref{Intro:BCL2model}  we say that $(\cF, P, U, W_1, W_2)$ is a {\em BCL2-model} for $(V_1, V_2)$.
It is easily checked that  in either case $(V_1, V_2)$ is a commuting isometric pair.  In fact there is a simple correspondence between BCL1 and BCL2
models for a given $(V_1, V_2)$:  {\em $(\cF, P, U, W_1, W_2)$ is a BCL1-tuple for $(V_1, V_2)$ if and only if
$(\cF, U^* P^\perp U, U^*)$ is a BCL2-tuple  for $(V_1, V_2)$.}

The result from \cite{B-C-L} is the converse:

\smallskip

\noindent
\textbf{Theroem B.}  (See Theorem \ref{Thm:BCLmodel} below.)  {\sl Any commuting isometric pair $(V_1, V_2)$ on a Hilbert space $\cK$ is unitarily equivalent to
the BCL-model isometric pair (of either the BCL1 or BCL2  form)  for some BCL-tuple $(\cF, P, U, W_1, W_2)$.  If $(\cF', P', U', W_1', W_2')$ is another BCL-tuple
giving rise to a BCL model  commuting isometric pair $(V_1', V_2')$  of the same form (BCL1 or BCL2) as $(V_1, V_2)$ which is unitarily equivalent to $(V_1, V_2)$,
 then  $(\cF, P, U, W_1, W_2)$
and $(\cF', P', U', W_1',$ $W_2')$ {\em coincide} in the sense that there are unitary transformations $\omega \colon \cF \to \cF'$ and
$\tau \colon \cH_u \to \cH'_u$ so that}
$$
\omega P = P' \omega, \quad \omega U = U' \omega, \quad \tau W_j = W'_j \tau \text{ for } j=1,2.
$$

\smallskip

The goal of this manuscript is to develop a more complete analogue of the Sz.-Nagy--Foias dilation theory and operator model theory for the 
commuting contractive pair setting.  This proceeds in several steps.

\smallskip

\noindent
\textbf{1.  Parametrization of And\^o lifts.}   Suppose that $(T_1, T_2)$ is a commuting contractive pair on $\cH$ and $(\bPi, \bV_1, \bV_2)$ is a minimal And\^o lift
for $(T_1, T_2)$ on $\bcK$, where $\bPi \colon \cH \to \bcK$ is an isometric embedding of $\cH$ into $\bcK$.  Then up to a unitary equivalence we have that $T = T_1 T_2$ is in the 
Sz.-Nagy--Foias functional-model form:
\begin{equation}  \label{NFfuncmodel}
T \underset{u}\cong T_{\Theta_T} := \left. P_{\cH_{\Theta_T}} \sbm{ M_z^{\cD_{T^*}} & 0 \\ 0 & M_\zeta|_{\overline{\Delta_{\Theta_T} L^2(\cD_T)}} } \right|_{\cH_{\Theta_T}}
\end{equation}
and the operators $T_1$ and $T_2$
 are then commuting contraction operators on $\cH_{\Theta_T}$ which factor $T_{\Theta_T}$:  $T_{\Theta_T}  = T_1 T_2 = T_2 T_2$.  
 Furthermore we may assume that $(\bV_1, \bV_2)$ is in the BCL2-model form  \eqref{Intro:BCL2model} for some BCL-tuple $(\cF, P, U, W_1, W_2)$ acting on a space of the form 
 $\sbm{ H^2(\cF) \\ \cH_u}$.  Then it remains to describe the (concrete) isometric identification map $\bPi \colon \cH_{\Theta_T} \to \sbm{ H^2(\cF) \\ \cH_u}$.  The result is as follows 
 (see Theorem \ref{Thm:NFmodel} and Corollary \ref{Cor:NFCoin} below) 

 \smallskip
 
 \noindent
 \textbf{Theorem C.}  {\sl Without loss of generality we may assume that  the space $\cH_u$  is equal to the second component of the Sz.-Nagy--Foias functional-model space
 $\cH_u = \overline{\Delta_{\Theta_T} L^2(\cD_T)}$ and the operators $W_1$, $W_2$ are the operators $W_{\sharp 1}$, $W_{\sharp 2}$ canonically uniquely determined by the
 commuting, contractive operator-pair $(T_1, T_2)$ (see Theorem \ref{T:flats} together with the notation \eqref{WandVNFs} below), 
 and that the  operators $\bV_1, \bV_2$ are in the BCL2-model form
 \begin{equation}  \label{CNF1}
 (\bV_1, \bV_2) = \left( \begin{bmatrix} M_{U^* P^\perp  + z U^* P} & 0 \\ 0 & W_{\sharp 1} \end{bmatrix},
 \begin{bmatrix} M_{ P U + z  P^\perp U } & 0 \\ 0 & W_{\sharp 2} \end{bmatrix} \right)  \text{ on } \begin{bmatrix} H^2(\cF) \\ \overline{\Delta_{\Theta_T} L^2(\cD_T)}\end{bmatrix}
 \end{equation}
 Then there is an isometric operator $\Lambda \colon \cD_{T^*} \to \cF$ such that the embedding operator
 $ \bPi$ is given by
 \begin{equation}  \label{CNF2}
    \bPi = \left. \begin{bmatrix} I_{H^2} \otimes \Lambda  & 0 \\ 0 & I_{\overline{\Delta_{\Theta_T} L^2(\cD_T)}} \end{bmatrix} \right|_{\cH_{\Theta_T}} \colon \cH_{\Theta_T} \to \begin{bmatrix}
    H^2(\cF) \\ \overline{\Delta_{\Theta_T} L^2(\cD_T)} \end{bmatrix} 
 \end{equation}
 and the augmentation $(\cF, \Lambda, P, U)$  of the BCL-tuple $(\cF, P, U)$ is a Type I And\^o tuple for $(T_1^*, T_2^*)$ in the sense that $\Lambda$ 
 must satisfy two compatibility operator equations
 involving $T_1^*$, $T_2^*$ and the BCL-tuple parameters $(P, U)$  (namely, equations \eqref{AndoTuple1} and \eqref{AndoTuple2} with the subscript $*$'s dropped).
 Furthermore for the associated lift to be minimal, the And\^o tuple should also satisfy an additional
 minimality condition (see Definition \ref{D:min-Ando} below.).
  
 Moreover, if $(\bPi, \bV_1, \bV_2)$ and $(\bPi',\bV_1', \bV_2')$ are two such lifts corresponding to minimal Type I $(T_1^*, T_2^*)$-And\^o tuples $(\cF, \Lambda, P, U)$ and 
 $(\cF', \Lambda', P', U')$ respectively, then $(\bPi, \bV_1, \bV_2)$ and $(\bPi',\bV_1', \bV_2')$ are unitarily equivalent as lifts if and only if $(\cF, \Lambda, P, U)$ and 
 $(\cF', \Lambda', P', U')$ coincide in the sense that there is a unitary change of basis $\omega \colon \cF \to \cF'$ so that}
 $$
  \omega \Lambda = \Lambda', \quad  \omega P = P' \omega, \quad  \omega U = U' \omega.
 $$ 
 
 Putting all the pieces together, we can say:  {\em unitary-equivalence classes of minimal And\^o lifts $(\bV_1, \bV_2)$ of a given commuting contractive pair $(T_1, T_2)$ are in 
 one-to-one correspondence with coincidence-equivalence classes of minimal Type I And\^o tuples $(\cF, \Lambda, P, U)$ of the commuting contractive pair $(T_1^*, T_2^*)$.}
 
 \smallskip
 
 We note that it is not immediately clear that the system of operator equations \eqref{AndoTuple1} - \eqref{AndoTuple2} has a solution $(\cF, \Lambda, P, U)$ for a given commuting pair $(T_1, T_2)$.  However, there is a type of pre-And\^o tuple $(\cF, \Lambda, P, U)$ which we call a
 {\em special And\^o tuple} for which it is possible to verify by direct computation that equations \eqref{AndoTuple1} - \eqref{AndoTuple2} do hold.  This is discussed in
 Section 4.3 below.  Let us point out that a direct construction of an And\^o lift for a commuting contractive pair $(T_1, T_2)$ such that the product operator $T = T_1 T_2$ is of class $C_{\cdot 0}$ appears in the 2017 paper of Das-Sarkar-Sarkar \cite{D-S-S}.
 
 \smallskip
 
 \noindent
 \textbf{2. Pseudo-commuting contractive lifts.}  One take-away from the preceding discussion is that And\^o lifts for a given commuting contractive pair $(T_1, T_2)$
 always exist, but are not necessarily uniquely determined (up to lift-unitary equivalence) by the pair $(T_1, T_2)$.  We now introduce a weaker type of lift which is uniquely
 determined up to lift-unitary equivalence by $(T_1, T_2)$ and arguably is a better parallel to the Sz.-Nagy--Foias minimal isometric lift of a single contraction operator $T$
 for the pair case $(T_1, T_2)$.  

Towards this end, let us suppose that $(\bPi, \bV_1, \bV_2)$ is the Sz.-Nagy--Foias-like model for a minimal And\^o lift of $(T_1, T_2)$ as in \eqref{CNF1} and \eqref{CNF2}
with the product operator $T = T_1 T_2 = T_2 T_1$ again assumed to be in the Sz.-Nagy-Foias model form as in \eqref{NFfuncmodel}.
 Note that the map $\bPi$ has an obvious extension to a map $\widehat \bPi$ acting on all of $\cK_{\Theta_T}$ with range still in 
 $\sbm{ H^2(\cF) \\ \overline{\Delta_{\Theta _T}L^2(\cD_T)}}$:
 \begin{equation}  \label{hatbPi}
  \widehat \bPi = 
  \begin{bmatrix}  
  I_{H^2} \otimes \Lambda & 0 \\ 0 & I_{\overline{\Delta_{\Theta_T} L^2(\cD_T)}} 
  \end{bmatrix} \colon 
  \begin{bmatrix} H^2(\cD_{T^*}) \\
  \overline{ \Delta_{\Theta_T} L^2(\cD_T)} \end{bmatrix}  \to \begin{bmatrix} H^2(\cF) \\  \overline{ \Delta_{\Theta _T}L^2(\cD_T)} \end{bmatrix}.
 \end{equation}
 This map is still isometric and has the intertwining property
 $$
   \widehat \bPi \begin{bmatrix} M_z^{\cD_{T^*}} & 0 \\ 0 & M_\zeta |_{\overline{ \Delta_{\Theta_T} L^2(\cD_T)}} \end{bmatrix}
    =   \begin{bmatrix} M_z^{\cF} & 0 \\ 0 & M_\zeta |_{\overline{ \Delta_\Theta L^2(\cD)}} \end{bmatrix}  \widehat \bPi =
  $$
 i.e.,
 $$
    \widehat \bPi V_{\Theta_T} = \bV \widehat \bPi
  $$
  where $\bV = \bV_1 \bV_2$ and where $V_{\Theta_T}$ as in \eqref{VThetaT} is the minimal isometric lift of $T_{\Theta_T}$ on $\cK_{\Theta_T}$. Furthermore, 
  as the Sz.-Nagy--Foias model lift $V_{\Theta_T}$ of $T_{\Theta_T}$ is minimal, we know that
  $$
     \cK_{\Theta_T} = \bigvee_{n \ge 0} V_{\Theta_T}^n \cH_{\Theta_T}.
  $$
  We conclude that 
 \begin{align*} \bcK_{00} :=  & \bigvee_{n \ge 0} \bV^n \operatorname{Ran} \Pi
   = \bigvee_{n \ge 0} \bV^n \widehat \bPi \cH_{\Theta_T} = \widehat \bPi \big( \bigvee_{n \ge 0} V_{\Theta_T}^n \cH_{\Theta_T} \big)
   = \widehat \bPi \cK_{\Theta_T}  \\
   =  & \begin{bmatrix} H^2(\operatorname{Ran} \Lambda) \\ \overline{  \Delta_{\Theta_T} L^2(\cD_T)} \end{bmatrix}
\end{align*}
  It now follows that  $(\bPi, \bV|_{\bcK_{00}})$ is a lift of $T_{\Theta_T}$ unitarily equivalent to the Sz.-Nagy--Foias model lift 
 $(\iota_{\cH_{\Theta_T} \to \cK_{\Theta_T}}, V_{\Theta_T})$ of $T_{\Theta_T}$.
In particular $\bV|_{\bcK_{00}}$ is unitarily equivalent to $V_{\Theta_T}$ via the unitary identification map $\widehat \bPi \colon \cK_{\Theta_T} \to \bcK_{00}$. 
The next idea is to compress the And\^o lift $(\bV_1, \bV_2)$ augmented by the product operator $\bV = \bV_1 \bV_2$ defined on $\bcK$ to the copy of the embedded 
minimal lift space $\bcK_{00}$ for the single contraction operator $T _{\Theta_T} = T_1 T_2$, which in turn can be represented as operators on the Sz.-Nagy--Foias model space 
$\cK_{\Theta_T}$ 
by again making use of the unitary identifaction $\widehat \bPi \colon \cK_{\Theta_T} \to \bcK_{00}$:
$$
  (\bbW_1, \bbW_2, \bbW) = \widehat \bPi^* ( \bV_1, \bV_2, \bV) \widehat \bPi \text{ acting on } \cK_{\Theta_T}.
$$
Plugging in the formulas \eqref{CNF1} and \eqref{hatbPi} for $\bV_1$, $\bV_2$,   $\widehat \bPi$ then gives us the explicit formulas in terms of the
And\^o tuple $(\cF, \Lambda, P, U)$ associated with the And\^o lift $(\bV_1, \bV_2)$:
$$
(\bbW_1, \bbW_2, \bbW) = 
\left(  \begin{bmatrix} M_{G_1^* + z G_2} & 0 \\ 0 & W_{\sharp 1} \end{bmatrix}, \begin{bmatrix} M_{G_2^* + z G_1} & 0 \\ 0 & W_{\sharp 2} \end{bmatrix},
\begin{bmatrix} M_z^{\cD_{T^*}} & 0 \\ 0 & M_\zeta|_{ \overline{\Delta_{\Theta_T} L^2(\cD_T)}}  \end{bmatrix}  \right)
$$
where we set $ G_1 = \Lambda^*  P^\perp U \Lambda, \quad G_2 = \Lambda^* U^* P \Lambda.$ Such triples have their own abstract characterization:   any such triple is a {\em pseudo-commuting, contractive operator-triple} meaning that $(\bbW_1, \bbW_2, \bbW)$
``almost commute"   in the sense of Definition \ref{D:pcc} below.  In addition $(\bbW_1, \bbW_2, \bbW)$ is a lift of $(T_1, T_2, T_{\Theta_T} = T_1 T_2)$ with $\bbW$ 
being equal to the minimal isometric lift  $V_{\Theta_T}$ of $T_{\Theta_T}$ on $\cK_{\Theta_T}$, and conversely:  {\sl any
pseudo-commuting, contractive lift $(\bbW_1, \bbW_2, \bbW)$ of $(T_1, T_2, T_{\Theta_T})$ with $\bbW$ equal to the minimal isometric lift $V_{\Theta_T}$ of $T_{\Theta_T}$ 
is unique and arises in this way as the
compression of any choice of  minimal And\^o lift of $(T_1, T_2)$} (see Theorem \ref{T:comp=pcc}).

Moreover, there is an independent characterization of the operators $(G_1, G_2)$ which shows that they are independent of the choice of minimal And\^o lift, or equivalently,
of minimal Type I And\^o tuple $(\cF, \Lambda, P, U)$ for $(T_1^*, T_2^*)$.  In fact $(G_1, G_2)$ can be alternatively characterized as the unique solution of a  certain system of
operator equations involving only $(T_1^*, T_2^*)$ and not involving a choice of Type I And\^o tuple $( \cF, \Lambda, P, U)$ for
$(T_1^*, T_2^*)$ (see Theorem \ref{T:FundOps}).  Following the precedent set in \cite{B-P-SR, Tirtha-tetrablock}, we call such $(G_1, G_2)$ arising in this way to be the 
{\em Fundamental-Operator pair} for $(T_1^*, T_2^*)$.   
In summary, we conclude that, unlike the case for minimal And\^o lifts, {\em the compression $(\bbW_1, \bbW_2, \bbW)$ of a minimal And\^o lift $(\bV_1, \bV_2, \bV)$ of 
$(T_1, T_2, T_{\Theta_T} = T_1 T_2)$
to the minimal isometric lift space $\cK_{\Theta_T}$ for $T_{\Theta_T}$, or equivalently, any pseudo-commuting, contractive lift $(\bbW_1, \bbW_2, \bbW)$ of
$(T_1, T_2, T_{\Theta_T})$ with third component $\bbW$ equal to the minimal isometric lift $V_{\Theta_T}$ of $T_{\Theta_T}$, is uniquely determined by the factor contrations
$T_1$ and $T_2$.} 
While $(\bbW_1, \bbW_2, \bbW)$ is no longer commuting, it does have a functional model representation of a much simpler form than that of a general commuting contractive 
pair $(T_1, T_2)$.  This is what leads to a Sz.-Nagy--Foias-like functional model for the commuting contractive pair $(T_1, T_2)$ discussed next.

\smallskip

\noindent
\textbf{3.  The Sz.-Nagy--Foias-like functional model for a contractive pair.}
The idea behind the Sz.-Nagy--Foias functional model for a single contraction operator $T$ is to obtain a relatively simple functional model for the essentially unique 
minimal isometric lift $V$ of $T$, and then compress the action of $V$ to its $*$-invariant subspace to arrive at a functional model for $T$.  A key point is the
uniqueness:  {\em there is a one-to-one correspondence between unitary equivalence classes of contraction operators and unitary equivalence classes of minimal isometric lifts.}
When we consider the pair case and use a minimal And\^o lift for the pair  $(T_1, T_2)$ in place of a minimal isometric lift for the single operator $T$,  this one-to-one correspondence
fails; going to the minimal And\^o lift introduces what one might call {\em noise} (extraneous data which has nothing to do with the original object of study, namely the commuting
contractive pair).  On the other hand, if we use the pseudo-commuting, contractive lift $(\bbW_1, \bbW_2, \bbW)$ for $(T_1, T_2,T)$ in place of the minimal isometric
lift $V$ for $T$, the situation is more parallel to the classical case.  Given a commuting, contractive pair,  we define the collection 
$$
\Xi_{(T_1, T_2)}: = ((G_1, G_2), (W_{\sharp 1}, W_{\sharp 2}), \Theta_T)
$$
 to be the {\em characteristic triple} of $(T_1, T_2)$, where $\Theta_T$ is the characteristic function for $T = T_1 T_2$, $(G_1, G_2)$ is
the Fundamental-Operator pair for $(T_1^*, T_2^*)$, and $(W_{\sharp 1}, W_{\sharp 2})$ is the commuting pair of unitary operators with product equal to 
$ M_\zeta |_{\overline{\Delta_{\Theta_T} L^2(\cD_T)}}$  on the space $\overline{\Delta_{\Theta_T} L^2(\cD_T)}$ canonically and uniquely associated with $(T_1, T_2)$ 
appearing in \eqref{CNF1}.  This characteristic triple turns out to be a {\em complete unitary invariant} for $(T_1, T_2)$ in the following sense:
{\em the commuting contractive pair $(T_1, T_2)$ is unitarily equivalent to the commuting contractive pair $(T_1', T_2')$ if and only if the associated 
characteristic triples $\Xi_{(T_1, T_2)}$ and $\Xi_{(T_1', T_2')}$ coincide in a certain natural sense.}  The reverse procedure in the classical case
relies on the clean characterization of the coincidence envelope of characteristic functions $\Theta_T$ as the set of purely contractive analytic functions
$\Theta$; for the pair case, the characterization of the coincidence envelope of characteristic triples, namely what we call {\em admissible triples}
$((G_1, G_2), (W_1, W_2), \Theta)$ (see Definition \ref{D:AdmissTriple}), is less tractable. Nevertheless this analysis provides some insight into the structure of 
commuting contractive pairs in general and can be tractable in some special cases.

\smallskip

This manuscript is organized as follows:  Following the Introduction, Chapter 2 reviews the unitary-dilation/isometric-lift/operator-model theory for a single contraction operator
from four points of view:  (i) the coordinate-free geometric picture as found in Chapter I of the classic book \cite{Nagy-Foias},  (ii) the Douglas model theory as in
\cite{Doug-Dilation}, the Sch\"affer model theory \cite{Schaffer}, and the Sz.-Nagy--Foias model theory as found in Chapter VI  of \cite{Nagy-Foias}. 

Chapter 3 develops from first principles the Berger-Coburn-Lebow model theory for a commuting contractive pair $(V_1, V_2)$ \cite{B-C-L} with inclusion of many illustrative examples.

In addition to the Sz.-Nagy--Foias model for a minimal And\^o lift of a commuting contractive pair,  Chapter 4 develops from first principles the
Douglas and Sch\"affer models for a minimal And\^o lift $(\bV_1, \bV_2)$ for a commuting contractive operator-pair $(T_1, T_2)$;  in fact the Douglas model is developed
first and then used as a bridge for understanding the Sz.-Nagy--Foias model.  The {\em Fundamental-Operator pair} $(F_1, F_2)$ for the
commuting contractive pair $(T_1^*, T_2^*)$ appears here for the first time in connection with characterizing when $(T_1, T_2)$ has a strongly minimal And\^o lift
(see Theorem \ref{T:StrongFund} below):
{\sl $(T_1, T_2)$ has a strongly minimal And\^o lift if and only if the Fundamental-Operator pair $(G_1, G_2)$ for $(T_1^*, T_2^*)$ satisfies the
additional system of operator equations \eqref{strong-fund} given below};  equivalently, {\sl there is a projection $P$ and a unitary operator $U$ on $\cF$
so that $(G_1, G_2) = (P^\perp U, U^*P)$} (see Lemma \ref{relations-of-E-lem} below) and the operator pair $(M_{F_2^* + F_1 z}, M_{F_1^* + F_2 z})$ on $H^2(\cF)$ 
assumes the BCL2-model form
$(M_{U^*P^\perp + z U^* P z}, M_{ P U + z  P^\perp U})$ for a commuting isometric pair.  In view of the results in Chapter 6 on pseudo-commuting contractive lifts,
this is just the statement that the pseudo-commuting contractive lift of $(T_1, T_2)$ is actually an And\^o lift  which is also exactly the case when
any minimal And\^o lift is unique up to unitary equivalence of lifts.
Also developed in Chapter 4  is the equivalence between existence of a strongly minimal And\^o lift for $(T_1, T_2)$ and the condition that the factorizations
$T = T_1 \cdot T_2$ and $T = T_2 \cdot T_1$ are both {\em regular} in the sense of Sz.-Nagy--Foias (see \cite{Nagy-Foias}).

Chapter 5 lays out the one-to-one correspondence between unitary-equivalence classes of minimal And\^o lifts for a commuting contractive pair $(T_1, T_2)$ on the one hand,
and coincidence-equivalence classes of the corresponding minimal And\^o tuples, both in the Douglas-model setting (where one works with
minimal Type I  And\^o tuples for $(T_1^*, T_2^*)$) and in the Sch\"affer-model setting (where one works with minimal strong Type II And\^o tuples for $(T_1, T_2)$).

Chapter 6 focuses on the pseudo-commuting contractive lift for a given commuting contractive pair $(T_1, T_2)$.  Some preliminary results can be obtained
in the abstract framework but other results (e.g., that {\sl the final component $\bbW$ of a pseudo-commuting contractive lift $(\bbW_1, \bbW_2, \bbW)$ 
of $(T_1, T_2, T = T_1 T_2 = T_2 T_1)$ uniquely determines the other components $\bbW_1, \bbW_2$}) makes use of a functional model
(any of Douglas, Sz.-Nagy--Foias, or Sch\"affer) for $(\bbW_1, \bbW_2, \bbW)$.

Chapter 7 develops the Sz.-Nagy--Foias-like functional model for a given commuting contractive pair $(T_1, T_2)$ while Chapter 8 obtains an analogue of the
Sz.-Nagy--Foias correspondence between invariant subspaces for $T$ and regular factorizations $\Theta(\zeta) = \Theta''(\zeta) \Theta'(\zeta)$
of the characteristic function $\Theta = \Theta_T$ of $T$ (see \cite{Nagy-Foias}).  

 
Finally, let us note that our companion paper \cite{BS-Douglas} extends some of the framework of this manuscript
to the higher-order tuple setting (commuting $d$-tuples $\underline{T} = (T_1, \dots, T_d)$ of contraction operators
on a Hilbert space with $d > 2$), and that this manuscript essentially subsumes the preliminary report \cite{sauAndo} posted on arXiv.

\numberwithin{section}{chapter}
\numberwithin{equation}{section}
\numberwithin{theorem}{section}

\chapter[Functional models]{Functional models for isometric lifts of a contraction operator}

The Sz.-Nagy dilation theorem asserts that any contraction operator $T$ on the Hilbert space $\cH$  can be dilated to a unitary operator $\cU$ on a space $\widetilde \cK
\supset \cH$, i.e., there is a unitary operator $\cU$ on a Hilbert space $\widetilde \cK$ containing $\cH$ so that 
$T^n = P_\cH \cU^n|_\cH$  for  $n = 0,1,2, \dots$.  By a lemma of Sarason, this is the same as saying that $\widetilde\cK$ has an orthogonal decomposition
$\widetilde\cK = \cK_- \oplus \cH \oplus \cK_+$ and with respect to this decomposition $\cU$ has the block-matrix form $\cU = \sbm{ * & 0 & 0 \\ * & T & 0  \\ * & * & * }$.
The dilation is said to be {\em minimal} if it is the case that $\widetilde \cK$ is the smallest reducing subspace for $\cU$ containing $\cH$. 
An operator $V$ on a Hilbert space $\cK \supset \cH$ is said to be an {\em isometric lift} of $T$ if $V$ is isometric, $\cK \ominus \cH$ is invariant for $V$
and $T^n = P_\cH V^n |_\cH$ for $n=0,1,2,\dots$.     Equivalently, it works out that $V$ on $\cK$ being a lift of $T$ is the same as $\cH \subset \cK$ being an
invariant subspace for $V^*$ and furthermore $V^*|_\cH = T^*$, i.e., with respect to the decomposition $\cK = \cH \oplus (\cK \ominus \cH)$ the operator $V$ has the
block-matrix representation $V = \sbm{ T & 0 \\ * & * }$.
There is a close connection between minimal unitary dilations and minimal isometric lifts, namely:
{\em if $\cU$ is a unitary dilation of $T$  on $\widetilde \cK \supset \cH$ and if we set $\cK = \bigvee_{n\ge 0} \cU^n \cH$ and define $V$ on $\cK$ as
$V = \cU|_\cK$, then $V$ is a minimal isometric lift of $T$}, and  conversely, {\em if $V$ is a minimal isometric lift of $T$, one can always extend $V$ to a unitary operator
$\cU$ on $\widetilde \cK \supset \cK$ so that $\cU$ is a minimal unitary dilation of $V$. }  For our purposes here, it is convenient to work almost exclusively with isometric lifts
rather than unitary dilations.    

To make various constructions to come more canonical, we shall make systematic use of  a more general notion of lift where we do not insist that the space $\cH$ on which $T$
acts is a subspace of the space $\cK$ but rather allow a isometric identification map $\Pi \colon \cH \to \cK$.  Thus we say that, for a given contraction operator
$T$ on $\cH$ the collection of objects $(\Pi, V)$ is a {\em isometric lift} of $T$ if 
\begin{itemize}
\item $\Pi \colon \cH \to \cK$ is an isometric embedding of $\cH$ into $\cK$, and

\item $\operatorname{Ran} \Pi$ is invariant for $V^*$ and furthermore $V^* \Pi = \Pi T^*$.
\end{itemize}

In this chapter we shall discuss three types of functional models (Sch\"affer, Douglas, and Sz.-Nagy--Foias) for a given Hilbert-space contraction operator.

\section{The Sch\"affer functional model for the minimal isometric lift}  \label{S:Schaffer}
For a (coefficient) Hilbert space $\cF$, we shall use the notation $H^2(\cF)$\index{$H^2(\cF)$} for the {\em Hardy space}
$$
H^2(\cF):=\{f:\bD\to\cF: f(z)=\sum_{n=0}^\infty z^nf_n\;\mbox{ and }\; \sum_{n=0}^\infty \|f_n\|_\cF^2<\infty\}.
$$
When $\cF=\mathbb C$, we shall denote $H^2(\cF)$ simply by $H^2$.\index{$H^2$} For a contraction operator $T$ acting on a Hilbert space $\cH$, we shall have use of the \textit{defect operators}\index{defect operator}\index{$D_T$}\index{$D_{T^*}$}
$$
   D_T = (I_\cH- T^*T)^{\frac{1}{2}}, \quad D_{T^*} = (I_\cH - T T^*)^{\frac{1}{2}}
$$
and the \textit{defect spaces}\index{defect space}
$$
 \cD_T = \overline{\operatorname{Ran}}\, D_T, \quad \cD_{T^*} = \overline{\operatorname{Ran}}\, D_{T^*}.
 $$\index{$\cD_T$}\index{$\cD_{T^*}$}
For $\cF$ any coefficient space, we shall use the notation 
\begin{equation}   \label{bev0cF}\index{$\bev_{0, \cF}$}
\bev_{0, \cF} := \bev_0 \otimes I_\cF \colon f \mapsto f(0)
\end{equation}
for the evaluation-at-0 map \index{evaluation at zero} on the vector-valued Hardy space $H^2(\cF)$ with the adjoint
$$
   \bev_{0, \cF}^* \colon \cF \to H^2(\cF)
$$
given by the identification of an element $x \in \cF$ with the constant function $f(z) = x$ considered as an element of $H^2(\cF)$. When $\cF=\mathbb C$, we shall denote $\bev_{0, \cF}$ simply by $\bev_0$.\index{$\bev_0$}

Let us write $\cK_{S}$ for the {\em Sch\"affer isometric lift space}\index{$\cK_S$}
\begin{equation}  \label{KS}
\cK_{S} = \begin{bmatrix} \cH \\ H^2(\cD_T) \end{bmatrix}
\end{equation}
and  $\Pi_S$ for the isometric embedding operator\index{$\Pi_S$}
\begin{equation}  \label{PiS}
\Pi_S = \sbm{ I_\cH \\ 0 } \colon \cH \to \cK_{S}
\end{equation}
and let $V_S$ on $\cK_{S}$ be the operator given by\index{$V_S$}
\begin{equation}   \label{VS}
  V_S \colon \begin{bmatrix} h \\ f \end{bmatrix} \mapsto \begin{bmatrix} T & 0 \\  \bev_{0,\cD_T}^* D_T & M_z^{\cD_T} \end{bmatrix}
   \begin{bmatrix} h \\ f \end{bmatrix} = \begin{bmatrix} T h \\ D_T h + z f(z) \end{bmatrix}.
\end{equation}
Then one can check that $V_S$ is isometric on $\cK_{S}$ and that  $\Pi_S T^* = V_S^* \Pi_S$ (due to the 
block lower-triangular form in the matrix representation of $V_S$), and hence $(\Pi_S, V_S)$ is an isometric lift of $T$.
Let us formally give this a name.

\begin{definition}  \label{D:S-isom-lift}\index{isometric lift of a contraction!Sch\"affer-model}
 If  $\cK_S$, $\Pi_S$, and $V_S$ are given as in \eqref{KS}, \eqref{PiS}, \eqref{VS}, then
$(\Pi_S, V_S)$ is an isometric lift of $T$ on $\cK_S$  and we shall say that $(\Pi_S, V_S)$ is the \textit{Sch\"affer-model isometric lift} of $T$.
\end{definition}

It is easy to check that the Sch\"affer-model isometric lift $(\Pi_S, V_S)$ is minimal, i.e., that
$$
  \bigvee_{n \in {\mathbb Z}_+} V_S^n \begin{bmatrix} \cH \\ 0 \end{bmatrix} = \cK_S:= \begin{bmatrix} \cH \\ H^2(\cD_T) \end{bmatrix}.
$$

Note that the Sch\"affer isometric-lift space $\cK_S = \sbm{ \cH \\ H^2(\cD_T) }$ has first component
$\cH$ equal to the original abstract Hilbert space while the second component $H^2(\cD_T)$ is
a functional Hilbert space,  so strictly speaking the Sch\"affer model is only a semi-functional model.
The original Sch\"affer model as presented in \cite[Section  I.5.1]{Nagy-Foias} has a purely matricial form
as the second component $H^2(\cD_T)$ is written in matricial form as $\ell^2_{{\mathbb Z}^+}(\cD_T)$ rather than in the functional form $H^2(\cD_T)$. 

\section{The Douglas functional model for the minimal isometric dilation}   \label{S:Douglas}
For $T$ a contraction operator on $\cH$, we shall have great use of the non-negative definite operator $Q_{T^*}$ given by the strong limit
\begin{align}\label{Q}
Q_{T^*}^2=\operatorname{SOT-}\lim T^nT^{*n}.
\end{align}A fundamental map for the construction of the Douglas model is the map \index{$\cO_{D_{T^*}, T^*}$}
$$
  \cO_{D_{T^*}, T^*}  \colon \cH \to H^2(\cD_{T^*})
$$
given by  
\footnote{
The notation $\cO_{D_{T^*}, T^*}$ is suggested by the fact that the operator $\cO_{D_{T^*}, T^*}$ can be viewed as the 
{\em frequence-domain observability operator} for the discrete-time state/output linear system 
$$  \left\{ \begin{array}{ccl}  x(t+1) & = & T^* x(t) \\
  y(t) & = & D_{T^*} x(t) \end{array} \right., \,  t=0,1,2,\dots
$$
($x(t)$ equal to the {\em state} at time $t$, $y(t)$ equal to the output at time $t$)
since running the system with initial condition $x(0)$ to produce an output string $\{y(t) \}_{n \in {\mathbb Z}_+}$
results in the $Z$-transform $\widehat y(z) := \sum_{n=0}^\infty y(n) \, z^n$ of the output string 
$\{y(n)\}_{n \in {\mathbb Z}^+}$ being given by
$$
    \widehat y(z) = \cO_{D_{T^*}, T^*} x_0.
$$}
\begin{equation}   \label{obsop}
\cO_{D_{T^*}, T^*} \colon h \mapsto \sum_{n=0}^\infty D_{T^*} T^{*n} h \, z^n 
  = D_{T^*} (I - z T^*)^{-1} h.
\end{equation}
The easy computation
\begin{align*}
\| \cO_{D_{T^*}, T^*} x \|^2_{H^2(\cD_{T^*})} & = \sum_{n=0}^\infty \| D_{T^*} T^{*n} x \|^2
= \sum_{n=0}^\infty \langle T^n (I - T T^*) T^{*n} x, x \rangle  \\
& = \| x \|^2 - \| Q_{T^*} x \|^2
\end{align*}
shows that $\cO_{D_{T^*}, T^*}$ is contractive as an operator from $\cH$ into the Hardy space $H^2(\cD_{T^*})$, 
and is an isometry exactly when $Q_{T^*}= 0$.

We now note that $T Q_{T^*}^2 T^* = Q_{T^*}^2$; hence the
formula
\begin{equation}  \label{defX}
X^*Q_{T^*}h = Q_{T^*} T^* h
\end{equation}
defines an isometry $X^*$ on $\operatorname{Ran} Q_{T^*}$ which extends by continuity to an isometry (still denoted as
$X^*$) on $\overline{\operatorname{Ran}} \, Q_{T^*}$.  If not already unitary, this operator  has a
minimal unitary extension on a space $\cQ_{T^*} \supset \overline{\operatorname{Ran}} \, Q_{T^*}$ which we 
denote by $W_D^*$\index{$W_D$}. 
A dense subspace of $\cQ_{T^*} $\index{$\cQ_{T^*}$} is 
\begin{equation}  \label{cQT*}
\bigcup_{n=0}^\infty W_D^{n} \operatorname{Ran} Q_{T^*} \text{ dense in } \cQ_{T^*}
\end{equation}
 and then the extension $W_D^*$ is given densely by
\begin{align} 
 & W_D^* W_D^n Q_{T^*} h = W_D^{n-1} X^*Q_{T^*} h = W_D^{n-1} Q_T^* T^*h \text{ for } n \ge 1,   \notag  \\
 & W_D^* Q_{T^*} h = X^* Q_{T^*} h =  Q_{T^*} T^* h.   \label{IntofQ}
\end{align}

Let us introduce the Hilbert space \index{$\cK_D$} 
\begin{equation}  \label{KD}
\cK_D := \begin{bmatrix} H^2(\cD_{T^*}) \\ \cQ_{T^*} \end{bmatrix}
\end{equation}
and define an  isometric operator
$V_D$ on $\cK_{D}$ by
\begin{align}\label{VD}\index{$V_D$}
  V_D = \begin{bmatrix} M_z & 0 \\ 0 & W_D \end{bmatrix}.
\end{align}
There is a canonical isometric embedding operator $\Pi_D \colon \cH \to \cK_D$ given by\index{$\Pi_D$}
\begin{align}\label{PiD}
  \Pi_D \colon h \mapsto \begin{bmatrix} \cO_{D_{T^*}, T^*} \\ Q_{T^*} \end{bmatrix}h.
\end{align}
Furthermore, we have the intertwining relation
\begin{equation}   \label{Dintertwine}
    \Pi_D T^* = (V_D)^*  \, \Pi_D.
\end{equation}
Let us give all this a formal name.
\begin{definition} \label{D:D-isom-lift}\index{isometric lift of a contraction!Douglas-model}
Define the space $\cK_D$, the operator $\Pi_D$ and the operator $V_D$ as in \eqref{KD}, \eqref{PiD}, and \eqref{VD}.
Then $(\Pi_D, V_D)$ is a lift of $T$ on the space $\cK_D$ which we shall refer to as the {\em Douglas-model isometric lift} of $T$.
\end{definition}

One can also see as a consequence of Lemma 1 in Douglas's paper
\cite{Doug-Dilation} that this Douglas-model lift is a minimal isometric lift.   We include here a simple direct proof making use of the formulation which
we are using here.  

\begin{proposition} \label{P:Douglas-min}  The Douglas-model isometric lift $(\Pi_D, V_D)$ (\eqref{VD} and \eqref{PiD}) of a contraction operator $T$ is 
minimal.
\end{proposition}

\begin{proof}  Verification of the minimality of $(\Pi_D, V_D)$ amounts to showing that
\begin{equation}  \label{Douglas-min}
\cK_{\rm min}: = \bigvee_{n = 0,1,2,\dots} \begin{bmatrix} M_z^n  & 0 \\ 0 & W_D^n \end{bmatrix} \begin{bmatrix} \cO_{D_{T^*}, T^*} \\ Q_{T^*} \end{bmatrix} \cH = 
\begin{bmatrix} H^2(\cD_{T^*}) \\ \cQ_{T^*} \end{bmatrix}.
\end{equation}
To see this, note first that for each $n \in {\mathbb Z}_+$ we have in particular that 
$$
\begin{bmatrix}  M_z^n & 0 \\ 0 & W^n_D \end{bmatrix} \begin{bmatrix} \cO_{D_{T^*}, T^*} \\ Q_{T^*} \end{bmatrix} T^{*n} h \in \cK_{\rm min}
$$
where
\begin{align}
& \begin{bmatrix}  M_z^n & 0 \\ 0 & W^n_D \end{bmatrix} \begin{bmatrix} \cO_{D_{T^*}, T^*} \\ Q_{T^*} \end{bmatrix} T^{*n} h  
=  \begin{bmatrix}  M_z^n & 0 \\ 0 & W^n_D \end{bmatrix} \begin{bmatrix} M_z^{*n} & 0 \\ 0 & W_D^{*n} \end{bmatrix}
 \begin{bmatrix} \cO_{D_{T^*}, T^*} \\ Q_{T^*} \end{bmatrix}  h   \notag  \\
 & \quad = \begin{bmatrix} M_z^n M^{*n}_z & 0 \\ 0 & W_D^n W_D^{*n} \end{bmatrix} \begin{bmatrix} \cO_{D_{T^*}, T^*} \\ Q_{T^*} \end{bmatrix} h
 =  \begin{bmatrix} M_z^n M_z^{*n} \cO_{D_{T^*}, T^*} h \\  Q_{T^*} h \end{bmatrix}
 \label{compute1}
 \end{align}
 where 
 $$
\lim_{n \to \infty} \|   M_z^n M_z^{*n} \cO_{D_{T^*}, T^*} h \| = \lim_{n \to \infty}  \| M_z^{*n} \cO_{D_{T^*}, T^*} h \|  = 0
$$
since $M_z$ is a shift operator.  Combining this observation with \eqref{compute1} leads us to
$$
\lim_{n \to \infty}  \begin{bmatrix}  M_z^n & 0 \\ 0 & W^n_D \end{bmatrix} \begin{bmatrix} \cO_{D_{T^*}, T^*} \\ Q_{T^*} \end{bmatrix} T^{*n} h  
= \begin{bmatrix} 0 \\ Q_{T^*} h \end{bmatrix}   \in \cK_{\rm min}.
$$
Subtracting this element off from $\sbm{ \cO_{D_{T^*}, T^*} \\ Q_{T^*} } h \in \cK_{\rm min}$, we conclude that 
$$
\sbm{ \cO_{D_{T^*}, T^*} h \\ 0} h \in \cK_{\rm min}.
$$
But then we also have, for all $h \in \cH$,
$$
\begin{bmatrix}  (\cO_{D_{T^*}, T^*} - M_z \cO_{D_{T^*}, T^*} T^*) h \\ 0 \end{bmatrix} \in \cK_{\rm min} 
$$
where 
\begin{align}
& \big( ( \cO_{D_{T^*}, T^*} - M_z \cO_{D_{T^*}, T^*} T^*) h \big)(z)  = \big( D_{T^*} (I - z T^*)^{-1} - z D_{T^*} (I - z T^*)^{-1} T^* \big) h  \notag \\
& \quad = \bev_{0, \cD_{T^*}}^* D_{T^*} h.
\label{cool-id}
\end{align}
Thus $\sbm{ \bev_{0, \cD_{T^*}}^* D_{T^*} h \\ 0 } \in \cK_{\rm min}$ and hence also
\begin{align*}
& \bigvee_{n=0,1,2,\dots} \begin{bmatrix} M_z^n & 0 \\ 0 &  W_D^n \end{bmatrix} \begin{bmatrix} \bev_{0, \cD_{T^*}}^*  D_{T^*} \cH \\ 0 \end{bmatrix}
= \begin{bmatrix} \bigvee_{n=0,1,2,\dots} M_z^n \bev_{0, \cD_{T^*}}^* \cD_{T^*} \\ 0 \end{bmatrix} \\
& \quad  = \begin{bmatrix} H^2(\cD_{T^*}) \\ 0 \end{bmatrix}
\subset \cK_{\rm min}.
\end{align*}

Finally let us recall that $\cup_{n \ge 0} W_D^n Q_{T^*} \cH$ is dense in $\cQ_{T^*}$ and hence
$$
\bigvee_{n \ge 0} \begin{bmatrix} M_z^n & 0 \\ 0 & W_D^n \end{bmatrix} \begin{bmatrix} 0 \\ Q_{T^*} \cH \end{bmatrix}
= \begin{bmatrix} 0 \\ \cQ_{T^*} \end{bmatrix} \subset \cK_{\rm min}.
$$
Putting all the pieces together then gives us
$$
\cK_D: = \begin{bmatrix} H^2(\cD_{T^*}) \\ \cQ_{T^*} \end{bmatrix} \subset \cK_{\rm min} \subset \cK_D
$$
and \eqref{Douglas-min} follows as wanted.
\end{proof}

\begin{remark}   \label{R:Douglas-func}
As the Douglas isometric-lift space $\cK_D = \sbm{H^2(\cD_{T^*}) \\ \cQ_{T^*}}$ has first component
$H^2(\cD_{T^*})$ equal to a functional Hilbert space but second component $\cQ_{T^*}$ equal to an abstract 
(non-functional) Hilbert space, so strictly speaking one can think of  the Douglas model is really being only a  semi-functional model.
However the space $\cQ_{T^*}$ comes equipped with a unitary operator $W_D$.  By the direct-integral version of the spectral theorem for 
unitary (more generally normal) operators (see \cite[Theorem I.6.1]{Dix}, one can convert this space to a direct-integral $L^2$-space
$\cQ_{T^*} \cong \bigoplus \int_{\mathbb T} \cH_\zeta {\tt d \nu}(t)$ for a scalar spectral measure $\nu$
and multiplicity function $m(\zeta) = \dim \cH_\zeta$, with $W_D$ then given as the multiplication operator
$W_D = M_\zeta \colon f(\zeta) \mapsto \zeta \cdot f(\zeta)$ (here $\zeta \mapsto f(\zeta) \in \cH_\zeta$ is a measurable
square-integrable cross-section of $\{\cH_\zeta\}_{\zeta \in {\mathbb T}}$).  This becomes more precise when we make connections with the Sz.-Nagy--Foias model in the next section.
\end{remark}

\begin{remark} \label{R:Durszt}
While the Douglas model is mostly explicitly constructed in terms of the operator $T$, one could ask for a more explicit construction of the residual part $\cQ_{T^*}$ where the unitary part $W_D$
of the minimal isometric lift $V_D = M_z \oplus W_D$ is defined.
There is a later construction due to Durszt \cite{Durszt}, which, while arguably more difficult to work with, 
gives such an explicit construction,
with first component exactly the same as in the Douglas construction (after interchanging the contraction
operator $T$ with its adjoint $T^*$), while the second component is a bilateral unitary shift operator
with a more complicated but completely explicit formula for the map from $\cH$ into this 
$\ell^2_{{\mathbb Z}}$-space.
\end{remark}

\section[The Sz.-Nagy--Foias model]{The Sz.-Nagy--Foias functional model for the minimal unitary dilation} \label{S:NFmodel}

As in the approach of Sz.-Nagy and Foias towards a functional model for unitary dilations, we assume that we are given a completely non-unitary (c.n.u.) contraction operator $T$ together with its minimal unitary dilation $\cU$ on $\widetilde \cK$ and minimal isometric lift $V$  on $\cK$ with $\cH \subset \cK \subset \widetilde \cK$ and $V = \cU|_\cK$.  Following \cite{Nagy-Foias}[Sections II.1 and II.2], one can see that the subspaces
$$
 \cL = \overline{( \cU - T) \cH}, \quad \cL_* = \overline{(I - \cU T^*) \cH}
$$
contained in $\cK \subset \widetilde \cK$ are wandering subspaces for $\cU$ in the sense that
$$
 \cU^n \cL \perp  \cU^m \cL, \quad \cU^n \cL_* \perp  \cU^m \cL_* \text{ for } m \ne n \text{ in } {\mathbb Z}
$$
and hence it makes sense to define subspace $M_{\pm}(\cL)$ and $M_{\pm}(\cL_*)$ as well as $M(\cL)$ and $M(\cL_*)$ via the internal direct sums in $\widetilde \cK$:
\begin{align*}
&  M_+(\cL) := \bigoplus_{n \ge 0}  \cU^n \cL, \quad M_-(\cL) := \bigoplus_{n < 0} \cU^n \cL_, \\
&  M_+(\cL_*) := \bigoplus_{n \ge 0}  \cU^n \cL_*, \quad  M_-(\cL_*) := \bigoplus_{n < 0}  \cU^n \cL_*,   \\
&  M(\cL): = M_-(\cL) \oplus M_+(\cL) = \bigoplus_{n \in {\mathbb Z}} 
   \cU^n \cL, \\
& M(\cL_*): = M_-(\cL) \oplus M_+(\cL)  = \bigoplus_{n \in {\mathbb Z}}   \cU^n \cL_*.
\end{align*}
Note that $M_+(\cL)$ and $M_+(\cL_*)$ are invariant for $ \cU$ while $M_-(\cL)$, $M_-(\cL_*)$ are invariant for $ \cU^*$ and
$M(\cL)$ and $M(\cL_*)$ are reducing for $ \cU$.
If we set $\cR = \widetilde \cK \ominus M(\cL_*)$,  then $\cR$ is reducing for $ \cU$ and we have the two orthogonal decompositions of the space $\widetilde \cK$:
$$
 \widetilde \cK = M(\cL_*) \oplus \cR = M_-(\cL_*) \oplus \cH \oplus M_+(\cL),
$$
and the space $\cK$ on which the minimal isometric lift $V$ of $T$ acts has the two orthogonal decompositions
\begin{equation}  \label{cKdecoms}
  \cK = M_+(\cL_*) \oplus \cR = \cH \oplus M_+(\cL).
\end{equation}
In particular we see that $\cH \subset \cK$; from the fact that $\cK$ is invariant for $ \cU$ and by definition $V =  \cU|_\cK$, we see that the wandering subspaces
$\cL$ and $\cL_*$ can equally well be defined as
$$
 \cL = \overline{ (V-T) \cH}, \quad \cL_* = \overline{(I - V T^*) \cH}.
$$
In summary, 
given any minimal isometric lift of the form $(i_{\cH \to \cK}, V)$ (i.e., having $\cH$ equal to a subspace of the dilation space $\cK$ on which $V$ is acting),  the space
$\cK$ then has the two-fold orthogonal decomposition \eqref{cKdecoms};  let us refer to this structure for $\cK$ as the {\em coordinate-free Sz.-Nagy--Foias model}
for a minimal isometric lift $V$ of $T$.

We next use this coordinate-free model to arrive at the Sz.-Nagy--Foias functional model for a minimal isometric lift of $T$ as follows.
It is a straightforward  computation to show that the maps
\begin{equation}  \label{iota/iota*}
 \iota \colon (V-T) h \mapsto D_T h, \quad \iota_* \colon (I - VT^*) h \mapsto D_{T^*} h
\end{equation}
extend to unitary identification maps
\begin{align}\label{ioeta/iota*}
 \iota \colon \cL \to \cD_T, \quad \iota_* \colon  \cL_* \mapsto \cD_{T^*}.
\end{align}
From the two decompositions for $\cK$ in \eqref{cKdecoms} we also see that $M_+(\cL) \perp M_-(\cL_*)$ and hence
\begin{equation}  \label{analyticity}
  P_{M(\cL_*)}  M_+(\cL) \subset M_+(\cL_*)
\end{equation}
It can also be shown that if $T$ is completely nonunitary (as we are assuming), then we recover the so-called {\em residual space} $\cR$ from the wandering subspace $\cL$ via 
the formula
\begin{equation}  \label{cLcR}
  \cR = \overline{( I - P_{M(\cL_*)}) M(\cL)}.
\end{equation}

Define the operator $\bTheta \colon M(\cL) \to M(\cL_*)$ as the restricted projection
$$
    \bTheta = P_{M(\cL_*)}|_{M(\cL)}.
$$
Let $\biota_* \colon M(\cL_*) \to L^2(\cD_{T^*})$ and $\biota \colon M(\cL) \to L^2(\cD_T)$ be the  extensions of the operators $\iota_*$ and $\iota$ \eqref{ioeta/iota*} to Fourier representation operators
\begin{equation}   \label{biota-biota*-2sided}
\biota_* \colon \sum_{n=-\infty}^\infty \cU^n \ell_{*n} \mapsto \sum_{n=-\infty}^\infty (\iota_* \ell_{*n}) \zeta^n, \quad
\biota \colon \sum_{n=-\infty}^\infty \cU^n \ell_{n} \mapsto \sum_{n=-\infty}^\infty (\iota \ell_n) \zeta^n
\end{equation}
where $\zeta$ is the independent variable on the unit circle ${\mathbb T}$.
 Then it is easily checked that
 $$
   \bTheta (\cU|_{M(\cL)}) = (\cU|_{M(\cL_*)}) \bTheta.
 $$
 Let $\widehat \bTheta = \biota_* \bTheta \biota^* \colon L^2(\cD_T) \to L^2(\cD_{T^*})$.  Then the previous intertwining relation becomes
 the function-space intertwining
 $$
    \widehat \bTheta M^{\cD_T}_{\zeta} = M^{\cD_{T^*}}_{\zeta} \widehat \bTheta.
 $$
 By a standard result (see e.g.\ \cite[Lemma V.3.1]{Nagy-Foias}), it follows that $\widehat \bTheta$  is a multiplication operator
 $$
    \widehat \bTheta \colon h(\zeta) \mapsto \Theta(\zeta) \cdot h(\zeta)
 $$
 for a measurable $\cB(\cD_T, \cD_{T^*})$-valued function $\zeta \mapsto \Theta(\zeta)$ on the unit circle ${\mathbb T}$.  
 Here, for two Hilbert spaces $\cE$ and $\cF$, the notation $\cB(\cE,\cF)$ stands for the space of all bounded linear operators from $\cE$ into $\cF$; when $\cE=\cF$ we simply use $\cB(\cE)$.\index{$\cB(\cE,\cF)$} \index{$\cB(\cE)$}  As $\bTheta$ is a restricted projection,
 it follows that $\| \bTheta \| \le 1$, and also $\| M_\Theta \| \le 1$ as an operator from $L^2(\cD_T)$ to $L^2(\cD_{T^*})$, from
 which it follows that $\| \Theta(\zeta) \| \le 1$ for almost all $\zeta$ in the unit circle.  Furthermore, by applying the Fourier transform to the subspace inclusion \eqref{analyticity},
 we see that $M_\Theta$ maps $H^2(\cD_T)$ into $H^2(\cD_{T^*})$; thus in fact $\Theta$ is a contractive
 $H^\infty$-function with values in $\cB(\cD_T, \cD_{T^*})$, known as the {\em Sz.-Nagy--Foias characteristic function}\index{characteristic function} of $T$, 
 given by the explicit formula
 \begin{align}\label{char-func}\index{characteristic function}\index{$\Theta_T$}
 \Theta(z)=\Theta_T(z)=-T+zD_{T^*}(I_{\cH}-zT^*)^{-1}D_{T}|_{\cD_T} \colon \cD_T \to \cD_{T^*}
 \end{align}
 (see \cite[Proposition VI.2.2]{Nagy-Foias}).

 Suppose next that $k \in \cR$ has the form $k = P_\cR \bell$ for some $\bell  = \bigoplus_{n \in {\mathbb Z}_+}  \cU^n \ell_n \in M(\cL)$.  Then
 $$
 \| k \|^2  = \| P_\cR \bell \|^2 = \| \bell \|^2 - \| \bTheta \bell \|^2 =
 \| \biota \bell \|^2 - \| \Theta_T \cdot  \biota \bell \|^2 = \| \Delta_{\Theta_T} \cdot \biota \bell \|^2
 $$
 where we let $\Delta_{\Theta_T}$ be the $\cD_T$-valued operator function on the unit circle ${\mathbb T}$ given by \index{$\Delta_{\Theta_T}$}
 $$
   \Delta_{\Theta_T}(\zeta) := (I - \Theta_T(\zeta)^* \Theta_T(\zeta))^{1/2}
 $$
 (the pointwise defect operator of $\Theta_T$:  $\Delta_{\Theta_T}(\zeta) = D_{\Theta_T(\zeta)}$).
 As we are assuming that $T$ is c.n.u.,  by formula \eqref{cLcR} we get that
 the space $(I - P_{M(\cL_*)}) M(\cL) = P_\cR M(\cL)$ is dense in $\cR$.
Hence we can define a unitary map $\omega_{\rm NF}$ from $\cR$ to $\overline{ \Delta_{\Theta_T}L^2(\cD_T) }$
densely defined on $P_\cR M(\cL)$ by
\begin{equation}   \label{omegaNF}\index{$ \omega_{\rm NF}$}
 \omega_{\rm NF} \colon P_\cR \bell \mapsto \Delta_{\Theta_T} \cdot \biota \bell.
\end{equation}
From this formula, we can read off the validity of the intertwining relation
\begin{equation} \label{omegaNF'}
  \omega_{\rm NF} (V|_\cR) = \big(  M^{\cD_T}_\zeta |_{\overline{ \Delta_{\Theta_T} L^2(\cD_T)}} \big) \, \omega_{\rm NF}.
\end{equation}
Let us  introduce the Sz.-Nagy--Foias functional-model lift space (which depends only on the characteristic function $\Theta_T$)
$\cK_{\Theta_T}$ \index{$\cK_{\Theta_T}$}by
\begin{equation}  \label{KNF}\index{$\cK_{\Theta_T}$}
\cK_{\Theta_T} := \begin{bmatrix} H^2(\cD_{T^*}) \\ \overline{ \Delta_{\Theta_T}L^2(\cD_T) }\end{bmatrix}.
\end{equation}   
We next
 define a unitary identification map $U_{\rm NF}$\index{$U_{\rm NF}$} from $\cK = M_+(\cL_*) \oplus \cR$ to  $\cK_{\Theta_T}$ by
\begin{align}\label{Unf}
  U_{\rm NF} k = \begin{bmatrix} \biota_* P_{M_+(\cL_*)} k \\ \omega_{\rm NF} P_\cR k \end{bmatrix}.
\end{align}
Since $\biota_*$ is unitary from $M_+(\cL_*)$ to $H^2(\cD_{T^*})$, $\omega_{\rm NF}$ is unitary from $\cR$ to
$\overline{ \Delta_{\Theta_T}L^2(\cD_T) }$, and $\cK$ has the internal orthogonal decomposition
$\cK = M_+(\cL_*) \oplus \cR$, we see that $U_{\rm NF}$ so defined is unitary from $\cK$ onto
$\cK_{\Theta_T}$.  Observing the intertwining relation $M^{\cD_{T^*}}_z \biota_*|_{M_+((\cL_*)} = \biota V|_{M_+(\cL_*)}$
and recalling \eqref{intertwine1'}, we arrive at the intertwining relation
\begin{equation}   \label{NFintwin}
  U_{\rm NF} V = V_{\rm NF} U_{\rm NF}
\end{equation}
where we set $V_{\rm NF}$ equal to the isometric operator on $\cK_{\Theta_T}$ given by\index{$V_{\rm NF}$}
\begin{align}\label{Vnf}
  V_{\rm NF} = \begin{bmatrix} M^{\cD_T}_z & 0 \\ 0 & M_{\zeta}|_{\overline{ \Delta_{\Theta_T} L^2(\cD_T)}} \end{bmatrix}.
\end{align}
Define the isometric embedding $\Pi_{\rm NF}$ of $\cH$ into $\cK_{\Theta_T} $ as\index{$\Pi_{\rm NF}$}
\begin{align}\label{Unf&Pinf}
  \Pi_{\rm NF} = U_{\rm NF} |_\cH.
\end{align}
Then we make the formal definition:

\begin{definition}  \label{D:NFfuncmod} Given a c.n.u.~contraction operator $T$ and any choice of minimal isometric lift of the form $(i_{\cH \to \cK}, V)$, we refer to
$( \Pi_{\rm NF}, V_{\rm NF})$ given by \eqref{Unf&Pinf}  and \eqref{Vnf}  as the {\em Sz.-Nagy--Foias functional model} for a minimal isometric lift $T$.
\end{definition}

\begin{remark} \label{R:DougSchaf-vs-NF}
The reader will note that for the case of the Sch\"affer and Douglas models discussed in Sections 2.1 and 2.2 respectively,
 we were able to define the embedding maps $\Pi_S \colon \cH \to \cK_S$ and $\Pi_{D} \colon \cH \to \cK_D$ explicitly in terms of $T$ whereas in Sz.-Nagy--Foias
 case the embedding map $\Pi_{\rm NF}$ is defined more implicitly by first introducing the map $U_{\rm NF}$ identifying the Sz.-Nagy--Foias coordinate-free
 space $\cK$ \eqref{cKdecoms} and then getting $\Pi_{\rm NF}$ as the restriction of $U_{\rm NF}$ to $\cH$.  To repair this lack of symmetry, we shall now do two tasks:
 \begin{itemize}
 \item[(i)] By Theorem I.4.1 in \cite{Nagy-Foias}, any two minimal isometric lift of a contraction are unitarily equivalent. Find explicitly the unitary operators $U_S \colon \cK \to \cK_S$ and $U_D \colon \cK \to \cK_D$ establishing the unitary equivalence of the coordinate-free Sz.-Nagy--Foias  lift $(i_{\cH \to \cK}, V)$ acting on $\cK$ to the Sch\"affer-model lift $(\Pi_S, V_S)$ acting on $\cK_S$ and  the Douglas-model lift $(\Pi_D, V_D)$ acting on $\cK_D$, respectively. 
  \item[(ii)]  Find a more explicit form for the Sz.-Nagy--Foias embedding operator $\Pi_{\rm NF} \colon \cH \to \cK_{\Theta_T}$.
 \end{itemize}
 
 \smallskip
 
 \noindent
 (i)   For the case of the Sch\"affer model, we use the second of the decompositions \eqref{cKdecoms} to see that
 \begin{equation}   \label{US}
   U_S = \begin{bmatrix}   I_\cH&0\\0& \biota P_{M_+(\cL)} \end{bmatrix} \colon \cK =  \begin{bmatrix} \cH \\ M_+(\cL)\end{bmatrix} \to  \cK_S = \begin{bmatrix} \cH \\ H^2(\cD_T) \end{bmatrix}
 \end{equation}
 does the job, i.e.,
 $$
  U_S V = \begin{bmatrix} T & 0 \\ \bev_{0,\cD_T}^*D_T & M_z^{\cD_T} \end{bmatrix} U_S = V_S U_S, \quad \Pi_S := \begin{bmatrix} I_\cH \\ 0 \end{bmatrix} = U_S|_\cH.
 $$
 Thus $U_S \colon \cK \to \cK_S$
 establishes a unitary equivalence between the coordinate-free lift $(i_{\cH \to \cK}, V)$ and the Sch\"affer-model lift $(\Pi_S, V_S)$.
 
 For the case of the Douglas lift, we make use of the first decomposition of $\cK$ in \eqref{cKdecoms}:  $\cK = M_+(\cL_*) \oplus \cR$.  Toward this end we need to make
 use of various connections between the space $\cR$ and the operator $Q_{T^*}$ appearing in the Douglas model.  We first note the connection
 that, for $h \in \cH$ we have (see \cite[Proposition II.3.1]{Nagy-Foias})
 $$
   \| P_\cR h \|^2 = \langle Q_{T^*}^2 h, h \rangle.
 $$
 Hence the map  $\omega_D \colon \cR \to \cQ_{T^*}$ defined densely by
 $$
   \omega_D \colon P_\cR h \mapsto Q_{T^*} h \in \cQ_{T^*}
 $$
 is isometric.   Furthermore, it is known that, in case $V$ is a minimal isometric lift of $T$,  we have that
 $$
    \cD_R:= \bigcap_{n=0}^\infty V^n P_\cR \cH \subset \cR
 $$
 is dense in $\cR$ (see \cite[Proposition II.3.1]{Nagy-Foias}).
 It is a routine matter to extend the map $\omega_D$ to the space $\cD_R$ via the formula
 \begin{equation}   \label{omegaD}
   \omega_D \colon V^n (P_\cR h) \mapsto W_D^n Q_{T^*} h
\end{equation}
with the result that $\omega_D$ is still an isometry.  We can then extend by continuity to a well-defined unitary identification map $\omega_D$ from $\cR$ onto $\cQ_{T^*}$.
We may then make use of the first decomposition $\cK = M_+(\cL_*) \oplus \cR$ in \eqref{cKdecoms} to define the unitary identification $U_D \colon \cK \to \cK_D =
\sbm{ H^2(\cD_{T^*}) \\ \cQ_{T^*} }$ via the formula
\begin{equation}   \label{UD}
 U_D \colon k \mapsto \begin{bmatrix} \iota_* P_{M_+(\cL_*)}k \\ \omega_D P_\cR k \end{bmatrix} \colon \cK \to \cK_D:= \begin{bmatrix} H^2(\cD_{T^*}) \\ \cQ_{T^*} \end{bmatrix}.
 \end{equation}
 It is now a matter of checking that
 $$
 U_D V  = V_D U_D, \quad \Pi_D = U_D|_\cH.
 $$
 We verify here only the first component of the second identity as follows:
 \begin{align*}
 \biota_* P_{M_+(\cL_*)}h & - \sum_{n=0}^\infty z^n \iota_* (I - V V^*) V^{*n} h \text{ (by shift analysis)} \\
  & = \sum_{n=0}^\infty z^n \iota_* (I - V T^*) T^{*n} h \text{ (since $T^* = V^*|_\cH$)} \\
  & = \sum_{n=0}^\infty z^n D_{T^*} T^{*n} h \text{ (by definition of $\iota_*$)} \\
  & = D_{T^*} (I - z T^*)^{-1} h = \cO_{D_{T^*}, T^*} h.
 \end{align*}
 We conclude that $U_D$ implements the unitary equivalence of $(i_{\cH \to \cK}, V_{\rm NF})$ and $( \Pi_D, V_D)$.
 
 \smallskip
 
 \noindent
 (ii)   By definition $\Pi_{\rm NF} = U_{\rm NF} |_\cH$, so for $h \in \cH$ we have
 \begin{align}\label{PiNF}
   \Pi_{\rm NF} \colon h \mapsto \begin{bmatrix} \iota_* P_{M_+(\cL_*)} h \\ \omega_{\rm NF} P_\cR h \end{bmatrix}.
 \end{align}
 We have seen from the final computation done in item (i) above that, for $h \in \cH$, we have
 $$
 \biota_* P_{M_+(\cL_*)} h = \cO_{D_{T^*}, T^*} h.
 $$
 Thus the first component of the Sz.-Nagy--Foias functional model embedding operator $\Pi_{\rm NF}$ agrees with the first component of the Douglas functional-model
 embedding operator $\Pi_D$.  For $h \in \cH$ let us next compute the second component of the Sz.-Nagy-Foias embedding operator $\Pi_{\rm NF}$ applied to $h$:
 \begin{align*}
 \omega_{\rm NF} P_\cR h  & = \omega_{\rm NF} P_\cR h  = \omega_{\rm NF} (\omega_D)^* \cdot \omega_D P_\cR h \\
   & = \omega_{\rm NF,D} Q_{T^*} h
 \end{align*}
 where we set
 \begin{equation}   \label{omegaNFD}
  \omega_{\rm NF,D} : = \omega_{\rm NF} \, (\omega_D)^* 
 \end{equation}
 is a unitary identification map from the second component of the Douglas model space $\cQ_{T^*}$ onto the second component of the the Sz.-Nagy--Foias
 model space $\overline{ \Delta_{\Theta_T} L^2( \cD_T)}$.  One can argue that the map $\omega_{\rm NF, D}$ is again not particularly explicit, but this will be a convenient
 place to hide the lack of explicitness for our purposes here.
\end{remark}

While it is problematical to identify the space $U_{\rm NF} \cH \subset \sbm{ H^2(\cD_{T^*}) \\ \overline{ \Delta_T L^2(\cD_T)} }$ explicitly as detailed in part (ii) of
Remark \ref{R:DougSchaf-vs-NF},    its orthogonal complement
in $\cK_{\Theta_T}$, namely the space $U_{\rm NF} M_+(\cL)$,  can be identified explicitly as follows.
For $\bell \in M_+(\cL)$,
\begin{align*}
 U_{\rm NF} \bell & =  \begin{bmatrix} \biota_* P_{M_+(\cL_*)} \bell \\ \omega_{\rm NF} P_\cF \bell \end{bmatrix} \\
  & = \begin{bmatrix}  \biota_* P_{M_+(\cL_*)} \biota^* \biota \bell \\ \Delta_{\Theta_T} \biota \bell \end{bmatrix} \\
  & = \begin{bmatrix} \Theta_T \\ \Delta_{\Theta_T} \end{bmatrix} \cdot \biota \bell
 \end{align*}
and hence the space $\cH_{\Theta_T} :=\Pi_{\rm NF}\cH = U_{\rm NF} \cH$ (which turns out to depend only on the characteristic function $\Theta_T$) is given by
\begin{align}\label{HNF}\index{$  \cH_{\Theta_T}$}
  \cH_{\Theta_T} = \begin{bmatrix} H^2(\cD_{T^*}) \\ \overline{ \Delta_{\Theta_T}L^2(\cD_T)} \end{bmatrix}
  \ominus \begin{bmatrix} \Theta_T \\ \Delta_{\Theta_T} \end{bmatrix} \cdot H^2(\cD_T).
\end{align}
Note that the subspace $U_{\rm NF} M_+(\cL) = \sbm{ \Theta_T \\ \Delta_{\Theta_T} } \cdot H^2(\cD_T)$ is invariant for
$V_{\rm NF}$ and hence $\cH_{\rm \Theta_T}$ is invariant for $V_{\rm NF}^*$.
Rewrite \eqref{NFintwin} in the form
$$
  U_{\rm NF} V^* = V_{\rm NF}^* U_{\rm NF}
$$
and restrict this identity to $\cH$ to arrive at
$$
   \Pi_{\rm NF} T^* = V_{\rm NF}^* \Pi_{\rm NF}.
$$
This suggests the following formal definition.

\index{isometric lift of a contraction!Sz.-Nagy--Foias functional-model}
\begin{definition}  \label{D:NF-isom-lift} Let the space $\cK_{\Theta_T}$, the operators $\Pi_{\rm NF} \colon \cH \to \cK_{\Theta_T}$ and $V_{\rm NF}$ on $\cK_{\Theta_T}$ 
be given as in \eqref{KNF}, \eqref{Unf&Pinf}, and \eqref{Vnf}.  Then $(\Pi_{\rm NF}, V_{\rm NF})$ is a minimal isometric lift of $T$ on $\cK_{\Theta_T}$ which we shall refer to as the
{\em Sz.-Nagy--Foias functional-model isometric lift of $T$}.
\end{definition}

Let us next note that the operator $U_{\rm NF} \colon \cK \to \cK_{\Theta_T}$ implements a unitary equivalence of the Sz.-Nagy--Foias isometric lift
$(\Pi_{\rm NF}, V_{\rm NF})$ with the coordinate-free minimal isometric lift $( \iota_{\cH \to \cK}, V)$, and hence
the Sz.-Nagy--Foias functional-model isometric lift $(\Pi_{\rm NF}, V_{\rm NF})$ is also minimal.

We note that both components of the Sz.-Nagy--Foias isometric-lift space $\cK_{\Theta_T} = \sbm{ H^2(\cD_{T^*}) \\
\overline{\Delta_{\Theta_T} L^2(\cD_T)} }$ are functional Hilbert spaces (albeit with first component $H^2(\cD_{T^*})$ consisting
of holomorphic functions while with second component consisting of $L^2$-measurable functions).   
For this reason
it makes sense to say that $\cK_{\Theta_T}$ is a {\em functional model space} and that $V_{\rm NF}$
is the {\em Sz.-Nagy--Foias functional-model isometric lift} of $T$.

\begin{remark}  \label{R:NFmodel}
The isometric embedding operator $\Pi_{\rm NF}$ for the Sz.-Nagy--Foias model (the analogue of the 
operator $\Pi_S$  \eqref{PiS} for the Sch\"affer model and of the operator $\Pi_D$ \eqref{PiD} for the
Douglas model) is the embedding operator
$$
  \Pi_{\rm NF} \colon \cH \to \cK_{\Theta_T} = \begin{bmatrix} H^2(\cD_{T^*}) \\
  \overline{ \Delta_{\Theta_T} L^2(\cD_T)} \end{bmatrix}
$$
with range given by
$$
  \operatorname{Ran} \Pi_{\rm NF} = \cH_{\Theta_T}: =
  \begin{bmatrix} H^2(\cD_{T^*}) \\ \overline{ \Delta_{\Theta_T} L^2(\cD_T)} \end{bmatrix} 
  \ominus \begin{bmatrix} \Theta_T \\ \Delta_{\Theta_T} \end{bmatrix} H^2(\cD_T)
$$
and
\begin{equation}   \label{NFlift}
\Pi_{\rm NF} T^* = V_{\rm NF}^* \Pi_{\rm NF}.
\end{equation}
Let us write $\Pi_{\rm NF, 0}$ for the operator $\Pi_{\rm NF}$ but considered to have codomain equal to its range $\cH_{\Theta_T}$ rather than all of $\cK_{\Theta_T}$.
Then $\Pi_{\rm NF,0} \colon \cH \to \cH_{\Theta_T}$ is unitary and from \eqref{NFlift} one can see that
$$
  T =  \Pi_{\rm NF,0}^* ( P_{\cH_{\Theta_T}}  V |_{\cH_{\Theta_T} })  \Pi_{\rm NF, 0},
$$
thereby showing that $T$ is unitarily equivalent to its functional-model operator $T_{\Theta_T}$:
$$
  T \underset{u} \cong T_{\Theta_T} := P_{\cH_{\Theta_T}}  \left. \begin{bmatrix} M^{\cD_{T^*}}_z & 0 \\ 0 & M^{\cD_T}_\zeta|_{ \overline{\Delta_{\Theta_T} L^2(\cD_T)}} \end{bmatrix} \right|_{\cH_{\Theta_T}}.
$$
It turns out that the characteristic operator function $\Theta_T$ \eqref{char-func} is a {\em complete unitary invariant} for a c.n.u. contraction operator $T$ 
in the sense that that two c.n.u.~contraction operators $T$ and $T'$ are unitarily equivalent if and only if $\Theta_T$ and $\Theta_{T'}$
{\em coincide} (i.e., are the same after a unitary change of basis on the input coefficient Hilbert space and on the output coefficient Hilbert space). 

But there is more.  Let us first note that any
\index{$(\cD_T,\cD_{T^*},\Theta_T)$} characteristic function $\Theta_T$ (written as $(\cD_T,\cD_{T^*},\Theta_T)$ when we wish to emphasize that the
values of $\Theta_T$ are operators from $\cD_{T}$ to $\cD_{T^*}$) is a {\em contractive analytic function}  \index{contractive analytic function} in the terminology of \cite{Nagy-Foias}, meaning that 
$\Theta_T$ an analytic function on the unit disk with contractive operator values.
It works out that the {\em coincidence envelope} of the characteristic functions consists exactly of those contractive analytic functions $(\cD, \cD_*, \Theta)$ 
which are also {\em purely contractive}\index{purely contractive} in the sense that  $\| \Theta(0) d \| < \| d \| $ for $0 \ne d \in \cD$\index{contractive analytic function!purely}, and the model-theory point of view can be reversed: given any purely contractive contractive analytic function
$(\cD, \cD_*, \Theta)$ one can form the space $\cH_\Theta$ exactly as in \eqref{HNF} but with $(\cD, \cD_*, \Theta)$ in place of $(\cD_T, \cD_{T^*}, \Theta_T)$: \index{$\cH_\Theta$}
$$
\cH_\Theta := \begin{bmatrix} H^2(\cD_*) \\  \overline{\Delta_\Theta \cdot H^2(\cD)} \end{bmatrix} \ominus \begin{bmatrix} \Theta \\ \Delta_\Theta \end{bmatrix} H^2(\cD),
$$where $\Delta_\Theta(\zeta)=(I_\cD-\Theta(\zeta)^*\Theta(\zeta))^{1/2}$ for $\zeta\in\mathbb T$.\index{$\Delta_\Theta$} Define the {\em model operator} associated with the purely contractive analytic function $(\cD, \cD_*, \Theta)$: \index{$T_\Theta$}
$$
T_\Theta : =\left. P_{\cH_\Theta} \begin{bmatrix} M^\cD_z & 0 \\ 0 & M^{\cD_*}_\zeta|_{ \overline{\Delta_\Theta L^2(\cD)}} \end{bmatrix} \right|_{\cH_\Theta}.
$$
Then it can be shown that the model operator $T_\Theta$ is a c.n.u.~contraction operator on $\cH_\Theta$, and its 
characteristic function $(\cD_{T_\Theta}, \cD_{T^*_\Theta}, \Theta_{T_\Theta})$ coincides with the original purely contractive analytic function $(\cD, \cD_*, \Theta)$.
Consequently, the study of abstract c.n.u.~contraction operators $T$ on a Hilbert space $\cH$ is equivalent to the
study of concrete functional-model operators of the form $T_\Theta$ associated with a purely contractive analytic function $\Theta$.
Furthermore, the restriction that
$T$ be c.n.u.~ is not really a restriction since any contraction operator can be decomposed as $T_{\rm cnu} \oplus T_u$ where $T_{\rm cnu}$ is c.n.u.~ and $T_u$
is unitary where unitaries are essentially well understood via spectral theory.  In Chapter \ref{C:NFmodel} below we 
present an extension of all these ideas to the setting of a commuting contractive pair $(T_1, T_2)$.
\end{remark}

\chapter{Pairs of commuting isometries}  \label{C:isometries}

Berger, Coburn and Lebow in \cite[Theorem 3.1]{B-C-L} gave a concrete model for $d$-tuples of commuting isometries
which played a
basic role in their investigation of the structure of the $C^*$-algebra generated by commuting isometries and Fredholm
theory of its elements.
In this chapter, we review the Berger-Coburn-Lebow model for pairs of commuting isometries for the pair case ($d=2$).

\section{Models for commuting pairs of isometries}  \label{S:Models-isom}

Before discussing the Berger-Coburn-Lebow (BCL) model for commuting pairs of isometries, we need a couple of lemmas.
Let us first recall that the Wold decomposition represents any isometry $V$ as the direct sum $V = S \oplus W$ with $S$
equal to a shift
operator ($S$ is an isometry such that $S^{*n} \to 0$ strongly as $n \to \infty$ or equivalently $\cap_{n=0}^\infty \operatorname{Ran} S^n = \{0\}$),
while $W$ is unitary.  As was mentioned in the previous chapter, we shall use the term shift operator and pure isometry interchangeably.

The first gives a model for commuting partial isometries of a special form which is a key ingredient in the proof of the Berger-Coburn-Lebow model theory for commuting isometries. For completeness, we provide a detailed proof here.

\begin{lemma}\label{relations-of-E-lem}
Let $\mathcal{F}$ be any Hilbert space and $E_1,E_2$ be operators on $\mathcal F$. Then $E_1,E_2$ are partial isometries of the form
\begin{equation}   \label{PIforms}
(E_1,E_2)=(U^* P^\perp, P U)
\end{equation}
for some projection $P$ and unitary $U$ in $\mathcal{B}(\mathcal{F})$ if and only if $E_1$, $E_2$ satisfy
\begin{align}\label{PIconds}
E_1E_2=E_2E_1=0 \text{ and } E_1E_1^*+E_2^*E_2=E_1^*E_1+E_2E_2^*=I_{\mathcal F}.
\end{align}
\end{lemma}

\begin{proof}
Suppose that $(E_1, E_2)$ has the form \eqref{PIforms}.  By direct substitution and making use of the unitary property of $U$ one sees that
then $(E_1, E_2)$ satisfies \eqref{PIconds}.

Conversely, suppose that $(E_1, E_2)$ satisfies conditions \eqref{PIconds}.  Then one can use conditions \eqref{PIconds} to see that
$$
E_1 E_1^* E_1 = E_1 (E_1^* E_1 + E_2 E_2^*) = E_1, \quad
E_2 E_2^* E_2 = E_2 (E_2^* E_2 + E_1 E_1^*) = E_2.
$$
Hence we have
 \begin{equation}   \label{part-isom}
   E_1 E_1^* E_1 = E_1, \quad E_2 E_2^* E_2 = E_2
\end{equation}
from which it follows that $E_1$ and $E_2$ are partial isometries.  This in turn is equivalent to all of $E_1E_1^*, E_1^*E_1,E_2E_2^*,
 E_2^*E_2$ being projections onto $\operatorname{Ran} E_1$, $\operatorname{Ran} E_1^*$, $\operatorname{Ran}E_2$,
 $\operatorname{Ran} E_2^*$, respectively.   Therefore conditions (\ref{PIconds}) can be reformulated as
\begin{equation}  \label{decoms}
\operatorname{Ran} E_1 \oplus \operatorname{Ran} E_2^*=\mathcal{F}=\operatorname{Ran} E_1^* \oplus \operatorname{Ran}E_2.
\end{equation}
By the polar-decomposition theorem, we have unitaries $U_1:\operatorname{Ran} E_1\to \operatorname{Ran} E_1^*$ and 
$U_2:\operatorname{Ran }E_2^*\to \operatorname{Ran} E_2$ such that 
$$
E_1^*=U_1(E_1E_1^*)^{\frac{1}{2}} \text{ and }E_2=U_2(E_2^*E_2)^{\frac{1}{2}}.
$$
Let us define a unitary operator $U$ on $\cF$ (making use of the decompositions \eqref{decoms} by
$$
U:=U_1\oplus U_2: \operatorname{Ran }E_1 \oplus \operatorname{Ran }E_2^*\to \operatorname{Ran }E_1^* \oplus \operatorname{Ran }E_2.
$$
More explicitly, making use of the fact that $E_1$ and $E_2$ are partial isometries with $E_1$ and $E_2^*$ having complementary ranges, we can
get a formula for the action of $U$:
$$
Ux =  U_1 E_1 E_1^* x + U_2 E_2^* E_2 x   =  (E_1^* +  E_2) x  \text{ for } x \in \cF.
$$
Let us now define a projection operator $P$ and a unitary operator $U$ on $\cF$ by 
\begin{equation}  \label{defPU}
P = E_2 E_2^*, \quad U:= E_1^* + E_2.
\end{equation}
It is now a straightforward exercise to verify that we recover $(E_1, E_2)$ from $(P,U)$ according to the formula \eqref{PIforms}.
\end{proof}

The first part of the next result is well-known (see e.g.\ \cite[page 227]{NFintertwine}) and the second part is an easy corollary
of the first part; we include short proofs of both results for completeness.

\begin{lemma}\label{L:AuxLemma}
\begin{enumerate}
\item
The only bounded linear operator intertwining a unitary operator with a shift operator is the zero operator, i.e.:  if
$\cK$ and $\cK'$ are Hilbert spaces, $U$ is a unitary operator on $\cK$, $S$ is a shift operator on $\cK'$, and $\Gamma \colon \cK \to \cK'$
is a bounded linear operator such that $\Gamma U = S \Gamma$, then $\Gamma = 0$.

\item Suppose that $S$ and $S'$ are shift operators on $\cK_1$ and $\cK_1'$ respectively, $U$ and $U'$ are unitary operators on
$\cK_2$ and $\cK'_2$ respectively, and $\Gamma = \sbm{ \Gamma_{11} & \Gamma_{12} \\ \Gamma_{21} & \Gamma_{22} }$
is a unitary operator from $\cK_1 \oplus \cK_2$ to $\cK'_1 \oplus \cD'_2$ which intertwines $S \oplus U$ with $S' \oplus U'$:
\begin{equation}   \label{block-intertwine}
\begin{bmatrix} \Gamma_{11}  & \Gamma_{12}  \\ \Gamma_{21} & \Gamma_{22} \end{bmatrix} \begin{bmatrix} S & 0 \\ 0 & U \end{bmatrix}
= \begin{bmatrix} S' & 0 \\ 0 & U' \end{bmatrix} \begin{bmatrix} \Gamma_{11}  & \Gamma_{12}  \\ \Gamma_{21} & \Gamma_{22} \end{bmatrix}.
\end{equation}
Then $\Gamma$ is block-diagonal, i.e.:  $\Gamma_{12} = 0$ and $\Gamma_{21} = 0$.
\end{enumerate}
\end{lemma}

\begin{proof}  \textbf{(1):}
The intertwining condition $\Gamma U = S \Gamma$ implies that $\Gamma U^n = S^n \Gamma$ for all $n=0,1,2,\dots$.
As $U$ is unitary, $\operatorname{Ran } U^n$ is the whole space $\cK$ for all $n=0,1,2,\dots$ and we conclude that
$\operatorname{Ran} \Gamma \subset \cap_{n=0}^\infty \operatorname{Ran} S^n$.  As $S$ is a shift,
$\cap_{n=0}^\infty \operatorname{Ran} S^n = \{0\}$, and we are forced to conclude that $\Gamma$ is the zero operator.

\smallskip

\textbf{(2):}  From the (1,2)-entry of \eqref{block-intertwine} we see that $\Gamma_{12} U = S' \Gamma_{12}$.  
From part (1) of the lemma
we conclude that $\Gamma_{12} = 0$.

As $\Gamma$ is unitary, it follows that \eqref{block-intertwine} can be rewritten as
$$
  \begin{bmatrix} S & 0 \\ 0 & U \end{bmatrix} \begin{bmatrix} \Gamma_{11}^* & \Gamma_{21}^* \\ 0 & \Gamma_{22}^* \end{bmatrix}
  = \begin{bmatrix} \Gamma_{11}^* & \Gamma_{21}^* \\ 0 & \Gamma_{22}^* \end{bmatrix} \begin{bmatrix} S' & 0 \\ 0 & U' \end{bmatrix}.
$$
The (1,2)-entry of this equality gives
$$
  S \Gamma_{21}^* = \Gamma_{21}^* U'.
$$
Again by part (1) of the Lemma, we conclude that $\Gamma_{21}^* = 0$, and hence also $\Gamma_{21} = 0$.
\end{proof}

\begin{remark}  \label{R:AuxLemma}
We note that part (1) of Lemma \ref{L:AuxLemma} fails if the hypothesis is changed to: {\em $U$ is unitary on 
$\cK$, $S$ is a shift on $\cK'$
and $X \colon \cK' \to \cK$ is such that $X S = U X$.}  As an example, take $U = M_\zeta$ on $\cK = L^2$, 
$S = M_z$ on $\cK' = H^2$
and $X \colon H^2 \to L^2$ equal to the embedding $X \colon f(z) \mapsto f(\zeta)$ of $H^2$ into $L^2$.
\end{remark}

We now have all the preparations needed to derive the BCL model for a commuting pair of isometries.  We shall actually have use for two
such models, each of which is easily derived from the other.  A somewhat different proof follows from Theorem 2.1 and Lemma 2.2 in
\cite{BDF1}.

For a Hilbert space $\cH$, we shall use $\cB(\cH)$ for the space of all bounded linear operators from $\cH$ to $\cH$. 
\begin{theorem}\label{Thm:BCLmodel}  {\rm (See Berger-Coburn-Lebow \cite{B-C-L}.)}
Let $(V_1,V_2)$ be a pair of commuting isometries on a Hilbert space $\cH$. 

\begin{enumerate}
\item Then there exist  Hilbert spaces $\cF$ and $\cK_{u}$,
a unitary identification map $\tau_{\rm BCL} \colon \cH \to \sbm{ H^2(\cF) \\ \cK_{u}}$,
a projection $P$ in $\cB(\cF)$, a unitary $U$ in $\cB(\cF)$ and commuting unitaries $W_{1}, W_{2}$ in $\cB(\cK_{u})$
such that
\begin{equation}   \label{BCL1model}
\tau_{\rm BCL}  V_1  = \begin{bmatrix} M_{P^\perp U + z PU} & 0 \\ 0 &  W_{1} \end{bmatrix} \tau_{\rm BCL}, \quad
\tau_{\rm BCL} V_2   = \begin{bmatrix} M_{U^*P + z U^*P^\perp} & 0 \\ 0 &  W_{2} \end{bmatrix}   \tau_{\rm BCL}.
\end{equation}
Explicitly one can take
\begin{equation}  \label{tauBCL-def}
  \cF = \cD_{V^*}, \quad \tau_{\rm BCL} = \begin{bmatrix}  \cO_{D_{V^*}, V^*} \\ Q_{V^*} \end{bmatrix}
  \colon \cH \to \begin{bmatrix} H^2(\cD_{V^*}) \\ \cQ_{V^*} \end{bmatrix}.
\end{equation}
where $\cO_{D_{V^*}, V^*}$, $Q_{V^*}$, and $\cQ_{V^*}$ are as in \eqref{obsop}, \eqref{Q}, \eqref{cQT*}
with $V^*$ in place of $T^*$.
We shall say that a pair of operators of the form
$$
\bigg(  \begin{bmatrix} M_{P^\perp U + z PU} & 0 \\ 0 & W_{1}\end{bmatrix}, 
\begin{bmatrix}M_{U^*P + z U^* P^\perp} & 0 \\ 0 &  W_{2} \end{bmatrix} \bigg)
$$
acting on $H^2(\cF) \oplus \cK_u$ (with $P$, $U$ as above) is a {\em BCL1 model}
\index{BCL1!model}
for a pair of commuting isometries.

\item Equivalently,  there exist  Hilbert spaces $\cF$ and $\cK_{u}$,
a unitary identification map $\tau_{\rm BCL}\colon \cH \to \sbm{ H^2(\cF) \\ \cK_{u}}$ ,
a projection $P^\ff$ in $\cB(\cF)$, a unitary $U^\ff$ in $\cB(\cF)$ and commuting unitaries $W_{1}^\ff, W_{2}^\ff$ in 
$\cB(\cK_{u})$
such that
\begin{equation}  \label{BCL2model}
\tau_{\rm BCL}V_1  = \begin{bmatrix}  M_{U^{\ff *}P^{\ff \perp} + z U^{\ff *}P^\ff} & 0 \\ 0 &  W_{1}^\ff \end{bmatrix}
  \tau_{\rm BCL}, \quad
  \tau_{\rm BCL} V_2   = \begin{bmatrix} M_{P^\ff U^\ff + z P^{\ff \perp}U^\ff} & 0 \\ 0 &  W_{2}^\ff
  \end{bmatrix} \tau_{\rm BCL}
\end{equation}
where again one can take $\cF$, $\cK_u$ and $\tau_{\rm BCL}$,as in \eqref{tauBCL-def}.
We shall say that a pair of operators of the form
$$
 \begin{bmatrix} M_{U^{\ff *}P^{\ff \perp} + z U^{\ff *}
P^\ff } & 0 \\ 0 & W_{1}^\ff \end{bmatrix}, \,
\begin{bmatrix}  M_{P^\ff U^\ff + z P^{\ff \perp}U^\ff} & 0 \\ 0 &  W_{2}^\ff  \end{bmatrix}
$$
acting on $\sbm{ H^2(\cF^\ff) \\ \cK^\ff_u}$ (with $P^\ff$, $U^\ff$ as above) is a {\em BCL2 model}
\index{BCL2!model}
for a pair of commuting isometries.
\end{enumerate}
\end{theorem}

\begin{proof}
We note that the {\em flip-transformation} acting on BCL-data sets $(\cF, P, U)$ given by
\begin{equation}   \label{flip}
\ff \colon  (\cF, P, U, W_1, W_2) \mapsto (\cF^\ff, P^\ff, U^\ff, W^\ff_1, W^\ff_2): = (\cF, U^* P U, U^*, W_1, W_2)
 \end{equation}
 transforms the BCL1 model \eqref{BCL1model} into the form of the BCL2 model \eqref{BCL2model} and vice versa.
 Alternatively, note that one converts a BCL1 model to a BCL2 model by interchanging the indices (1,2) on $V_1, V_2$ and interchanging $P$ with $P^\perp$,
 and vice-versa.
 Hence it suffices to verify only one of the
 statements (1) and (2).  We shall work out the details for the BCL2 model.  As all the details will be worked out
 only for this setting, we drop the superscript-$\ff$ from the notation.

Let $V $ be the isometry $V= V_1 V_2$,
set $D_{V^*} = (I - V V^*)^{\frac{1}{2}}$ equal to  the {\em defect operator} for $V^*$, and let $\cD_{V^*} = \overline{\operatorname{Ran}}\, D_{V^*}$.
Since $V$ is an isometry, in fact $D_{V^*} = (I - V V^*)$ is just the orthogonal projection onto $(\operatorname{Ran} V)^\perp$
and $\cD_{V^*} = \operatorname{Ran} D_{V^*} = (\operatorname{Ran} V)^\perp$.  By an iterative and limiting procedure, one can show that
any $h \in \cH$ decomposes orthogonally as
$$
  h = \left(\bigoplus_{n=0}^\infty V^n D_{V^*} V^{*n} h \right) \oplus h_u \text{ where } h_u = \lim_{n \to \infty} V^n V^{*n} h \in \cH_u : = \bigcap_{n=0}^\infty
  \operatorname{Ran} V^n.
$$
Hence the space $\cH$ decomposes as
$$
   \cH = \left( \bigoplus_{n=0}^\infty V^n \cD_{V^*} \right) \oplus \cH_u,
$$
amounting to the coordinate-free version of the {\em Wold decomposition} for the isometry $V$
(see \cite{vonN-Wold, Wold}).
To convert this decomposition to a more functional form, we introduce a unitary Fourier representation operator
$$
\tau_{\rm BCL} \colon \cH \to \begin{bmatrix}  H^2(\cD_{V^*}) \\  \cH_u  \end{bmatrix}
$$
given by
\begin{equation}   \label{tauBCL}
  \tau_{\rm BCL} \colon h \mapsto  \begin{bmatrix} \sum_{n=0}^\infty (D_{V^*} V^{*n} h) z^n \\
   \lim_{n \to \infty} V^n V^{*n}  \end{bmatrix} h = :  \begin{bmatrix} \cO_{D_V^*, V^*} \\ \cQ_{V^*} \end{bmatrix}
\end{equation}
Then  one easily checks that $\tau_{\rm BCL}$ has the intertwining property
 $$
 \tau_{\rm BCL} V = (M_z^{\cD_{V^*}} \oplus W) \tau_{\rm BCL}
 $$
where $M_z^{\cD_{V^*}}$ is the forward shift on $H^2(\mathcal{D}_{V^*})$  ($ M_z^{\cD_{V^*}} \colon f(z) \mapsto z f(z)$) and where $W=V|_{\mathcal{H}_u}$ is a unitary.
Let us set
$$
   \widetilde V_1 = \tau_{\rm BCL} V_1 \tau_{\rm BCL}^*,  \quad \widetilde V_2 = \tau_{\rm BCL} V_2 \tau_{\rm BCL}^*.
 $$
 Write out block-matrix representations
 \begin{equation}  \label{tildeV-1}
   \widetilde V_j = \begin{bmatrix} \widetilde V_{j,11} & \widetilde V_{j,12} \\ \widetilde V_{j,21} & \widetilde V_{j, 22} \end{bmatrix}, \quad j = 1,2
 \end{equation}
 for the operators $\widetilde V_j$ with respect to the decomposition $\sbm{ H^2(\cD_{V^*})  \\ \cH_u }$ on which they act.
 The commutativity of each $V_1,V_2$ with $V$ implies the commutativity of each $\widetilde V_1,
 \widetilde V_2$ with $\sbm{ M_z^{\cD_{V^*}} & 0 \\ 0 & W}$.  In particular, we get the corner intertwining conditions
 $$
 \widetilde V_{j, 12} W = M_z^{\cD_{V^*}}   \widetilde V_{j, 12}
 $$
 for $j=1,2$.  As $W$ is unitary and $M_z^{\cD_{V^*}}$ is a shift,  part (1) of Lemma \ref{L:AuxLemma} implies that $\widetilde V_{j,12} = 0$ for $j=1,2$,
 and the representation \eqref{tildeV-1} collapses to
 \begin{equation}   \label{tildeV-2}
 \widetilde V_j = \begin{bmatrix} \widetilde V_{j,11} & 0 \\ \widetilde V_{j,21} & \widetilde V_{j, 22} \end{bmatrix}, \quad j = 1,2.
 \end{equation}
 From the fact that $V_1 V_2 = V_2 V_1 = V$, we know that
 $$
 \widetilde V_1 \widetilde V_2 = \widetilde V_2 \widetilde V_1 = \begin{bmatrix} M_z^{\cD_{V^*}} & 0 \\  0 & W \end{bmatrix}.
 $$
 In particular, we must have
 $$
   \widetilde V_{1, 22} \widetilde V_{2, 22} =  \widetilde V_{2, 22} \widetilde V_{1, 22} = W
 $$
 is unitary.  Furthermore, the fact that each  $\widetilde V_j$ is an isometry implies that each $\widetilde V_{j,22}$ is an isometry.
 Putting the pieces together, we see that each $\widetilde V_{j, 22}$ is a surjective isometry, i.e., each $\widetilde V_{j, 22}$ is unitary.
 As each $\widetilde V_j$ is an isometry, we see that
 $$
 \begin{bmatrix} \widetilde V_{j,11}^* & \widetilde V_{j,21}^* \\ 0 & \widetilde V_{j,22}^* \end{bmatrix}
 \begin{bmatrix} \widetilde V_{j,11} & 0 \\ \widetilde V_{j,21} & \widetilde V_{j,22} \end{bmatrix}  = 
 \begin{bmatrix} I_{H^2(\cD_{V^*})}& 0 \\
 0 & I_{\cH_u} \end{bmatrix}.
 $$
 In particular, equality of the (1,2)-entries gives that $\widetilde V_{j,21}^* \widetilde V_{j,22} = 0$.  As we have 
 already noted that $\widetilde V_{j,22}$
 is surjective, it follows that $V_{j,21}^* = 0$, and hence also $V_{j,21} = 0$.  Thus the representation 
 \eqref{tildeV-2} collapses further to
 \begin{equation}   \label{tildeV-3}
 \widetilde V_j = \begin{bmatrix} \widetilde V_{j,11} & 0 \\ 0 & \widetilde V_{j, 22} \end{bmatrix}, \quad j = 1,2,
 \end{equation}
 i.e., the decomposition $\sbm{H^2(\cD_{V^*}) \\ 0 } \oplus \sbm{ 0 \\ \cH_u}$
   is reducing for each  $\widetilde V_j$.

As each $\widetilde V_j$ commutes with $\sbm{ M_z^{\cD_{V^*}} & 0 \\ 0 & W}$,  it then follows that $\widetilde V_j=
\sbm{ M_{\varphi_j} & 0 \\ 0 & W_j}$ for $j= 1,2$,  where $\varphi_j$ is an $H^\infty(\mathcal{D}_{V^*})$-function and 
$(W_1,W_2)$ is a pair of commuting unitaries such that $W_1W_2=W$. Since $V_1=V_2^*V$,  consideration of 
the power series expansion of $\varphi_1$ and $\varphi_2$ enables one to conclude that $\varphi_1(z)=E_1+zE_2^*$ and $\varphi_2(z)=E_2+zE_1^*$ for some operators $E_1,E_2$ acting on
 $\mathcal{D}_{V^*}$.  Since $M_{\varphi_1}$ is an isometry, we have
\begin{align*}
(I_{H^2}\otimes E_1+M_z\otimes E_2^*)^*(I_{H^2}\otimes E_1+M_z\otimes E_2^*)=I_{H^2}\otimes I_{\cD_{V^*}},
\end{align*}
which implies that
\begin{align}\label{BCLEqn1}
E_1^*E_1+E_2E_2^*=I_{\cF}.
\end{align}
Similarly since $M_{\varphi_2}$ is an isometry, we have
\begin{align}\label{BCLEqn2}
E_2^*E_2+E_1E_1^*=I_{\cD_{V^*}}.
\end{align}
Also, since $V=V_1V_2$, we have $M_z=M_{E_1+zE_2^*}M_{E_2+zE_1^*}=M_{E_2+zE_1^*}M_{E_1+zE_2^*}$, which readily implies that
\begin{align}\label{BCLEqn3}
 E_1E_2=0=E_2E_1.
\end{align}
From equations (\ref{BCLEqn1}), (\ref{BCLEqn2}) and (\ref{BCLEqn3}), we conclude by 
Lemma \ref{relations-of-E-lem} that
there exist a projection $P$ and a unitary $U$ in $\cB(\cD_{V^*})$ such that $E_1,E_2$ are as in (\ref{PIforms}) 
(with $\cF$ taken to be $\cF = \cD_{V^*}$). Consequently,
$$
(M_{\varphi_1},M_{\varphi_2})=(M_{U^*P^\perp +zU^*P}, M_{PU+z P^\perp U}).
$$
and the theorem now follows.
\end{proof}

\begin{definition}\label{D:BCLTuple}
For a pair $(V_1,V_2)$ of commuting isometries, let the Hilbert space $\cF$ and the operators $P$, $U$ in $\cB(\cF)$, $W_1,W_2$ in $\cB(\cH_u)$
be as in Theorem \ref{Thm:BCLmodel}. Then a tuple $(\cF,P,U,W_1,W_2)$ associated with the BCL1 model \eqref{BCL1model} for $(V_1, V_2)$ will be called a
{\em BCL1 tuple}
\index{BCL1!tuple} 
for $(V_1,V_2)$, while a tuple $(\cF^\ff, P^\ff, U^\ff, W^\ff_{1}, W^\ff_{*2})$ associated
with a BCL2 model
\eqref{BCL2model} for $(V_1, V_2)$ will be called a {\em BCL2 tuple}
\index{BCL2!tuple} 
for $(V_1, V_2)$.

If $P$ is a projection and $U$ is a unitary operator on a Hilbert space $\cF$ with no pair of commuting isometries or of
BCL-model-type specified, we shall say simply that the tuple $(\cF, P, U)$ is a {\em BCL tuple}.
\end{definition}

The following uniqueness result was observed in \cite{B-C-L} but not proved there. We outline the proof here.

\begin{theorem}\label{Thm:BCLcoin}
Let $(V_1,V_2)$ on $\cH$ and $(V_1',V_2')$ on $\cH'$ be two pairs of commuting isometries with $(\cF,P,U,W_1,W_2)$ and
$(\cF',P',U',W_1',W_2')$ as respective BCL2 tuples. Then $(V_1,V_2)$ and $(V_1',V_2')$ are
unitarily equivalent if and only if the associated BCL2-tuples are unitarily equivalent in the sense that there exist
unitary operators $\tau:\cF\to\cF'$ and $\tau_u:\cH_u\to\cH_u'$ such that
\begin{equation}  \label{BCLcoin}
(\tau P\tau^*,\tau U\tau^*)=(P',U'), \quad (\tau_u W_1\tau_u^*, \tau_u W_2\tau_u^*)=(W_1',W_2').
\end{equation}

Moreover, a BCL2 model uniquely determines the associated BCL tuple in the following sense:  if $(\cF, P, U, W_1, W_2)$
and $(\cF, P', U', W_1', W_2')$ are
two BCL2 tuples with the same  model coefficient spaces $\cF = \cF'$, $\cH_u = \cH_u'$ such that the associated model commuting isometric pairs are the same
\begin{align*}
& \left( \sbm{ M_{U^*P^\perp + z U^*P} & 0 \\ 0 &  W_1}, \sbm{M_{PU + z P^\perp U} & 0 \\ 0 & W_2}  \right) =
\left( \sbm{M_{U^{\prime *}P^{\prime \perp} + z U^{\prime *} P'} & 0 \\ 0 &  W_1'},
\sbm{ M_{P' U'+ z P^{\prime \perp} U'} & 0 \\ 0 &  W_2' } \right),
\end{align*}
then in fact the BCL2 tuples are identical:
$$
(\cF, P, U, W_1, W_2)   =  (\cF, P', U', W_1', W_2').
$$

Similar statements hold true with BCL1 model and BCL1 tuples in place of BCL2 model and BCL2 tuples.
\end{theorem}

\begin{proof}
Due to the correspondence \eqref{flip} between BCL1 tuples and BCL2 tuples, it suffices to prove the result for BCL2 tuples.
Again we write a BCL2 tuple simply as $(\cF, P, U)$.

  If there exist unitary operators $\tau:\cF\to\cF'$ and $\tau_u:\cH_u\to\cH_u'$ such that (\ref{BCLcoin}) holds, 
  then the pairs
  $$
  \left( \sbm{M_{U^*P^\perp +zU^*P} & 0 \\ 0 & W_1 }, 
 \sbm{M_{PU+zP^\perp U} & 0 \\ 0 &  W_2 } \right), \quad 
  \left( \sbm{M_{U'^*P'^\perp +zU'^*P'} & 0 \\ 0 & W_1'},  \sbm{ M_{P' U+zP'^\perp U'} & 0 \\ 0 &  W_2'} \right)
  $$
   are unitarily equivalent via the unitary similarity
  \begin{align*}
   \sbm{ I_{H^2}\otimes \tau & 0 \\ 0 & \tau_u} \colon \sbm{ H^2(\cF) \\ \cH_u} \to 
   \sbm{ H^2(\cF') \\ \cH_u'}.
  \end{align*}
  Then by Theorem \ref{Thm:BCLmodel} the pairs $(V_1,V_2)$ and $(V_1',V_2')$ are unitarily equivalent.

  Conversely, suppose that the pairs 
  $$
  \left( \sbm{ M_{U^*P^\perp +z U^* P} & 0 \\ 0 &  W_1}, \sbm{M_{PU+zP^\perp U} & 0 \\ 0 &  W_2}\right), \quad 
 \left( \sbm{M_{U'^* P'^\perp +z U'^* P'} & 0 \\  &  W_1'},  
 \sbm{M_{P' U'+zP'^\perp U'} & 0 \\ 0 &  W_2'} \right)
 $$ 
 are unitarily equivalent via the unitary similarity
  $$
  \hat\tau=\begin{bmatrix}
               \tau' & \tau_{12} \\
               \tau_{21} & \tau_u
             \end{bmatrix} \colon \begin{bmatrix} H^2(\cF) \\ \cH_u \end{bmatrix} \to \begin{bmatrix} H^2(\cF') \\ \cH_u' \end{bmatrix}.
 $$
  From the intertwining $\hat{\tau}V_1V_2=V_1'V_2'\hat{\tau}$, we have
  \begin{equation}  \label{intertwine1'}
      \begin{bmatrix} \tau' & \tau_{12} \\ \tau_{21} & \tau_u \end{bmatrix}
    \begin{bmatrix} M_z^\cF & 0 \\ 0 & W \end{bmatrix}
    =\begin{bmatrix} M_z^{\cF'} & 0 \\ 0 & W'  \end{bmatrix}
    \begin{bmatrix} \tau' & \tau_{12} \\  \tau_{21} & \tau_u \end{bmatrix}
  \end{equation}
  where we write $M_z^\cF$ for multiplication by $z$ on $H^2(\cF)$ and $M_z^{\cF'}$ for multiplication by $z$ on $H^2(\cF')$.
  By part (2) of Lemma \ref{L:AuxLemma} we conclude that $\tau_{12} = 0$, $\tau_{21} = 0$
and hence $\hat\tau$ collapses to the diagonal form
 $$
       \hat\tau = \begin{bmatrix} \tau' & 0 \\ 0 & \tau_u \end{bmatrix}.
 $$
  Therefore the unitary $\tau_u$ intertwines $(W_1,W_2)$ with $(W_1',W_2')$ and the unitary $\tau'$ intertwines  $M_z^\cF$ with $M_z^{\cF'}$
 forcing $\tau$ to have the form $\tau'=I_{H^2}\otimes\tau$ for some unitary $\tau:\cF\to\cF'$.
 Since $\tau'$ intertwines $M_{U^*(P^\perp +zP)}$ with $M_{U'^*(P'^\perp +zP')}$,
 we have $\tau U^* P^\perp = U'^* P'^\perp \tau$ and $\tau U^* P=U'^* P' \tau$. Therefore
 $$
    \tau U^*=\tau U^* (P^\perp+P)=U'^* (P'^\perp+P') \tau=U'^* \tau
  $$
from which we see that
$$
U'^* \tau P  = \tau U^* P = U'^* P' \tau \Rightarrow \tau P = P' \tau.
$$
In particular, if $\hat \tau$ is the identity operator, then $\tau$ and $\tau_u$ are identity operators, implying that
$(\cF, P, U, W_1, W_2) = (\cF', P', U', W_1', W_2')$, thereby verifying the last statement.
This completes the proof.
\end{proof}

Theorem \ref{Thm:BCLcoin} suggests that there should be a canonical choice of BCL tuple generating a
BCL1 (or BCL2) model for a given
commuting isometric pair.  This is indeed the case and is the content of the next result (see also Proposition 7.1 in \cite{BDF3}).
Since the unitary part can be handled separately by spectral theory,
we assume that the product isometry $V = V_1 V_2 = V_2 V_1$ is a pure isometry ($V^{*n} \to 0$ strongly as $n \to \infty$).

\begin{theorem}  \label{T:BCLcanonical}  Suppose that $(V_1, V_2)$  is a commuting isometric pair on $\cH$ such that $V = V_1 V_2 = V_2 V_1$ is a
pure isometry. Then:

\smallskip

\noindent
(1)  The operators
$$
D_{V^*}, \, D_{V_1^*}, \, D_{V_2^*}, \, V_1 D_{V_2^*}V_1^*, \, V_2 D_{V_1^*} V_2, \, D_{V_2^*} V_1^* + V_2 D_{V_1^*}, \,
D_{V_2^*} V_1^* + V_2 D_{V_1^*}
$$
all have range and cokernel contained in $\cD_{V^*} = \operatorname{Ran} D_{V^*}$ and therefore, when restricted to $\cD_{V^*}$, can be
viewed as elements of $\cB(\cD_{V^*})$ (bounded linear operators mapping $\cD_{V^*}$ into itself).

\smallskip

\noindent
(2) The  operators  $D_{V^*}$, $D_{V_1^*}$, $D_{V_2^*}$, $V_1 D_{V_2^*} V_1^*$, $V_2 D_{V_1^*} V_2^*$ all map $\cD_{V^*}$ into itself
and thus the restriction of these operators to $\cD_{V^*}$ may be considered as elements of $\cB(\cD_{V^*})$ (bounded linear operators on 
$\cD_{V^*}$).  When this is done, all are orthogonal projections on $\cD_{V^*}$ (with $D_{V^*}|_{\cD_{V^*}} = I_{\cD_{V^*}}$) and we have
the orthogonal decompositions:
\begin{equation}   \label{ortho-decom}
 \cD_{V^*} = \operatorname{Ran}  D_{V_1^*} \oplus \operatorname{Ran} V_1 D_{V_2^*} V_1^*,\quad
 \cD_{V^*} = \operatorname{Ran} V_2 D_{V_1^*} V_2^* \oplus \operatorname{Ran} D_{V_2^*}.
\end{equation}

\smallskip

\noindent
(3)  When considered as operators on $\cD_{V^*}$, the operators
$$
 U: = (V_1 D_{V_2^*} + D_{V_1^*} V_2^*)|_{\cD_{V^*}}, \quad U_*: = (D_{V_2^*} V_1^* + V_2 D_{V_1^*})|_{\cD_{V^*}}
$$
are unitary.

\smallskip

\noindent
(4) A BCL1 tuple for $(V_1, V_2)$ is given directly in terms of $(V_1, V_2)$ by
\begin{equation}   \label{BCL1canonical}
(\cF, P, U) = \left(\cD_{V^*}, D_{V_1^*}|_{\cD_{V^*}}, (V_1 D_{V_2^*} + D_{V_1^*} V_2^*)|_{\cD_{V^*}} \right)
\end{equation}
while a BCL2 tuple for $(V_1, V_2)$ is similarly given by
\begin{equation}  \label{BCL2canonical}
(\cF_*, P_*, U_*) = \left( \cD_{V^*}, V_2 D_{V_1^*}V_2^*|_{\cD_{V^*}},  (D_{V_2^*} V_1^* + V_2 D_{V_1^*}) |_{\cD_{V^*}} \right).
\end{equation}
 \end{theorem}

 \begin{proof}  From the fact that $V_1$ and $V_2$ are isometries (so $V_1^* V_1 = V_2^* V_2 = V^* V = I_\cH$), it is easily checked that
 each of the operators $X$ in statement (2) is a projection (i.e.,  $X = X^*$ and $X^2 = X$).   The two orthogonal decompositions in statement (2)
 then follow from the general identities
 \begin{equation}   \label{defid}
   I - V V^* = (I - V_1 V_1^*) + V_1 (I - V_2 V_2^*) V_1^* = I - V_2 V_2^* + V_2 (I - V_1 V_1^*)V_2^*.
 \end{equation}
 These identities also show that all the projection operators have range and cokernel (i.e., also the range for a projection operator) a subspace of $\cD_{V^*}$.

 As for $U_* =  D_{V_2^*} V_1^* + V_2 D_{V_1^*}$,  noting that $D_{V_2^*} V_2 = 0$ since $V_2$ is an isometry, we can compute
 \begin{align*}
& U_*^* U_*   = (V_1 D_{V_2^*} + D_{V_1^*} V_2^*)(D_{V_2^*} V_1^* + V_2 D_{V_1^*})  \\
 & \quad  = V_1 (I - V_2 V_2^*) V_1^* + (I - V_1 V_1^*)  = V_1 V_1^* - V V^* + I - V_1 V_1^* = I - V V^* = D_{V^*}
 \end{align*}
 and
 $$
 U_* U_*^* = (D_{V_2^*} V_1^* +  V_2 D_{V_1^*})
 (V_1 D_{V_2^*} + D_{V_1^*} V_2^*) = D_{V_2^*} + V_2 D_{V_1^*} V_2^* = D_{V^*}.
 $$
 These identities not only verify that $U_*$ has range and cokernel inside $\cD_{V^*}$, but furthermore that $U_*$ is unitary when considered as an
 operator on $\cD_{V^*}$.  Similar computations verify the corresponding properties for $U =  D_{V_2^*} V_1^* + V_2 D_{V_1^*}$.  This completes the
 verification of statements (1), (2), and (3).

Let us now assume that $V_1, V_2$ are presented in the form of a BCL2 model
$$
V_1= M_{U_*^*(P_*^\perp + z P_*)}, \quad V_2 = M_{(P_* + z P_*^\perp) U_*}
$$
acting on the Hardy space $H^2(\cF_*)$.  Our next goal is to understand how to recover the operator pair $(P,U)$ directly in terms of the operator pair $(V_1,
V_2)$.
Note first that then
\begin{align*}
& V_1^* - V_2 V^*  =  \\
&  \left( (I_{H^2} \otimes U_*^* P_*^\perp + M_z \otimes U_*^* P_*)^* - (I_{H^2} \otimes P_* U_* + M_z \otimes P_*^\perp U_*) (M_z \otimes I_{\cD_{V^*}})^* \right)  \\
& = \left( (I_{H^2} \otimes P_*^\perp U_* + M_z^* \otimes P_* U_*) - (M_z^* \otimes P_* U_* + M_z M_z^* \otimes P_*^\perp U_*) \right)  \\
& = \left( (I_{H^2} - M_z M_z^*) \otimes P_*^\perp U_* \right) = \bev_{0, \cF_*}^*  P_*^\perp U_*.   
\end{align*}
If we identify the coefficient space $\cF_*$ with the $V^*$-defect space $\cD_{V^*} = \operatorname{Ran} I  - V V^*$, then it is convenient to view the 
operators $P_*$ and $U_*$
as operators on $\cD_{V^*}$ rather than just on $\cF_*$, and then we have
$$
D_{V^*} (V_1^* - V_2 V^*) = P_*^\perp U_* D_{V^*} \text{ and } D_{V^*} V^{*n} (V_1^* - V_2 V^*) = 0 \text{ for } n=1,2,3,\dots.
$$
It follows that $V_1^* - V_2 V^*$ has range in $\cD_{V^*}$, and hence we actually have
\begin{equation}  \label{U*1}
V_1^* - V_2 V^* = P_*^\perp U_* D_{V^*} = D_{V^*} P_*^\perp U_* D_{V^*}.
\end{equation}
A similar computation gives us
\begin{equation}  \label{U*2}
V_2^* - V_1 V^* = D_{V^*} U_*^* P_* D_{V^*}.
\end{equation}
Next note that
\begin{align*}
D_{V^*} U_*^* D_{V^*} & = D_{V^*} U_*^* (P_*^\perp + P_*) D_{V^*} = (V_1^* - V_2 V^*)^* + (V_2^* - V_1 V^*) \\
& = V_1 - V_1 V_2 V_2^* + V_2 - V_1 V_1^* V_2^* \\
& = V_1 D_{V_2^*} + D_{V_1^*} V_2^*.
\end{align*}
As we have already checked that the operator $V_1 D_{V_2^*} + D_{V_1^*} V_2^*$ has range and cokernel contained in $\cD_{V^*}$, we may cancel the projection
$D_{V^*}$ on the left and on the right of the left-hand side and deduce that necessarily $U_*^* = (V_1 D_{V_2^*} + D_{V_1^*} V_2^*)|_{\cD_{V^*}}$.
Hence $U_* = (D_{V_2^*} V_1^* + V_2 D_{V_1^*})|_{\cD_{V^*}}$ as in \eqref{BCL2canonical}.

To find $P_*$ we see from \eqref{U*2} that
\begin{align*}
P_* & = U_* (V_2^* - V_1 V^*) = (D_{V_2^*} V_1^* + V_2 D_{V_1^*}) (V_2^* - V_1 V^*)  \\
& = D_{V_2^*} V_1^* + V_2 D_{V_1^*}) (V_2^* - V_1 V^*) \\
& = D_{V_2^*} V^* - D_{V_2^*} V^* + V_2 D_{V_1^*} V_2^* - V_2 D_{V_1^*} V_1 V^*  \\
& = V_2 D_{V_1^*} V_2^* \text{ (since $D_{V_1^*} V_1 = 0$)}
\end{align*}
as in \eqref{BCL2canonical}.

Parallel computations can be used to show that \eqref{BCL1canonical} is a BCL1 tuple for $(V_1, V_2)$ in case $(V_1, V_2)$ are in the model BCL1 form.
Alternatively, it suffices to show that the flip map
\eqref{flip} applied to \eqref{BCL2canonical} produces \eqref{BCL1canonical}, i.e., that, in the notation of \eqref{BCL1canonical}--\eqref{BCL2canonical},
\begin{equation}   \label{BCL2can-flip}
(U_*^* P_* U_*, U_*^*)  = (P, U).
\end{equation}
The second relation in \eqref{BCL2can-flip} is clear by inspection.  As for the first, let us compute, using the various isometry identities, e.g.,
 $D_{V_2^*} V_2 = 0$, $V_2^* D_{V_2^*} = 0$, $V_2^* V_2 = I$, $D_{V_1}^2 = D_{V_1}$,  $V_2^* V_2 = I$,
\begin{align*}
& U_*^* P_* U_*  = (V_1 D_{V_2^*} + D_{V_1^*} V_2^*) V_2 D_{V_1^*} V_2^* (D_{V_2^*} V_1^* + V_2 D_{V_1^*}) \\
& \quad  = D_{V_1^*} V_2^* (V_2 D_{V_1^*} V_2^*) V_2 D_{V_1^*} = D_{V_1^*} (V_2^* V_2) D_{V_1^*} (V_2^* V_2) D_{V_1^*} = D_{V_1^*}^3   = D_{V_1^*}  \\
& \quad  = P
\end{align*}
and \eqref{BCL2can-flip} follows.

It remains to argue that the same formulas hold in case the commuting isometric pair $(V_1, V_2)$ is not presented in a BCL1 or BCL2 model form.  However by
Theorem  \ref{Thm:BCLmodel} we know that any commuting isometric pair $(V_1, V_2)$ (with $V = V_1 V_2$ having no unitary part) is unitarily equivalent to a
BCL2-model pair
$$
(V_1', V_2') = (M_{U^*_* P_*^\perp + z U^*_* P_*}, M_{P_*U_* + z P_*^\perp U_*})
$$
on $H^2(\cF')$, with $\cF' = \cD_{V^*}$.  It is now a routine observation that the unitary identification map $\omega \colon \cH \to H^2(\cF')$ washes through
all the formulas in \eqref{BCL2canonical} so as to give a BCL2 tuple $(\cF', U_*', P_*')$ unitarily equivalent to the BCL2 tuple $(\cF, U_*, P_*)$ built directly from the
model commuting isometric pair $(V_1', V_2')$.

\end{proof}

To facilitate computation, it is convenient to replace the space $\cD_{V^*}$ known to have the internal direct-sum decompositions \eqref{ortho-decom}
with the external direct sum $\sbm{ \cD_{V_1^*} \\ \cD_{V_2^*}}$ via either of the identification maps
$$
  \Phi = \begin{bmatrix} D_{V_1^*} & V_1 D_{V_2^*} \end{bmatrix} \colon \begin{bmatrix} \cD_{V_1^*} \\ \cD_{V_2^*} \end{bmatrix} \to \cD_{V^*},
  \quad \Phi_* = \begin{bmatrix} V_2 D_{V_1^*} & D_{V_2^*} \end{bmatrix}   \colon \begin{bmatrix} \cD_{V_1^*} \\ \cD_{V_2^*} \end{bmatrix} \to \cD_{V^*}.
$$
The same idea appears in \cite[Corollary 7.2]{BDF3}.  We then have the following result.

\begin{corollary}  \label{C:BCLdata}  Given a commuting isometric pair $(V_1, V_2)$ such that $V: = V_1 V_2$ is a shift, an alternative explicit BCL2-tuple
for $(V_1, V_2)$ is
\begin{equation}   \label{BCL2tuple'}
(\cF_*, P_*, U_* ) = \left( \begin{bmatrix} \cD_{V_1^*} \\ \cD_{V_2^*} \end{bmatrix},
\begin{bmatrix} I & 0 \\ 0 & 0 \end{bmatrix},  \begin{bmatrix} D_{V_1^*} V_2|_{\cD_{V_1^*} }  & D_{V_1^*} |_{\cD_{V_2^*}} \\
V_1^* V_2 |_{\cD_{V_1^*}}  & V_1^*|_{\cD_{V_2^*}} \end{bmatrix} \right).
\end{equation}
\end{corollary}

\begin{proof}   We have already noted that
$$
( \cF_*', P_*', U_*')  = (\cD_{V^*}, V_2 D_{V_1^*} V_2^*, (D_{V_2^*} V_1^* + V_2 D_{V_1^*})|_{\cD_{V^*}})
$$
is a BCL2-tuple for $(V_1, V_2)$.  One can see as a consequence of the  identities \eqref{defid} that $\Phi_* $ acting from $ \sbm{ \cD_{V_1^*} \\ \cD_{V_2^*} }$ to $ \cD_{V^*}$
given by
$$
\Phi_* = \begin{bmatrix} V_2 D_{V_1^*}  & D_{V_2^*} \end{bmatrix} \colon \begin{bmatrix} \cD_{V_1^*} \\ \cD_{V_2^*} \end{bmatrix} \to \cD_{V^*}
$$
is unitary.  It is then a matter of again repeatedly using the identities \eqref{defid} to see that
\begin{align*}
 \Phi_*^* P_*'  \Phi_* & =  \begin{bmatrix} D_{V_1^*} V_2^* \\ D_{V_2^*} \end{bmatrix} V_2 D_{V_1^*} V_2^* \begin{bmatrix} V_2 D_{V_1^*}  & D_{V_2^*} \end{bmatrix}
 = \begin{bmatrix} I_{\cD_{V_2^*}} & 0 \\ 0 & 0 \end{bmatrix}  = P_*, \\
 \Phi_*^* U_*' \Phi_* & = \begin{bmatrix} D_{V_1^*} V_2^* \\ D_{V_2^*} \end{bmatrix}
(D_{V_2^*} V_1^* + V_2 D_{V_1^*} ) \begin{bmatrix} V_2 D_{V_1^*} & D_{V_2^*} \end{bmatrix} \\
& = \begin{bmatrix} D_{V_1^*} V_2 D_{V_1^*}  & D_{V_1^*} V_2^* D_{V_2^*}  \\
 D_{V_2^*} V_1^* V_2 D_{V_1^*} & D_{V_2^*} V_1^* D_{V_2^*} \end{bmatrix}
 \end{align*}
 Let us note on the side that
 \begin{align*}
 D_{V_2^*} V_1^* V_2 D_{V_1^*} & = (I - V_2 V_2^*) V_1^* V_2 D_{V_1^*} = V_1^* V_2 D_{V_1^*} - V_2 V_1^* V_2^* V_2 D_{V_1^*} \\
 & = V_1^* V_2 D_{V_1^*} - V_2 V_1^* D_{V_1^*} = V_1^* V_2 D_{V_1^*}.
 \end{align*}
 Hence the formula for $\Phi_*^* U_*' \Phi_*$ can be completed to
 $$
 \Phi_*^* U_*' \Phi_*  = \begin{bmatrix} D_{V_1^*} V_2 D_{V_1^*}  & D_{V_1^*} V_2^* D_{V_2^*}  \\
V_1^* V_2 D_{V_1^*}  & D_{V_2^*} V_1^* D_{V_2^*} \end{bmatrix} =
 \begin{bmatrix} D_{V_1^*} V_2|_{\cD_{V_1^*} }  & D_{V_1^*} |_{\cD_{V_2^*}} \\
V_1^* V_2 |_{\cD_{V_1^*}}  & V_1^*|_{\cD_{V_2^*}} \end{bmatrix}  = U_*.
$$
As $(\cF_*, P_*, U_*)$ is unitarily equivalent to the known BCL2-tuple $(\cF'_*, P_*', U_*')$ for $(V_1, V_2)$, it follows that
$(\cF_*, P_*, U_*)$ is also a BCL2-tuple for $(V_1, V_2)$.
\end{proof}

Finally the following model-characterization of joint reducing subspaces for a BCL-model pair of commuting
isometries will be useful in the sequel.

\begin{remark}  \label{R:BCLunitary}
We show here how to use the spectral theory for unitary operators to work out a BCL-model for the commuting 
unitary operator-pair $(W_1, W_2)$ as appearing in the second component of the general BCL-model for a pair 
of commuting isometries as follows.   Suppose that $(W_1, W_2)$ is a pair of commuting unitary operators
on a Hilbert space $\cH$.  Then the product $W = W_1 W_2 = W_2 W_1$ is also unitary, and hence, by the direct-integral version
of the spectral theorem for normal operators (see \cite[Theorem II.6.1]{Dix}), $W$  can be represented as a diagonalized operator on
the direct integral space $\bigoplus \int_{\mathbb T}
\cH(\zeta)\, {\tt d}\nu(\zeta)$  with the fiber space $\cH(\zeta)$ having dimension equal to the multiplicity
function $n(\zeta) = \dim \cH(\zeta)$ well-defined $\nu$-a.e.   For $j=1,2$, the unitary operator $W_j$
commutes with $W$ and hence is {\em decomposable} (see \cite[Theorem II.2.1]{Dix},
meaning that there are measurable  operator-valued functions
$\zeta \mapsto \varphi_j(\zeta) \in \cB(\cH(\zeta))$ so that $W_j$ is represented as a multiplication operator
$$
   W_j = M_{\varphi_j} \colon h(\zeta) \mapsto  \varphi_j(\zeta) h(\zeta).
 $$
 As $M_{\varphi_j}$ is unitary, it must be the case that the multiplier value $\varphi_j(\zeta)$ is unitary on $\cH(\zeta)$
 for a.e. $\zeta \in {\mathbb T}$.  As $W_1 W_2 = W_2 W_1 = W$, it then must also be the case that
 \begin{equation}   \label{phi-product}
   \varphi_1(\zeta) \varphi_2(\zeta) = \zeta I_{\cH(\zeta)} \text{ for a.e. } \zeta \in {\mathbb T}.
 \end{equation}
 To parametrize the set of all such pairs $(\varphi_1, \varphi_2)$ simply let $\varphi_1$ be an arbitrary measurable 
 unitary-operator-valued  function $\zeta \mapsto  \varphi_1(\zeta)$.  Then we may solve \eqref{phi-product}
 to see that $\varphi_2(\zeta)$ is unique and is given by
 $$
   \varphi_2(\zeta) = \zeta \cdot \varphi_1(\zeta)^*.
 $$
Thus $(W_1, W_2) = (M_{\varphi_1},  \zeta \cdot M_{\varphi_1^*})$ with $M_{\varphi_1}$ equal to an 
arbitrary unitary decomposable operator
on $\bigoplus \int_{\mathbb T}  \cH(\zeta) {\tt d} \nu(\zeta)$ is the form for an arbitrary pair of unitary operators
on $\bigoplus \int_{\mathbb T} \cH(\zeta) {\tt d} \nu$ having product $W = W_1 W_2$ equal to $M_{\zeta I_{\cH(\zeta)}}$
on $\bigoplus \int_{\mathbb T} \cH(\zeta) {\tt d} \nu(\zeta)$.
\end{remark}

We next seek a characterization of the joint reducing subspaces for the shift part of a  
commuting isometric pair $(V_1, V_2)$ in terms of the associated BCL2 model:
\begin{equation}  \label{BCL-shift-model}
 (V_1, V_2) = (M_{U^*(P^\perp + z P)}, \, M_{(P + z P^\perp) U}) \text{ on } H^2(\cF)
\end{equation}
for a BCL tuple $(\cF, P, U)$.  It is convenient to first introduce a definition.

\begin{definition} Suppose $(\cF, P, U)$ is a BCL-tuple (with commuting unitary operators $W_1, W_2$ assumed
to be trivial).  Suppose that $\cF_0$ is a subspace of $\cF$ such that
\begin{itemize}
\item[(i)] $\cF_0$ is invariant for $P$, and
\item[(ii)] $\cF_0$ is reducing for $U$.
\end{itemize}
Set $P_0  = P|_{\cF_0}$ and $U_0  = U|_{\cF_0}$. Then we say that the And\^o tuple $(\cF_0, P_0, U_0)$
is a reduced sub-And\^o tuple of $(\cF, P, U)$.
\end{definition}

Then we have the following result.

\begin{theorem} \label{T:reduced-sub-tuple}  Suppose that $(V_1, V_2)$ is the BCL2 model
commuting isometric pair \eqref{BCL-shift-model} associated with the And\^o tuple $(\cF, P, U)$.
Then joint reducing subspaces for $(V_1, V_2)$ are in one-to-one correspondence with
reducing sub-And\^o tuples $(\cF_0, P_0, U_0)$ with associated reducing subspace equal to
$H^2(\cF_0)$ viewed as a subspace of $H^2(\cF)$ in the natural way.
\end{theorem}

\begin{proof}
Suppose that $\cM \subset H^2(\cF)$ is a joint reducing subspace for $(V_1, V_2)$ as in \eqref{BCL-shift-model}.
Then in particular $\cM$ is reducing for $V = V_1 V_2 = M_z^\cF$ on $H^2(\cF)$.
The Beurling-Lax theorem characterizes the invariant subspaces $\cM$ for the shift operator $M_z^\cF$ on a vectorial Hardy space $H^2(\cF)$ as those of the form $\Theta \cdot H^2(\cE)$ for a inner function $\Theta$
(i.e., an $\cB(\cU, \cY)$-valued function $z \mapsto \Theta(z)$ on the unit disk with radial-limit boundary-value function $\zeta \mapsto \Theta(\zeta)$ having isometric values a.e. on ${\mathbb T}$).  
If $\cM$ is reducing for $M_z$ then both $\cM$ and $\cM^\perp$ have Beurling-Lax representations
$$
  \cM = \Theta \cdot H^2(\cE), \quad \cM^\perp = \Psi \cdot H^2(\cE')
$$
for inner functions $\Theta$ and $\Psi$ with values in $\cB(\cE, \cF)$ and $\cB(\cE', \cF)$  for appropriate coefficient Hilbert spaces $\cE$ and $\cE'$ respectively.
Furthermore we have the orthogonal decomposition
$$
   H^2(\cF) = \cM \oplus \cM^\perp = \Theta  H^2(\cE) \oplus \Psi H^2(\cE')
 $$
 implying that $\begin{bmatrix} \Theta & \Psi \end{bmatrix}$ is also inner as a  function with values in $\cB(\sbm{ \cE \\ \cE'}, \cF)$.
 As in general the Beurling-Lax representer for a given shift-invariant subspace $\cM$ is unique up to a unitary-constant right factor, we see from all this that
 $\begin{bmatrix} \Theta & \Psi \end{bmatrix}$ is a unitary constant from $\sbm{ \cE \\ \cE'}$ onto $\cF$.  In particular we see that $\Theta$ must be equal to a constant
 $\Theta(z)  = \Theta(0)$ isometric embedding of $\cE$ onto a subspace $\cF_0$ of $\cF$ and 
 $\cM$ has the form 
$\cM = H^2(\cF_0) \subset H^2(\cF)$ for the subspace $\cF_0 = \Theta(0) \cE$ of the coefficient space $\cF$.

It remains to understand when a subspace of this form is also reducing for $V_1$ and $V_2$.
Let $f(z) = f_0$ where $f_0 \in \cF_0$.  Then $V_1 \colon f(z) \mapsto U^* P^\perp f_0 + z U^* P f_0 \in
H^2(\cF_0)$ forces $U^* P^\perp f_0 \in \cF_0$, $U^* P f_0 \in \cF_0$, i.e., 
\begin{itemize}
\item[(i)] invariance of $H^2(\cF_0)$ under $V_1$ implies invariance of $\cF_0$  under $U^* P^\perp$ and $U^* P$.  
\end{itemize}
Similarly, 
\begin{itemize}
\item[(ii)] invariance of $H^2(\cF_0)$ under $V_1^*$ implies invariance of $\cF_0$ under $P^\perp U$,
\item[(iii)] invariance of $H^2(\cF_0)$ under $V_2$ implies invariance of $\cF_0$ under $PU$ and $P^\perp U$.
\item[(iv)] invariance of $H^2(\cF_0)$ under $V_2^*$ implies invariance of $\cF_0$ under $U^*P$.
\end{itemize}
By summing the two operators in item (i) and in item (ii) respectively, we see that $\cF_0$ is invariant under
$U^*$ and under $U$.   Then from either (ii) or (iv) we see that $\cF_0$ is invariant under $P$  or $P^\perp$
(and hence also under $P^\perp = I - P$ or $P = I - P^\perp$).    Conversely, if $\cF_0$ is invariant under
$U$, $U^*$, and $P$ (and hence also $P^\perp$), it is routine to verify by direct computation that
$H^2(\cF_0)$ is invariant under all of $V_1$, $V_1^*$, $V_2$, $V_2^*$, and hence is jointly reducing
for $(V_1, V_2)$.  Then restriction of $(V_1, V_2)$ to $H^2(\cF_0)$ amounts to the BCL2 model 
corresponding to the reduced sub-BCL tuple $(\cF_0, P_0, U_0)$ as expected.
\end{proof}

\begin{remark}  \label{R:unitary-reducing} A result parallel to Theorem \ref{T:reduced-sub-tuple} can be obtained
for the direct-integral model for a unitary commuting pair $(W_1, W_2)$ as in Remark \ref{R:BCLunitary}.
The result is:  {\em a subspace $\cM$ of $\bigoplus \int_{\mathbb T} \cH(\zeta) \, {\tt d}\nu(\zeta)$ is reducing
for the commuting unitary pair $(W_1, W_2) = (M_{\varphi_1}, \zeta M_{\varphi_1^*})$ on 
$\bigoplus \int_{\mathbb T} \cH(\zeta) {\tt d} \nu(\zeta)$ if and only if  $\cM$ has the form
$$
  \cM = \bigoplus \int_{\mathbb T} P(\zeta) \cH(\zeta)\, {\tt d}\nu(\zeta)
$$
where $\zeta \mapsto P(\zeta)$ is a measurable function with $P(\zeta)$ equal to a orthogonal projection on
$\cH(\zeta)$  which is reducing for $\varphi_1(\zeta)$ (and hence also for $\varphi_2(\zeta) = \zeta \cdot \varphi_1(\zeta)^*$)
for a.e.~$\zeta \in {\mathbb T}$.}  Combining this result with Theorem \ref{T:reduced-sub-tuple} then leads
to a characterization of the reducing subspaces for a general BCL-model as in Theorem \ref{Thm:BCLmodel}.
\end{remark}

\section[Commuting unitary extensions]{Commuting unitary extension of a commuting pair of isometries}

 It is well known (see \cite[Section I.6]{Nagy-Foias}) that an arbitrary family of commuting isometries
 can always be extended to a family of
 commuting unitaries. The following result shows that when the family is finite and one of the isometries
 in the family is the product of the rest of the isometries, then the family can be extended to a family of commuting
 unitaries with additional structure.

\begin{lemma} \label{L:special-ext}
Let $(V_1,V_2)$ be a pair of commuting isometries. Then $(V_1,V_2)$ has a commuting unitary
extension $(Y_1,Y_2)$ such that $Y=Y_1Y_2$ is the minimal unitary extension of $V=V_1V_2$.
\end{lemma}

\begin{proof}
Theorem \ref{Thm:BCLmodel} plays a pivotal role in the proof of this result. We can assume without loss of generality that $\cH=\sbm{H^2(\cF) \\ \cH_u}$ and
$$
(V_1,V_2,V_1V_2)= \left( \sbm{M_{P^\perp U +zP U} & 0 \\ 0 &  W_1},
\sbm{ M_{U^* P+z U^* P^\perp } & 0 \\ 0 &  W_2}, \sbm{M_z & 0 \\ 0 &  W} \right),
$$
where $(\cF,P,U,W_1,W_2)$ is a BCL2 tuple for $(V_1,V_2)$. Now define a pair of  operators $Y_1, Y_2$
on $L^2(\mathcal F)\oplus\mathcal{H}_u$ by
  $$
  (Y_1,Y_2):=\left( \sbm{M_{P^\perp U+ \zeta PU} & 0 \\ 0 & W_1}, \sbm{M_{U^*P+\zeta U^*P^\perp} & 0 \\
  0 &  W_2} \right).
  $$
where $\zeta$ is the coordinate variable on the unit circle ${\mathbb T}$.
  Then one can check that $(Y_1,Y_2)$ is a pair of commuting unitaries and that $(Y_1,Y_2)$ is an extension of
  $(V_1,V_2)$, where $\sbm{H^2(\cF) \\  \cH_u}$ is identified as a subspace of 
  $\sbm{ L^2(\mathcal F) \\ \mathcal{H}_u}$
  via
   \begin{align*}
\begin{bmatrix}  z^n\xi \\ \eta  \end{bmatrix} \mapsto  \begin{bmatrix} \zeta^n \xi \\ \eta  \end{bmatrix}
\text{ for all }\xi\in\mathcal{F} \text{ and }\eta\in\cH_u \text{ for } n = 0,1,2,\dots.
   \end{align*}
   Then
  $Y=Y_1Y_2=\sbm{ M_\zeta & 0 \\ 0 &  W}$ on $\sbm{ L^2(\mathcal F) \\ \mathcal{H}_u}$ is clearly the 
  minimal unitary extension of $V=V_1V_2=\sbm{ M_z & 0 \\ 0 &  W}$ on $\sbm{ H^2(\mathcal F) \\ \mathcal{H}_u}$.
\end{proof}

\section{Doubly commuting pairs of isometries}

Arguably (see \cite{BKPS}), the BCL-model for a commuting pair of isometries has proven to be of limited utility for
understanding the finer geometric structure
of a commuting pair of isometries.  Consequently, there has been some investment in the use of other approaches (beginning with
multivariable analogs of the {\em Wold decomposition}) toward this goal  (see \cite{Burdak, BKS, SY, Slo1980}).  While the most general case
still remains mysterious, a particularly tractable special case is the case of a {\em doubly commuting isometric pair}\index{double commutativity}, i.e., a commuting pair of
isometries $(V_1, V_2)$ such that $V_1^* V_2 = V_2 V_1^*$ (and hence also $V_2^* V_1 = V_1 V_2^*$);
see \cite{SarkarLAA, Mandrekar1988, SSW}.
The next result characterizes the double commutativity property for a commuting isometric pair in terms of an associated
BCL2 tuple $(\cF, P, U, W_1, W_2)$ for $(V_1, V_2)$;.
this characterization was already observed by Berger-Coburn-Lebow \cite{B-C-L} with further elaboration by Ga\c{s}per-Ga\c{s}per
\cite{GG}, Bercovci-Douglas-Foias \cite[Proposition 2.10]{BDF1}, and Bhattacharyya-Rostogi-Kashari \cite[Lemma 3.2]{BRK}. We include yet another proof which fits in with
the ideas here.

\begin{theorem}  \label{T:doublycomisom}
Let $(V_1,V_2)$ be a pair of commuting isometries and let $(\cF,P,U,$ $W_1,W_2)$ be a choice of BCL2 tuple for $(V_1, V_2)$.
Then $(V_1,V_2)$ is doubly commuting if and only if
\begin{equation}   \label{DC-BCLtuple}
P^\perp UP =0, \text{ i.e., } \operatorname{Ran} P \text{ is invariant for } U.
\end{equation}
\end{theorem}

\begin{proof}
By Theorem \ref{Thm:BCLmodel} we assume without loss of generality  that $(V_1, V_2)$ is given by a BCL2 model:
$$
(V_1,V_2)=(M_{U^* P^\perp +z U^* P} \oplus W_1, M_{ PU +zP^\perp U} \oplus W_2)
$$
on $H^2(\cF) \oplus \cH_u$.   As commuting unitaries are automatically doubly commuting, we see that
$V_1$ double commutes with $V_2$ if and only if  $V_{1s}: = M_{U^*(P^\perp +zP)}$ double commutes with
$V_{2s}: = M_{(P+zP^\perp)U}$.

It is convenient to view $H^2(\cF)$ as the tensor product Hilbert space $H^2 \otimes \cF$ and then to write
\begin{align*}
& V_{1s} = M_{U^* P^\perp  + z U^* P} = I_{H^2} \otimes U^* P^\perp   + M_z \otimes U^* P, \\
& V_{2s} = M_{PU+z P^\perp U} = I_{H^2} \otimes P U  +  M_z \otimes  P^\perp U.
\end{align*}
We may then compute
\begin{align*}
& V_{1s}^* V_{2s}  = \left( (I_{H^2} \otimes P^\perp U) + (M_z^* \otimes PU) \right)
\left( (I_{H^2} \otimes PU) + (M_z \otimes P^\perp U ) \right)  \\
& = (I_{H^2} \otimes P^\perp U P U )  + (M_z \otimes P^\perp U P^\perp U )  +
     (M_z^* \otimes PUPU) + (I_{H^2} \otimes P U P^\perp U )
\end{align*}
while
\begin{align*}
& V_{2s} V_{1s}^* = \left( (I_{H^2} \otimes PU) + (M_z \otimes P^\perp U) \right)
\left( (I_{H^2} \otimes P^\perp U) + (M_z^* \otimes PU ) \right)  \\
& =(I_{H^2} \otimes  PUP^\perp U)  + (M_z \otimes P^\perp U P^\perp U )
+(M_z^* \otimes PUPU ) + (M_z M_z^* \otimes P^\perp U P U).
\end{align*}
Thus
\begin{align*}
&V_{1s}^* V_{2s} - V_{2s} V_{1s}^*  =  \\
& \quad I_{H^2} \otimes P^\perp U P U  + M_z \otimes P^\perp U P^\perp U  + M_z^* \otimes PUPU  + I_{H^2} \otimes PUP^\perp U  \\
& \quad - I_{H^2} \otimes PUP^\perp U  - M_z \otimes P^\perp U P^\perp U   -  M_z^* \otimes PUPU  - M_z M_z^* \otimes P^\perp U P U  \\
& = (I - M_z M_z^*) \otimes P^\perp U P U.
\end{align*}
As $ I_{H^2} - M_z M_z^*  = \bev_0^* \bev_0$ (where $\bev_0$ is the evaluation-at-0 map)  and $\bev_0^* \bev_0$ is the projection on the constant functions on  $H^2$ and hence is not zero, we see that $V_{1s}$ double commutes with $V_{2s}$ exactly when
$P^\perp U P U = 0$.  As $U$ is unitary, an equivalent formulation is $P^\perp U P = 0$.   The theorem now follows.
\end{proof}

As an illustration of Theorem \ref{T:doublycomisom} we now compute the BCL2 model for
a standard example of doubly commuting isometries, namely the bidisk shift operators $(M_{z_1}, M_{z_2})$ acting on $H^2_{{\mathbb D}^2}$.

\begin{example}  \label{E:BCLbidisk}  Consider the commuting pair of shift operators $(V_1, V_2) = (M_{z_1}, M_{z_2})$ acting on the Hardy space 
over the bidisk \index{$H^2_{{\mathbb D^2}}$}
$$
H^2_{{\mathbb D^2}}:= \{ f(z_1, z_2) = \sum_{(m,n) \in {\mathbb Z}^2_+} a_{ij} z_1^i z_2^j \colon \sum_{(m,n) \in {\mathbb Z}^2_+} | a_{ij}|^2 < \infty\}.
$$
Note that the operators $M_{z_1}$ and $M_{z_2}$ are shifts on $H^2_{{\mathbb D}^2}$ so the $(W_1, W_2)$-component in a BCL2 tuple for $(M_{z_1}, M_{z_2})$ is trivial.
We shall show: {\em  a BCL2 tuple for  $(M_{z_1}, M_{z_2})$ is
\begin{equation}   \label{BCLtuple-bidisk}
(\cF, P, U)  = (\ell^2_{\mathbb Z}, P_{\ell^2_{[1, \infty)}}, \bS)
\end{equation}
where $\ell^2_{\mathbb Z}$\index{$\ell^2_{\mathbb Z}$} is  the space of absolutely square-summable sequences  indexed
by the integers ${\mathbb Z}$,  $P_{\ell^2_{[1, \infty)}}$ is the orthogonal projection
on $\ell^2_{\mathbb Z}$ with range equal to the subspace of sequences supported on the subset
$\{ n \in {\mathbb Z} \colon 1 \le n \}$, and $\bS$\index{$\bS$} is the bilateral shift operator
$$
  \bS \colon \be_n \mapsto \be_{n+1}
$$
on $\ell^2_{\mathbb Z}$ (where $\{ \be_n \colon n \in {\mathbb Z}\}$ is the standard orthonormal basis for $\ell^2_{\mathbb Z}$).\index{$\be_j$}}

To construct a BCL2 model for $(V_1, V_2)$ according to the construction in the proof of Theorem \ref{Thm:BCLmodel}, we need to compute the wandering subspace for the shift $V_1V_2 = M_{z_1 z_2}$.  Note that $\operatorname{Ran} M_{z_1 z_2}$
consists of functions with Taylor coefficients $a_{ij}$ supported on the set $\{(i,j) \in {\mathbb Z}^2_+ \colon i,j \ge 1\}$.  Hence
$\cF = ( \operatorname{Ran} M_{z_1 z_2})^\perp$ is the subspace
\begin{equation}  \label{Fbd}
 \cF = \{ f(z_1, z_2) =  a_{00} + \sum_{i>0} a_{i0} z_1^i + \sum_{j>0} a_{0j} z_2^j \colon |a_{00}|^2 + \sum_{i> 0} |a_{i0}|^2 + \sum_{j > 0} | a_{0j} |^2 < \infty\}.
 \end{equation}
It is convenient to identify $\cF$ with $\ell^2_{\mathbb Z}$  via the map
$\tau_{\rm bd}$ (the subscript bd suggesting {\em bidisk}) defined on the orthonormal basis of monomials for $\cF$ according to the formula\index{$ \tau_{\rm bd}$}
\begin{equation}   \label{taubd}
 \tau_{\rm bd} \colon z_1^i \mapsto \be_{-i} \text{ for } i \ge 0, \quad \tau_{\rm bd} \colon z_2^j \mapsto \be_j \text{ for } j \ge 0.
 \end{equation}
 We wish to extend $\tau_{\rm bd}$ to a map
 from all of $H^2_{{\mathbb D}^2}$ to $H^2(\ell^2_{\mathbb Z}) : = H^2 \otimes \ell^2_{\mathbb Z}$ so that we have the intertwining
 $\tau_{\rm bd} M_{z_1 z_2} = M_z \tau_{\rm bd}$.  Thus we require that
 $$
   \tau_{\rm bd} \left((z_1 z_2)^k z_1^i\right) = z^k \tau_{\rm bd}( z_1^i)   = z^k \be_{-i}, \quad
   \tau_{\rm bd}\left( (z_1 z_2)^k z_2^j\right) = z^k \tau_{\rm bd}  (z_2^j)  = z^k \be_{j}
 $$for $i \ge 0$ and $j \ge 0$,
 or in a more closed form,
 \begin{equation}  \label{taubd-gen}
   \tau_{\rm bd} \colon z_1^i z_2^j \mapsto \begin{cases} \be_{j-i} z^j &\text{for } i \ge j, \\
             \be_{j-i} z^i &\text{for } i \le j.  \end{cases}
\end{equation}
As $\tau_{\rm bd}$ so defined is a well-defined bijection from an orthonormal basis for $H^2_{{\mathbb D}^2}$ to an
orthonormal basis for $H^2(\ell^2_{\mathbb Z})$,
$\tau_{\rm bd}$ extends to a well-defined unitary map from the scalar-valued Hardy space over the bidisk
$H^2_{{\mathbb D}^2}$ onto the $\ell^2_{\mathbb Z}$-valued
Hardy space over the disk $H^2(\ell^2_{\mathbb Z})$ which satisfies the intertwining property
$$
    \tau_{\rm bd} M_{z_1 z_2} = M_z \tau_{\rm bd}.
$$
By the construction in the proof of Theorem \ref{Thm:BCLmodel} we are guaranteed that there is a projection operator $P$ and a unitary operator $U$ on
$\cF \cong \ell^2_{\mathbb Z}$ so that
$$
 \tau_{\rm bd} M_{z_1} = M_{U^*(P^\perp + z P)} \tau_{\rm bd}, \quad
 \tau_{\rm bd} M_{z_2} = M_{(P + z P^\perp) U} \tau_{\rm bd}.
 $$
 Once one discovers the candidate, it is a matter of direct checking to see that $P = P_{\ell^2_{[1, \infty)}}$, $U = \bS$
 on $\cF = \ell^2_{\mathbb Z}$ does the job.
 Note that $\operatorname{Ran} P = \ell^2_{[1, \infty)}$ is invariant under $U = \bS$, as is to be expected from Theorem \ref{T:doublycomisom} since $(M_{z_1},
 M_{z_2})$ is doubly commuting.
\end{example}

We can use the result of Theorem \ref{T:doublycomisom}  combined with Example \ref{E:BCLbidisk} to obtain the following Wold decomposition
for a doubly commuting pair of isometries due to  S\l oci\'nski \cite[Theorem 3]{Slo1980}.
We present a new proof using the structure of the BCL2 model for doubly commuting isometries given by Theorem \ref{T:doublycomisom}
combined with the classical Wold decompositions for 
$U|_{\operatorname{Ran} P}$ and $U^*|_{\operatorname{Ran} P^\perp}$ and recognition
of the BCL2 model for the bidisk shift-pair given in Example \ref{E:BCLbidisk}.  Similar results have been obtained by Ga\c{s}per-Ga\c{s}per \cite{GG}.

\begin{theorem}  \label{T:doublycomisom'}
Suppose that $(V_1, V_2)$ is a doubly commuting pair of isometries on the Hilbert space $\cH$.  Then $\cH$ has an orthogonal direct sum
decomposition
\begin{equation}   \label{H-decom}
   \cH = \cH_{\rm dcs} \oplus \cH_{\rm su} \oplus \cH_{\rm us} \oplus \cH_{\rm uu}
\end{equation}
such that
\begin{enumerate}
\item[(i)] each of $\cH_{\rm dcs}$, $\cH_{\rm su}$, $\cH_{\rm us}$, $\cH_{\rm uu}$ is reducing for $V_1$ and $V_2$,

\item[(ii)] $(V_1|_{\cH_{\rm dcs}}, V_2|_{\cH_{\rm dcs}})$ is a doubly commuting pair of shift operators,

\item[(iii)] $(V_1|_{\cH_{\rm su}}, V_2|_{\cH_{\rm su}})$ is a commuting pair of operators such that
$V_1|_{\cH_{\rm su}}$ is a shift operator while $V_2|_{\cH_{\rm su}}$ is unitary,

\item[(iv)]  $(V_1|_{\cH_{\rm us}}, V_2|_{\cH_{\rm us}})$ is a commuting pair of operators such that
$V_1|_{\cH_{\rm us}}$ is unitary while $V_2|_{\cH_{\rm us}}$ is a shift operator,        and

\item[(v)] $(V_1|_{\cH_{\rm uu}}, V_2|_{\cH_{\rm uu}})$ is a commuting pair of unitary operators.
\end{enumerate}
Conversely, any pair of operators $(V_1, V_2)$ on $\cH$ with a decomposition \eqref{H-decom} satisfying conditions {\rm (i)--(v)}
is a doubly commuting pair of isometries.
\end{theorem}

\begin{proof}  Suppose first that $V_1, V_2$ on $\cH$ with $\cH$ decomposing as in \eqref{H-decom} satisfies (i)--(v).  Then clearly
$(V_1, V_2)$ is a commuting pair of isometries since the restriction to each piece is commuting.  The restriction to $\cH_{\rm dcs}$
is doubly commuting by condition  (ii). In general, if $(S, W)$ is a commuting operator pair with $W$ unitary, then
$$
  S^* W = W W^* S^* W = W S^* W^* W = W S^*
$$
and hence $(S, W)$ is in fact doubly commuting (this in fact holds with $W$ any normal operator by the Putnam-Fuglede
theorem -- see \cite{Rosenblum} for a slick proof).  Hence the restrictions of $(V_1, V_2)$ to $\cH_{\rm su}$, $\cH_{\rm us}$,
$\cH_{\rm uu}$ are all doubly commuting as well, and it follows that the full commuting pair $(V_1, V_2)$ is
doubly commuting.

Conversely, suppose that $(V_1, V_2)$ on $\cH$ is a doubly commuting pair of isometries.
By Theorem \ref{T:doublycomisom} $(V_1, V_2)$ is unitarily equivalent to a BCL2 model $(M_{U^*(P^\perp + zP)} \oplus W_1,
M_{(P + zP^\perp)U} \oplus W_2)$ on
$H^2(\cF) \oplus \cK_u$ for some coefficient Hilbert space $\cF$ and a Hilbert space $\cK_u$, where the BCL2 tuple
$(\cF, P, U, W_1, W_2)$ for $(V_1, V_2)$ has the additional property that $P^\perp U P = 0$.  Let us consider the
Wold decomposition for the isometry $U|_{\operatorname{Ran} P}$:
 {\em $\operatorname{Ran} P = \cP_{s} \oplus \cP_{ u}$ with
$\cP_{s}$ and $\cP_{u}$ invariant for $U$ with $U_{\cP_{s}} : = U|_{\cP_{ s}}$ equal to a shift operator and
$U_{\cP_{u}} : = U|_{\cP_{u}}$ equal to a unitary operator.}   Similarly $U^*|_{\operatorname{Ran} P^\perp}$ has a Wold decomposition:
{\em $\operatorname{Ran} P^\perp = \cP_{\perp s} \oplus \cP_{\perp u}$ with $\cP_{\perp s}$ and $\cP_{\perp u}$ invariant for
$U^*$ with $U^*_{\cP_{_{\perp s}} }: = U^*|_{\cP_{_{\perp s}}}$
equal to a shift operator and $U^*_{\cP_{\perp u} }: = U^*|_{\cP_{\perp u}}$ equal to a unitary operator.}  With respect to the decomposition
\begin{equation}   \label{Fdecom}
    \cF = \operatorname{Ran} P^\perp \oplus \operatorname{Ran} P =
    \cP_{\perp u} \oplus \cP_{\perp s} \oplus \cP_{s} \oplus \cP_{ u},
\end{equation}
$U$ has a $(4 \times 4)$-block matrix decomposition of the form
$$
   U = \begin{bmatrix} U_{\cP_{\perp u}} & 0 & 0 & 0 \\ 0 & U_{\cP_{\perp s}} & 0 & 0  \\
   X & Y & U_{\cP_{s}} & 0 \\ Z & W & 0 & U_{\cP_{u}}  \end{bmatrix}
$$
where $U_{\cP_u}$ is unitary, $U_{\cP_s}$ is  a shift, $U_{\cP_{\perp s}}$ is the adjoint of a shift, and $U_{\cP_{\perp u}}$ is unitary.
The fact that $U$, $U_{\cP_u}$ and $U_{\cP_{\perp u}}$ are all unitary forces $X=0$, $Z= 0$, $W=0$ as well as
$$
Y^* Y = I  -  U_{\cP_{\perp s}}^* U_{\cP_{\perp s}}, \quad Y^* U_{\cP_s} = 0, \quad
Y Y^* = I  - U_{\cP_s} U_{\cP_s}^*, \quad Y U^*_{\cP_{\perp s}} = 0.
$$
Hence $Y$ is a partial isometry with initial space equal to
$$
\cD_{U_{\cP_{\perp s}}} = \operatorname{Ran} \,  (I - U_{\cP_{\perp s}}^* U_{\cP_{\perp s}})  \subset \cP_{\perp s}
$$
and with final space equal to
$$
\cD_{U_{\cP_{ s}}^*} =  \operatorname{Ran} \, (I - U_{\cP_{\perp s}} U_{\cP_{\perp s}}^*)  \subset \cP_{s}.
$$
We then use the operator
$$
W|_{\cD_{U_{\cP_{\perp s}}}} \colon    \cD_{U_{\cP_{ \perp s}}} \to \cD_{U_{ \cP_{s}}^*}
$$
as a unitary identification map to identify $\cD_{U_{\cP_{\perp s}}}$ and $\cD_{U_{\cP_{ s}}^*}$ with a common coefficient space which we
shall call $\cE$.  As $U_{\cP_s}$ is shift with wandering subspace identified with $\cE$ while $U_{\cP_{\perp s}}^*$ is a shift with
wandering subspace also identifiable with $\cE$, we may view $\cP_s$ and $\cP_{\perp s}$ as having the respective forms
$$
   \cP_s = \bigoplus_{n=1}^\infty U_{\cP_s}^{n-1} \cE, \quad \cP_{\perp s} = \bigoplus_{m=0}^\infty U_{\cP_{\perp s}}^{*m} \cE.
 $$
 Let us introduce additional unitary identification maps
$$ \tau_- \colon \cP_{\perp s } \to \ell^2_{(-\infty, 0]}(\cE)\mbox{ and }\tau_+ \colon \cP_{s} \to \ell^2_{[1, \infty)}(\cE)$$ given by
\begin{equation} \label{tau}
\tau_- \colon \bigoplus_{n = 0}^{-\infty} U_{\cP_{\perp s}}^{*-n} e_n \mapsto \{ e_{n} \}_{n \le -1}, \quad
\tau_+ \colon \bigoplus_{m=1}^\infty U_{\cP_{ s}}^{m-1} e_m \mapsto \{ e_m \}_{m \ge 1}.
\end{equation}
Then we see that
\begin{equation}  \label{tau-intertwine}
 \begin{bmatrix} \tau_- & 0 \\ 0 & \tau_+ \end{bmatrix} \begin{bmatrix} U_{\cP_{s \perp}} & 0 \\ Y & U_{\cP_{s}} \end{bmatrix} =
 {\mathbf S} \begin{bmatrix} \tau_- & 0 \\ 0 & \tau_+ \end{bmatrix}
\end{equation}
where ${\mathbf S}$ is the bilateral shift operator on $\ell^2_{\mathbb Z}(\cE)$:
$$
{\mathbf S} \colon \{ \be_n \}_{n \in {\mathbb Z}} \mapsto \{ \be_{n+1} \}_{n \in {\mathbb Z}}.
$$
Note that $\sbm{ U_{\cP_s} & 0 \\ Y & U_{\cP_{\perp s}} }$ amounts to $U|_{\cP_s \oplus \cP_{\perp s}}$; we have thus shown to this point
that {\em $U|_{\cP_s \oplus \cP_{\perp s}}$ is unitarily equivalent to the bilateral shift $\mathbf S$ on $\ell^2_{\mathbb Z}(\cE)$.}

Let us now rewrite the decomposition \eqref{Fdecom} as
\begin{equation}   \label{Fdecom'}
\cF = \cP_{\perp u} \oplus \ell^2_{\mathbb Z}(\cE) \oplus \cP_{u}
\end{equation}
where we use the identification
$$
\cP_{\perp s} \oplus \cP_s \cong \ell^2_{\mathbb Z}(\cE)
$$
implemented by the unitary identification map
$$
 \tau = \begin{bmatrix} \tau_- & 0 \\ 0 & \tau_+ \end{bmatrix} \colon \begin{bmatrix} \cP_{\perp s} \\ \cP_{ s} \end{bmatrix}
 \to \begin{bmatrix} \ell^2_{(-\infty,0]}(\cE) \\ \ell^2_{[1, \infty)}(\cE) \end{bmatrix} \cong \ell^2_{\mathbb Z}(\cE)
 $$
 given by \eqref{tau}.  If we let $U_u$ and $U_{\perp u}$ be the unitary operators
 $$
   U_u = U|_{\cP_u}, \quad U_{\perp u} = U|_{\cP_{\perp u}},
 $$
 then as a consequence of \eqref{tau-intertwine} we see that in these coordinates the first three objects in the BCL2 tuple $(\cF, U, P, W_+, W_-)$
 assume the more detailed form
 \begin{equation}   \label{FPUdecom}
  \cF = \begin{bmatrix} \cP_{\perp u} \\ \ell^2_{\mathbb Z}(\cE) \\ \cP_{u} \end{bmatrix}, \quad
  U = \begin{bmatrix} U_{\perp u} & 0 & 0 \\ 0 & {\mathbf S} & 0 \\ 0 & 0 & U_{u} \end{bmatrix}, \quad
  P = \begin{bmatrix} 0 & 0 & 0 \\ 0 & P_{_+} & 0 \\ 0 & 0 & P_{_{\cP_u}} \end{bmatrix}
 \end{equation}
 where $P_+$ is the orthogonal projection on $\ell^2_{\mathbb Z}(\cE)$ with range equal to $\ell^2_{[1, \infty)}(\cE)$
 (considered as the subspace of $\ell^2_{\mathbb Z}(\cE)$ having all coordinates with indices in $(-\infty, 0]$ equal to zero).

 The BCL2 model for $(V_1, V_2)$ is to take
 $$
 V_1 \cong \begin{bmatrix} M_{U^*(P^\perp  + z P )} & 0 \\ 0 & W_1 \end{bmatrix}, \quad
 V_2 \cong \begin{bmatrix} M_{( P + z P^\perp) U} & 0 \\ 0 & W_2 \end{bmatrix}
 $$
 acting on
 $$
 \cH \cong \begin{bmatrix} H^2(\cF) \\ \cK_u \end{bmatrix}.
 $$
 From the decompositions \eqref{FPUdecom} for $\cF$, $P$, $U$ we see that $H^2(\cF) \oplus \cK_u$
 has the finer decomposition
 $$
  \cH \cong H^2(\cF) \oplus \cK_u \cong H^2(\cP_{\perp u}) \oplus H^2(\ell^2_{\mathbb Z}(\cE)) \oplus H^2(\cP_{u}) \oplus \cK_u
  $$
  which split $V_1$ and $V_2$ as four-fold direct sums
  $$
  V_1 \cong V_{1, \perp u} \oplus V_{1,s} \oplus V_{1, u} \oplus W_1, \quad
  V_2 \cong V_{2, \perp u} \oplus V_{2,s} \oplus V_{2,  u} \oplus W_2
  $$
  where
  \begin{align*}
  & V_{1,\perp u} = I_{H^2} \otimes U_{\cP_{\perp u}}^* , \quad V_{2, \perp u} = M_z \otimes U_{\cP_{\perp u}}  \text{ on } H^2(\cP_u), \\
  & V_{1,s} = M_{{\mathbf S}^* (P_+^\perp + z P_+)}, \quad V_{2,s} = M_{(P_+ + z  P_+^\perp) \bS}
  \text{ on } H^2(\ell^2_{\mathbb Z}(\cE)),  \\
  & V_{1, u} = M_z \otimes U_{\cP_u}^* , \quad V_{2,  u} = I_{H^2} \otimes U_{\cP_u} \text{ on } H^2(\cP_{\perp u}),
  \end{align*}
  and where $(W_1, W_2)$ on $\cK_u$ is the commuting pair of unitary operators coming from the original BCL2 tuple
  $(\cF, U, P, W_1, W_2)$ for $(V_1, V_2)$.  It is easily checked that $V_{1, \perp u}$ is a shift operator commuting with the unitary operator
  $V_{2,\perp u}$ and that $V_{1, u}$ is a shift operator commuting with the unitary operator $V_{2,  u}$.
  Let us set
  $$
  \cH_{\rm dcs} = H^2(\ell^2_{\mathbb Z}(\cE)), \quad
  \cH_{\rm us} = H^2(\cP_{ \perp u}), \quad
  \cH_{\rm su} = H^2(\cP_{u}), \quad
  \cH_{\rm uu}  = \cK_u.
 $$
 Then the above analysis shows that conditions (i), (iii), (iv), (v) in Theorem \ref{T:doublycomisom'} are all verified.  Hence it remains only to
 verify condition (ii), i.e., we must show that {\em the operator pair
 \begin{equation}  \label{model1}
    (V_{1,s} = M_{\bS^* (P_+^\perp  + z P_+)}, V_{2,s} = M_{ (P_+ + z  P_+^\perp) \bS})
    \text{ on } H^2(\ell^2_{\mathbb Z}(\cE))
 \end{equation}
 is a doubly commuting pair of shift operators.}  But we recognize \eqref{model1} as just the BCL2 model for the bidisk shift operators $(M_{z_1}, M_{z_2})$
 on $H^2_{{\mathbb D}^2}$ computed in Example \ref{E:BCLbidisk} tensored with the coefficient Hilbert space $\cE$, i.e., \eqref{model1} is the BCL2 model
 for the doubly commuting shift-operator pair $(M_{z_1}, M_{z_2})$ acting on $H^2_{{\mathbb D}^2}(\cE)$.  Thus the theorem follows.
 \end{proof}

 As a corollary we recover the another result of  S\l oci\'nski \cite[Theorem 1]{Slo1980} characterizing doubly commuting
 shift-pairs.

 \begin{corollary}  \label{C:Slo}  A pair of operators $(V_1, V_2)$ is a doubly commuting pair of shift operators if and only if
 $(V_1, V_2)$ is unitarily equivalent to the concrete pair of shift operators $(M_{z_1}, M_{z_2})$ acting on the
 vector-valued Hardy space over the bidisk $H^2_{{\mathbb D}^2}(\cE)$ for some coefficient Hilbert space $\cE$.
 \end{corollary}

 \begin{proof}
 Suppose that $(V_1, V_2)$ is a doubly commuting pair of shift operators.  Then in particular $(V_1, V_2)$ has a
 Wold decomposition as in
 Theorem \ref{T:doublycomisom'}.  But the only piece of this decomposition which involves a pair of shift operators is the piece
 $(V_1|_{\cH_{\rm dcs}},  V_2|_{\cH_{\rm dcs}})$, and as we have seen in the proof of Theorem \ref{T:doublycomisom'},
 this pair in turn is
 unitarily equivalent to $(M_{z_1}, M_{z_2})$ on $H^2_{{\mathbb D}^2}(\cE)$ for some coefficient Hilbert space $\cE$.
\end{proof}

\section[Not doubly-commuting commuting isometries]{Commuting isometries which are not doubly commuting: examples}   \label{S:non-doublycom}

In this section, we look at some examples of commuting isometric pairs which are not doubly commuting. 

\begin{example}  \label{E:1nondoubly}  \textbf{$(M_{z}, M_\Theta)$ on $H^2(\cY)$.}
For this example we take
$$
 \cH = H^2(\cY), \quad V_1 = M_z, \quad V_2 = M_\Theta
 $$
 where $\Theta$ is an inner function with values in $\cB(\cY)$.  This in fact is the functional model of Bercovici-Douglas-Foias for a commuting isometric
 pair $(V_1, V_2)$ such that $V_1$ is a shift (see \cite{BDF1, BDF3}). To get a BCL2-tuple for this $(V_1, V_2)$, we convert the BCL2-tuple given by Corollary \ref{C:BCLdata} for
 a general commuting isometric pair to a more functional form for this specific $(V_1, V_2)$.  The first step is to look at the Wold decomposition for the shift
 operator $V = V_1 V_2 = M_{z \Theta(z)}$ acting on $H^2(\cY)$.  Note that the wandering subspace $\cD_{V^*}$ is given by
 $$
  \cD_{V^*} = H^2(\cY) \ominus z \cdot \Theta(z) H^2(\cY)
$$
which has the internal direct-sum decomposition
$$
\cD_{V^*} = V_2 \cD_{V_1^*} \oplus \cD_{V_2^*} = \Theta \cdot \cY \oplus {\mathfrak H}(\Theta)
$$where ${\mathfrak H}(\Theta)$ is the Sarason/de Branges-Rovnyak model space associated with the inner function $\Theta$:
$$
{\mathfrak H}(\Theta) = H^2(\cY) \ominus \Theta \cdot H^2(\cY).
$$
Then the Wold decomposition for $V = M_{z \Theta(z)}$ on $H^2(\cY)$ has the concrete form
$$
f(z) = \sum_{n=0}^\infty z^n \Theta(z)^n (\Theta(z) y_n + h_n(z))
$$
for any $f \in H^2(\cY)$ where $y_n \in \cY$ and $h_n \in {\mathfrak H}(\Theta)$ are determined by
$$
  \Theta(z) y_n + h_n(z) = P_{\cD_{V^*}} V^{*n} h.
$$
We then define an identification map $\tau_\Theta \colon H^2(\cY) \to H^2\left( \sbm{ \cY \\ {\mathfrak H}(\Theta)} \right) \cong \sbm{ H^2(\cY) \\ H^2({\mathfrak H}(\Theta))}$ by
\begin{equation}   \label{tau-Theta}
\tau_\Theta \colon \sum_{n=0}^\infty z^n \Theta(z)^n (\Theta(z) y_n + h_n(z)) \mapsto
\sum_{n=0}^\infty z^n \begin{bmatrix} y_n \\ h_n(w) \end{bmatrix} =: \begin{bmatrix} y(z) \\ h(z,w) \end{bmatrix}
\end{equation}
where we set
$$
y(z) = \sum_{n=0}^\infty y_n z^n \in H^2(\cY), \quad  h(z,w) = \sum_{n=0}^\infty z^n h_n(w) \in H^2({\mathfrak H}(\Theta)).
$$
Then it is straightforward to see that $\tau_\Theta$ implements the intertwining identity
$$
  \tau_\Theta M_{z \Theta(z)} = M_z \tau_\Theta.
$$
It is less obvious to identify by inspection the inner operator pencils $\Psi_1(z)$ and $\Psi_2(z)$ so that
\begin{equation}  \label{tau-Theta-intertwine}
\tau_\Theta M_z = M_{\Psi_1(z)} \tau_\Theta, \quad \tau_\Theta M_\Theta = M_{\Psi_2(z)} \tau_\Theta.
\end{equation}
However, from the general formulas obtained in Theorem \ref{T:BCLcanonical} and Corollary \ref{C:BCLdata} we know that
\begin{equation}  \label{BCLmodel}
(M_{\Psi_1}, M_{\Psi_2}) \text{ with } \Psi_1(z) = U^* P^\perp + z U^* P, \quad \Psi_2(z) = P U + z P^\perp U
\end{equation}
where we choose
\begin{equation}   \label{BCLdata-BDFmodel}
 P = \begin{bmatrix} I_\cY & 0 \\ 0 & 0 \end{bmatrix}, \quad
 U =  \begin{bmatrix} \bev_{0,\cY} M_\Theta|_\cY  & \bev_{0,\cY}|_{{\mathfrak H}(\Theta)}  \\ M_w^* M_\Theta|_\cY & M_w^*|_{{\mathfrak H}(\Theta)} \end{bmatrix}
 =: \begin{bmatrix} D & C \\ B & A \end{bmatrix} \colon \begin{bmatrix} \cY \\ {\mathfrak H}(\Theta) \end{bmatrix} \to
 \begin{bmatrix} \cY \\ {\mathfrak H}(\Theta) \end{bmatrix}
\end{equation}
does the job (where we now use $w$ as the independent variable for functions in ${\mathfrak H}(\Theta)$), as is seen by specializing the formulas in \eqref{BCL2tuple'}
to the case where
$$
\cD_{V_1^*} = \cY, \quad \cD_{V_2^*} = {\mathfrak H}(\Theta), \quad V_1 = M_z, \quad V_2 = M_{\Theta(z)}.
$$
In summary, formulas \eqref{BCLmodel}, \eqref{BCLdata-BDFmodel} gives the explicit conversion of the BDF-model  $(M_z, M_\Theta)$ to a 
BCL-model ($M_{\Psi_1}, M_{\Psi_2}$).

Let us note next that  the action of $U$ in \eqref{BCLdata-BDFmodel} can be given a more explicit form
\begin{equation}   \label{U-action}
U \colon \begin{bmatrix} y \\ h(w) \end{bmatrix} \mapsto \begin{bmatrix} \Theta(0) y + h(0) \\
\frac{ \Theta(w) - \Theta(0)}{w} y + \frac{ h(w) - h(0)}{w} \end{bmatrix}.
\end{equation}
If we use the identification $\tau$ from  $\cY \oplus \fH(\Theta)$ to $\fH( w \cdot \Theta)$
(where we set $\fH(\Theta) = H^2(\cY) \ominus \Theta \cdot H^2(\cY)$ and similarly $\fH( w \cdot \Theta) = H^2(\cY) \ominus (w \cdot \Theta) H^2(\cY)$
with $w$ here used as the independent variable for functions in
$\fH(\Theta)$ or $\fH( w \cdot \Theta)$ contained in $H^2(\cY)$) given by
$$
 \tau \colon y \oplus h(w) \mapsto y + w h(w),
$$
then we can  get a possibly more convenient BCL2-tuple $(\widetilde \cF, \widetilde P, \widetilde U)$  for $(M_z, M_\Theta)$ on $H^2(\cY)$, namely:
\begin{align}  
&\widetilde \cF = \fH(w \cdot \Theta), \quad  \widetilde P \colon y + w h(w) \mapsto y, \notag  \\
& \widetilde U \colon y + w h(w) \mapsto \big(\Theta(0) y + h(0) \big) + w \bigg( \frac{ \Theta(w) - \Theta(0)}{w} y + \frac{h(w) - h(0)}{w} \bigg).
\label{tildeBCLTuple}
\end{align}

Curiously, from the point of view of system theory, $U$ is just the system matrix for the canonical functional-model
de Branges-Rovnyak transfer-function realization for $\Theta$:
\begin{equation}  \label{Theta-realize}
\Theta(z) = D + z C (I - z A)^{-1} B
\end{equation}
where $D,C,A,B$ are as in \eqref{BCLdata-BDFmodel} (see Theorem 1.2 in \cite{BB-Montreal} for this point of view).
In terms of the  original presentation of $(V_1, V_2)$ as $V_1 = M_z$ and $V_2 = M_{\Theta(z)}$ on $H^2(\cY)$, it is easy to derive an
alternative representation of $\Theta(z)$ as
\begin{equation} \label{Theta-realize'}
\Theta(z) = P_\cY (I - z V_1^*)^{-1} V_2|_\cY.
\end{equation}
 Indeed, let us use the notation $S$ for the shift operator $M_z$ on $H^2$ and then expand
$\Theta(z)$ in its power series representation
$\Theta(z) = \sum_{n=0}^\infty \Theta_n z^n$.  If we identify $\cY$ with the constant functions in $H^2(\cY)$ (whichever is more convenient for the particular
context), we can then write $V_2|_\cY$ as
$$
V_2|_\cY = M_\Theta |_\cY  = \sum_{n=0}^\infty S^n  \Theta_n \colon \cY \to H^2(\cY).
$$
We then can write, for $z$ in the unit disk ${\mathbb D}$,
\begin{align*}
&  P_\cY (I - z V_1^*)^{-1} V_2|_\cY = P_\cY \bigg( \sum_{k=0}^\infty z^k  S^{*k}\bigg) \bigg( \sum_{n=0}^\infty  S^n  \Theta_n  \bigg)  \\
& \quad = P_\cY  \sum_{k\geq0,n\geq k}^\infty z^k S^{n-k} \Theta_n    = \sum_{n=0}^\infty z^n \Theta_n = \Theta(z).
\end{align*}
thereby verifying \eqref{Theta-realize'}.  If we make use of the intertwining relations  \eqref{tau-Theta-intertwine} and use the map $\tau_\Theta M_\Theta|_\cY \colon
y \mapsto \sbm{ y \\ 0}$ to identify the coefficient input/output space $\cY \subset H^2(\cY)$ with the input/output space $\sbm{ \cY \\ 0} \subset \sbm{ H^2(\cY) \\
{\mathfrak H}(\Theta)}$,
 we see that the realization \eqref{Theta-realize'} leads immediately to an alternative realization involving the operators $M_{\Psi_1}$ and $M_{\Psi_2}$
 on $\sbm{ H^2(\cY) \\ H^2({\mathfrak H}(\Theta)) }$:
 \begin{equation}   \label{Theta-realize''}
 \Theta(z) = \begin{bmatrix} P_{ \cY} & 0 \end{bmatrix}  (I - z M_{\Psi_1}^*)^{-1} M_{\Psi_2} \begin{bmatrix} I_\cY \\ 0 \end{bmatrix}.
 \end{equation}
 \end{example}

\begin{example} \label{E:2nondoubly}  \textbf{ $(M_{z_1}, M_{z_2})$ on $H_\diamond({\mathbb D}^2)$.}

We now consider the subspace $\cH_\diamond: = H^2_{{\mathbb D^2}} \ominus \{ \text{constant functions}\}$ of $H^2_{{\mathbb D}^2}$.\index{$\cH_\diamond$}  It is clear that
$\cH_\diamond$ is invariant under $(M_{z_1}, M_{z_2})$ so we can consider the rank-one perturbation of the Example \ref{E:BCLbidisk}, namely the commuting pair
of shift operators
$$
V_1 = M_{z_1}|_{\cH_\diamond}, \quad V_2 = M_{z_2}|_{\cH_\diamond}.
$$
Let us note that this pair $(V_1,V_2)$ is not doubly commuting:  one way to see this is to observe that
$ V_2^*V_1(z_2)=V_2^*(z_1z_2)=z_1\neq 0=V_1V_2^*(z_2)$.

Note that both $V_1$ and $V_2$ are shifts so the $(W_1, W_2)$-component of a BCL2 tuple for $(V_1, V_2)$ is trivial.  We shall show:
{\em a BCL2 tuple for $(M_{z_1}|_{\cH_\diamond}, M_{z_2}|_{\cH_\diamond})$ is
\begin{equation}  \label{BCLtuple-diamond}
(\cF, P, U ) = (\ell^2_{\mathbb Z}, P_{\ell^2_{\{0\} \cup [2, \infty)}}, \bS)
\end{equation}
where $P_{\ell^2_{\{0\} \cup [2, \infty)}}$ is the orthogonal projection of $\ell^2_{\mathbb Z}$ onto the subspace of absolutely square-summable sequences with support
on the subset $\{0\} \cup [2, \infty) \subset {\mathbb Z}$ and where $\bS \colon \be_n \mapsto \be_{n+1}$ is the forward bilateral shift operator on
$\ell^2_{\mathbb Z}$.} Note that the criterion for double commutativity fails by one-dimension:  while $\operatorname{Ran} P = \ell^2_{\{0\} \cup [2, \infty)}$
is not invariant under $U = \bS$, it does have a codimension-one subspace, namely $\ell^2_{[2, \infty)}$, which is $\bS$-invariant, fitting with the fact that
the $(V_1, V_2)$ in this example is only a rank-one perturbation of the $(V_1, V_2)$ in Example \ref{E:BCLbidisk}.

To verify that \eqref{BCLtuple-diamond} is a BCL2 tuple for $(M_{z_1}|_{\cH_\diamond}, M_{z_2}|_{\cH_\diamond})$, proceed as follows.
 Note that elements $f$ of $\cH_\diamond$ have the form
$$
  f(z_1, z_2) = \sum_{i=1}^\infty a_{i0} z_1^i + \sum_{j=1}^\infty a_{0j} z_2^j + \sum_{i,j = 1}^\infty a_{ij} z_1^i z_2^j.
 $$
 Then
 \begin{align*}
 (Vf)(z_1, z_2) & = \sum_{i=1}^\infty a_{i0} z_1^{i+1} z_2 + \sum_{j=1}^\infty a_{0j} z_1 z_2^{j+1} + \sum_{i,j = 1}^\infty a_{ij} z_1^{i+1} z_2^{j+1} \\
& = \sum_{i=2}^\infty a_{i-1,0} z_1^i z_2 + \sum_{j=2}^\infty a_{0,j-1} z_1 z_2^j + \sum_{i,j=2}^\infty a_{i-1, j-1} z_1^i z_2^j.
\end{align*}
so  $\operatorname{Ran} V$ consists of all functions in $H^2_{{\mathbb D}^2}$ with Taylor coefficients supported on the set
$$
{\mathfrak S} = \{ (i,j) \colon i \ge 2 \text{ and }  j=1, \text{ or }  i=1 \text{ and }  j \ge 2, \text{ or } i \ge 2 \text{ and }  j \ge 2\}.
$$
It is now a counting exercise to see that the complement of this set inside ${\mathbb Z}^2_+ \setminus \{(0,0)\}$ is
$$
{\mathfrak S}' = \{ (i,j)\colon i \ge 1 \text{ and } j = 0, \text{ or }  i = 0 \text{ and } j \ge 1, \text{ or }  (i,j) = (1,1) \}.
$$
Thus the space $\cH_\diamond \ominus (\operatorname{Ran} V)^\perp$ can be described as
$$
\cF = \{ f \in \cH_\diamond \colon f(z_1, z_2) = \sum_{i > 0} a_{i0} z_1^i + \sum_{j> 0} a_{0j} z_2^j + a_{11} z_1 z_2 \}.
$$
from which we see that $\cF$ has the set
$$
 \cS_\cF = \{z_1^i \colon i \ge 1\} \cup \{ z_2^j \colon j \ge 1 \} \cup \{ z_1 z_2 \}
$$
as an orthonormal basis.  Let us introduce the map $\tau_\diamond \colon \cF \to \ell^2_{\mathbb Z}$ by defining it to map
the orthonormal basis $\cS_\cF$ for $\cF$ onto the standard orthonormal basis for $\ell^2_{\mathbb Z}$ by\index{$\tau_\diamond$}
\begin{equation}  \label{taudiamond1}
\tau_\diamond \colon z_1^i z_2^j \mapsto \begin{cases}  \be_{-i} &\text{if } (i,j) = (i,0) \text{ with } i > 0, \\
\be_0 &\text{if } (i,j) = (1,1), \\ \be_{j} &\text{if } (i,j) = (0,j) \text{ with } j > 0. \end{cases}
\end{equation}
As $\cF$ is the wandering subspace for the shift operator $M_{z_1 z_2}$ on $\cH_\diamond$,  it follows that $\cH_\diamond$ has as an orthonormal basis
the set $\{ z_1^k z_2^k z_1^i z_2^j \colon k \ge 0, z^i z^j \in \cS_\cF\}$, i.e., an orthonormal basis for $\cH_\diamond$ is
$$
 \cS_\diamond = \{ z_1^k z_2^k z_1^i \colon k \ge 0,\, i> 0\} \cup \{ z_1^{k+1} z_2^{k+1} \colon k \ge 0\} \cup \{ z_1^k z_2^{k+j} \colon k \ge 0,\, j> 0\},
 $$
 or in a more closed form
 $$
 \cS_\diamond = \{ z_1^i z_2^j \colon i > j \ge 0 \text{ or } i=j \ge 1 \text{ or } 0 \le i < j \}.
 $$
 We wish to extend $\tau_\diamond$ to a unitary map from all of $\cH_\diamond$ onto $H^2(\ell^2_{\mathbb Z})$ so that we have the intertwining
 \begin{equation}   \label{taudiamond-intertwine}
   \tau_\diamond M_{z_1 z_2} = M_z \tau_\diamond.
 \end{equation}
 This requires that
 \begin{align*}
& \tau_\diamond ( z_1^k z_2^k z_1^i) = z^k \tau_\diamond (z_1^i) = z^k \be_{-i} \text{ for } k \ge 0, \, i> 0, \\
& \tau_\diamond (z_1^k z_2^k z_1 z_2) = z^k \tau_\diamond (z_1 z_2) = z^k \be_0, \\
& \tau_\diamond (z_1^k z_2^k z_2^j) = z^k \tau_\diamond ( z_2^j) = z^k \be_j \text{ for } j > 0,
 \end{align*}
 or, in better closed form,
 \begin{equation}  \label{taudiamond-monomials}
 \tau_\diamond \colon z_1^i z_2^j \mapsto \begin{cases} z^j \be_{j-i} &\text{if } i > j \ge 0, \\
     z^{i-1} \be_0 = z^{j-1} \be_0 &\text{if } i=j \ge 1, \\
     z^i \be_{j-i} &\text{if } 0 \le i < j.  \end{cases}
 \end{equation}
 Extending $\tau_\diamond$ by linearity to a map $\tau \colon \cH_\diamond \to H^2(\ell^2_{\mathbb Z})$ gives us a unitary identification from
 $\cH_\diamond$ to $H^2(\ell^2_{\mathbb Z})$  \eqref{taudiamond-monomials} satisfying the intertwining \eqref{taudiamond-intertwine}.

 By Theorem \ref{Thm:BCLmodel} we are guaranteed the existence of a projection operator $P$ and a unitary operator $U$ on $\ell^2_{\mathbb Z}$ so that
 we have the intertwinings
 \begin{equation}   \label{diamond-intertwine}
\tau_\diamond M_{z_1} =  M_{U^* P^\perp + zU^* P} \tau_\diamond, \quad
 \tau_{\diamond} M_{z_2} =  M_{PU+zP^\perp U} \tau_\diamond.
 \end{equation}
 Using the above formulas, one can compute that
 \begin{equation} \label{diamondMz1}
  \tau_\diamond M_{z_1} (z_1^i z_2^j) = \begin{cases}  z^i \be_0 &\text{if } i+1 = j \ge 0, \\
  z^j \be_{j-i-1}  &\text{if } i+1 > j >0, \\
  z^{i+1} \be_{j-i-1} &  \text{if } 0 \le i+1 < j.   \end{cases}
  \end{equation}
   Careful bookkeeping making use of the formulas \eqref{taudiamond-monomials}  shows that the three formulas in \eqref{diamondMz1} force the following respective conditions
  on the operator pair $(P,U)$:
  $$ \begin{cases}
    P^\perp \be_1 = \be_1 \text{ and } U^* \be_1 = \be_0, \\
    P \be_0 = \be_0 \text{ and } U^* \be_0 = \be_{-1} \text{ as well as } P^\perp \be_k = \be_k \text{ and } U^* \be_k = \be_{k-1} \text{ for } k<0, \\
    P \be_k = \be_k \text{ and } U^* \be_k = \be_{k-1} \text{ for } k>1
    \end{cases}
 $$
 for which the only solution is $(P, U)$ as in \eqref{BCLtuple-diamond}.
 Alternatively, once one has discovered this candidate,  it is possible to check directly that it satisfies the first intertwining condition in \eqref{diamond-intertwine}.
 By general principles, the second is then automatic, as can also be checked directly.
 \end{example}

%
%
%
%
%
%
%
%
%

\chapter[Models for And\^o lifts]{Models for And\^o lifts of a commuting contractive pair}  \label{C:Ando}

\section{Preliminaries}  \label{S:prelim}
Let $(T_1, T_2)$ be a commuting pair of contraction operators on $\cH$.  We say that the triple
$(\Pi, V_1, V_2)$ is an {\em And\^o isometric lift} or, simply an {\em And\^o lift}\index{And\^o lift}, of $(T_1, T_2)$ if
(i) there is a Hilbert space $\cK$ such that
$\Pi \colon \cH \to \cK$ is an isometric embedding of $\cH$ into $\cK$,  and (ii) $(V_1, V_2)$ is a commuting
pair of isometries on $\cK$ such that $V_j^* \Pi  = \Pi T_j^*$ for $j=1,2$.  We shall be particularly interested in the case 
where the And\^o lift is {\em minimal},\index{And\^o lift!minimal} i.e., the case where the smallest jointly invariant subspace for $(V_1, V_2)$ containing $\operatorname{Ran} \Pi$
 is the whole space $\cK$:
\begin{equation}  \label{minAndoLift}
\cK = \bigvee_{n_1, n_2 \ge 0} V_1^{n_1} V_2^{n_2} \operatorname{Ran} \Pi.
\end{equation}
We say that two such
And\^o  lifts $(\Pi, V_1, V_2)$ with $\Pi \colon \cH \to \cK$  and $(\Pi',  V_1', V_2')$ with
$\Pi' \colon \cH \to \cK'$ are {\em unitarily equivalent}\index{unitary equivalence of two!And\^o lifts}
 if there is a unitary operator
$\tau \colon \cK \to \cK'$ such that
\begin{equation}  \label{equiv-lifts}
\tau \Pi = \Pi', \quad \tau V_1 = V_1' \tau, \quad \tau V_2 = V_2' \tau.
\end{equation}
In the single-variable case it is known that any two minimal isometric lifts are unitarily equivalent (see \cite[Theorem I.4.1]{Nagy-Foias}).
We shall see that this result fails in the bivariate setting
of And\^o lifts for a commuting, contractive pair $(V_1, V_2)$ (see Chapter  \ref{C:classification} to come); more precisely, there are additional invariants which must
be equivalent in the appropriate sense before two minimal And\^o lifts can be unitarily equivalent. As in the single-variable case, 
given an And\^o  lift $(\Pi, V_1, V_2)$, there is always a unitarily equivalent And\^o  lift
$(\Pi', V_1, V_2)$ so that $\cH$ is equal to a subspace of $\cK'$ and $\Pi' \colon \cH \to \cK'$ is just the inclusion
map.  To see this, simply set $\cK' = \sbm{ \cH \\ (I - \Pi \Pi^*)\cK}$ and observe that the map
$$
\tau = \begin{bmatrix} \Pi^* \\ I - \Pi \Pi^* \end{bmatrix} \colon \cK \to \cK'
$$
is unitary.  If we then set
$$
   \quad \Pi' = \sbm{ I_\cH \\ 0 } \colon \cH \to \cK', \quad
V_i' = \tau V_i \tau^* \text{ for } i=1,2
$$
we see that all of conditions \eqref{equiv-lifts} are satisfied.  Furthermore, identifying $\cH$ with $\sbm{ \cH \\ 0} \subset \cK'$
makes $\cH$ a subspace of $\cK'$ and then $\Pi'$ is just the inclusion map.   When $\Pi \colon \cH \to \cK$ is an inclusion map we write simply $(V_1, V_2)$ rather than 
$(\iota_{\cH \to \cK}, V_1, V_2)$ for the And\^o isometric lift.

In this chapter, given a commuting contractive operator-pair $(T_1,T_2)$, we give two new proofs of the existence of And\^o isometric lifts and exhibit three distinct models for an And\^o isometric lift of $(T_1,T_2)$ associated with the names Douglas, Sz.-Nagy--Foias and Sch\"affer. A basic ingredient in all
three models is the notion of {\em pre-And\^o tuple} defined as follows.

\begin{definition}  \label{D:preAndoTuple}
\index{And\^o tuple!pre}
A collection of objects of the form $(\cF,\Lambda,P,U)$\index{$(\cF,\Lambda,P,U)$} is said to be a  {\em pre-And\^o tuple}\index{pre-And\^o tuple}
for the commuting contractive operator-pair $(T_1, T_2)$ if $\cF$ is a Hilbert space, $\Lambda:\cD_{T_1T_2} \to \cF$
is an isometry,  $P$ is a projection operator on $\cF$,
and  $U$ is a unitary operator on $\cF$.

Two pre-And\^o tuples $(\cF,\Lambda,P,U)$ and $(\cF',\Lambda',P',U')$ for $(T_1, T_2)$ are said to {\em coincide}
if there is a unitary operator $\tau \colon \cF \to \cF'$ such that
$$
\tau \Lambda = \Lambda', \quad \tau P \tau^* = P', \quad \tau U \tau^* = U'.
$$
\end{definition}

Thus a pre-And\^o tuple amounts to a BCL-tuple (see Definition \ref{D:BCLTuple}), but with the added ingredient of
the isometry $\Lambda \colon \cD_{T_1 T_2} \to \cF$, while coincidence of pre-And\^o tuples is a natural 
extension of the notion of unitary equivalence of BCL-tuples in the sense of Theorem \ref{Thm:BCLcoin}
(with the possible commuting pairs of unitary operators $(W_1, W_2)$ and $(W'_1, W'_2)$ ignored).
When a pre-And\^o tuple satisfies some additional natural conditions
to be discussed below, we shall refer to the collection $(\cF,\Lambda,P,U)$ simply as an  And\^o tuple.

We shall have need of two distinct types of And\^o tuples:  {\em And\^o tuples of Type I} (see Section \ref{S:Douglas2}) and {\em And\^o tuples of Type II}  
(see Section \ref{S:S2}).

We define a notion of irreducibility for pre-And\^o tuples.
\begin{definition}  \label{D:irred-Ando}\index{pre-And\^o tuple!irreducible}
The pre-And\^o tuple $(\cF, \Lambda, P, U)$ is said to be {\em irreducible} if the smallest subspace of $\cF$
invariant under $U$, $U^*$, $P$ and containing $\operatorname{Ran} \Lambda$ is the whole space $\cF$.
\end{definition}

In view of Theorem \ref{T:reduced-sub-tuple}, the pre-And\^o tuple $(\cF, \Lambda, P, U)$ is irreducible
if and only if the only reduced sub-BCL tuple $(\cF_0, P_0, U_0)$ of the BCL tuple $(\cF, P, U)$ such that $\cF_0 \supset \operatorname{Ran} \Lambda$ is the whole And\^o tuple $(\cF, P, U)$. Another equivalent statement is:  {\em $(\cF, \Lambda, P, U)$ is an irreducible pre-And\^o tuple if and only if the smallest joint reducing subspace for $(M_{U^*(P^\perp + z P)}, M_{(P + z P^\perp) U})$ 
containing $H^2(\operatorname{Ran} \Lambda)$ is the whole space $H^2(\cF)$.}   We shall need a notion of minimality for a pre-And\^o tuple which we shall call {\em minimal}.

\begin{definition}  \label{D:min-Ando}\index{pre-And\^o tuple!Douglas minimal} The pre-And\^o tuple $(\cF, \Lambda, P, U)$  is said to be 
{\em Douglas minimal} if the smallest joint invariant subspace for the BCL2 model
$(M_{U^*(P^\perp + zP)}, M_{(P + z P^\perp) U)})$ containing the space $H^2(\operatorname{Ran} \Lambda)$
is the whole space $H^2(\cF)$.
\end{definition}

Since any reducing subspace is also invariant, it is at the level of a tautology to see that minimality of an
And\^o tuple implies its irreducibility.  In all the examples which we have checked, the converse also holds,
but to this point, we have not been able to determine if the converse holds in general.

\section[Type I And\^o tuples and Douglas model]{Type I And\^o tuples and Douglas model for an And\^o  lift}  \label{S:Douglas2}

Let us recall that given a contraction operator $T$ on a Hilbert space $\cH$, we
may define a positive semidefinite operator $Q_{T^*}^2$ on $\cH$ as the strong limit
\begin{align}\label{DQ}
  Q_{T^*}^2:=\operatorname{SOT-}\lim_{n\to\infty}T^nT^{*n},
\end{align}
We set $Q_{T^*}$ equal to the positive-semidefinite square root of $Q_{T^*}^2$.
As explained in Section \ref{S:Douglas} (see \eqref{defX}), the identity $T Q_{T^*}^2 T^* = Q_{T^*}^2$ implies that the formula
\begin{equation}   \label{defX'}
X^* Q_{T^*} h = Q_{T^*} T^* h
\end{equation}
 extends by continuity to a well-defined isometry $X^*$ on $ \overline{\operatorname{Ran}}\, Q_{T^*}$
which has a minimal unitary extension, denoted as $W_D^*$, on the Hilbert space $\cQ_{T^*}$ equal to
the closure of $\cup_{n=0}^\infty W_D^n \operatorname{Ran} Q_T$ defined densely as in formula \eqref{IntofQ}:
$$
W_D^* W_D^n Q_{T^*} h = W_D^{n-1} Q_{T^*} h \text{ for } n \ge 1, \quad
W_D^* Q_{T^*} h = X^* Q_{T^*} h = Q_{T^*} T^* h.
$$
Then, as explained in Section  \ref{S:Douglas},  $(\Pi_D,V_D)$ is the Douglas minimal isometric lift for the single
contraction operator $T$, where the isometry $V_D$ on $\sbm{ H^2(\cD_{T^*}) \\ \cQ_{T^*}}$ is defined as
$ V_ D= \sbm{ M_z & 0 \\ 0 &  W_D}$
and $\Pi_D \colon \cH \to \sbm{ H^2(\cD_{T^*}) \\ \cQ_{T^*}}$ is the isometric embedding of $\cH$ defined as
$ \Pi_D  =  \sbm{  D_{T^*}(I - zT^*)^{-1}   \\ Q }$.
The goal of this section is to give a Douglas-type model for an And\^o isometric lift of a given commuting contractive pair
$(T_1,T_2)$.   We first need some preliminaries.

\subsection{Canonical pair of commuting unitaries} \label{S:flats}

As a preliminary for extending the Douglas-model isometric lift to the commuting contractive pair setting, we shall need the following simple but telling result of Douglas.  Here we use the standard notation:
if $X$ and $Y$ are operators on a Hilbert space $\cH$, we write $X \preceq Y$ if
$Y-X$ is {\em positive semidefinite}, i.e., $\langle (Y -X) h, h \rangle_\cH \ge 0$ for all $h \in \cH$.

\begin{lemma}[Douglas Lemma \cite{Douglas}]    \label{DougLem} \index{Douglas Lemma}
Let $A$ and $B$ be two bounded operators on a Hilbert space $\mathcal{H}$. Then there exists a 
contraction $C$ such that
$A=BC$ if and only if $$AA^*\preceq BB^*.$$
\end{lemma}

The Douglas-model minimal isometric lift for the single contraction operator $T$ as summarized in the
introductory part of Section \ref{S:Douglas2}, but now applied to the case where $T = T_1 T_2$ is the product  
operator coming from the
commuting contractive operator-pair $(T_1, T_2)$, leads to the following result which will play
a significant role in what follows. 

\begin{theorem}  \label{T:flats}  Given a commuting contractive operator pair $(T_1, T_2)$ on $\cH$, set
$T = T_1 T_2$ and let $Q_{T^*}$ be given by \eqref{Q} and \eqref{cQT*}.  Then there exist
a pair of unitary operators $(W_{\flat 1}, W_{\flat 2})$ on $\cQ_{T^*}$ so that
\begin{equation}  \label{DefWflats}\index{$W_{\flat 1},W_{\flat 2}$}
W_{\flat1}W_{\flat2}=W_D \text{ and } W_{\flat j}^*Q_{T^*}=Q_{T^*}T_j^*, \text{ for each } j=1,2.
\end{equation}
\end{theorem}

\begin{proof}
First note that when $Q_{T^*}$ is as in \eqref{Q} and $T=T_1T_2=T_2T_1$, then
\begin{eqnarray*}
\langle T_iQ_{T^*}^2T_i^*h,h\rangle = \lim_{n \to \infty} \langle T^n(T_iT_i^*){T^*}^nh,h \rangle
 \leq \lim_{n \to \infty} \langle T^n{T^*}^nh,h\rangle = \langle Q_{T^*}^2h,h\rangle
\end{eqnarray*}
for all $h \in \cH$ from which we conclude that
\begin{equation} \label{prop1}
T_1Q_{T^*}^2T_1^* \preceq Q_{T^*}^2 \text{ and }T_2Q_{T^*}^2T_2^* \preceq Q_{T^*}^2.
\end{equation}
By Lemma \ref{DougLem}, the inequalities in (\ref{prop1}) imply that there 
exist two contraction operators $X_1^*$ and $X_2^*$ on $\overline{\operatorname{Ran}}\;Q_{T^*}$ such that
\begin{equation}    \label{X1X2}
X_1^*Q_{T^*}=Q_{T^*}T_1^*, \quad X_2^*Q_{T^*}=Q_{T^*}T_2^*.
\end{equation}
From the equalities in (\ref{X1X2}) and (\ref{defX'}) it is clear that $X_1$ and $X_2$ commute and that
\begin{equation}  \label{X*commute}
X^*=X_1^*X_2^*.
\end{equation}
Since $X^*$ is an isometry, both $X_1^*$ and $X_2^*$ are isometries, as a consequence of the general fact that,
whenever $T$ is an isometry with factorization $T = T_1 T_2$ for some commuting contractions $T_1$ and $T_2$,
then in fact $T_1$ and $T_2$ are also isometries;  one way to see this is to look at  the following norm equalities
easily derived by using the commutativity of the contractive pair $(T_1, T_2)$:
$$
\|D_{T_1}T_2h\|^2+\|D_{T_2}h\|^2=\|D_Th\|^2=\|D_{T_1}h\|^2+\|D_{T_2}T_1h\|^2 \text{ for all }h \in \mathcal{H}.
$$
By Lemma \ref{L:special-ext} we get a commuting unitary extension $(W_{\flat1}^*,W_{\flat2}^*)$ of 
the commuting isometric pair $(X_1^*, X_2^*)$ on $\cQ_{T^*} = \overline{\text{span}}\{W_D^{ n}x: x\in \operatorname{Ran} Q_{T^*}   \text{ and } n\geq0\}$.  Thus the product
$W_D^*=W_{\flat1}^*W_{\flat2}^*$ is  the minimal unitary extension of the product $X^*=X_1^*X_2^*$.
\end{proof}

The pair $(W_{\flat1},W_{\flat2})$ will be referred to as the {\em canonical pair of commuting 
unitaries}\index{canonical pair of commuting unitaries} associated with the contractive pair $(T_1,T_2)$. 

\smallskip

We next address uniqueness of a canonical pair of unitaries $(W_{\flat1}, W_{\flat2})$ for a given 
commuting contractive
pair $(T_1, T_2)$.

\begin{lemma}  \label{L:uniqueness-flats}
Let $(T_1,T_2)$ on $\mathcal{H}$ and $(T_1',T_2')$ on $\mathcal{H'}$ be two pairs of commuting contractions and 
$(W_{\flat1},W_{\flat2})$ on $\cQ_{T^*}$ and $(W_{\flat1}',W_{\flat2}')$ on $\cQ_{T'}$ be the respective pairs 
of commuting unitaries obtained from them as above. If $(T_1,T_2)$ is unitarily equivalent to $(T_1',T_2')$ 
via the unitary similarity $\phi \colon \mathcal{H} \to \mathcal{H'}$, then $(W_{\flat1},W_{\flat2})$ and 
$(W_{\flat1}',W_{\flat2}')$ are unitarily equivalent via the induced unitary transformation 
$\tau_\phi \colon \cQ_{T^*}\to \cQ_{T^{\prime *}}$ determined by 
$\tau_\phi \colon W_D^nQ_{T^*}h \to W_D'^n Q_{T^{\prime *}} \phi h$.  In particular, if $(T_1, T_2) = (T_1', T_2')$, then
$(W_{\flat1},W_{\flat2}) = (W_{\flat1}',W_{\flat2}')$.
\end{lemma}

\begin{proof}
Let $\phi:\mathcal{H}\to\mathcal{H'}$ be a unitary that intertwines $(T_1,T_2)$ and $(T_1',T_2')$. Let us
denote $T=T_1T_2$ and $T'=T_1'T_2'$. Let $Q_{T^*}$ and $Q_{T^{\prime *}}$ be the limits of $T^nT^{* n}$ and
$T'^nT^{\prime * n}$, respectively, in the strong operator topology. Clearly, $\phi$ intertwines $Q_{T^*}$ and $Q_{T'^*}$. Therefore
$\phi$ takes $\overline{\operatorname{Ran}}\,Q_{T^*}$ onto
$\overline{\operatorname{Ran}}\,Q_{T'^*}$. We denote the restriction of $\phi$ to
$\overline{\operatorname{Ran}}\,Q_T$ by $\phi$ itself. Let $(X_1,X_2)$ on 
$\overline{\operatorname{Ran}}\,Q_{T^*}$ and
$(X_1',X_2')$ on $\overline{\operatorname{Ran}}\,Q_{T'^*}$ be the pairs of commuting co-isometries corresponding to
the pairs $(T_1,T_2)$ and $(T_1',T_2')$ as in (\ref{X1X2}), respectively. It is easy to see from the definition that
$$
\phi(X_1,X_2)=(X_1',X_2')\phi.
$$Let $(W_{\flat1},W_{\flat2})$ on $\cQ_{T^*}$ and $(W_{\flat1}',W_{\flat2}')$ on $\cQ_{T'^*}$
be the pairs of commuting unitaries corresponding to $(T_1,T_2)$ and $(T_1',T_2')$, respectively. 
Remembering the formula (\ref{cQT*}) for the spaces $\cQ_{T^*}$ and $\cQ_{T'^*}$, we 
can densely define $\tau_\phi: \cQ_{T^*} \to  \cQ_{T'^*}$ by
\begin{align}\label{tauphi}
\tau_\phi: W_D^{ n}x\mapsto W_D'^{ n}\phi x, \text{ for every } x\in \mathcal{Q}_{T^*} \text{ and } n\geq 0
\end{align}
and extend linearly and continuously. Trivially, $\tau_\phi$ is unitary and intertwines $W_D$ and $W_D'$.
For a non-negative integer $n$ and $x$ in $\cQ_{T^*}$, we have using (\ref{DefWflats})
\begin{align*}
\tau_\phi W_{\flat1}(W_D^{n}x) & =\tau_\phi W_D^{n+1}(W_{\flat2}^*x)= W_D'^{ n+1}\phi(X_2^*x)\\
&=W_D'^{n+1}W_{\flat2}'^*\phi x=W_{\flat1}'W_D'^{n}\phi x=W_{\flat1}'\tau_\phi(W_D^{ n}x).
\end{align*}
A similar computation shows that $\tau_\phi$ intertwines $W_{\flat2}$ and $W_{\flat2}'$ too.
\end{proof}
Note that in the above lemma, when $(T_1,T_2)=(T_1',T_2')$, then
\begin{align*}
\cQ_{T^*}=  \cQ_{T'^*}, \quad  \phi=I_{\cH} \quad \text{and}  \quad \tau_\phi=I_{\cQ_{T^*}}.
\end{align*} 
Therefore the following is a straightforward consequence of the above lemma.

\begin{corollary}\label{C:CanComUni}
Let $(T_1,T_2)$ be a commuting contractive operator-pair and $Q_{T^*}$ be as in (\ref{DQ}) where $T=T_1T_2$. Let 
$(W_1,W_2)$ be some pair of commuting unitaries on $\cQ_{T^*}$ such that
\begin{align}\label{DefiningWs}
W_D=W_1W_2 \text{ and }W_j^*Q_{T^*}=Q_{T^*}T_j^* \text{ for }j=1,2.
\end{align}Then $(W_1,W_2)=(W_{\flat1},W_{\flat2})$.
\end{corollary}

\subsection{Douglas-type structure of a general And\^o  lift} \label{S:DougArbMod}
\index{And\^o lift!Douglas-type model}

We shall see that minimal And\^o lifts $(V_1, V_2)$ of a given commuting contractive-pair $(T_1, T_2)$
are in one-to-one correspondence with And\^o tuples associated with $(T_1^*, T_2^*)$ which satisfy some
additional conditions as follows.

\begin{definition}  \label{D:AndoTupleI}  Given a commuting contractive operator-pair $(T_1, T_2)$ and 
a pre-And\^o tuple
$(\cF_*, \Lambda_*, P_*, U_*)$  for $(T_1^*, T_2^*)$, we say that
$(\cF_*, \Lambda_*, P_*, U_*)$ is a
{\em Type I And\^o tuple}
\index{And\^o tuple!type I} 
for $(T_1^*, T_2^*)$  if the following additional conditions are satisfied:
\begin{align}   
&  P_*^\perp U_*  \Lambda_*  D_{T^*} +   P_* U_* \Lambda_* D_{T^*} T^*  = \Lambda_* D_{T^*} T_1^*,
\label{AndoTuple1} \\
&  U_*^* P_* \Lambda_* D_{T^*}          + U_*^* P_*^\perp \Lambda_* D_{T^*} T^* = \Lambda_* D_{T^*} T_2^*.
 \label{AndoTuple2}
 \end{align}
 If the And\^o tuple for $(T_1^*, T_2^*)$ is Douglas-minimal when considered as a pre-And\^o tuple for $(T_1^*, T_2^*)$
(see Definition \ref{D:min-Ando}), we say that $(\cF_*, \Lambda_*, P_*, U_*)$ is a
{\em Douglas-minimal  Type * And\^o tuple} for $(T_1^*, T_2^*)$.  

Let us note that the notion of Type I And\^o tuple for a given commuting contractive pair  $(T_1, T_2)$ is 
{\em coordinate-free}
in the following sense:  {\em if $(\cF_*, \Lambda_*, P_*, U_*)$ is a Type I And\^o tuple for $(T_1, T_2)$ and if
$(\cF'_*, \Lambda'_*, P'_*, U'_*)$ is a pre-And\^o tuple for $(T_1, T_2)$ which coincides with
$(\cF_*, \Lambda_*, P_*, U_*)$ (in the sense of Definition \ref{D:preAndoTuple}), then $(\cF'_*, \Lambda'_*, P'_*, U'_*)$
is also a Type I And\^o tuple for $(T_1, T_2)$.}
\end{definition}

The next result provides a functional model for minimal And\^o lifts $(V_1, V_2)$ of a given commuting contractive
operator-pair $(T_1, T_2)$ in terms of a Douglas-minimal Type I And\^o tuples for $(T_1^*, T_2^*)$.

\begin{theorem}  \label{T:Dmodel}  Let $(T_1, T_2)$ be a pair of commuting contractions and let
$(\cF_*, \Lambda_*, P_*, U_*)$ be a Douglas-minimal Type I And\^o tuple for $(T_1^*, T_2^*)$.  Set $T = T_1 T_2$, let
the unitary operator $W_D$ on the Hilbert space $\cQ_{T^*}$  be as in the Douglas model for the minimal 
isometric lift of $T$ (see Section \ref{S:Douglas}),
and let $W_{\flat 1}$, $W_{\flat 2}$ be the canonical  pair of commuting unitaries associated with $(T_1, T_2)$
as in Theorem \ref{T:flats}.  Define operators and spaces by
\begin{align}\index{$\bcK_D$}\index{$\bPi_D$} \index{$\bV_{D,1}, \bV_{D,2}$}
 & \bcK_D = \begin{bmatrix} H^2(\cF_*) \\ \cQ_{T^*} \end{bmatrix}, \notag  \\
 & (\bV_{D,1}, \bV_{D,2}) =  \left( \begin{bmatrix} M_{U_*^*(P_*^\perp + z P_*)} & 0 \\ 0 & W_{\flat 1} \end{bmatrix},
 \begin{bmatrix} M_{(P_* + z P_*^\perp) U_*} & 0 \\ 0 & W_{\flat 2} \end{bmatrix} \right) \text{ acting on } \bcK_D,  \notag \\
 &   \bPi_D = \begin{bmatrix} (I_{H^2} \otimes \Lambda_*)  \cO_{D_{T^*}, T^*} \\ Q_{T^*} \end{bmatrix} \colon \cH \to \bcK_D.
 \label{DougAndoMod}
 \end{align}
 Then $(\bPi_D, \bV_{D,1}, \bV_{D,2})$ is a minimal  And\^o  lift for $(T_1, T_2)$.

 Conversely, given any minimal And\^o lift $(\bPi, \bV_1, \bV_2)$ for $(T_1, T_2)$, there is a Douglas-minimal 
 Type I And\^o tuple $(\cF_*, \Lambda_*, P_*, U_*)$ for $(T_1^*, T_2^*)$ so that the lift $(\bPi, \bV_1, \bV_2)$ is unitarily equivalent 
 to the Douglas-model And\^o lift \eqref{DougAndoMod} associated with $(\cF_*, \Lambda_*,$ $ P_*, U_*)$ as in  \eqref{DougAndoMod}.
\end{theorem}

\begin{proof}
Suppose that $(\cF_*, \Lambda_*,  P_*, U_*)$ is a minimal Type I And\^o tuple for the commuting 
pair of operators
$(T_1^*, T_2^*)$ acting on $\cH$ and let  $(\bPi_D, \bV_{D,1}, \bV_{D,2})$ be as in \eqref{DougAndoMod}.  
Note that $(\bV_{D,1}, \bV_{D,2})$ on $\bcK_D$ is a pair of commuting isometries since this pair has the form of the BCL model
for commuting isometric operator-pairs.  Note that $\bPi_D$ can be factored 
$$
 \bPi_D = \begin{bmatrix} I_{H^2} \otimes \Lambda_* & 0 \\ 0 & I_{\cQ_{T^*}} \end{bmatrix}  
  \begin{bmatrix} \cO_{D_{T^*}, T^*} \\  Q_{T^*} \end{bmatrix} 
$$ 
as the product of isometries, and hence is itself  an isometric embedding $\bPi_D \colon \cH \to \bcK_D$.  
To show that $(\bPi, \bV_{D,1}, \bV_{D,2})$ is an And\^o lift of $(T_1, T_2)$ it remains only to verify the intertwining relations
$ \bV_{D,j}^* \Pi = \bPi T_j^*$  for  $j=1,2$,  i.e.,
\begin{align*}  
& \begin{bmatrix} \big( M_{U^*_*P_*^\perp + z U_*^* P_*)}\big)^*& 0 
\\ 0 & W_{\flat 1}^* \end{bmatrix} \begin{bmatrix} (I_{H^2}  \otimes \Lambda_* ) \cO_{D_{T^*}, T^*} \\ Q_{T^*} \end{bmatrix}
= \begin{bmatrix} (I_{H^2}  \otimes \Lambda_* ) \cO_{D_{T^*}, T^*}  \\ Q_{T^*}  \end{bmatrix} T_1^*,  \\
 &  \begin{bmatrix} \big( M_{P_* U_* + z P_*^\perp U_* }\big)^*& 0 
\\ 0 & W_{\flat 2}^* \end{bmatrix} \begin{bmatrix} (I_{H^2}  \otimes \Lambda_* ) \cO_{D_{T^*}, T^*} \\ Q_{T^*} \end{bmatrix}
= \begin{bmatrix} (I_{H^2}  \otimes \Lambda_* ) \cO_{D_{T^*}, T^*}  \\ Q_{T^*}  \end{bmatrix} T_2^*,
 \end{align*}
 or equivalently, the system of equations
 \begin{align}
 & ( I_{H^2} \otimes P_*^\perp U_* + M_z^* \otimes P_* U_*)(I_{H^2} \otimes \Lambda_* )\cO_{D_{T^*}, T^*}
 = (I_{H^2} \otimes \Lambda_*) \cO_{D_{T^*}, T^*} T_1^*,   \label{intertwine1}  \\
 &  W_{\flat 1}^* Q_{T^*} = Q_{T^*} T_1^*,  \label{intertwine2}  \\
 & (I_{H^2} \otimes U_*^* P_* + M_z^* \otimes U_*^* P_*^\perp) (I_{H^2} \otimes \Lambda_*) 
 \cO_{D_{T^*}, T^*}  = (I_{H^2} \otimes \Lambda_*) \cO_{D_{T^*}, T^*} T_2^*, \label{intertwine3}  \\
 & W_{\flat 2}^* Q_{T^*} = Q_{T^*} T_2^*.  \label{intertwine4}
\end{align}
Note that equations \eqref{intertwine2} and \eqref{intertwine4} follow immediately from Theorem \ref{T:flats}
(see the second identity in   \eqref{DefWflats}).  Applying  \eqref{intertwine1} to a generic
vector $h \in \cH$ gives
$$ 
\sum_{n=0}^\infty \big(P_*^\perp U_* \Lambda_* D_{T^*} T^{*n} h + P_* U_* \Lambda_* D_{T^*} T^{*(n+1)}h \big) \, z^n
= \sum_{n=0}^\infty \big( \Lambda_* D_{T^*} T^{*n} T_1^* h ) \, z^n.
$$
This identity between power series is the same as the matching of power series coefficients of the two sides
holding for all $h \in \cH$:
\begin{equation}  \label{reduced-sys}
P_*^\perp U_* \Lambda_* D_{T^*} T^{*n} + P_* U_* \Lambda_* D_{T^*} T^{*(n+1)} = 
 \Lambda_* D_{T^*} T^{*n} T_1^* \text{ for all } n = 0,1,2,\dots.
\end{equation}
Since $T_1^*$ and $T^*$ commute, we see that the identity for the case of a general $n$ follows from the special case
of $n=0$ by multiplying on the right by $T^{*n}$;   hence the system of equations \eqref{reduced-sys}
is equivalent to the single equation
$$
P_*^\perp U_* \Lambda_* D_{T^*} + P_* U_* \Lambda_* D_{T^*} T^* = 
 \Lambda_* D_{T^*} T_1^*
 $$
 which amounts to the operator equation \eqref{AndoTuple1}.  In a similar way condition \eqref{intertwine3}
 reduces to \eqref{AndoTuple2}.   

 With notation as in  \eqref{DougAndoMod} let us set\index{$\bcK_{D,0}$}
 $$
  \bcK_{D, 0} = \bigvee_{n_1, n_2 \ge 0}  \bV_{D,1}^{n_1} \bV_{D,2}^{n_2} \operatorname{Ran} \bPi_D 
$$
To show that $(\bPi_D, \bV_{D,1}, \bV_{D,2})$ is a minimal And\^o lift for $(T_1, T_2)$, it remains only to show that
$\bcK_D = \bcK_{D, 0}$.   

As an intermediate step, we shall first identify the smaller space
\begin{equation}  \label{Kcircmin}\index{$\bcK_{D,00}$}
\bcK_{D,00} : =  \bigvee_{n \ge 0} \bV_{D,1}^n \bV_{D,2}^n \operatorname{Ran} \bPi_D   = \bigvee_{n \ge 0}  \bV_D^n  \operatorname{Ran} \bPi_D  \subset \bcK_{D,0}
\end{equation}
where $\bV_D = \bV_{D,1} \bV_{D,2} = \sbm{ M_z^{\cF_*} & 0 \\ 0 & W_D}$ on $\sbm{ H^2(\cF_*) \\ \cQ_{T^*} }$.    
Once $\cK_{D,00}$ is identified, we can find
$\bcK_{D, 0}$ via the formula
\begin{equation} \label{Kmin-Kcircmin}
\bcK_{D, 0} = \bigvee_{n_1, n_2 = 0,1,2,\dots} \bV_{D,1}^{n_1} \bV_{D,2}^{n_2} \bcK_{D,00}.
\end{equation}

Note that $\bPi_D$ given by  \eqref{DougAndoMod} factors as
$$
  \bPi_D = \begin{bmatrix} I_{H^2} \otimes \Lambda_* & 0 \\ 0 & I_{\cQ_{T^*}} \end{bmatrix} \Pi_D
$$
where
$$
  \Pi_D = \begin{bmatrix} \cO_{D_{T^*}, T^*} \\ \cQ_{T^*} \end{bmatrix} \colon \cH \to \cK_D = \begin{bmatrix} H^2(\cD_{T{^*}}) \\ \cQ_{T^*} \end{bmatrix}
$$
is the embedding operator for the Douglas-model isometric lift of $T$ as in \eqref{PiD}. 
Let us also note the intertwining
$$
\bV_D^n \begin{bmatrix} I_{H^2} \otimes \Lambda_* & 0 \\ 0 & I_{\cQ_{T^*}} \end{bmatrix}
= \begin{bmatrix} I_{H^2} \otimes \Lambda_* & 0 \\ 0 & I_{\cD_{T^*}} \end{bmatrix} V_D^n
$$
where $V_D = \sbm{ M_z^{\cD_{T^*}} & 0 \\ 0 & W_D}$ is the isometric lift of $T$ in the Douglas model as in \eqref{VD}.  
Hence we see that, for $h \in \cH$,
\begin{align*}
\bV_D^n \bPi_D h & = \bV_D^n \begin{bmatrix} I_{H^2} \otimes \Lambda_* & 0 \\ 0 & I_{\cQ_{T^*}} \end{bmatrix} \Pi_D h  = 
\begin{bmatrix} I_{H^2} \otimes \Lambda_* & 0 \\ 0 & I_{\cQ_{T^*}} \end{bmatrix} V_D^n \Pi_D h
\end{align*}
where
$$
\bigvee_{n=0,1,2,\dots, \, h \in \cH}  V_D^n \Pi_D h = \cK_D : = \begin{bmatrix} H^2(\cD_{T^*}) \\ \cQ_{T^*} \end{bmatrix}
$$
since the Douglas-model lift of $T$ is minimal, as seen in Proposition \ref{P:Douglas-min}.   We have now identified the space $\bcK_{D,00}$ as being exactly equal to
\begin{align} 
\bcK_{D,00} & = \begin{bmatrix} I_{H^2} \otimes \Lambda_* & 0 \\ 0 & I_{\cQ_{T^*}} \end{bmatrix} \begin{bmatrix} H^2(\cD_{T^*}) \\ \cQ_{T^*} \end{bmatrix} \\
& =  \begin{bmatrix} H^2( \operatorname{Ran} \Lambda_*) \\ \cQ_{T^*} \end{bmatrix}.
 \label{kcircmin-id}
\end{align}

We now use the formula \eqref{Kmin-Kcircmin} to see that
$$
 \bcK_{D, 0} = \bigvee_{n_1, \, n_2 \ge 0}  \bV_{D,1}^{n_1} \bV_{D,2}^{n_2} \begin{bmatrix}  H^2(\operatorname{Ran} \Lambda_*) \\ \cQ_{T^*} \end{bmatrix} 
 = \begin{bmatrix}  \bigvee_{n_1, \, n_2 \ge 0} M_1^{n_1} M_2^{n_2} H^2(\operatorname{Ran} \Lambda_*) \\ \cQ_{T^*} \end{bmatrix},
 $$
 where we use the short-hand notation
\begin{equation}   \label{shorthand}
  M_1: = M_{U_*^* (P_*^\perp + z P_*)}, \quad M_2:= M_{(P_* + z P_*^\perp) U_*}.
\end{equation}
 This space is equal to the whole space $\bcK_D = \sbm{ H^2(\cF_*) \\ \cQ_{T^*} }$ exactly when
 $$
  \bigvee_{n_1, \, n_2 \ge 0} M_1^{n_1} M_2^{n_2} H^2(\operatorname{Ran} \Lambda_*) = H^2(\cF_*),
  $$
i.e., when the original And\^o tuple $(\cF_*, P_*, U_*, \Lambda_*)$ is minimal as an And\^o tuple.

\smallskip
 
 Conversely, suppose that $(\bPi, \bV_1, \bV_2)$ is any And\^o lift for the commuting, contractive operator-pair
 $(T_1, T_2)$ on $\cH$.  Thus  $(\bV_1, \bV_2)$ is a commuting  isometric operator-pair on some Hilbert  space 
 $\bcK$ and 
 $\bPi$ is an isometric embedding of $\cH$ into $\bcK$ so that
 $$
   \bV_j^* \bPi = \bPi T_j^* \text{ for } j = 1,2.
 $$
 By Theorem \ref{Thm:BCLmodel} we know that the commuting pair of isometries $(\bV_1, \bV_2)$ is unitarily
equivalent to its BCL2 model as in \eqref{BCL2model} acting on the space
$\cK:=\sbm{ H^2(\cD_{\bV^*} ) \\ \cQ_{\bV^*}}$  via the unitary identification map
$\tau_{\rm BCL} \colon \bcK  \to \cK_0$   given by
$$
\tau_{\rm BCL} = \begin{bmatrix} \cO_{D_{\bV^*}, \bV^*}  \\ Q_{\bV^*} \end{bmatrix}
  \colon \bk \mapsto \begin{bmatrix} \cO_{D_{\bV^*}, \bV^*}  \\ Q_{\bV^*} \end{bmatrix} \bk.
 $$
 Here $\bV := \bV_1 \bV_2$ satisfies the intertwining relation
 $\cO_{D_{\bV^*}, \bV^*} \bV = M_z^{\cD_{\bV^*}} \cO_{D_{\bV^*}, \bV^*}$.
 For the case here where $\bV$ is an isometry, the operator $\cO_{D_{\bV^*}, \bV^*}$ is a partial isometry
 onto the model space $H^2(\cD_{\bV^*})$  for the shift-part of $\bV$ in its Wold decomposition, while
 $Q_{\bV^*} = \operatorname{SOT-}\lim_{n \to \infty} \bV^n \bV^{*n}$ is the projection of $\bcK$ onto the unitary subspace $\cH_u
 = \bigcap_{n \ge 0} \operatorname{Ran} \bV^n$ in the Wold decomposition of $\bV$.
 Moreover we have the intertwining relations involving the operators $(\bV_1, \bV_2)$ in the factorization of $\bV$:
$$
\cO_{D_{\bV^*}, \bV^*} \bV_1 = M_1 \cO_{D_{\bV^*}, \bV^*}, \quad
\cO_{D_{\bV^*}, \bV^*} \bV_2 = M_2 \cO_{D_{\bV^*}, \bV^*}
$$
where $M_1,M_2$ are as in \eqref{shorthand} and the projection operator $P_*$ and the unitary operator $U_*$ on $\cD_{\bV^*}$ form a Type II BCL-tuple  $(\cD_{\bV^*}, P_*, U_*)$ as in Theorem \ref{Thm:BCLmodel}.
 Let us also introduce operators $W_1$, $W_2$, $W$ on $\cQ_{\bV^*} = \cH_u$ according to
 $$
  W_1 = \bV_1|_{\cQ_{\bV^*}}, \quad W_2 = \bV_2|_{\cQ_{\bV^*}}, \quad W = \bV|_{\cQ_{\bV^*}}
 $$
 and define an isometric embedding map $\Pi \colon \cH \to \cK = \sbm{ H^2(\cD_{\bV^*}) \\ \cQ_{\bV^*}}$
 according to
 \begin{equation}  \label{modelPi}
   \Pi = \tau_{\rm BCL} \bPi = \begin{bmatrix} \cO_{D_{\bV^*}, \bV^*} \\ Q_{\bV^*} \end{bmatrix} \bPi.
 \end{equation}
 Then it is easily checked that 
 \begin{equation}  \label{AndoLift2}
  \bigg( \Pi,  \begin{bmatrix} M_1 & 0 \\ 0 & W_1 \end{bmatrix},
  \begin{bmatrix} M_2 & 0 \\ 0 & W_2 \end{bmatrix} \bigg)
  \end{equation}
  is also an And\^o lift of $(T_1, T_2)$ which is unitarily equivalent (as an And\^o lift) to our original more abstract And\^o lift $(\bPi, \bV_1, \bV_2)$.
 
 Embedded in the lift \eqref{AndoLift2} is an isometric lift $\left( \Pi,   \sbm{M_z^{\cD_{\bV^*}} & 0 \\ 0 & W } \right)$ 
 for the product contraction operator
 $T = T_1 T_2$ but this lift is not necessarily minimal as a lift of $T$.  However, we can always restrict the
 lift space $\cK$ appropriately to arrive at a minimal lift for $T$.  Namely, if we set 
 \begin{align}
&  \cK_{00} = \bigvee_{n \ge 0} \begin{bmatrix} M_z^{\cD_{\bV^*}} & 0 \\ 0 & W 
   \end{bmatrix}^n \operatorname{Ran} \Pi_0, 
   \notag \\
 & \Pi_{00}  \colon  \cH \to \cK_{00}  \text{ given by } \Pi_{00} h = \Pi  h \in \operatorname{Ran} \Pi  \subset \cK_{00}, 
 \label{Kmin-circ}
\end{align}
then
  $$
  \bigg( \Pi_{00},  \begin{bmatrix} M_z^{\cD_{\bV^*}}  & 0 \\ 0 & W \end{bmatrix} \bigg|_{\cK_{00}} \bigg)
  $$
  is a {\em minimal} isometric lift for the  product contraction operator $T = T_1 T_2$.

On the other hand, the Douglas model for the minimal isometric lift of $T$ (see Section 2.2.2) has the form 
$(\Pi_D, V_D)$ where $\Pi_D$ is the isometric embedding of $\cH$ into $\cK_D:=\sbm{ H^2(\cD_{T^*}) \\ \cQ_{T^*}}$ given by
$$
  \Pi_D \colon h \mapsto \begin{bmatrix} \cO_{D_{T^*}, T^*} \\ Q_{T^*} \end{bmatrix} h.
$$
and where
$$
V_D = \begin{bmatrix} M_z^{\cD_{T^*}} & 0 \\ 0 & W_D \end{bmatrix} \colon  
\begin{bmatrix}  H^2(\cD_{T^*}) \\ \cQ_{T^*} \end{bmatrix}  \to \begin{bmatrix}  H^2(\cD_{T^*}) \\ \cQ_{T^*} \end{bmatrix}.
$$
By the uniqueness of minimal isometric lifts,  there exists a unique unitary operator $\widehat \Gamma_0 \colon \cK_D \to \cK_{00} \subset \cK$ so that
\begin{align} 
& \widehat \Gamma_0 \begin{bmatrix} M_z^{\cD_{T^*}} & 0 \\ 0 & W_D \end{bmatrix}= 
\bigg( \begin{bmatrix} M_z^{\cD_{\bV^*}} & 0 \\ 0 & W \end{bmatrix} \bigg |_{\cK_{00}}  \bigg) 
\widehat \Gamma_0,    \label{uniqueness1} \\
& \widehat \Gamma_0 \begin{bmatrix} \cO_{D_{T^*}, T^*} \\ Q_{T^*} \end{bmatrix}  =  
 \Pi_{00} \colon \cH \to \cK_{00} \subset \cK.
\label{uniqueness2'}
\end{align}

To gain added flexibility let us view the unitary operator $\widehat\Gamma_0$ from $\sbm{ H^2(\cD_{T^*}) \\ \cQ_{T^*} }$
to $\cK_{00}$ as actually an isometry, now denoted more simply as $\widehat \Gamma$,
mapping into $\cK = \sbm{ H^2(\cD_{\bV^*}) \\ \cQ_{\bV^*}}$.   
 Let us write out $\widehat \Gamma$ as a $2 \times 2$-block operator matrix 
 \begin{equation} \label{Gammahat-block}
 \widehat \Gamma = \begin{bmatrix} \Gamma & \Gamma_{12} \\ \Gamma_{21} & \Gamma_{22} \end{bmatrix}
  \colon \cK_D  = \begin{bmatrix} H^2(\cD_{T^*}) \\ \cQ_{T^*} \end{bmatrix}
 \to  \cK= \begin{bmatrix} H^2(\cD_{\bV^*}) \\ \cQ_{\bV^*} \end{bmatrix}.
 \end{equation}
 Recalling the formula  \eqref{Kmin-circ} for $\Pi_{\rm min}$ and the formula for $\Pi$ in \eqref{modelPi} we may rewrite
\eqref{uniqueness2'} as
 \begin{equation}   \label{uniqueness2}
 \widehat \Gamma \begin{bmatrix} \cO_{D_{T^*}, T^*} \\ Q_{T^*} \end{bmatrix} = \begin{bmatrix} \cO_{D_{\bV^*}, \bV^*} \\ Q_{\bV^*} \end{bmatrix} \bPi.
 \end{equation}
 From  \eqref{uniqueness1} we see that
 \begin{align}
& \Gamma M_z^{\cD_{T^*}} = M_z^{\cD_{\bV^*}} \Gamma, \quad 
 \Gamma_{12} W_D = M_z^{\cD_{\bV^*}} \Gamma_{12},  \notag  \\
 & \Gamma_{21} M_z^{\cD_{T^*}} = W \Gamma_{21},  \quad 
 \Gamma_{22} W_D = W \Gamma_{22}   \label{matrixintertwine1}
 \end{align}
 As $W_D$ is unitary and $M_z^{\cD_{\bV^*}}$ is a shift, we see from item (1) in Lemma \ref{L:AuxLemma} that
 the second intertwining condition in \eqref{matrixintertwine1} implies that $\Gamma_{12} = 0$ and hence
 $\widehat \Gamma$ collapses to
 \begin{equation}  \label{Gamma12=0}
 \widehat \Gamma = \begin{bmatrix} \Gamma & 0 \\ \Gamma_{21} & \Gamma_{22} \end{bmatrix}.
 \end{equation}
 From  \eqref{uniqueness2} we then pick up the additional relations
 \begin{equation}  \label{matrixintertwine2}
 \Gamma \cO_{D_{T^*}, T^*}  = \cO_{D_{\bV^*}, \bV^*} \bPi, \quad 
 \Gamma_{21} \cO_{D_{T^*}, T^*} + \Gamma_{22} Q_{T^*}
= Q_{\bV^*} \bPi.
 \end{equation}
  
 Consider next the following variants of the identity \eqref{cool-id}:
 \begin{align}
 & \cO_{D_{T^*}, T^*} - M_z^{\cD_{T^*}} \cO_{D_{T^*}, T^*} T^* = \bev_{0,\cD_{T^*}}^* D_{T^*}, \notag \\
 & \cO_{D_{\bV^*}, \bV^*} - M_z^{\cD_{\bV^*}} \cO_{D_{\bV^*}, \bV^*} \bV^* = \bev_{0,\cD_{V^*}}^* D_{\bV^*}.
 \label{variants}
 \end{align}
 Use the first relation in \eqref{matrixintertwine2} together with the fact that $(\bPi, \bV)$ is a lift of $T$ to get
 \begin{align}
 M_z^{\cD_{\bV^*}} (\Gamma \cO_{D_{T^*}, T^*}) T^* & = M_z^{\cD_{\bV^*}} (\cO_{D_\bV^*, \bV^*} \bPi) T^*  = M_z^{\cD_{\bV^*} }\cO_{D_{\bV^*}, \bV^*} \bV^* \bPi.
 \label{double-intertwine}
 \end{align}
 Using the second relation in \eqref{variants}  together with again \eqref{matrixintertwine2} then enables us to compute
 \begin{align}
\notag \Gamma \bev_{0,\cD_{T^*}}^* D_{T^*} & = \Gamma \bigg( \cO_{D_{T^*}, T^*} - M_z^{\cD_{T^*}} \cO_{D_{T^*}, T^*} T^* \bigg) 
 \text{ (by \eqref{variants})}  \notag \\
 & = \bigg( \cO_{D_{\bV^*}, \bV^*} - M_z^{\cD_{\bV^*}} \cO_{D_{\bV^*}, \bV^*} \bV^* \bigg) \bPi   = \bev_{0,\cD_{\bV^*}}^* D_{\bV^*} \bPi ,  \label{Lambda*}
 \end{align}
 i.e., the operator $\Gamma \colon H^2(\cD_{T^*}) \to H^2(\cD_{\bV^*})$ maps constant functions in $H^2(\cD_{T^*})$ to constant functions
 in $H^2(\cD_{\bV^*})$.
Since $\Gamma \colon H^2(\cD_{T^*}) \to H^2(\cD_{\bV^*})$ has the shift intertwining property
 (the first relation in \eqref{matrixintertwine1},  we conclude that there must be an operator 
 $\Lambda_* \colon \cD_{T^*} \to \cD_{\bV^*}$ so that $\Gamma$ is ``multiplication by a constant":
 $\Gamma = I_{H^2} \otimes \Lambda_*$.  From the implicit formula \eqref{Lambda*} we see that
 $\Lambda_*$ must be given by
 \begin{equation}  \label{Lambda*explicit}
   \Lambda_* \colon D_{T^*} h \mapsto D_{\bV^*} \bPi h.
 \end{equation}
 Let us note next that
 \begin{align*}
 \| D_{T^*} h \|^2 & = \| h \|^2 - \| T^* h \|^2  = \| \bPi h \|^2 - \| \bPi T^* h \|^2  = \| \bPi h \|^2 - \| \bV^* \bPi h \|^2  \\
 &  = \langle (I - \bV \bV^*) \bPi h, \bPi h \rangle  =  \| D_{\bV^*}  \bPi h \|^2 = \| \Lambda_* D_{T^*} h \|^2
 \end{align*}
 so $\Lambda_*$ is an isometry.  Combining the representation \eqref{Gamma12=0} with the fact that
 $\Gamma = I_{H^2} \otimes \Lambda_*$ is an isometry since $\Lambda_*$ is isometric and that fact that
 $\widehat \Gamma$ is also an isometry, we are now finally able to conclude that $\Gamma_{21} = 0$ as well. We thus now have reduced
 $\widehat \Gamma$ to the diagonal form
 \begin{equation}   \label{Gamma21=0}
\widehat   \Gamma = \begin{bmatrix} I_{H^2} \otimes \Lambda_* & 0 \\ 0 & \Gamma_{22} \end{bmatrix} \colon
\begin{bmatrix} H^2(\cD_{T^*}) \\ \cQ_{T^*} \end{bmatrix} \to 
\begin{bmatrix} H^2(\cD_{\bV^*}) \\ \cQ_{\bV^*} \end{bmatrix}
 \end{equation}
 where $\Gamma_{22}$ is an isometry from $\cQ_{T^*}$ to $\cQ_{\bV^*}$. We note also that the four 
 intertwining conditions \eqref{matrixintertwine1} now collapse to the two intertwining conditions
 \begin{equation}   \label{intertwine2'}
 (I_{H^2} \otimes \Lambda_*)  M_z^{\cD_{T^*}} = M_z^{\cD_{\bV^*}} (I_{H^2} \otimes \Lambda_*), 
 \quad \Gamma_{22} W_D = W \Gamma_{22},
 \end{equation}
 and the condition \eqref{uniqueness2} splits into the two operator equations
 \begin{equation}   \label{22intertwine'}
 \cO_{D_{\bV^*}, \bV^*} \bPi = (I_{H^2} \otimes \Lambda_*)  \cO_{D_{T^*}, T^*}, \quad
 Q_{\bV^*} \bPi = \Gamma_{22} Q_{T^*}.
 \end{equation}
 We conclude that the transcription \eqref{uniqueness2} of the lifting-embedding map for the And\^o lift 
 \eqref{AndoLift2} can be rewritten as
 \begin{equation}  \label{lifting-embeddings2}
 \begin{bmatrix} \cO_{D_{\bV^*}, \bV^*} \\ Q_{\bV^*} \end{bmatrix} \bPi =
 \begin{bmatrix} (I_{H^2} \otimes \Lambda_*) \cO_{D_{T^*}, T^*} \\ \Gamma_{22} Q_{T^*} \end{bmatrix}.
 \end{equation}
 and the lifting property for the collection \eqref{AndoLift2} can be rewritten as
 \begin{align}
 & \begin{bmatrix} M_j^* & 0 \\ 0 & W_j^* \end{bmatrix} 
  \begin{bmatrix} (I_{H^2} \otimes \Lambda_*) \cO_{D_{T^*}, T^*} \\ \Gamma_{22} Q_{T^*} \end{bmatrix}
  =  \begin{bmatrix} (I_{H^2} \otimes \Lambda_*) \cO_{D_{T^*}, T^*} \\ \Gamma_{22} Q_{T^*} \end{bmatrix} T_j^*
\text{ for } j=1,2 \notag \\
&  \begin{bmatrix} (M_z^{\cD_{\bV^*}})^* & 0 \\ 0 & W^* \end{bmatrix}
 \begin{bmatrix} (I_{H^2} \otimes \Lambda_*) \cO_{D_{T^*}, T^*} \\ \Gamma_{22} Q_{T^*} \end{bmatrix}
 = \begin{bmatrix} (I_{H^2} \otimes \Lambda_*) \cO_{D_{T^*}, T^*} \\ \Gamma_{22} Q_{T^*} \end{bmatrix} T^*.\label{LiftPropRewrite}
 \end{align}
 Finally let us note that the fact that we have now identified $\Lambda_* \colon \cD_{T^*} \to \cD_{\bV^*}$ as an isometry means that the tuple $(\cD_{\bV^*}, P_*, U_*, \Lambda_*)$
  is a pre-And\^o tuple for $(T_1^*,T_2^*)$. As  in the forward direction (the analysis following equations \eqref{intertwine1}-\eqref{intertwine4}), a comparison of the coefficients of
 $$
 M_j^*(I_{H^2} \otimes \Lambda_*) \cO_{D_{T^*}, T^*} =(I_{H^2} \otimes \Lambda_*) \cO_{D_{T^*}, T^*} T_j^* \text{ for } j=1,2 
 $$yields that $(\cD_{\bV^*}, P_*, U_*, \Lambda_*)$ is actually a Type I And\^o tuple for $(T_1^*,T_2^*)$. 
 
By construction the ambient space $\cK_{00}$ \eqref{Kmin-circ} for the minimal lift 
 $\sbm{ M_z^{\cD_{\bV^*}} & 0 \\ 0 & W} \bigg |_{\cK_{00}}$
 of the product contraction operator $T = T_1 T_2$ embedded inside $\cK_0$ is given by
 \begin{equation}   \label{Kmin'}
   \cK_{00} = \widehat \Gamma \begin{bmatrix} H^2(\cD_{T^*}) \\ \cQ_{T^*} \end{bmatrix}
    = \begin{bmatrix} (I_{H^2} \otimes \Lambda_*)  H^2(\cD_{T^*}) \\ \Gamma_{22}  \cQ_{T^*} \end{bmatrix},
\end{equation}
i.e., we finally see that the subspace $\cK_{00}$ initially defined as in \eqref{Kmin-circ} splits with respect to the
orthogonal block-decomposition  appearing in the ambient space $\cK = \sbm{ H^2(\cD_{\bV^*}) \\ \cQ_{\bV^*}}$.

We would like next to analyze the copy of the minimal And\^o isometric lift  for the commuting, contractive pair
$(T_1, T_2)$ obtained by restricting
$\sbm{ M_1 & 0 \\ 0 & W_1}, \sbm{ M_2 & 0 \\ 0 & W_2}$ (($M_1$, $M_2)$ as in \eqref{shorthand}) to the 
subspace $\cK_0$ given by
\begin{align*}
\cK_{0} & = \bigvee_{n_1, n_2 \ge 0} \bigg\{ \begin{bmatrix} M_1 & 0 \\ 0 & W_1 \end{bmatrix}^{n_1}
\begin{bmatrix} M_2 & 0 \\ 0 & W_2 \end{bmatrix}^{n_2} 
\begin{bmatrix} \cO_{D_{\bV^*}, \bV^*}\\ Q_{\bV^*}  \end{bmatrix} 
\operatorname{Ran} \bPi \bigg\}  \\
& = \bigvee_{n_1, n_2 \ge 0} \bigg\{ \begin{bmatrix} M_1 & 0 \\ 0 & W_1 \end{bmatrix}^{n_1}
\begin{bmatrix} M_2 & 0 \\ 0 & W_2 \end{bmatrix}^{n_2}  \begin{bmatrix} (I_{H^2} \otimes \Lambda_*) 
\cO_{D_{T^*}, T^*} \\ \Gamma_{22} Q_{T^*} \end{bmatrix} \cH.
\end{align*}
As $\cK_0$ a priori is given by the same formula but with  the non-negative integers $n_1, n_2$ constrained to satisfy
$n_1 = n_2$,  it is clear that  we have the nesting of subspaces
$$
\cK_{00} \subset \cK_0 \subset \cK.
$$
Due to the splitting appearing in the 
formula \eqref{Kmin'} for $\cK_{00}$, we see that $\cK_0$ can be rewritten as
\begin{align}
 \cK_0 & =  \bigvee_{n_1, n_2 \ge 0} \bigg\{ \begin{bmatrix} M_1 & 0 \\ 0 & W_1 \end{bmatrix}^{n_1}
\begin{bmatrix} M_2 & 0 \\ 0 & W_2 \end{bmatrix}^{n_2}   \begin{bmatrix} (I_{H^2} \otimes \Lambda_*) H^2(\cD_{T^*}) \\ \Gamma_{22} \cQ_{T^*} \end{bmatrix}   \notag \\
& = \begin{bmatrix}  \bigvee_{n_1, n_2 \ge 0} M_1^{n_1} M_2^{n_2} \Gamma  \, H^2(\cD_{T^*}) \\
  \bigvee_{n_1, n_2 \ge 0} W_1^{n_1} W_2^{n_2} \Gamma_{22}  \, \cQ_{T^*} \end{bmatrix} 
   \label{Kmin}
\end{align}

We first analyze the bottom component
$$
 \bigvee_{n_1, n_2 \ge 0} W_1^{n_1} W_2^{n_2} \Gamma_{22}  \, \cQ_{T^*}.
$$
 From the second intertwining relation in \eqref{intertwine2'} we see that $\operatorname{Ran} \Gamma_{22}$
 is invariant under $W$.   On the other hand, we see from \eqref{LiftPropRewrite} and  the intertwining
 properties \eqref{IntofQ} that
 $$
 W^* \Gamma_{22} Q_{T^*} = \Gamma_{22} Q_{T^*} T^* = \Gamma_{22} W_D^* Q_{T^*}.
 $$
 More generally, for $k=1,2,\dots$ we have
 \begin{align*}
 W^* \Gamma_{22} (W_D^k Q_{T^*}) &  = W^* W^k \Gamma_{22} Q_{T^*} = W^{k-1} \Gamma_{22} Q_{T^*}
 = \Gamma_{22} W_D^{k-1} Q_{T^*}  \notag \\
 & = \Gamma_{22} W_D^* (W_D^k Q_{T^*}).
 \end{align*}
 As $\cup_{k=0}^\infty W_D^k \operatorname{Ran} Q_{T^*}$ is dense in $\cQ_{T^*}$ by definition (see \eqref{cQT*}), 
 we conclude
 that $\operatorname{Ran} \Gamma_{22}$ is also invariant under $W^*$ and furthermore we have the intertwining
 $W^* \Gamma_{22} = \Gamma_{22} W_D^*$. Thus we conclude that  $\Gamma_{22} \colon \cQ_{T^*}
 \to \Gamma_{22} \cQ_{T^*}$ is unitary and implements a unitary equivalence between the unitary operator
 $W|_{\operatorname{Ran} \Gamma_{22}}$ on $\operatorname{Ran} \Gamma_{22}$ and
 the unitary operator $W_D$ on $\cQ_{T^*}$.
 
 We next argue that $\operatorname{Ran} \Gamma_{22}$ is reducing for $W_j$ ($j=1,2$) as well 
 and that the unitary operator $\Gamma_{22} \colon \cQ_{T^*} \to \operatorname{Ran} \Gamma_{22}$
 implements a unitary equivalence between the operator $W_{\flat j}$ on $\cQ_{T^*}$ (given by  Theorem
 \ref{T:flats}) and $W_j|_{\operatorname{Ran}
 \Gamma_{22}}$ on $\operatorname{Ran} \Gamma_{22}$.  Indeed, from \eqref{LiftPropRewrite} together
 with the intertwining given in Theorem \ref{T:flats},  we see that, for $j=1,2$,
 $$ 
 W_j^* \Gamma_{22} Q_{T^*} = \Gamma_{22} Q_{T^*} T_j^* = \Gamma_{22} W_{\flat j}^* Q_{T^*}
 $$
 and more generally, for $k=0, 1,2, \dots$, 
 \begin{align}
 W_j^* \Gamma_{22} (W_D^k Q_{T^*}) &  = W^k W_j^* \Gamma_{22} Q_{T^*}
 = W^k \Gamma_{22} W_{\flat j}^* Q_{T^*} = \Gamma_{22} W_D^k W_{\flat j}^* Q_{T^*}  \notag \\
 & = \Gamma_{22} W_{\flat j}^* (W_D^k Q_{T^*}).
 \label{need1}
\end{align}
 As $\cup_{k=0}^\infty \operatorname{Ran} W_D^k Q_{T^*}$ is dense in $\cQ_{T^*}$, these computations show 
 not only that $\operatorname{Ran} \Gamma_{22}$ is invariant under $W_j^*$ but also the intertwining:
 $W_j^* \Gamma_{22} = \Gamma_{22} W_{\flat j}^*$ for $j=1,2$. Also, by taking adjoint of these intertwining relations and using the fact that the operators $W_{\flat j}$ and $W_j$ (for $j=1,2$) all are unitary operators, we see that $\Gamma_{22}W_{\flat j}=W_j\Gamma_{22}$ for $j=1,2$. This implies that $\operatorname{Ran} \Gamma_{22}$  is invariant under $W_1$ and $W_2$ as well.
 
 
 Let us now recall the expression \eqref{Kmin} for $\cK_0$.  From the preceding analysis, we see that
 $$
 \cK_0 = \begin{bmatrix} \bigvee_{n_1, n_2 \ge 0} M_1^{n_1} M_2^{n_2} H^2(\cD_{T^*} ) \\
 \operatorname{Ran} \Gamma_{22} \end{bmatrix}.
 $$
Note that  minimality of the And\^o lift \eqref{AndoLift2} just means that 
$$
\cK_{0}  =   \cK = \begin{bmatrix} H^2(\cD_{\bV^*}) \\ \cD_{\bV^*} \end{bmatrix}
$$
which is equivalent to the following two conditions holding simultaneously:
\begin{align} 
&  H^2(\cD_{\bV^*}) = \bigvee_{n_1, n_2 \ge 0}  M_1^{n_1} M_2^{n_2} ( I_{H^2} \otimes \Lambda_*) H^2(\cD_{T^*}) 
  \label{min1}   \\
&  \cQ_{\bV^*} = \Gamma_{22} \cQ_{T^*}.   \label{min2}
\end{align}
By definition, condition \eqref{min1} holds exactly when the And\^o tuple 
$$
(\cD_{\bV^*}, \Lambda_*, P_*, U_*)
$$
is {\em minimal}, while the second condition \eqref{min2} holds exactly when $\Gamma_{22}$ is
{\em unitary} (i.e., a surjective isometry).  When this is the case, the preceding analysis shows that we can 
use $\Gamma_{22}$ to
identify $\cQ_{T^*}$ with $\cQ_{\bV^*}$ and that under this unitary identification the operators $W$, $W_1$, $W_2$
on $\cQ_{\bV^*}$ become the operators $W_D$, $W_{\flat,1}$, $W_{\flat,2}$ on $\cQ_{T^*}$ and the
collection \eqref{AndoLift2} has the form of the functional-model And\^o lift \eqref{DougAndoMod} based on
the Type I And\^o tuple $(\cD_{\bV^*}, \Lambda_*, P_*, U_*)$ for $(T_1^*, T_2^*)$.

Finally, let us recall that the And\^o lift \eqref{AndoLift2} is constructed so as to be unitarily equivalent to the 
original abstract And\^o lift 
$(\bPi, \bV_1, \bV_2)$.  Hence minimality of $(\bPi, \bV_1, \bV_2)$ is equivalent to minimality of
the collection \eqref{AndoLift2} as an And\^o lift of $(T_1, T_2)$.  Thus, if we start with a minimal
abstract And\^o lift, we have shown that it is unitarily equivalent to a functional model And\^o lift
\eqref{DougAndoMod} constructed from a minimal Type I And\^o tuple for $(T_1^*,T_2^*)$, in this case
$(\cD_{\bV^*}, \Lambda_*,  P_*, U_*)$.
 \end{proof}
 
 We now illustrate the Douglas model for And\^o lifts with a couple of special cases.
 
 \begin{example} \label{E:Dex}
 \textbf{Illustrative special cases for Douglas-model And\^o lifts.}
 
 \smallskip
 
\textbf{1. $(T_1, T_2)$ = BCL2-model commuting isometric operator-pair:}  Let us consider the special case
where $(T_1, T_2) = (V_1, V_2)$ is a BCL2-model for a commuting pair of isometries with product $V= V_1 V_2$ equal to a shift
\begin{align*}
&  (T_1, T_2, T:=T_1T_2 = T_2 T_1)   = (M_{U^*P^\perp + z U^*P}, M_{PU + z P^\perp U}, M_z^\cF)  \\
  & \quad  = ((I_{H^2} \otimes U^* P^\perp) + (M^\cF_z \otimes U^* P),  (I_{H^2} \otimes PU) + (M_z^\cF\otimes P^\perp U), M_z^\cF)
\end{align*}
on $H^2(\cF)$ (where the latter tensor-product formulation is often more convenient for computations),
where $(\cF, P, U)$ is a BCL-tuple ($\cF$ = a coefficient Hilbert space while $P$ is a projection and $U$ is a unitary operator on $\cF$).
In this case $(T_1, T_2)$ is already a commutative isometric pair, so $(V_1, V_2) = (T_1, T_2)$ is a minimal isometric lift of itself $(T_1, T_2)$.  
When we go through the construction in the proof of Theorem \ref{T:Dmodel}, we see that the Douglas-minimal Type II And\^o tuple which we are led to
consists of the original BCL-tuple $(\cF, P, U)$ augmented by the isometric embedding map $\Lambda \colon \cD_{T^*} \to \cF$ given by
$\Lambda \colon D_{T^*} h = D_{V^*} h \mapsto h(0) \in \cF$ for $h \in H^2(\cF)$.  It is easy to check directly that $(\cF, \Lambda, P, U)$ so defined is a
Douglas-minimal Type I And\^o tuple for $(T_1^*, T_2^*) = (V_1^*, V_2^*)$ (with trivial commuting unitary-operator piece $(W_1, W_2)$), 
as expected from the general computations underlying the proof of Theorem \ref{T:Dmodel}.
This makes precise how a BCL2-model embeds into a Douglas-model And\^o lift.

More abstractly, using the BCL2-model for any commuting isometric operator-pair $(V_1, V_2)$ such that $V_1 V_2$ is a shift operator, we can assert:
{\em  if $(T_1, T_2)$ is a commuting isometric pair such that $T = T_1 T_2$ is a shift operator, then any Douglas-minimal Type I And\^o tuple $(\cF, \Lambda, P, U)$
for $(T_1^*, T_2^*)$ has the property that $\Lambda \colon \cD_{T^*} \to \cF$ is unitary and $(\cF, P, U)$ is a BCL2 tuple for the commuting isometric pair $(T_1, T_2)$.}

\smallskip

\textbf{2. $(T_1, T_2)$ = the adjoint of a BCL2-model commuting isometric operator pair:}  We now consider the case where $(T_1, T_2) = (V_1^*, V_2^*)$
is a commuting co-isometric operator-pair, where $(V_1, V_2)$ is as in the previous example (1).  Thus
\begin{align*}
& (T_1, T_2, T) = (V_1^*, V_2^*, V^*) = \\
& \quad ((I_{H^2} \otimes P^\perp U) + ((M_z^\cF)^* \otimes PU),  (I_{H^2} \otimes U^*P) + ( (M_z^\cF)^* \otimes U^* P^\perp ), (M_z^\cF)^*).
\end{align*}
The And\^o lifting problem comes down to the commuting-unitary extension problem treated in Lemma \ref{L:uniqueness-flats}.  The result is that the minimal isometric
lift consists of the commuting unitary operators $(M_{\varphi_1^*}, M_{\varphi_2^*})$ acting on $L^2(\cF)$, where
$$
 \varphi_1^*(\zeta) = P^\perp U + \overline{\zeta} PU, \quad \varphi_2^*(\zeta) = U^* P + \overline{\zeta} U^* P^\perp.
 $$
 The BCL2-tuple has trivial shift part $(\cF, \Lambda, P, U)$ and non-trivial commuting unitary part $(M_{\varphi_1^*}, M_{\varphi_2^*})$ acting on $L^2(\cF)$.
  \end{example}

\section{Special And\^o tuples}\label{S:SpcATuple}

In this section, starting with a commuting contractive operator-pair $(T_1,T_2)$, we construct a class of
pre-And\^o tuples for $(T_1,T_2)$ which we refer to as {\em special And\^o tuples} for $(T_1, T_2)$.
Special And\^o tuples, when put in canonical form, are defined in a constructive manner, so there is no issue
with their existence as is the case for Type I And\^o tuples defined in terms of the existence of solutions of some operator
equations.  We shall see that any special And\^o tuple for $(T_1^*, T_2^*)$ is also
a Type I And\^o tuple for $(T_1^*, T_2^*)$, thereby confirming the existence of Type I And\^o tuples.
 Conversely, for each computable simple example that we have been able
to work out the converse holds:  any Type I And\^o tuple for $(T_1^*, T_2^*)$ turns out to be special.  However,
the validity of the converse statement in full generality remains an open problem.
We think it instructive to push the formalism as far as possible without a commutativity
assumption, and then see the further consequences arising from a commutativity assumption.

Suppose that $T \colon \cH_2 \to \cH_0$ is a (possibly non-square) Hilbert-space contraction operator having a factorization
$$
   T = T' T''
$$
where $T' \colon \cH_1 \to \cH_0$ and $T'' \colon \cH_2 \to \cH_1$ are also contraction operators.  Then the identity
\begin{align}
  D_T^2 & = I - T^*T = I - T^{\prime \prime *} T^{\prime *} T' T'' = T^{\prime \prime *}( I - T^{\prime *} T' ) T''+ I -  T^{\prime \prime *} T''  \notag \\
  & = T^{\prime \prime *} D_{T'}^2 T''  +  D_{T''}^2
  \label{id1-nc}
\end{align}
implies that the map $Z \colon \cD_T \to \cD_{T'} \oplus \cD_{T''}$ defined densely by
\begin{equation} \label{Z}\index{$Z$}
Z \colon D_T h \mapsto D_{T'}T'' h \oplus D_{T''} h \text{ for } h \in \cH
\end{equation}
extends to an isometry.  Such isometries come up in the characterization of invariant subspaces for a contraction operator in terms of its
Sz.-Nagy--Foias functional model (see \cite[Chapter VII]{Nagy-Foias} and Chapter \ref{C:invsub} below).  An important special case is the
case where $T = T' \cdot T''$ is a {\em regular factorizaton} as defined next.

\begin{definition} \label{D:RegFact} 
\index{regular factorization for!contractions} 
Given a contraction operator $T \colon \cH_0 \to \cH$
with factorization $T = T'  \cdot T''$ for contraction operators $T'  \colon \cH_1 \to \cH$ and $T'' \colon \cH_0 \to \cH_ 2$, we say that
the  factorization
$T = T' \cdot T''$ is {\em regular} if it is the case that the operator
$Z \colon \cD_T \to \cD_{T'} \oplus \cD_{T''}$ in \eqref{Z} is surjective (so $Z \colon \cD_T \to \cD_{T'} \oplus \cD_{T''}$ is unitary).
\end{definition}

Now let us suppose that $(T_1, T_2)$ is a commuting contractive pair on $\cH$ and that we set $T = T_1 \cdot T_2 = T_2 \cdot T_1$.
Then we have two versions of \eqref{id1-nc} corresponding to setting $T' = T_1$, $T''= T_2$ or the reverse $T' = T_2$, $T'' = T_1$:
\begin{equation}   \label{id1}
D_T^2 = T_2^* D_{T_1}^2 T_2  +  D_{T_2}^2 , \quad D_T^2 = D_{T_1}^2 + T_1^* D_{T_2}^2 T_1
\end{equation}
from which we conclude in particular that
\begin{equation}   \label{id2}
T_2^*D_{T_1}^2T_2 + D_{T_1}^2 = D_{T_1}^2 + T_1^* D_{T_2}^2 T_1.
\end{equation}
As a consequence of the first identity in \eqref{id1}, we see  that the operator $\Lambda_\dag \colon \cD_T \to \cD_{T_1} \oplus \cD_{T_2}$ defined densely by
\begin{equation}   \label{defLambda}
  \Lambda_\dag \colon D_T h \mapsto D_{T_1} T_2 h \oplus D_{T_2} h \text{ for all } h \in \cH
\end{equation}
is an isometry from $\cD_T$ into $\cD_{T_1} \oplus \cD_{T_2}$.
Let us introduce the notation\index{$\cD_{U_0}$}
\index{$\cR_{U_0}$}
\begin{align}
& \cD_{U_0} : =\operatorname{clos.}  \{ D_{T_1} T_2 h \oplus D_{T_2} h \colon h \in \cH\} \subset \cD_{T_1} \oplus \cD_{T_2}, \notag \\
& \cR_{U_0}:=\operatorname{clos.}  \{ D_{T_1} h \oplus D_{T_2} T_1 h \colon h \in \cH\} \subset \cD_{T_1} \oplus \cD_{T_2}.
\label{U0dom-codom}
\end{align}
As a consequence of \eqref{id2}, we see that the operator $U_0 \colon \cD_{U_0} \to \cR_{U_0}$
defined densely by
\begin{equation}  \label{U0}
U_0 \colon D_{T_1} T_2 h \oplus D_{T_2} h \mapsto  D_{T_1} h \oplus D_{T_2} T_1 h \text{ for all } h \in \cH
\end{equation}
is unitary.   If $\cD_{U_0}^\perp$ and $\cR_{U_0}^\perp$ have the same dimension, we can find a unitary identification map
$U_{00} \colon \cD_{U_0}^\perp \to \cR_{U_0}^\perp$ and then define a unitary operator 
$$
U_\dag \colon \cD_{T_1} \oplus \cD_{T_2} \to \cD_{T_1} \oplus \cD_{T_2} 
$$
by setting 
\begin{equation}   \label{Udagger}
U_\dag|_{\cD_{U_0}} = U_0, \quad U_\dag|_{ \cD_{U_0}^\perp} = U_{00}
\end{equation}
and then extending by linearity.  Even if it is the case that $\cD_{U_0}^\perp$ and $\cR_{U_0}^\perp$ have different dimensions, we may introduce 
an infinite-dimensional Hilbert space $\cF_{\dag 0}$ so that $\cD_{U_0}^\perp \oplus \cF_{\dag 0}$ and $\cR_{U_0}^\perp \oplus \cF_{\dag 0}$
have the same dimension (here we are assuming that all Hilbert spaces are separable), introduce a unitary identification map 
$U_{00} \colon  \cD_{U_0}^\perp \oplus \cF_{\dag 0} \to \cR_{U_0}^\perp \oplus \cF_{\dag 0}$,
and then obtain a unitary operator $U_\dag$ on the larger space $\cF_\dag := \cD_{U_0} \oplus \cR_{U_0} \oplus \cF_{\dag 0}$ by setting
$$
  U_\dag|_{\cD_{U_0}}= U_0, \quad U_\dag|_{\cD_{U_0}^\perp \oplus \cF_{\dag 0} }= U_{00}
$$
and extending by linearity.

In addition we let $P_\dagger$ be any projection operator on $\cF_{\dagger}$ which is an extension of the projection operator
$$
  P_0 \colon  d \oplus r \mapsto d \oplus 0
$$
defined on $\cD_{T_1} \oplus \cD_{T_2}$, i.e., $P_\dagger$ is any projection operator on $\cF_\dagger$ of the form
\begin{equation}  \label{Pdagger}
P_\dagger \colon d \oplus r \oplus f_0 \mapsto d \oplus 0 \oplus P_{\dagger 0} f_0 \text{ for } d \oplus r \oplus f_0 \in \cD_{T_1} \oplus \cD_{T_2} \oplus \cF_{\dagger 0}
= \cF_\dagger
\end{equation}
where $P_{\dagger 0}$ is any choice of orthogonal projection on $\cF_{\dagger 0}$.
For future reference we now introduce the formal definition of {\em special And\^o tuple}. \index{And\^o tuple!special}

\begin{definition} \label{D:SpcATuple}
Given a commuting contractive operator-pair $(T_1,T_2)$ on a Hilbert space
$\cH$,\index{$(\cF_\dag, \Lambda_\dag, P_\dag , U_\dag)$}\index{$(\cF_{\dag*}, \Lambda_{\dag*}, P_{\dag*}, U_{\dag*})$}
any collection of spaces and operators $(\cF_\dag, \Lambda_\dag, P_\dag , U_\dag)$ constructed as in
\eqref{defLambda}, \eqref{Udagger}, \eqref{Pdagger}
will be called a {\em canonical special And\^o tuple} for the pair $(T_1, T_2)$.   We shall
say that  pre-And\^o tuple
$(\cF'_\dag, \Lambda'_\dag, P'_\dag , U'_\dag)$ coinciding (in the sense of Definition \ref{D:preAndoTuple})
with a canonical special And\^o tuple 
$(\cF_\dag, \Lambda_\dag, P_\dag , U_\dag)$ for $(T_1, T_2)$ is simply a {\em special And\^o tuple} for $(T_1, T_2)$
(not necessarily in canonical form).
Canonical special And\^o tuples for the adjoint pair $(T_1^*,T_2^*)$
will often be denoted by $(\cF_{\dag*}, \Lambda_{\dag*}, P_{\dag*}, U_{\dag*})$.
We shall say that $(\cF_\dag, \Lambda_\dag, P_\dag , U_\dag)$ is an {\em irreducible special And\^o tuple}
if in addition it is the case that the smallest subspace of $\cF_\dagger$ containing $\operatorname{Ran} \Lambda_\dag$ which is invariant for
$U_\dagger$, $U_\dagger^*$, and $P_\dagger$  is the whole space $\cF_\dagger$.  If
$(\cF_{\dag}, \Lambda_{\dag}, P_{\dag}, U_{\dag})$ is a canonical special And\^o tuple,  irreducibility means that the
smallest reducing subspace for $U_\dag$ containing $\cD_{T_1} \oplus \cD_{T_2}$ is the whole space $\cF_\dag$.
\end{definition}

\index{And\^o tuple!canonical special}Note that the point of the distinction between {\em canonical special And\^o tuple} and {\em special And\^o tuple}
for $(T_1, T_2)$ is that the notion of {\em canonical special And\^o tuple} is not coordinate-free.  The enlarged class
of {\em special And\^o tuples} as in Definition \ref{D:SpcATuple} then is the coincidence envelope of the
{\em canonical special And\^o tuples} (the smallest collection of And\^o tuples containing the canonical special And\^o tuples which is invariant
under the relation of coincidence).

\begin{notation}  \label{N:flat}
For a given commuting contractive operator-pair $(T_1,T_2)$, the Douglas model for And\^o isometric lifts of $(T_1,T_2)$
corresponding to some canonical special And\^o tuple  for $(T_1^*,T_2^*)$ will be denoted by $(\Pi_\flat, V_{\flat1},V_{\flat2})$ and
$V_\flat=V_{\flat1}V_{\flat2}$.
\end{notation}

\begin{remark}  \label{R:specATuple}
We note that canonical special And\^o tuples  are easily constructed.   There are various scenarios possible
which we discuss in turn.

\smallskip

\noindent
\textbf{Scenario 1:   $(T_1, T_2)$ such that both $T = T_1 \cdot T_2$ and $T = T_2 \cdot T_1$ are regular factorizations.}

\smallskip

\noindent
 We have seen that the first identity in \eqref{id1}
implies that $\Lambda_\dagger$ given by \eqref{defLambda} is an isometry.  Let us now note also that the second identity in \eqref{id1}
implies that the map $\Lambda'_\dagger \colon \cD_T \to \cD_{T_1} \oplus \cD_{T_2}$ given by
$$
\Lambda'_\dagger \colon D_T h \mapsto D_{T_1} h \oplus D_{T_2} T_1 h \text{ for } h \in \cH
$$
is an isometry.  Note also from the definitions that
$$
   \operatorname{Ran} \Lambda_\dagger = \cD_{U_0}, \quad \operatorname{Ran} \Lambda_\dagger' = \cR_{U_0}.
$$
By definition  $T = T_1 \cdot T_2$ is a regular factorization exactly when $\overline{\operatorname{Ran}}\, \Lambda_\dagger$ is the whole space
$\cD_{T_1} \oplus \cD_{T_2}$, i.e.,
$$
  \cD_{U_0} = \cD_{T_1} \oplus \cD_{T_2}.
$$
Similarly, $T = T_2 \cdot T_1$ is a regular factorization exactly when $\operatorname{Ran} \Lambda'_\dagger = \cD_{T_1} \oplus \cD_{T_2}$,
i.e., when
$$
  \cR_{U_0} = \cD_{T_1} \oplus \cD_{T_2}.
$$
Thus, when it is the case that both $T = T_1 \cdot T_2$ and $T = T_2 \cdot T_1$ are regular factorizations, $U_0$ given by \eqref{U0}
actually already defines a unitary operator on $\cF = \cD_{T_1} \oplus \cD_{T_2}$ and conversely:  if $\cD_{U_0} = \cR_{U_0}
= \cD_{T_1} \oplus \cD_{T_2}$ so $U_0$ already defines a unitary operator on $\cD_{T_1} \oplus \cD_{T_2}$, then both
$T = T_1 \cdot T_2$ and $T = T_2 \cdot T_1$ are regular factorizations.
If this is the case, then any other unitary extension $\widetilde U$ must agree with $U_0$ on $\cD_{U_0} = \cD_{T_1} \oplus \cD_{T_2}$
and hence $\cD_{T_1} \oplus \cD_{T_2}$ is already reducing for $\widetilde U$ forcing us to the conclusion that $\widetilde U = U_0$.
We conclude: {\em  in case both $T_1 \cdot T_2$ and $T_2 \cdot T_1$ are regular factorizations, then the commuting, contractive pair
$(T_1, T_2)$ has a unique  irreducible special And\^o tuple in canonical form.}

\smallskip

\noindent
 \textbf{Scenario 2: $\dim \cD_{U_0}^\perp = \dim \cD_{R_0}^\perp > 0$}
 (here the orthogonal complements are taken with respect to the ambient space $\cD_{T_1} \oplus \cD_{T_2}$.) In this case, we can extend
$U_0$ to a unitary operator  $U$ on all of $ \cF:= \cD_{T_1} \oplus \cD_{T_2}$ by choosing any unitary identification map
$U'_0 \colon \cD_{U_0}^\perp \to \cR_{U_0}^\perp$, then defining
$$
  U|_{\cD_{U_0}} = U_0 \colon \cD_{U_0} \to \cR_{U_0}, U|_{\cD_{U_0}^\perp} = U'_0 \colon  \cD_{U_0}^\perp \to \cR_{U_0}^\perp
$$
and then extending $U$ to a unitary operator  $U$ on all of
$$
\cD_{T_1} \oplus \cD_{T_2} = \cD_{U_0} \oplus \cD_{U_0}^\perp = \cR_{U_0} \oplus \cR_{U_0}^\perp
$$
by linearity.  Then any such $U$ induces a irreducible And\^o tuple $(\cF = \cD_{T_1} \oplus \cD_{T_2}, P, U)$
since the subspace $\cD_{T_1} \oplus \cD_{T_2}$ is already reducing for $U$.  However, we see that there is freedom
in the choice of the operator $U'_0 \colon \cD_{U_0}^\perp \to \cR_{U_0}^\perp$.  We conclude that
in this case {\em irreducible special And\^o tuples exist but do not have a unique representative in canonical form.}
 In fact, as we shall see
in examples to come, there are canonical minimal special And\^o tuples with the ambient Hilbert space $\cF$ properly
containing $\cD_{T_1} \oplus \cD_{T_2}$.

 \smallskip

 \noindent
 \textbf{Scenario 3: $\dim \cD_{U_0}^\perp \ne \dim \cD_{R_0}^\perp$.}
In this case, we may enlarge the
ambient space $\cD_{T_1} \oplus \cD_{T_2}$
to the space $\cF: = \cD_{T_1} \oplus \cD_{T_2} \oplus \ell^2$ (here $\ell^2$ can be taken to be any separable
infinite-dimensional Hilbert space).  Then we are in the Scenario 2 setting
$$
\dim \cD_{U_0}^\perp = \dim \cD_{R_0}^\perp = \infty
$$
but where we now take the ambient space to be $\cF = \cD_{T_1} \oplus \cD_{T_2} \oplus \ell^2$ rather than just
$\cD_{T_1} \oplus \cD_{T_2}$.  We may then proceed as in Scenario 2 to construct a unitary extension of
$U_0$.  We can always arrange for the resulting canonical-form special And\^o tuple $(\cF_\dagger, \Lambda_\dagger, U_\dagger, P_\dagger )$
to be irreducible by cutting $\cF_\dagger$ down to the smallest reducing subspace for $U_\dagger$ containing
$\cD_{T_1} \oplus \cD_{T_2}$.  Due to the freedom in the choice of unitary extension of the partially defined $U_0$,
it is clear that irreducible special And\^o tuples for $(T_1, T_2)$ are not unique.
\end{remark}

As a corollary of the extended discussion in Remark \ref{R:specATuple} we get:

\begin{corollary}  \label{C:UniqueMinAndoTuple}
(1)  A commuting contractive pair $(T_1, T_2)$ always has a
special And\^o tuple.
\smallskip

\noindent
(2)  $(T_1, T_2)$ has a unique (up to coincidence)
irreducible special And\^o tuple if and only if $T1 \cdot T2$ and $T_2 \cdot T_1$ are regular factorizations.
\end{corollary}

The question arises as to whether it suffices to assume that only one of the factorizations $T_1 \cdot T_2$
and $T_2 \cdot T_1$ is regular.  This issue is resolved
by the following result.

\begin{theorem}  \label{T:regfact}
   Let $(T_1, T_2)$ be a commuting, contractive pair of contraction operators
on a Hilbert space $\cH$.
\begin{enumerate}
\item Assume that all defect spaces $\cD_{T}$, $\cD_{T_1}$,  $\cD_{T_2}$ are finite-dimensional. Then
$T_1 \cdot T_2$ is a regular factorization if and only if $T_2 \cdot T_1$ is a regular factorization.

\item In the infinite-dimensional setting, it is possible for one of the factorizations $T_1 \cdot T_2$ (respectively $T_2 \cdot T_1$) to be regular
while the other $T_2 \cdot T_1$ (respectively $T_1 \cdot T_2$) is not regular.

\item Suppose that $(T_1, T_2)$ is a commuting pair of isometries (so that trivially both $T_1 \cdot T_2$ and $T_2 \cdot T_1$ are regular factorizations).
Then $(T_1^*, T_2^*)$ has the {\em double regular-factorization property,}  namely:  both $T_1^* \cdot T_2^*$ and $T_2^* \cdot T_1^*$ are regular factorizations. 
\end{enumerate}
\end{theorem}

\begin{proof}[\sf{Proof of (1):}]
Let us write $\Lambda_r$ for the counterpart of $\Lambda$ when the roles of $T_1$ and $T_2$ are reversed:
$$
\Lambda_r \colon D_T h \mapsto D_{T_1} T_2 h \oplus D_{T_2} h.
$$
Then as a consequence of the second identity in \eqref{id1} we see that
$\Lambda_r$ is an isometry from $\cD_T$ onto $\cR(U_0)$.  Moreover, if $T_1 \cdot T_2$ is a regular factorization, then
\begin{align*}
 \dim \cD_T & =  \dim \cD(U_0) \text{ (since $\Lambda$ is an isometry from $\cD_T$ onto $\cD_{U_0}$)} \\
 & = \dim (\cD_{T_1} \oplus \cD_{T_2}) \text{ (by regularity of $T_1 \cdot T_2$)} \\
 & \ge \dim \cR_{U_0} \text{ (since $\cR_{U_0} \subset \cD_{T_1} \oplus \cD_{T_2}$) }\\
 & = \dim \cD_T \text{ (since $\Lambda_r$ is an isometry from $\cD_T$ onto $\cR_{U_0}$).}
  \end{align*}
 Thus we see that necessarily the inequality in line 3 must be equality.  If we are in the finite-dimensional setting
 ($\cD_T$, $\cD_{T_1}$, $\cD_{T_2}$ all finite-dimensional), we necessarily then have
 $\cR_{U_0} = \cD_{T_1} \oplus \cD_{T_2}$ which is the statement that the factorization $T_2 \cdot T_1$
 is also regular.

 \smallskip

 \noindent
 {\sf Proof of (2):}
Let $(T_1, T_2)$ be the following pair of contractions on
$\cH: = H^2 \oplus H^2$:
$$
  (T_1, T_2) =\left( \begin{bmatrix} 0 & 0 \\ I & 0 \end{bmatrix}, \begin{bmatrix}T_z & 0 \\ 0 & T_z \end{bmatrix} \right)
$$
where $T_z$ is the Toeplitz operator $f(z) \mapsto z f(z)$ on $H^2$.  Note that
$$
 T_1 T_2 = \begin{bmatrix} 0 & 0 \\ T_z & 0 \end{bmatrix} = T_2 T_1 =: T.
$$
Then
\begin{align*}
& D_T^2 = \begin{bmatrix} I_{H^2} & 0 \\ 0 & I_{H^2} \end{bmatrix}  -
\begin{bmatrix} 0 & T_z^* \\ 0 & 0 \end{bmatrix} \begin{bmatrix} 0 & 0 \\ T_z & 0 \end{bmatrix}
 = \begin{bmatrix}  0 & 0 \\ 0 & I_{H^2} \end{bmatrix} = D_T, \quad \cD_T = \begin{bmatrix} 0 \\ H^2 \end{bmatrix} \\
& D_{T_1}^2 = \begin{bmatrix} I & 0 \\ 0 & I \end{bmatrix} - \begin{bmatrix} 0 & I \\ 0 & 0 \end{bmatrix}
\begin{bmatrix} 0 & 0 \\ I & 0 \end{bmatrix} = \begin{bmatrix} 0 & 0 \\ 0 & I \end{bmatrix} = D_{T_1}, \quad
\cD_{T_1} = \begin{bmatrix}  0 \\ H^2 \end{bmatrix}, \\
& D_{T_2}^2 = \begin{bmatrix} I & 0 \\ 0 & I \end{bmatrix} - \begin{bmatrix} T_z^* & 0 \\ 0 & T_z^* \end{bmatrix}
\begin{bmatrix} T_z & 0 \\ 0 & T_z \end{bmatrix} = \begin{bmatrix} 0 & 0 \\ 0 & 0 \end{bmatrix}, \quad
\cD_{T_2} = \{0 \},
\end{align*}
and  $\cD_{T_1} \oplus \cD_{T_2} = \sbm{ 0 \\ H^2}$.

We let $\Lambda$ be the map associated with the factorization $T_1 \cdot T_2$ given by \eqref{defLambda}
while $\Lambda_r$ is the same map associated with the factorization
$T_2 \cdot T_1$ (i.e., \eqref{defLambda} but with the indices and then the components interchanged).
For $h = \sbm{h_1 \\ h_2 } \in \cH = \sbm{ H^2 \\ H^2}$
we compute
$$
\Lambda \colon D_T h = \sbm{ 0 \\ h_2 } \mapsto
D_{T_1} T_2 h \oplus D_{T_2} h = \sbm{ 0 \\ T_z h}
$$
and we conclude that
$$
  \operatorname{Ran} \Lambda = \sbm{ 0 \\ zH^2} \underset{\ne}\subset \sbm{ 0 \\ H^2 } = \cD_{T_1} \oplus
  \cD_{T_2}
$$
implying that $T_1 \cdot T_2$ is not a regular factorization.  On the other hand,
$$
\Lambda_r \colon D_T h = \sbm{ 0 \\ h_2 } \mapsto
D_{T_1} h \oplus D_{T_2} T_1 h = \sbm{ 0 \\ h_2} \oplus 0
$$
from which we see that
$$
\operatorname{Ran} \Lambda_r  =  \sbm{ 0 \\ H^2} = \cD_{T_1} \oplus \cD_{T_2}
$$
implying that the factorization $T_2 \cdot T_1$ is regular.

{\sf Proof of (3):}   This is an immediate corollary  of Proposition VII.3.2 (b) from \cite{Nagy-Foias} which asserts:  for a contractive pair $(T_1, T_2)$ on a Hilbert space 
$\cH$, {\em the factorization $T = T_1 T_2$
is regular whenever $T_1$ or $T_2^*$ is isometric.}  However this fact in turn can be seen as an immediate corollary of the 
following alternative characterization of regular factorization discovered somewhat later  (see \cite{SzNF-RegularFactor}):
{\em  $T_1 \cdot T_2$ is a regular factorization if and only if $\cD_{T_1} \cap \cD_{T_2^*}
= \{0\}$.} Let us note that this criterion also comes up in the characterization of triviality of overlapping spaces in the
de Branges-Rovnyak model theory (see \cite{Ball-Memoir}).
For a direct proof of item (3) in Theorem \ref{T:regfact} , see \cite[Lemma 27]{sauAndo}.
\end{proof}

\begin{remark}  \label{R:ArovGrossman}
The general issue arising here is the following:  {\em given subspaces  $\cD_{U_0}$ and $\cR_{U_0}$ of some Hilbert space
$\cH_0$ and a unitary map $U_0\colon \cD_{U_0} \to \cR_{U_0}$, find a unitary extensions $U \colon \cH \to \cH$ on a
possibly larger Hilbert space
$\cH \supset \cH_0$ so that $U|_{\cD_{U_0}} = U_0 \colon \cD_{U_0} \to \cR_{R_0}$ and $U$ is {\em minimal} in the sense that
the smallest reducing subspace for $U$ containing $\cH_0$ is all of $\cH$.}  This problem is the core of the
{\em lurking isometry technique} in interpolation theory (see e.g.\ \cite{Ball-Winnepeg}) but has a much earlier history as well (see e.g.\ \cite{ArovGrossman} 
for a thorough treatment).
\end{remark}

We note that the definitions of Type I And\^o tuples is existential:  we have not verified that one can solve the equations \eqref{AndoTuple1} and\eqref{AndoTuple2} for 
$\Lambda_*$, $P_*$, $U_*$, and at this stage we have not ruled out the
possibility that the set of Type I And\^o tuples is in fact empty.  On the other hand, the
preceding discussion shows in particular that it is always possible to construct special And\^o tuples (for $(T_1, T_2)$ or for $(T_1^*, T_2^*)$)
and in general, there are many such choices:   the possibilities are parametrized
by a choice of unitary extension $U$ for the partially defined isometry $U_0$.  The next result, therefore, has crucial significance
as it demonstrates that these constructions lead to a new proof of And\^o's theorem on the existence of And\^o  lifts for a given
commuting contractive operator pair.

\begin{theorem} \label{Thm:special}  Let $(T_1, T_2)$ be any commuting contractive operator-pair on $\cH$. Then
any special And\^o tuple for $(T_1^*, T_2^*)$ is also a Type I And\^o tuple for $(T_1^*, T_2^*)$.  In particular, these latter classes are not empty and And\^o  lifts of 
$(T_1, T_2)$ exist.
\end{theorem}

\begin{proof}
As the notion of Type I And\^o tuple for $(T_1^*, T_2^*)$ is coordinate-free, we may without loss of generality suppose
that the special And\^o tuple $(\cF_{\dag *}, \Lambda_{\dag *}, P_{\dag *}, U_{\dag *})$
for $(T^*_1, T_2^*)$ is given in canonical form.
Clearly this is a pre-And\^o tuple for $(T_1^*, T_2^*)$.
To show that this collection is a Type I And\^o tuple, we need only verify conditions \eqref{AndoTuple1} and \eqref{AndoTuple2}. We deal in detail only 
with \eqref{AndoTuple1} as verification of \eqref{AndoTuple2} is completely analogous. Making use of the defining properties of a canonical special And\^o tuple gives, 
for each $h \in\cH$,
\begin{align*}
&P_{\dag*}^\perp U_{\dag*}\Lambda_{\dag*} D_{T^*}h+P_{\dag*}U_{\dag*}\Lambda_{\dag*} D_{T^*}T^*h\\
&=P_{\dag*}^\perp U_{\dag*}(D_{T_1^*}T_2^*h\oplus D_{T_2^*}h)+P_{\dag*}U_{\dag*}(D_{T_1^*}T_2^*T^*h\oplus D_{T_2^*}T^*h)\\
&=P_{\dag*}^\perp (D_{T_1^*}h\oplus D_{T_2^*}T_1^*h)+P_{\dag*}(D_{T_1^*}T^*h\oplus D_{T_2^*}T_1^*T^*h)\\
&=(0\oplus D_{T_2^*}T_1^*h)+(D_{T_1^*}T_2^*T_1^*h\oplus 0)=\Lambda_{\dag*} D_{T^*}T_1^*h.
\end{align*}
As $h \in \cH$ is arbitrary, this verifies  \eqref{AndoTuple1} as wanted.

As we have seen in the extended Remark \ref{R:specATuple}, there always exist special And\^o tuples for any commuting contractive pair $(T_1^*, T_2^*)$. 
As the class of special And\^o tuples for $(T_1^*, T_2^*)$ forms a subclass of the class of  Type I And\^o tuples by the first part of the theorem, it follows that these 
latter classes are all non-empty. Then the constructions in Theorem \ref{T:Dmodel} based on a special And\^o tuple as a starting point leads to an explicit And\^o lift for 
$(T_1, T_2)$.
\end{proof}

\begin{remark} \label{R:reasonsD}  The reader may wonder why we use a BCL2 model for the And\^o lift of a given contractive pair $(T_1, T_2)$ rather than a BCL1 
model in our Douglas model for an And\^o lift.  Had we used a BCL1 model instead, we would have arrived at a notion of what we here call a Type I$'$ 
And\^o tuple $(\cF', \Lambda', P', Y')$ arising as follows. The form of the model \eqref{DougAndoMod} would have the adjusted form
\begin{align}
 & \cK = \sbm{ H^2(\cF') \\ \cQ_{T^*} }, \notag  \\
 & (V_1, V_2) =  \left( \sbm{ M_{(P'^\perp + z P')U'} & 0 \\ 0 & W_{\flat 1}},
  \sbm{ M_{U'^*(P' + z P'^\perp) U_*} & 0 \\ 0 & W_{\flat 2} } \right) \text{ acting on } \cK,  \notag \\
 &   \Pi = \sbm{ (I_{H^2} \otimes \Lambda')  \cO_{D_{T^*}, T^*} \\ Q_{T^*} } \colon \cH \to \cK.
 \label{RigidDiagMod'}
 \end{align}
 and the operator equations characterizing when such a collection of operators and spaces would actually yield an And\^o lift of $(T_1, T_2)$, i.e., 
 the analogue of \eqref{AndoTuple1} and \eqref{AndoTuple2} would be
 \begin{align}
 & U'^* P' \Lambda' D_{T^*} T^* + U'^* P'^\perp  \Lambda' D_{T^*} = \Lambda' D_{T^*} T_1^*  \label{AndoTuple1'} \\
 &  P'^\perp U' \Lambda' D_{T^*} T^* + P' U' \Lambda' D_{T^*} = \Lambda' D_{T^*} T_2^*. \label{AndoTuple2'}
 \end{align}
 These equations are obtained by replacing the part of the data $(P', U')$ with its flipped version (see \eqref{flip})
 $$
  (P'^\ff, U'^\ff)  =  (U'^* P' U',  U'^*)
 $$
 and then plugging this transformed data set into equations \eqref{AndoTuple1} and \eqref{AndoTuple2}.
 The drawback of this approach is that then the analogue of Theorem \ref{Thm:special} fails, i.e., it need not be the case
 that a special And\^o tuple is a Type I$'$ And\^o tuple.  An explicit example is given  in the Appendix (Section \ref{S:Appendix}) for the interested reader.
 \end{remark}
 
\section{The Sz.-Nagy--Foias model for an And\^o isometric lift}   \label{S:NFmodel2}
\index{And\^o lift!Sz.-Nagy--Foias-type functional model}
In this section, we convert the preceding analysis to a functional-model form to give a functional model for And\^o
lifts.

Let $\omega_D$ and $\omega_{\rm NF}$ be the unitaries as defined in (\ref{omegaD}) and (\ref{omegaNF}), respectively. We observed in 
part (ii) of Remark \ref{R:DougSchaf-vs-NF} that the unitary 
$$
\omega_{\text{NF,D}}:= \omega_{\rm NF}\omega_{\rm D}^*:\cQ_{T^*}\to\overline{\Delta_{\Theta_T} L^2(\cD_{\Theta_T})}
$$
intertwines $W_D$ with $M^{\cD_T}_\zeta |_{\overline{\Delta_{\Theta_T} L^2(\cD_T)}}$. 
Let us adopt the  notation
\index{$W_{\sharp 1},W_{\sharp 2}$}
\begin{align}\label{WandVNFs}
\notag&(W_{\sharp1},W_{\sharp 2},M^{\cD_T}_\zeta |_{\overline{\Delta_{\Theta_T}L^2(\cD_T)}})\\
&:=(\omega_{\text{NF,D}}W_{\flat1}\omega_{\text{NF,D}}^*,
\omega_{\text{NF,D}}W_{\flat2}\omega_{\text{NF,D}}^*, \omega_{\text{NF,D}}W_D\omega_{\text{NF,D}}^*),
\end{align}
where $(W_{\flat1},W_{\flat2})$ is the canonical pair of commuting unitaries for $(T_1,T_2)$ as in \eqref{S:flats}. Let $(\cF_{*},\Lambda_{*},P_{*},U_{*})$ be a Type I And\^o tuple for 
$(T_1^*,T_2^*)$ and $(\bPi_D, \bV_{D,1}, \bV_{D,2})$ be the  Douglas-model And\^o  lift of $(T_1,T_2)$ corresponding to
$(\cF_{*},\Lambda_{*},P_{*},U_{*})$
of $(T_1^*, T_2^*)$. By Theorem \ref{T:Dmodel} 
such a lift exists and is given as  $(\bPi_D, \bV_{D,1} \bV_{D,2})$ on $\bcK_D:= \sbm{ H^2(\cF_*) \\ \cQ_{T^*} }$ in \eqref{DougAndoMod}. Consider the unitary operator
\begin{align}\label{DNFintwin}
\bU_{\rm NF,D}:=  \begin{bmatrix} I_{H^2(\cF_*)} & 0 \\ 0 &  \omega_{\text{NF,D}}   \end{bmatrix} \colon
\begin{bmatrix}  H^2(\cF_{*}) \\ \cQ_{T^*}\end{bmatrix} \to \begin{bmatrix} H^2(\cF_{*}) \\ \overline{\Delta_{\Theta_T}L^2(\cD_T)} \end{bmatrix}
\end{align}
and define
\begin{equation}   \label{NF-AndoGenForm1}
(\bPi_{\rm NF}, \bV_{{\rm NF},1}, \bV_{{\rm NF},2}):=(\bU_{\rm NF,D} \bPi_D, \bU_{\rm NF,D}  \bV_{D,1} \bU_{\rm NF,D}^*,  \bU_{\rm NF,D} \bV_{D_2} \bU_{\rm NF,D}^*).
\end{equation}
Then we note that
\begin{align} \label{NF-AndoGenForm}\index{$\bV_{{\rm NF},1}, \bV_{{\rm NF},2}$}
&  (\bV_{{\rm NF},1}, \bV_{{\rm NF},2},  \bV_{{\rm NF},1} \bV_{{\rm NF},2}) \nonumber\\
=&  \left(  \begin{bmatrix} M_{U_{*}^* (P_{*}^\perp + z P_{*})} & 0 \\ 0 &  W_{\sharp1} \end{bmatrix},
\begin{bmatrix} M_{(P_{*} + z P^\perp_{*}) U_{*}} & 0 \\ 0 &  W_{\sharp2} \end{bmatrix},
\begin{bmatrix} M_z^{\cF_*} & 0 \\ 0 &  M_{\zeta}|_{\overline{\Delta_{\Theta_T} L^2(\cD_T)}} \end{bmatrix} \right)
\end{align}
and that the isometry 
\index{$\bPi_{\rm NF}$}
\index{$\bcK_{\rm{NF}}$}
$\bPi_{\rm NF}:\cH \to  \bcK_{\rm{NF}}:= \sbm{ H^2(\cF_{*}) \\    \overline{\Delta_{\Theta_T}L^2(\cD_T)}}$ is given by
\begin{equation}      \label{NF-iso}
 \bPi_{\rm NF}   = \bU_{\rm NF,D}  \bPi_D   = \begin{bmatrix} \Lambda_*  \cO_{D_{T^*}, T^*}  \\ \omega_{\text{NF,D}} Q_{T^*} \end{bmatrix}    
= \begin{bmatrix} I_{H^2} \otimes\Lambda_* & 0 \\ 0 &  I_{\overline{\Delta_{\Theta_T} L^2(\cD_T)}} \end{bmatrix} \Pi_{\rm NF} h
\end{equation}
where $\Pi_{\rm NF}$ is the embedding of $\cH$ into $\sbm{ H^2(\cD_{T^*}) \\ \overline{\Delta_{\Theta_T} L^2(\cD_T)} }$
given by \eqref{Unf&Pinf}(also by \eqref{PiNF}).
Consequently, by Theorem \ref{T:Dmodel} we have proved the following theorem, which gives a Sz.-Nagy--Foias functional model for an And\^o  
lift of commuting contractive operator-pair constructed canonically from a Type I And\^o tuple $(\cF_*, \Lambda_*, P_*,U_*)$ for $(T_1^*, T_2^*)$.

\begin{theorem}\label{Thm:NFmodel}
Let $(T_1,T_2)$ be a commuting pair of contractions on a Hilbert space $\cH$ and $(\cF_*,\Lambda_*, P_*,U_*)$ be a Douglas-minimal Type I And\^o tuple for 
$(T_1^*,T_2^*)$. Define a pair $(\bV_{{\rm NF},1},\bV_{{\rm NF},2})$ of commuting isometries on $\bcK_{\rm NF} := \sbm{ H^2(\cF_{*}) \\\overline{\Delta_{\Theta_T} L^2(\cD_T)}}$ 
as in (\ref{NF-AndoGenForm}) and an embedding
$\bPi_{\rm NF} \colon \cH \to \bcK_{\rm NF}$ as in \eqref{NF-iso}. Then $(\bPi_{\rm NF}, \bV_{{\rm NF},1},$ $\bV_{\rm{NF}, 2})$ is a minimal And\^o lift of $(T_1,T_2)$.

Conversely, any minimal And\^o  lift $(\bPi, \bV_1, \bV_2)$ of $(T_1, T_2)$ is unitarily equivalent (as an isometric lift) to an And\^o lift of the form (\ref{NF-AndoGenForm}) coming from 
a Douglas-minimal Type I And\^o tuple for $(T_1^*,T_2^*)$.
\end{theorem}

\section[Sch\"affer model]{Type II And\^o tuples and Sch\"affer models for an And\^o isometric lift}
\label{S:S2}
\index{And\^o lift!Sch\"affer-type model}

Let $(T_1,T_2)$ on $\cH$ be a pair of commuting contractions on the Hilbert space $\cH$ and let
$(\Pi, V_1,V_2)$ with $\Pi \colon \cH \to \cK$ and $V_1, V_2$ on $\cK$ be an And\^o  lift
of $(T_1,T_2)$.  Up to unitary equivalence we may arrange that $\cH \subset \cK$ and $\Pi \colon \cH \to \cK$ is
the inclusion map.  Hence with respect to the decomposition $\cK=\cH\oplus(\cK\ominus\cH)$ we have
$$
(V_1,V_2)=\left(\begin{bmatrix}   T_1 & 0 \\ C_1 & D_1\end{bmatrix},
\begin{bmatrix}  T_2 & 0 \\  C_2 & D_2 \end{bmatrix}  \right)
$$
for some operators
\begin{equation}  \label{CjDj}
 C_j \colon \cH \to  \cK \ominus \cH, \quad D_j \colon  \cK \ominus \cH \to \cK \ominus \cH
\end{equation}
for $j = 1,2$.  As $V_1$ and $V_2$ commute,  we must also have that
$$
\begin{bmatrix} T_1 & 0 \\ C_1 & D_1 \end{bmatrix} \begin{bmatrix} T_2 & 0 \\ C_2 &  D_2 \end{bmatrix}
 = \begin{bmatrix} T_2 & 0 \\ C_2 & D_2 \end{bmatrix} \begin{bmatrix} T_1 & 0 \\ C_1 & D_1 \end{bmatrix}
$$
leading to the matrix identity
\begin{equation}  \label{V1V2=V2V1}
V := V_1 V_2 = \begin{bmatrix} T_1 T_2 & 0 \\ C_1 T_2 + D_1 C_2 & D_1 D_2 \end{bmatrix}  =
  \ \begin{bmatrix} T_2 T_1 & 0 \\ C_2 T_1 + D_2 C_1 & D_2 D_1 \end{bmatrix} = V_2 V_1
\end{equation}
which gives us the operator identities (in addition to the assumed commutativity of $(T_1, T_2)$)
\begin{equation}   \label{identities1}
C_1 T_2 + D_1 C_2 = C_2 T_1 + D_2 C_1, \quad D_1 D_2 = D_2 D_1.
\end{equation}
From the fact that each $V_j$ ($j=1,2$) is an isometry, we have
$$
\begin{bmatrix} T_j^* & C_j^* \\ 0 & D_j^* \end{bmatrix} \begin{bmatrix} T_j & 0 \\ C_j & D_j \end{bmatrix} = \begin{bmatrix} I_\cH & 0 \\ 0 & I_{\cK \ominus \cH} \end{bmatrix} 
$$
giving us the identities
\begin{equation}   \label{identities2}
C_j^* C_j = I - T_j^* T_j,   \quad C_j^* D_j = 0, \quad D_j^* D_j = I_{\cK \ominus \cH} \text{ for } j=1,2.
\end{equation}

In particular, from the last identity in \eqref{identities1} and in \eqref{identities2} we see that the pair $(D_1, D_2)$ is a commuting isometric pair on
$\cK \ominus \cH$.
 In constructing the Sch\"affer-type model for And\^o isometric lift, for reasons to become clear later, we choose to work with
 a BCL1 model \eqref{BCL1model} for the commuting isometric pair $(D_1, D_2)$ rather than with a BCL2 model \eqref{BCL2model} for $(V_1, V_2)$ 
 as we did for the Douglas model.
Hence by Theorem \ref{Thm:BCLmodel} there exist a Hilbert space $\cF$,
a projection $P$ and unitary $U$ acting on $\cF$ and a pair $(Y_1,Y_2)$ of commuting unitaries acting on some Hilbert space
$\cY$ and a unitary identification map $\tau_{\rm BCL} \colon \cK \ominus \cH \to \bcK_S:=  \sbm{ H^2(\cF) \\ \cY}$ so that we have
$$
\tau_{\rm BCL}  \, (D_1,D_2)=  \left( \begin{bmatrix} M_{(P^\perp + z P)U} & 0 \\ 0 &  Y_1 \end{bmatrix},
\begin{bmatrix} M_{U^* (P + z P^\perp) } & 0 \\ 0 &  Y_2 \end{bmatrix} \right)
 \, \tau_{\rm BCL}
$$
and with the product operator $D = V_1 V_2$ satisfying
$$
 \tau_{\rm BCL}  D := \begin{bmatrix} M_z & 0 \\ 0 & Y \end{bmatrix} \text{ where } Y = Y_1 Y_2 = Y_2 Y_1.
$$
For $j=1,2$ let us define operators $\sbm{ C_{s1} \\ C_{u2}}, \sbm{ C_{s2} \\ C_{u2} }, \sbm{ C_s \\ C_u}$  from $\cH$ to $\sbm{ H^2(\cF) \\ \cY}$ by
$$
\begin{bmatrix} C_{s1} \\ C_{u1}\end{bmatrix} = \tau_{\rm BCL} C_1,
\quad  \begin{bmatrix} C_{s2} \\ C_{u2}\end{bmatrix} = \tau_{\rm BCL} C_2, \quad  \begin{bmatrix} C_s \\ C_u \end{bmatrix} = \tau_{\rm BCL} C
$$
where the operators $C_1$ and $C_2$ are as in \eqref{CjDj} and where $C$ is given by
$$
C = [V_1V_2]_{21} = C_1T_2 + D_1C_2 \text{ or } C = [V_2V_1]_{21} = C_2T_1 + D_2 C_1 \text{ as in \eqref{V1V2=V2V1}.}
$$
where the subscripts $s$ and $u$ indicate the {\em shift} component $H^2(\cF)$ and the {\em unitary}  component $\cY$ respectively.
Let us denote $V=V_1V_2$ and $Y=Y_1Y_2$. Then we have
$$
\begin{bmatrix} I_\cH & 0 \\ 0 & \tau_{\rm BCL} \end{bmatrix} (V_1, V_2, V) = (\bV_{S,1}, \bV_{S,2}, \bV_S) \begin{bmatrix} I_\cH & 0 \\ 0 & \tau_{\rm BCL} \end{bmatrix}
$$
where
\begin{align}
&  (\bV_{S,1}, \bV_{S,2} , \bV_S)=   \notag \\
  & \left(\begin{bmatrix}  T_1 & 0  & 0 \\  C_{s1}  & M_{(P^\perp + z P)U} & 0 \\ C_{u1} & 0 & Y_1 \end{bmatrix},
\begin{bmatrix}   T_2 & 0 & 0 \\   C_{s2} & M_{U^*(P + z P^\perp)} & 0 \\  C_{u2} & 0 &  Y_2  \end{bmatrix},
 \begin{bmatrix} T & 0 & 0 \\  C_s  & M_z & 0 \\ C_u & 0 &  Y \end{bmatrix}\right).
    \label{liftS}
\end{align}
The identities \eqref{identities1} lead to the equalities
\begin{align}    
 \begin{bmatrix} C_s \\ C_u \end{bmatrix}  & = 
 \begin{bmatrix} C_{s1} \\ C_{u1} \end{bmatrix} T_2+
\begin{bmatrix} M_{P^\perp U  + z P U} & 0 \\ 0 & Y_1 \end{bmatrix} \begin{bmatrix} C_{s2}  \\ C_{u2}  \end{bmatrix}  \notag \\
&  =
\begin{bmatrix} C_{s2} \\ C_{u2} \end{bmatrix} T_1+ \begin{bmatrix} M_{U^* P + z U^* P^\perp} & 0 \\ 0 & Y_2 \end{bmatrix} 
\begin{bmatrix} C_{s1} \\ C_{u1} \end{bmatrix}.
\label{C}
\end{align}
Similarly, \eqref{identities2} leads to
\begin{align}
 &  I_\cH = T_1^*T_1+C_{s1}^*C_{s1} + C_{u1}^* C_{u1} =   T_2^*T_2+ C_{s2}^*C_{s2} + C_{u2}^* C_{u2}   = T^*T+C_{s}^*C_s + C_u^* C_u,   \notag \\
& 0 = C_{s1}^* M_{(P^\perp+ zP)U} = C_{s2}^* M_{U^*(P + z P^\perp)} = C_s^* M_z = 0   \notag \\
& C_{u1}^* Y_1 =  C_{u2}^* Y_2 = C_u^* Y = 0.
\label{Facts-Iso}
\end{align}
Multiply the first equation in \eqref{C} on the left by $\sbm{ M_{(P^\perp + z P)U} & 0 \\ 0 &  Y_1}^*$ to get
\begin{equation}   \label{C2}
\begin{bmatrix} C_{s2} \\ C_{u2} \end{bmatrix}  = \begin{bmatrix} M_{P^\perp U + z P U} & 0 \\ 0 & Y_1 \end{bmatrix}^* \begin{bmatrix} C_s \\ C_u \end{bmatrix}
- \begin{bmatrix} (M_{P^\perp U+ z P U}  & 0 \\ 0 & Y_1\end{bmatrix}^*  \begin{bmatrix} C_{s1}  \\ C_{u1} \end{bmatrix} T_2.
\end{equation}
However, by taking adjoints in the second and third lines of \eqref{Facts-Iso} we see in particular that
$$
    \left( \begin{bmatrix} M_{ P^\perp U  + z P U} & 0 \\ 0 &  Y_1 \end{bmatrix} \right)^*  \begin{bmatrix} C_{s1}  \\ C_{u1} \end{bmatrix}  = 0.
$$
Thus equation \eqref{C2} simplifies to
$$
  \begin{bmatrix} C_{s2} \\ C_{u2}  \end{bmatrix} = \left( \begin{bmatrix} M_{ P^\perp U  + z P U} & 0 \\ 0 & Y_1 \end{bmatrix} \right)^* \begin{bmatrix} C_s \\ C_u \end{bmatrix}.
$$
or more simply
\begin{equation}   \label{C2'}
C_{s2} = \left( M_{ P^\perp U  + z P U} \right)^* C_s, \quad C_{u2} = Y_1^* C_u.
\end{equation}
A similar analysis starting with the second equation in \eqref{C} leads to
$$  
\begin{bmatrix} C_{s1} \\ C_{u1}  \end{bmatrix} = \left( \begin{bmatrix} M_{U^* P + z U^* P^\perp} & 0 \\ 0 & Y_2 \end{bmatrix} \right)^* \begin{bmatrix} C_s \\ C_u \end{bmatrix}
$$
or more simply
\begin{equation}  \label{C1'}
C_{s1} = \left(M_{U^*(P + zP^\perp)}\right)^* C_s, \quad C_{u1} = Y_2^* C_u.
\end{equation}
From the second and third lines in \eqref{Facts-Iso} we see in particular that
\begin{equation}  \label{eq*}
 M_z^* C_s = 0, \quad Y^* C_u = 0.
\end{equation}
The first item in \eqref{eq*} forces $C_s$ to have the form
\begin{equation}  \label{Cs0}
  C_s = \bev_{0, \cF}^* C_{s0} \text{ where } C_{s0}\colon \cF \to \cF.
\end{equation}
Since $Y$ is unitary, the second item in \eqref{eq*} forces 
\begin{equation}  \label{Cu}
C_u = 0.
\end{equation}
From the second equation in \eqref{C1'} and \eqref{C2'} we get the further collapsing
\begin{equation}   \label{Cj''}
C_{u1} = 0, \quad C_{u2} = 0.
\end{equation}
From the first line in \eqref{Facts-Iso} we see in particular that $I_\cH - T^*T = C_s^* C_s + C_u^* C_u$.
Since we now have established that $C_u = 0$, this simplifies to
$$
   I_\cH - T^* T = C_s^* C_s = C_{s0}^* C_{s0}.
$$
This identity in turn tells us that we can factor $C_{s0}$ as
\begin{equation}  \label{Cs0fact}
  C_{s0} = \Lambda D_T
\end{equation}
where $\Lambda \colon \cD_T \to \cF$ is an isometry.

Finally, by combining the first equations in \eqref{C2'}, \eqref{C1'} together with \eqref{Cs0}, \eqref{Cs0fact} and \eqref{Cj''}, we see that
\begin{align}
& \begin{bmatrix} C_{s1} \\ C_{u1} \end{bmatrix} = \begin{bmatrix} (M_{U^*(P + z P^\perp) })^* \bev_{0, \cF}^* \Lambda D_T \\ 0 \end{bmatrix}
     = \begin{bmatrix} \bev_{0,\cF}^*   P U \Lambda D_T \\ 0 \end{bmatrix},   \notag \\
& \begin{bmatrix} C_{s2} \\ C_{u2} \end{bmatrix}  = \begin{bmatrix} (M_{(P^\perp + zP)U})^* \bev_{0, \cF}^* \Lambda D_T \\ 0 \end{bmatrix}
  = \begin{bmatrix} \bev_{0,\cF}^*  U^* P^\perp  \Lambda D_T \\ 0 \end{bmatrix},  \notag  \\
&  \begin{bmatrix} C_s \\ C_u \end{bmatrix}  = \begin{bmatrix}  \bev_{0,\cF}^* \Lambda D_T \\ 0   \end{bmatrix}.
 \label{FindalCs}
\end{align}

Let us summarize:  given a commuting contractive operator-pair $(T_1, T_2)$ on $\cH$ having an And\^o isometric
lift $(\Pi, V_1, V_2)$,
we have now come upon a collection $(\cF, \Lambda, U, P)$ where $\cF$ is another Hilbert space, $\Lambda \colon \cD_T \to
\cF$ is an isometry (where $T = T_1 T_2$), $P$ and $U$ are operators on $\cF$ with $P$ equal to a projection and $U$ equal to a unitary operator, respectively, 
i.e., in the terminology defined in Definition \ref{D:preAndoTuple}, the collection $(\cF, \Lambda, P, U)$ is a
{\em pre-And\^o tuple for $(T_1, T_2)$}.  All this leads to a Sch\"affer-type functional model for the And\^o lift as follows.

\begin{theorem} \label{T:summary}
Given any And\^o  lift $(\Pi, V_1,V_2)$ of a given commuting contractive pair $(T_1,T_2)$, there is a pre-And\^o tuple $(\cF, \Lambda, P, U)$
for $(T_1, T_2)$, another Hilbert space $\cY$ along with a commuting pair $(Y_1, Y_2)$ of unitary operators on $\cY$,
together with a unitary operator $\btau$ from $\cK$ onto the Sch\"affer model space $\bcK_S$ defined below, so that
$$
(\btau \Pi, \btau V_1 \btau^*, \btau V_2 \btau^*, \btau V \btau^* ) = ( \bPi_S, \bV_{S,1}, \bV_{S,2}, \bV_S)
$$
where 
\begin{align}\index{$\bcK_S$} \index{$\bPi_S$}\index{$\bV_{S,1}, \bV_{S,2}$}
&\bcK_S = \sbm{ \cH \\ H^2(\cF)  \\  \cY },  \quad \bPi_S = \sbm{ I_\cH \\ 0  \\ 0 } \colon \cH \to \bcK_S,  \notag \\
 & (\bV_{S,1}, \bV_{S,2}) =  
  \left( \sbm{T_1 & 0  & 0 \\  \bev_{0,\cF}^* P U \Lambda D_T & M_{P^\perp U  + z P U} & 0 \\ 0 & 0 & Y_1},
  \sbm{
   T_2 & 0 & 0 \\ \bev_{0,\cF}^* U^* P^\perp  \Lambda D_T & M_{U^* P + z U^* P^\perp } & 0  \\ 0 & 0 &Y_2 } \right).
    \label{AndoGenForm}
\end{align}
 Furthermore one can choose the isometry $\Lambda$ so that
 \begin{equation}   \label{canonicalV}\index{$\bV_S$}
 \bV_S:= \bV_{S,1} \bV_{S,2} = \bV_{S,2} \bV_{S,1} = \sbm{T & 0 & 0  \\   \bev_{0,\cF}^* \Lambda D_T & M^\cF_z  & 0  \\ 0 & 0 &  Y }.
 \end{equation}
 where we set $T : = T_1 T_2 = T_2 T_1$, $Y: = Y_1 Y_2 = Y_2 Y_1$.
\end{theorem}

We now investigate the converse direction, i.e., given a pair of commuting contractions $(T_1,T_2)$,  a pre-And\^o
tuple for $(T_1, T_2)$, and a Hilbert space $\cY$ equipped with a commuting pair of unitary operators $(Y_1, Y_2)$,
when is the pair $(\bV_{S,1}, \bV_{S,2})$ as defined
in (\ref{AndoGenForm}) (clearly a lift of $(T_1, T_2)$ due to the triangular form in \eqref{AndoGenForm}) a commuting pair of isometries? 
It turns out that the answer to this question is negative in general: the pre-And\^o tuple must
satisfy some additional conditions which we now explore.
  This discussion motivates the following definition.

\begin{definition}\label{AndoTuple}
\index{And\^o tuple!type II}
Let $(T_1,T_2)$ be a pair of commuting contractions on a Hilbert space $\cH$. A pre-And\^o tuple $(\cF,\Lambda,P,U)$ for $(T_1,T_2)$ is called a {\em Type II And\^o tuple   } if
\begin{enumerate}
 \item[(i)] \textbf{Commutativity:}
 $$PU\Lambda D_TT_2+P^\perp \Lambda D_T=U^*P^\perp\Lambda D_TT_1+U^*PU\Lambda D_T;$$
  \text{ and}
\item[(ii)] \textbf{Isometry:}
$$
D_T\Lambda^*U^*PU\Lambda D_T=D_{T_1}^2, \quad D_T\Lambda^* P^\perp\Lambda D_T=D_{T_2}^2.$$
\end{enumerate}

We say that the Type II And\^o tuple $(\cF, \Lambda, P, U)$ is  a {\em strong Type II And\^o tuple}
\index{And\^o tuple!strong type II}
if item (i) is true in the strengthened form
\begin{enumerate}
\item[(i$^\prime$)]
 $PU\Lambda D_TT_2+P^\perp \Lambda D_T=U^*P^\perp\Lambda D_TT_1+U^*PU\Lambda D_T = \Lambda D_T$.
\end{enumerate}

Finally let us say that the And\^o tuple $(\cF, \Lambda, P, U)$ is {\em Sch\"affer-minimal} if it is the case that
\begin{align*}
& \bigvee_{n_1, n_2 \ge 0}   \begin{bmatrix} T_1 & 0 \\ \bev_{0, \cF}^* P U \Lambda D_T & M_{ P^\perp U + z PU} \end{bmatrix}^{n_1} 
\begin{bmatrix} T_2 & 0 \\ \bev_{0, \cF}^* U^* P^\perp \Lambda D_T & M_{U^* P + z U^* P^\perp} \end{bmatrix}^{n_2} \begin{bmatrix}  \cH \\ 0 \end{bmatrix}  \\
& \quad = \begin{bmatrix} \cH \\ H^2(\cF) \end{bmatrix}.
\end{align*}
\end{definition}

Let us note that both notions, {\em Type II And\^o tuple} and {\em strong Type II And\^o tuple}, are {\em coordinate-free}
in the following sense:  {\em if $(\cF, \Lambda, P, U)$ is a  Type II (respectively strong Type II) And\^o tuple
for the commuting contractive pair $(T_1,T_2)$ and $(\cF', \Lambda', P', U')$  is a pre-And\^o tuple for $(T_1, T_2)$
which coincides with  $(\cF, \Lambda, P, U)$ in the sense of Definition \ref{D:preAndoTuple}, then
$(\cF', \Lambda', P', U')$ is also a Type II (respectively strong Type II) And\^o tuple for $(T_1, T_2)$}.
We shall see in Section \ref{S:TypeIIvsStrongTypeII} below that the class of strong Type II And\^o tuples is strictly smaller than that of Type II And\^o tuples.

Then we have the following result.

\begin{theorem}\label{Thm:LiftFromTuple}
Let $(T_1,T_2)$ be a commuting contractive operator-pair on a Hilbert space $\cH$ and $(\cF,\Lambda,P,U)$ be a
Type II And\^o tuple for $(T_1,T_2)$. Let $(Y_1,Y_2)$ be any pair of commuting unitaries on some Hilbert space $\cY$. Let $\bcK_S$ be the space and $\bV_{S,1},
\bV_{S,2}$ on $\bcK_S$ be the operators as in \eqref{AndoGenForm} above. 
Then $(\bV_{S,1},\bV_{S,2})$ is a commuting pair of isometries on $\bcK_S$ and  is an And\^o lift of $(T_1,T_2)$. 
In case $(\cF, \Lambda, P, U)$ is a strong Type II And\^o tuple, then in addition
we have
\begin{equation}  \label{Vcanonical}
\bV_S:= \bV_{S,1} \bV_{S,2} = \bV_{S,2} \bV_{S,1} = \begin{bmatrix} T & 0 & 0  \\  \bev_{0, \cF}^*  \Lambda D_T & M^\cF_z  & 0 \\ 0 & 0  &  Y_1 Y_2 \end{bmatrix}.
\end{equation}

If  $(\bV_{S,1}, \bV_{S,2})$ is a minimal And\^o lift of $(T_1, T_2)$, then  $\cY = \{0\}$, the And\^o tuple $(\cF, \Lambda, P, U)$ is Sch\"affer-minimal, and \eqref{AndoGenForm}
simplifies to
\begin{align}
& \bcK_S = \begin{bmatrix} \cH \\  H^2(\cF) \end{bmatrix}, \quad   \bPi_S = \begin{bmatrix}  I_\cH \\  0 \end{bmatrix} \colon \cH \to \bcK_S,   \notag \\
& (\bV_{S,1},\bV_{S,2})=
 \left(\begin{bmatrix}
            T_1 & 0 \\ \bev_{0, \cF}^* PU\Lambda D_T & M_{P^\perp U+zP U} \end{bmatrix},
  \begin{bmatrix}
   T_2 & 0 \\ \bev_{0, \cF}^*  U^*P^\perp \Lambda D_T & M_{U^*P+z U^* P^\perp} \end{bmatrix}\right)
\label{Conv-AndoGenForm'}
\end{align}
and, in case $(\cF, \Lambda, P, U)$ is a strong Type II And\^o tuple, then the formula \eqref{Vcanonical} for $\bV_S:=\bV_{S,1} \bV_{S,2}$
simplifies to
\begin{equation}  \label{Vcanonical'}
\bV_S:= \bV_{S,1} \bV_{S,2} = \bV_{S,2} \bV_{S,1} = \begin{bmatrix} T & 0 \\  \bev_{0,\cF}^* \Lambda D_T & M_z^\cF \end{bmatrix}
\end{equation}

Conversely any minimal And\^o  lift $(\Pi, V_1, V_2)$ of $(T_1, T_2)$ is unitarily
equivalent (as a lift of $(T_1, T_2)$) to an And\^o   lift of the form \eqref{Conv-AndoGenForm'}--\eqref{Vcanonical'}  coming from a
Sch\"affer-minimal, strong Type II And\^o tuple for $(T_1, T_2)$.
\end{theorem}

\begin{proof}
A matrix computation from the formulas \eqref{AndoGenForm}  shows that
\begin{equation} \label{V1V2}
\bV_{S,1} \bV_{S,2} = \begin{bmatrix} T_1 T_2   & 0 & 0  \\ [\bV_{S,1}\bV_{S,2}]_{21}  & M_z^\cF  & 0 \\ 0 & 0 &  Y_1 Y_2 \end{bmatrix}
\end{equation}
 where
 \begin{align}
 [\bV_{S,1} \bV_{S,2}]_{21}  & =
 \bev_{0,\cF}^* PU\Lambda D_TT_2 + M_{P^\perp U +zP U}  \, \bev_{0,\cF}^* U^*P^\perp \Lambda D_T  \notag  \\
& = \bev_{0,\cF}^*  PU \Lambda D_T T_2+  \bev_{0,\cF}^*  P^\perp \Lambda D_T \notag \\
 &  =\bev_{0,\cF}^* \left(  PU\Lambda D_T T_2 +   P^\perp \Lambda D_T \right).
 \label{[V1V2]21}
\end{align}
Similarly one can compute that
\begin{equation}   \label{V2V1}
\bV_{S,2} \bV_{S,1} = \begin{bmatrix} T_2 T_1 & 0 & 0 \\ [\bV_{S,2} \bV_{S,1}]_{21} & M_z & 0 \\ 0 & 0 &  Y_2 Y_1 \end{bmatrix}
\end{equation}
where
\begin{align}
[\bV_{S,2} \bV_{S,1}]_{21} & =  \bev_{0,\cF}^* U^*P^\perp \Lambda D_T T_1+ (M_{U^* P+z U^* P^\perp}\oplus \bev_{0,\cF}^*  P U \Lambda D_T \notag \\
& = \bev_{0,\cF}^* (U^* P^\perp \Lambda D_T T_1 +  U^* P U \Lambda D_T).
\label{[V2V1]21}
\end{align}
We conclude that $\bV_{S,1} \bV_{S,2} = \bV_{S,2} \bV_{S,1}$ exactly when the following system of equations hold:
\begin{align}
T_1 T_2 & = T_2 T_1, \notag  \\ Y_1 Y_2 & =  Y_2 Y_1,  \notag \\
  PU\Lambda D_T T_2 +   P^\perp \Lambda D_T  & =  U^* P^\perp \Lambda D_T T_1 + U^* P U  \Lambda D_T.
  \label{V1V2com}
\end{align}
The first two equations are valid due to our assumptions that $(T_1, T_2)$ and $(Y_1, Y_2)$ are commuting pairs.
Note that the last equation is just condition (i) in \eqref{AndoTuple} (the commutativity condition).

In case $(\cF, \Lambda, P, U)$ is a strong And\^o tuple, then
$$
[\bV_{S,1} \bV_{S,2}]_{21} = [ \bV_{S,2} \bV_{S,1} ]_{21} = \bev_{0,\cF}^*  \Lambda D_T.
$$
From \eqref{V1V2} and \eqref{V2V1} we read off that
$$
  \bV_S: = \bV_{S,1} \bV_{S,2} = \bV_{S,2} \bV_{S,1} = \begin{bmatrix} T & 0  & 0 \\ \bev_{0,\cF}^* \Lambda D_T  & M^\cF_z & 0  \\ 0 & 0  & Y \end{bmatrix}
$$
where we set $T = T_1 T_2 = T_2 T_1$ and $Y = Y_1 Y_2 = Y_2 Y_1$.

We next argue that condition (ii) in Definition \ref{AndoTuple} implies that $\bV_{S,1}$ and $\bV_{S,2}$ are isometries.
Since $\sbm{ M_{(P^\perp + z P) U} & 0 \\ 0 & Y_1}$ is an isometry,
to verify that $\bV_{S,1}$ is an isometry it suffices to check that the $(1,1)$-entry of $\bV_{S,1}^*\bV_{S,1}$ equals $I_\cH$
and that the $(1,2)$-entry of $\bV_{S,1}^* \bV_{S,1}$ is equal to $0$, i.e., we need to show:
$$
  T_1^*T_1+D_T\Lambda^*U^*P\bev_{0,\cF} \bev_{0, \cF}^*PU\Lambda D_T=I_{\mathcal H}, \quad
  D_T\Lambda^*U^*P   \bev_{0,\cF}  M_{(P^\perp +zP)U} =0.
$$
The first identity is an immediate consequence of the first equality in condition (ii) (the isometry condition) in Definition \ref{AndoTuple}.
As for the second note that
\begin{align*}
&  D_T\Lambda^*U^*P   \bev_{0,\cF} M_{(P^\perp +zP)U} 
 = D_T \Lambda^* U^* P \bev_{0,\cF} M_{P^\perp U} = D_T \Lambda^* U^* P P^\perp U  \bev_{0, \cF}  = 0
\end{align*}
as wanted. Hence $\bV_{S,1}$ is an isometry. Similarly one can show that $\bV_{S,2}$ is an isometry by making use
of the second part of condition (ii) in Definition \ref{AndoTuple}.
This completes the proof of the direct side of Theorem \ref{Thm:LiftFromTuple}.

If $(\bV_{S,1}, \bV_{S,2})$ as in \eqref{AndoGenForm} is a minimal lift of $(T_1, T_2)$, then
$$
\begin{bmatrix} \cH  \\ H^2(\cF) \\ \cY \end{bmatrix} =
\bigvee_{n_1, n_2 \ge 0} \bV_{S,1}^{n_1} \bV_{S,2}^{n_2}  \begin{bmatrix} \cH \\ 0 \\ 0 \end{bmatrix}.
$$
From the triangular form of $V_1$ and $V_2$, we see that the right-hand side
of this last display is contained in $\sbm{ \cH \\ H^2(\cF) \\ \{0\}}$.  It thus follows that the space $\cY$ is trivial,
the formulas in \eqref{AndoGenForm} and \eqref{canonicalV} collapse to those in \eqref{Conv-AndoGenForm'} and \eqref{Vcanonical'}, and
the minimality condition can now be expressed as
$$
\begin{bmatrix} \cH \\ H^2(\cF) \end{bmatrix} = \bigvee_{n_1, n_2 \ge 0} \bV_{S,1}^{n_1} \bV_{S,2}^{n_2}  \begin{bmatrix} \cH \\ 0  \end{bmatrix}  
= \begin{bmatrix} \cH \\ H^2(\cF) \end{bmatrix} =: \bcK_S
$$
which is exactly the condition that $(\cF, \Lambda, P, U)$ be Sch\"affer-minimal.

\smallskip

Conversely, suppose that $( \Pi, V_1, V_2)$ is a minimal And\^o  lift of the commuting contractive pair
$(T_1, T_2)$.  The discussion preceding the theorem tells us that the unitary operator
$$
\btau:= \begin{bmatrix} I_\cH & 0 \\ 0 & \tau_{\rm BCL} \end{bmatrix} \colon \cK \cong \begin{bmatrix} \cH \\ \cK \ominus \cH \end{bmatrix} \to \bcK_S:= 
\begin{bmatrix}  \cH \\ H^2(\cF) \end{bmatrix}
$$
transforms the commuting isometric pair $(V_1, V_2)$ on $\cK$ to the pair $(\bV_{S,1}, \bV_{S,2})$ on $\bcK_S$
as given by \eqref{AndoGenForm} with $\bV_S = \bV_{S,1} \bV_{S,2} = \bV_{S,2} \bV_{S,1}$ 
as in \eqref{canonicalV},
where the tuple $(\cF, \Lambda, P, U)$  is as in condition (i) in Definition \ref{AndoTuple}.
By reversing the analysis given in the direct analysis, the fact that $\bV_{S,1} \bV_{S,2} = \bV_{S,2} \bV_{S,1}$ gives us the system of
equations \eqref{V1V2com} which then forces
 the commutativity condition (i) in Definition \ref{AndoTuple}) to hold, and the fact that $\bV_{S,1}$ and $\bV_{S,2}$
 are isometries forces the isometry condition  ((ii) in Definition \ref{AndoTuple}) to hold.  Furthermore, the fact
 that by construction $\bV_S = \bV_{S,1}\bV_{S,2} = \bV_{S,2}\bV_{S,1}$ has the form \eqref{Vcanonical} forces the strong
 commutativity condition (i$^{'}$) in Definition \ref{AndoTuple}) to hold, so in fact $(\cF, \Lambda, P, U)$
 is a strong Type II And\^o tuple. Furthermore, the assumption that $(\bPi_S, \bV_{S,1}, \bV_{S,2})$ is minimal forces the space $\cY$
 to be trivial and the tuple $(\cF, \Lambda, P, U)$ to be Sch\"affer-minimal (see Definition \ref{AndoTuple}).
\end{proof}

In parallel with Example \ref{E:Dex} for the Douglas-model And\^o lift, we next introduce some special cases of And\^o lifts for illustrative purposes.

\begin{example}  \label{E:Sex}
\textbf{Illustrative special cases for Sch\"affer-model And\^o lifts}

\smallskip

\textbf{1. $(T_1, T_2)$ = BCL1-model commuting isometric operator-pair:}  Suppose that $(T_1, T_2)  = (V_1, V_2)$ is a commutative isometric pair,
so $D_T = 0$.  Then we get a Type II (strong or otherwise) And\^o tuple for $(T_1, T_2)$ simply by taking the coefficient space $\cF$ to be the zero space
and hence all operators $\Lambda \colon \cD_T \to \cF$, projection $P$ on $\cF$ and unitary operator $U$ on $\cF$ are all zero.  Then the Sch\"affer model
for the lift becomes $(\bV_{S,1}, \bV_{S,2}) = (T_1, T_2)$, i.e., $(T_1, T_2)$ is the minimal And\^o lift of itself.  The reader should find the next special case to be more interesting.

\textbf{2. $(T_1, T_2)$ = adjoint of adjusted BCL2-model commuting isometric pair:} 
We have seen that any commuting isometric operator-pair $(V_1, V_2)$ with product $V = V_1 V_2$ equal to a shift can be modelled in the BCL2-model form:
\begin{equation}   \label{BCL2modelV}
(V_1, V_2, V) =
(M_{U^* P^\perp + z U^* P}, M_{P U + z P^\perp U}, M_z^\cF) \text{ on } H^2(\cF)
\end{equation}
where $(\cF, P, U)$ is a choice of BCL tuple ($\cF$ is a coefficient Hilbert space, $P$ is a projection and $U$ is a unitary operator on $\cF$).
 Let us introduce the space $L^2(\cF)$
consisting of functions of the form $\zeta \mapsto f\zeta)$ (with $\zeta$ equal to the independent variable on the unit circle ${\mathbb T}$)
such that $f(\zeta) \sim \sum_{n= -\infty} ^\infty \widehat f_n \zeta^n$ with $\sum_{n \in {\mathbb Z}} \| f \|^2_\cF < \infty$. Then we can view
$H^2(\cF)$ as the subspace of $L^2(\cF)$ consisting of such $L^2(\cF)$-functions  $f$ having negatively-indexed Fourier coefficients equal to zero:
$\widehat f_n = 0$ for $n<0$.  Its Hilbert-space orthogonal complement in $L^2(\cF)$ consists of $L^2(\cF)$ functions with nonnegatively-indexed Fourier coefficients equal to zero:
$ f_n= 0$ for $n \ge 0$.
Let us observe that the reflection operator 
$$
{\mathfrak r}  \colon f(\zeta) \mapsto \zeta^{-1} f(\zeta^{-1})
$$
is a unitary involution operator on $L^2(\cF)$ (so ${\mathfrak r} ={\mathfrak r}^{-1} = {\mathfrak r}^*$) which maps $H^2(\cF)$ unitarily onto $H^2(\cF)^\perp$ and
$H^2(\cF)^\perp$ unitarily onto $H^2(\cF)$.  We can view elements of $H^2(\cF)^\perp$ as analytic functions $f(z^{-1}) = \sum_{n=1}^\infty \widehat f_{-n} z^{-n}$
in the variable $z^{-1}$ with zero constant term representing analytic functions on the exterior of the unit disk ${\mathbb E} = \{ 1/z \colon z \in {\mathbb D}\setminus \{0\}\} \cup \{ \infty\}$
with value at $\infty$ equal to $0$.  Let us use the transformation
$$
{\mathfrak r}_+:={\mathfrak r}|_{H^2(\cF)} \colon H^2(\cF) \to H^2(\cF)^\perp
$$
to transform the model operators \eqref{BCL2modelV} acting on $H^2(\cF)$ to a unitarily equivalent version but acting on the space $H^2(\cF)^\perp$:
\begin{align}
& (\widetilde V_1, \widetilde V_2, \widetilde V)  =  {\mathfrak r}_+  (V_1, V_2, V) {\mathfrak r}_+^{-1}  \notag \\
& \quad = (M_{U^*P^\perp + z^{-1} U^*P}, M_{PU+ z^{-1} P^\perp U}, M^\cF_{z^{-1}}) \text{ on } H^2(\cF)^\perp.  \label{BCL2'}
\end{align}
Let us summarize the analysis to this point:  {\em given any commuting isometric operator pair $(\bV_1, \bV_2)$ with product $\bV = \bV_1 \bV_2$ equal to a shift operator,
there is a BCL tuple $(\cF, P, U)$ so that $(\bV_1, \bV_2)$ is unitarily equivalent to the  commuting isometric operator-pair $(\widetilde V_1, \widetilde V_2)$
given by \eqref{BCL2'}, and conversely any such pair $(\widetilde V_1, \widetilde V_2)$ is a commuting isometric operator-pair with product operator $\widetilde V = \widetilde V_1 
\widetilde V_2$ equal to the shift operator $ \widetilde V = M^\cF_{z^{-1}}$ on $H^2(\cF)^\perp$.}

Suppose now that $(T_1, T_2)= (V_1^*, V_2^*)$ is the commuting pair of coisometries with product $T = T_1 T_2 = V_1^* V_2^*$ equal to the adjoint of a shift operator.
Then by the preceding discussion we may assume that $(V_1, V_2) = (\widetilde V_1, \widetilde V_2)$ is in the model form \eqref{BCL2'}.  Let us compute
\begin{align}
(T_1, T_2) & = (\widetilde V_1^*, \widetilde V_2^*)  \label{SPair} \\
& \quad = P_{H^2(\cF)^\perp} ( M_{P^\perp U + \zeta PU}, M_{PU+ \zeta P^\perp U})|_{H^2(\cF)^\perp},  \notag \\
 & T := T_1 T_2 = P_{H^2(\cF)^\perp} M_\zeta^\cF|_{H^2(\cF)^\perp}   \label{T1T2T}
\end{align}
By inspection we see that 
$$
(\cU_1, \cU_2) : = (M_{P^\perp U + \zeta PU}, M_{PU^\perp + \zeta P^\perp U}) \text{ on } L^2(\cF)
$$
is a minimal commuting isometric (in fact commuting unitary) lift of $(T_1, T_2)$ with product $\cU_1 \cU_2 = M_\zeta$ on $L^2(\cF)$.

Our next goal is to fit this construction into the Sch\"affer model.   For these computations we view $H^2(\cF)^\perp$ as the subspace of $L^2(\cF)$ consisting of functions
$f$ of the form $f(\zeta) = \sum_{n = -1}^{-\infty} \widehat f_n \zeta^n$.  Then compute:
$$
D_T = D_T^2 = (I - T^* T) \colon f(\zeta) = \sum_{n=-1}^{-\infty} \widehat f_n \zeta^n \mapsto \widehat f_{-1} \zeta^{-1}.
$$
We define $\Lambda \colon \cD_T \to \cF$ by, for $f(\zeta) = \sum_{n=1}^{-\infty} \widehat f_n \zeta^n \in H^2(\cF)^\perp$,
\begin{align}\label{SLambda}
 \Lambda \colon D_T f = \widehat f_{-1} \zeta^{-1}  \mapsto f_{-1} \in \cF.
\end{align}
Then $\Lambda$ is a unitary operator from $\cD_T$ onto $\cF$.  We then define the model space $\bcK_S$ and the model operators
$(\bV_{S,1}, \bV_{S,2})$ as in \eqref{Conv-AndoGenForm'}.  Here the space $\cH$ on which $(T_1, T_2)$ is defined on $H^2(\cF)^\perp$,
so $\bcK_S = \sbm{ H^2(\cF)^\perp \\ H^2(\cF}$ has an easy identification with the space $L^2(\cF)$.  It is now straightforward to see that
the block matrix decompositions for $(\bV_{S,1}, \bV_{S,2})$ in \eqref{Conv-AndoGenForm'} amounts to the block matrix decompositions for the
operators $\cU_1, \cU_2)$ with respect to the decomposition of $L^2(\cF)$ as $H^2(\cF)^\perp \oplus H^2(\cF)$ and similarly the matrix decomposition
of $\bV_S = \bV_{S,1} \bV_{S,2}$ in \eqref{Vcanonical'} acting on $H^2(\cF)^\perp \oplus H^2(\cF)$ is the same as the matrix decomposition for the 
operator $M_\zeta$ on $L^2(\cF)$ with respect to the orthogonal decomposition $L^2(\cF) = H^2(\cF)^\perp \oplus H^2(\cF)$.  By the necessity side of
Theorem \ref{Thm:LiftFromTuple}, it follows that $(\cF, \Lambda, P, U)$ is a strong Type II And\^o tuple for $(T_1, T_2)$, as can also be checked directly.
This example gives the next
simplest illustration (after \#1 above) of the Sch\"affer-like model for And\^o lifts.
\end{example}

Given any commuting contractive operator-pair $(T_1, T_2)$ on $\cH$,
a special class of pre-And\^o tuples, called {\em special And\^o tuples}, associated with the pair $(T_1, T_2)$ as well as
with the adjoint pair $(T_1^*, T_2^*)$ was introduced in Section \ref{S:SpcATuple} (see Definition \ref{D:SpcATuple}).
There we saw that special And\^o tuples for $(T_1^*, T_2^*)$  are also Type I And\^o tuples (see Theorem \ref{Thm:special}) and hence led to an explicit construction 
of And\^o isometric lifts for $(T_1, T_2)$ via the Douglas-model lift construction (see Theorem \ref{T:Dmodel}).  We now show that, remarkably, special And\^o tuples 
for $(T_1, T_2)$ are automatically also strong Type II And\^o-tuples, as explained next.

\begin{theorem}  \label{Thm:special-canonical}
Suppose that $(T_1, T_2)$ is a commuting contractive operator-pair on $\cH$ and that
$(\cF_{\dag},\Lambda_\dag,P_\dag,U_\dag)$ is a special And\^o tuple for $(T_1, T_2)$ constructed as in
Definition \ref{D:SpcATuple}.  Then $(\cF_{\dag},\Lambda_\dag,P_\dag,U_\dag)$ is in fact a strong Type II
And\^o tuple for $(T_1, T_2)$, i.e., conditions {\rm (i$^\prime$)} and \rm{(ii)} in Definition \ref{AndoTuple} are satisfied.
\end{theorem}

\begin{proof}
let $(\cF_{\dag},\Lambda_\dag,P_\dag,U_\dag)$ be the special  And\^o tuple for $(T_1,T_2)$ constructed via the algorithm
explained in Definition \ref{D:SpcATuple}.  We must verify that conditions (i$^\prime$) and (ii) in Definition \ref{AndoTuple} are satisfied.
Let $h$ be an arbitrary element of $h$.
The computation
\begin{align*}
&P_\dag U_\dag \Lambda_\dag D_TT_2h+P_\dag^\perp\Lambda_\dag D_Th\\
&=P_\dag U_\dag(D_{T_1}T_2^2h\oplus D_{T_2}T_2h)+P_\dag^\perp(D_{T_1}T_2h\oplus D_{T_2}h)\\
&=P_\dag(D_{T_1}T_2h\oplus D_{T_2}T_1T_2h)+(0 \oplus D_{T_2}h)
=D_{T_1}T_2h \oplus D_{T_2}h=\Lambda_\dag D_Th
\end{align*}
together with the computation
\begin{align*}
&U_\dag^*(P_\dag^\perp\Lambda_\dag D_TT_1+P_\dag U_\dag\Lambda_\dag D_T)h\\
&=U_\dag^*(P_\dag^\perp (D_{T_1}T_2T_1h\oplus D_{T_2}T_1h)+P_\dag U_\dag(D_{T_1}T_2h\oplus D_{T_2}h))\\
&=U_\dag^*((0\oplus D_{T_2}T_1 h)+P_\dag(D_{T_1}h \oplus D_{T_2}T_1h)) \\
&=U_\dag^*((0\oplus D_{T_2}T_1 h )+  (D_{T_1} h \oplus 0))=U_\dag^*(D_{T_1} h \oplus D_{T_2}T_1 h )=\Lambda_\dag D_T h
\end{align*}
verifies condition (i$^\prime$).
 The following easy inner product computations for arbitrary $h, h' \in\cH$
\begin{align*}
 \langle D_T \Lambda_\dag^* U_\dag^* P_\dag U_\dag \Lambda_\dag D_T h, h'\rangle  & =
  \langle P_\dag U_\dag\Lambda_\dag D_Th,P_\dag U_\dag\Lambda_\dag D_Th' \rangle \\
  & =\langle D_{T_1}h\oplus 0,D_{T_1}h'\oplus 0\rangle =\langle D_{T_1}^2h,h' \rangle
\end{align*}
and
\begin{align*}
\langle D_T \Lambda_\dag^* P_\dag^\perp \Lambda_\dag D_T h, h' \rangle
&  = \langle P_\dag^\perp\Lambda_\dag D_T h,P^\perp\Lambda_\dag D_Th' \rangle  \\
 & =\langle 0\oplus D_{T_2}h,0\oplus D_{T_2}h'\rangle =\langle D_{T_2}^2h,h' \rangle.
\end{align*}
establish condition (ii).  
\end{proof}

\begin{notation}  \label{N:natural}\index{$V_{\natural 1}, V_{\natural 2}$}
For a given commuting contractive operator-pair $(T_1,T_2)$,  the And\^o  lift corresponding to the special And\^o tuple of
$(T_1,T_2)$ given as in (\ref{Conv-AndoGenForm'}) will be denoted by $(V_{\natural1},V_{\natural2})$ and $V_{\natural}=V_{\natural1}V_{\natural2}$. Note that because any special And\^o tuple of $(T_1,T_2)$ is a
strong Type II And\^o tuple, $V_{\natural} : = V_{\natural 1} V_{\natural 2}$ is given by
\begin{align}\label{Vnatural}
\begin{bmatrix}
  T_1T_2 & 0 \\
  \bev_{0,\cF}^*\Lambda_\dag D_T & M_z
\end{bmatrix}.
\end{align}
\end{notation}

\begin{remark}  \label{R:Schaffer-Ando}
We have already noted that the definition of special And\^o tuple (whether for $(T_1, T_2)$ or for $(T_1^*, T_2^*)$) is
constructive and hence in principal special And\^o tuples are easy to write down.  In particular we are assured that the class of
special And\^o tuples for $(T_1, T_2)$ is not empty.  Furthermore,
Theorem \ref{Thm:special-canonical} combined with Theorem \ref{Thm:LiftFromTuple} shows us how to use  special
And\^o tuples for $(T_1, T_2)$ to construct a class of And\^o isometric lifts for a given commuting contractive operator
pair $(T_1, T_2)$.  In this way we arrive at a second new proof of And\^o's theorem \cite{ando}, via the Sch\"affer-model
construction rather than by the Douglas-model construction as in Theorem \ref{Thm:special}.
 A priori it is conceivable that the unitary equivalence classes of And\^o isometric lifts arising via the Douglas-model
construction from special And\^o tuples for $(T_1^*, T_2^*)$ are distinct from those arising via the Sch\"affer-model
construction from special And\^o tuples for $(T_1, T_2)$.
\end{remark}

\begin{remark} \label{R:reasonsS}
In parallel with the discussion in Remark \ref{R:reasonsD} concerning the choice of  BCL model in the construction of
the Douglas model for an And\^o lift of $(T_1, T_2)$,  here we discuss the choice of BCL model in the construction of the Sch\"affer model for an And\^o lift 
of $(T_1, T_2)$.   In the representation for an And\^o lift
$(V_1, V_2)$ on $\cK$ for a commuting contractive pair $(T_1, T_2)$ on $\cH \subset \cK$, we could have used the BCL2 model (rather than the BCL1 model) for 
$(V_1, V_2)|_{\cK \ominus \cH}$ to arrive at the unitarily equivalent  And\^o lift having the form $(\bPi'_S, \bV'_{S,1}, \bV'_{S,2})$ where
\begin{align}
&  (\bV'_{S,1},\bV'_{S,2},\bV'_S)=   \notag \\
  & \left(\begin{bmatrix}
            T_1 & 0  & 0 \\
            C_{s1} & M_{U^* P^\perp + z U^* P} & 0 \\ C_{u1} & 0 & Y_1   \end{bmatrix},
  \begin{bmatrix}
   T_2 & 0 & 0  \\
   C_{s2} & M_{PU + z P^\perp U} & 0 \\ C_{u2}  & 0 & Y_2
   \end{bmatrix},
  \begin{bmatrix}
      T & 0 & 0  \\
      C_s  & M^\cF_z & 0 \\ C_u & 0 &  Y
    \end{bmatrix}\right).
    \label{liftS'}
\end{align}
in place of \eqref{liftS}.  Then the same analysis leading from \eqref{liftS} to \eqref{AndoGenForm} and \eqref{canonicalV}
would lead us instead to
\begin{align}
& (\bV'_{S,1},\bV'_{S,2}) =  \label{AndoGenForm''}  \\
& \left(\begin{bmatrix} T_1 & 0 & 0  \\  \bev_{0,\cF}^* U^* P  \Lambda D_T & M_{U^*P^\perp + zU^* P} & 0 \\ 0  & 0 &  Y_1 \end{bmatrix},
  \begin{bmatrix}
   T_2 & 0 & 0 \\
  \bev_{0,\cF}^* P^\perp  U \Lambda D_T & M_{ P U + z P^\perp U } & 0 \\ 0 & 0 &  Y_2 \end{bmatrix}\right),
\notag\\
& \bV'_S:= \bV'_{S,1} \bV'_{S,2} = \bV'_{S,2} \bV'_{S,1} = \begin{bmatrix} T & 0  & 0 \\   \bev_{0,\cF}^* \Lambda D_T & M^\cF_z & 0 \\  0 & 0 & Y\ \end{bmatrix}.
  \label{canonicalV''}
 \end{align}
 where we set $T : = T_1 T_2 = T_2 T_1$, $Y: = Y_1 Y_2 = Y_2 Y_1$ and where $\Lambda \colon \cD_T \to \cF$
 is an isometry; let us note that a short-cut way to see this is to make use of the flip transformation 
$\ff$ given by \eqref{flip}.

 Conversely, for a collection of spaces and operators of the form \eqref{AndoGenForm''}, \eqref{canonicalV''} to be an And\^o lift of $(T_1, T_2)$ 
 requires the additional compatibility conditions
 \begin{enumerate}
 \item[(a)] \textbf{Commutativity condition:}
 $$
 U^* P \Lambda D_T T_2 + U^* P^\perp U \Lambda D_T = P^\perp U \Lambda D_T T_1 + P \Lambda D_T ;
 $$
 \item[(b)]  \textbf{Isometry condition:}
 $$
  D_T \Lambda^* P \Lambda D_T = D_{T_1}^2, \quad
    D_T \Lambda^* U^* P^\perp U \Lambda D_T = D_{T_2}^2;
$$
\item[(a$'$)]   \textbf{Strengthened commutativity condition:}
$$
U^* P \Lambda D_T T_2 + U^* P^\perp U \Lambda D_T = P^\perp U \Lambda D_T T_1 + P \Lambda D_T
= \Lambda D_T.
$$
\end{enumerate}
We note that these equations are obtained simply by replacing $(P, U)$ by $(U^* P U, U^*)$ (i.e., by applying the flip transformation $\ff$ \eqref{flip}) 
in the conditions
(i), (ii), (i$'$) in the definition of Type II And\^o tuple (see Definition \ref{AndoTuple}).  Let us say that a pre-And\^o tuple for
$(T_1, T_2)$ is a {\em Type} II$'$   {\em And\^o tuple} for $(T_1, T_2)$ if the modified compatibility conditions (a), (b) hold,
and is a {\em strongType} II$'$ {\em And\^o tuple for} $(T_1, T_2)$ if the strengthened form of these conditions
(a$'$) (b) holds.

Just as in the Douglas-model setting, the drawback of this alternative approach is that a
special And\^o tuple for $(T_1, T_2)$ need not be a strong Type II$^\prime$ And\^o tuple for $(T_1, T_2)$, i.e., the analogue
of Theorem  \ref{Thm:special-canonical} with {\em strong Type} II$'$ {\em And\^o tuple} inserted in place of
{\em strong Type} II {\em And\^o tuple} fails in general.  For an explicit example we refer to the Appendix (Section \ref{S:Appendix} below).
\end{remark}

\section{Strongly minimal And\^o lifts via strongly minimal And\^o tuples}  \label{S:StrongMinimality}

We here introduce the notions of {\em strongly minimal And\^o tuple} and {\em strongly minimal And\^o lift} for a commutative contractive operator-pair $(T_1, T_2)$.

\begin{definition}  \label{D:StrongMinTuple}\index{pre-And\^o tuple!strongly minimal}
A pre-And\^o tuple $(\cF,\Lambda,P,U)$ for $(T_1,T_2)$ is said to be {\em strongly minimal} if the isometry $\Lambda:\cD_{T_1T_2}\to\cF$ is surjective. 
In case the pre-And\^o tuple $(\cF, \Lambda, P, U)$ is a Type I or strong Type II And\^o tuple and is strongly minimal as an And\^o pre-tuple, we say that
$(\cF, \Lambda, P, U)$ is a strongly minimal Type I (respectively strongly minimal strong Type II)  And\^o tuple. 
\end{definition}

The companion notion {\em strongly minimal And\^o lift} defined as follows.

\begin{definition}  \label{D:StrpmgMinLift}\index{And\^o lift!strongly minimal}
An And\^o lift $(\Pi, V_1,V_2)$ of a commuting contractive pair $(T_1,T_2)$ is said to be {\em strongly minimal} if it acts on the space $\cK_{00}$ given by
\begin{align*}
\cK_{00} =\bigvee_{n\geq 0} V_1^nV_2^n\operatorname{Ran}\Pi.
\end{align*}
\end{definition}

Note that, since $\operatorname{Ran}\Lambda =\cF$ for a strongly minimal pre-And\^o tuple $(\cF,\Lambda,P,U)$,  
it is obvious that a strongly minimal pre-And\^o tuple is  Dougls-minimal (see Definition \ref{D:min-Ando}). 

The next result makes precise the strong correlation between these two notions of strongly minimal.

\begin{theorem}  \label{T:sminLift=sminTuple}
Let $(T_1,T_2)$ be a commuting contractive operator-pair acting on $\cH$. Then:
\begin{enumerate}
\item $(T_1,T_2)$ has a strongly minimal And\^o lift $(\Pi, V_1,V_2)$ if and only if there is a strongly minimal Type I And\^o tuple for $(T_1^*,T_2^*)$
In more detail, Theorem \ref{T:Dmodel} continues to hold with the substitutions:
\begin{itemize}
\item {\em Douglas-minimal Type I And\^o tuple of $(T_1^*, T_2^*)$}  $\rightarrow$ {\em strongly minimal Type I And\^o tuple of $(T_1, T_2)$,}

\item {\em minimal And\^o lift of $(T_1, T_2)$}  $\rightarrow$ {\em strongly minimal And\^o lift of $(T_1, T_2)$}.
\end{itemize}

\smallskip

\item $(T_1,T_2)$ has a strongly minimal And\^o lift $(\Pi, V_1,V_2)$ if and only if there is a strongly minimal strong Type II And\^o tuple for $(T_1,T_2)$.
In more detail, Theorem \ref{Thm:LiftFromTuple} continues to hold with the substitutions
\begin{itemize}

\item  {\em Sch\"affer-minimal strong Type II And\^o tuple for $(T_1, T_2)$} $\rightarrow$ {\em strongly minimal strong Type II And\^o tuple for $(T_1, T_2)$,}

\item  minimal And\^o lift of $(T_1, T_2)$ $\rightarrow$  {\em strongly minimal And\^o lift of $(T_1, T_2)$.}
\end{itemize}
\end{enumerate}

Thus existence of a strongly minimal  Type I And\^o tuple for $(T_1^*, T_2^*)$ is equivalent to existence of a
strongly minimal strong Type II And\^o tuple for $(T_1, T_2)$.
\end{theorem}

\begin{proof}[Proof of (1):]
If $(\Pi, V_1,V_2)$ is a strongly minimal And\^o lift of $(T_1,T_2)$, then by definition $V=V_1V_2$ is a minimal isometric lift of $T=T_1T_2$. By Theorem \ref{T:Dmodel}, 
the And\^o lift $(\Pi, V_1,V_2)$ can be modeled as $(\bPi_D, \bV_{D,1}, \bV_{D,2})$ in \eqref{DougAndoMod} for some Douglas-minimal Type I And\^o tuple $(\cF_*, \Lambda_*,
P_*, U_*)$ for $(T_1^*, T_2^*)$.  Consider the minimal isometric lift $(\Pi_D,V_D)$ of 
$T$ acting on 
$\cK_D=\sbm{H^2(\cD_{T^*})\\ \cQ_{T^*}}$ as in \S \ref{S:Douglas}. By the unitary equivalence of any two Sz.-Nagy--Foias isometric lifts of a given contraction contractiion operator
and of the uniqueness of the implementing unitary transformation (see \cite[Theorem I.4.1]{Nagy-Foias}), there is a unique unitary 
\begin{align*}
\tau=\begin{bmatrix}
\tau'&\tau_{12}\\ \tau_{21}&\tau''
\end{bmatrix}:\begin{bmatrix}
H^2(\cD_{T^*})\\ \cQ_{T^*}
\end{bmatrix}\to\begin{bmatrix}
H^2(\cF_*)\\ \cQ_{T^*}
\end{bmatrix} 
\end{align*}such that
\begin{align}
\begin{bmatrix} \tau'&\tau_{12}\\ \tau_{21}&\tau'' \end{bmatrix}
\begin{bmatrix} M_z^{\cD_{T^*}}&0\\0&W_D \end{bmatrix} & =
\begin{bmatrix}  M_z^{\cF_*}&0\\0&W_D \end{bmatrix}  \begin{bmatrix}\tau'&\tau_{12}\\ \tau_{21}&\tau''\end{bmatrix}, \label{Inter1}\\  
\quad \begin{bmatrix}
\tau'&\tau_{12}\\ \tau_{21}&\tau'' \end{bmatrix}\begin{bmatrix} \cO_{\cD_{T^*},T^*}\\Q_{T^*}\end{bmatrix} &=
\begin{bmatrix} (I_{H^2}\otimes \Lambda_*)\cO_{\cD_{T^*},T^*}\\Q_{T^*} \end{bmatrix}.   \label{Inter2}
\end{align}
Applying Part (2) of Lemma \ref{L:AuxLemma} to \eqref{Inter1} we conclude that 
$$
\tau=\begin{bmatrix}
I_{H^2}\otimes \Lambda'&0\\ 0&\tau''
\end{bmatrix}
$$
for some unitary operators $\Lambda':\cD_{T^*}\to\cF_*$ and $\tau'':\cQ_{T^*}\to\cQ_{T^*}$. Equation \eqref{Inter2} therefore yields
\begin{align}
&(I_{H^2}\otimes\Lambda')\cO_{D_{T^*},T^*}=(I_{H^2}\otimes\Lambda_*)\cO_{D_{T^*},T^*},  \label{Eqn1} \\ 
&\tau'' Q_{T^*}=Q_{T^*} \quad \tau'' W_D=W_D\tau''.   \label{Eqn2}
\end{align}
Equating the constant coefficients in the series forms of \eqref{Eqn1}, we get $\Lambda'=\Lambda_*$. Consequently, the Type I And\^o tuple 
$(\cF_*,\Lambda_*,P_*,U_*)$ is strongly minimal.

Conversely, suppose $(\cF_*,\Lambda_*,P_*,U_*)$ is a Type I And\^o tuple of $(T_1^*,T_2^*)$ such that $\Lambda_*:\cD_{T^*}\to\cF_*$ is a unitary. 
By the forward direction of Theorem \ref{T:Dmodel}, the isometric operators $\bPi_D ,\bV_{D,1}, \bV_{D,2}$ given by
\begin{align*}
&   \bPi_D = \begin{bmatrix} (I_{H^2} \otimes \Lambda_*)  \cO_{D_{T^*}, T^*} \\ Q_{T^*} \end{bmatrix} \colon \cH \to 
\begin{bmatrix} H^2(\cF_*)\\ \cQ_{T^*} \end{bmatrix}, \\
 & (\bV_{D,1}, \bV_{D,2}) =  \left( \begin{bmatrix} M_{U_*^*(P_*^\perp + z P_*)} & 0 \\ 0 & W_{\flat 1} \end{bmatrix},
  \begin{bmatrix} M_{(P_* + z P_*^\perp) U_*} & 0 \\ 0 & W_{\flat 2} \end{bmatrix} \right) \text{ on } 
\begin{bmatrix} H^2(\cF_*)\\ \cQ_{T^*} \end{bmatrix}  
 \end{align*}
 constitute an And\^o lift of $(T_1,T_2)$. Note that the operator
 $$
\tau:=\begin{bmatrix}
I_{H^2}\otimes \Lambda_*&0\\0&I_{\cQ_{T^*}}
\end{bmatrix}:\begin{bmatrix}
H^2(\cD_{T^*})\\ \cQ_{T^*}
\end{bmatrix}\to\begin{bmatrix}
H^2(\cF_*)\\ \cQ_{T^*}
\end{bmatrix}
 $$ 
 is unitary from  $\sbm{ H^2(\cD_{T^*} \\ \cQ_{T^*} }$ onto $\sbm{ H^2(\cF_*) \\ \cQ_{T^*}}$ since $\Lambda_*$ is unitary from $\cD_{T^*}$ onto $\cF_*$, and
 has the properties
\begin{align*}
&\tau\Pi_D=
\begin{bmatrix}
I_{H^2}\otimes \Lambda_*&0\\0&I_{\cQ_{T^*}}
\end{bmatrix}\begin{bmatrix} \cO_{D_{T^*}, T^*} \\ Q_{T^*} \end{bmatrix}=
\begin{bmatrix} (I_{H^2} \otimes \Lambda_*)  \cO_{D_{T^*}, T^*} \\ Q_{T^*} \end{bmatrix}=\Pi\quad\mbox{and}\\
&\tau V_D=\begin{bmatrix}
I_{H^2}\otimes \Lambda_*&0\\0&I_{\cQ_{T^*}}
\end{bmatrix}\begin{bmatrix}
M_z^{\cD_{T^*}}&0\\0&W_D
\end{bmatrix}=\begin{bmatrix}
M_z^{\cF_*}&0\\0&W_D
\end{bmatrix}\begin{bmatrix}
I_{H^2}\otimes \Lambda_*&0\\0&I_{\cQ_{T^*}} \end{bmatrix}=V\tau,
\end{align*}
where $V=V_1V_2$ and $(\Pi_D,V_D)$ is the Douglas-model minimal isometric lift of $T=T_1T_2$ as in \S \ref{S:Douglas}. Consequently, the And\^o lift 
$(\Pi, V_1,V_2)$ of $(T_1,T_2)$ is strongly minimal. This completes the proof of (1).

{\sf{Proof of (2):}}  For this proof we use the Sch\"affer model rather than the Douglas model. Suppose that 
 $(\Pi, V_1,V_2)$ is a strongly minimal And\^o lift of $(T_1,T_2)$. Since a strongly minimal And\^o lift is obviously minimal, by the converse part of 
 Theorem \ref{Thm:LiftFromTuple} we know that
 there is a strong Type II And\^o tuple $(\cF,\Lambda,P,U)$ of $(T_1,T_2)$ so that the lift $(\Pi, V_1,V_2)$ is unitarily equivalent to the Sch\"affer-model lift
 $(\bPi_S, \bV_{S,1}, \bV_{S,2})$ \eqref{Conv-AndoGenForm'} determined by $(\cF, \Lambda, P, U)$
 \begin{align*}
 & \bPi_S = \begin{bmatrix} I_\cH \\ 0 \end{bmatrix} \colon \cH \to \bcK_S = \begin{bmatrix} \cH \\ H^2(\cF) \end{bmatrix}, \\
& (\bV_{S,1}, \bV_{S,2}) =
 \left(\begin{bmatrix} T_1 & 0 \\ \bev_{0,\cF}^* PU\Lambda D_T & M_{(P^\perp+zP)U} \end{bmatrix},
 \begin{bmatrix} T_2 & 0 \\ \bev_{0,\cF}^*  U^*P^\perp \Lambda D_T & M_{U^*(P+zP^\perp)} \end{bmatrix}\right)
\end{align*}
and also by \eqref{Vcanonical'}
\begin{align}\label{CanoIso}
\bV_S= \bV_{S,1}  \bV_{S,2}  = \begin{bmatrix} T & 0 \\  \bev_{0,\cF}^* \Lambda D_T & M_z^\cF \end{bmatrix} \colon \begin{bmatrix} \cH \\ H^2(\cF) \end{bmatrix}\
\to \begin{bmatrix} \cH \\ H^2(\cF) \end{bmatrix}.
\end{align}
By hypothesis, $V$ is a minimal lift of $T=T_1T_2$. On the other hand, consider the Sch\"affer-model minimal isometric lift  $(\Pi_S, V_S)$ of $T=T_1T_2$, given by
\begin{align}\label{SchafferIso}
& \Pi_S = \begin{bmatrix} I_\cH \\ 0 \end{bmatrix} \colon \cH \to \cK_S:= \begin{bmatrix} \cH \\ H^2(\cD_T) \end{bmatrix} , \quad
& V_S=\begin{bmatrix}
     T&0\\ \bev_{0,\cD_T}^*D_T & M_z^{\cD_T} \end{bmatrix}  \text{ on } \begin{bmatrix}  \cH \\ H^2(\cD_T) \end{bmatrix}
 \end{align}
as discussed in \S \ref{S:Schaffer}. By the uniqueness of minimal Sz.-Nagy--Foias lifts of single contraction operators \cite[Theorem I.4.1]{Nagy-Foias}, there exists a unitary
$$
\tau = \begin{bmatrix} \tau_{11} & \tau_{12} \\ \tau_{21} & \tau_{22} \end{bmatrix} \colon 
\begin{bmatrix} \cH \\ H^2(\cD_T) \end{bmatrix} \to  \begin{bmatrix}  \cH \\ H^2(\cF) \end{bmatrix}
$$
intertwining the isometries in \eqref{SchafferIso} and \eqref{CanoIso} and such that $\tau|_\cH = I_\cH$. 
Hence $\tau$ is unitary with $\tau_{11} = I_\cH$ forcing also $\tau_{12} = 0$, $\tau_{21} = 0$, so now 
$\tau = \sbm{ I_\cH & 0 \\ 0 & \tau_{22}}$ with $\tau_{22} \colon H^2(\cD_T) \to H^2(\cF)$ unitary.
The intertwining condition thus becomes
\begin{align}\label{Inter}
\begin{bmatrix} I_\cH & 0 \\ 0 & \tau_{22} \end{bmatrix}
\begin{bmatrix}  T&0\\ \bev_{0, \cD_T}^*D_T & M_z^{\cD_T}  \end{bmatrix}
=\begin{bmatrix} T&0\\ \bev_{0,\cF}^* \tau_{22} D_T & M_z^\cF \end{bmatrix}
\begin{bmatrix} I_\cH & 0 \\ 0 & \Lambda' \end{bmatrix}.
\end{align}
Comparing the (2,1) and (2,2) entries in the above matrices then gives
$$
\tau_{22} \bev_{0, \cD_T}^* D_T = \bev_{0, \cF}^* \Lambda D_T, \quad \tau_{22} M_z^{\cD_T} = M_z^\cF \tau_{22}.
$$
From the first equation we see that $\tau_{22}$ takes constants into constants.  From the second equation we see that
$\tau_{22}$ is a multiplication operator.  Putting these two conditions together says that $\tau_{22}$ is multiplication by a constant, i.e.,
$\tau_{22}$ has the form $\tau_{22} = I_{H^2}\otimes\Lambda'$ for some $\Lambda' \colon \cD_{T} \to \cF$.
Further inspection of the first equation tells us that the constant is $\Lambda$:  $\Lambda' = \Lambda$.   Moreover, as observed earlier, $\tau_{22}$
is unitary, i.e., the operator $I_{H^2}\otimes\Lambda$ is unitary.  This then forces $\Lambda$ to be unitary as an operator from $\cD_T$ to $\cF$.
This in turn means that the tuple $(\cF,\Lambda,P,U)$ is also strongly minimal as a Type II And\^o tuple.

Conversely, let $(\cF,\Lambda,P,U)$ be a strongly minimal strong Type II And\^o tuple of the pair $(T_1,T_2)$. By the forward direction of Theorem \ref{Thm:LiftFromTuple}, 
for any Type II And\^o tuple $(\cF,\Lambda,P,U)$ of $(T_1,T_2)$, $(\bPi_S, \bV_{S,1}, \bV_{S,2})$  given by \eqref{Conv-AndoGenForm'} 
is an And\^o lift of $(T_1, T_2)$ with $\bV_S:= \bV_{S,1} \bV_{S,2}$ as in \eqref{Vcanonical'}.
The strongly minimal property of the And\^o tuple $(\cF, \Lambda, P, U)$ means that the operator $\Lambda \colon \cD_T \to \cF$ is unitary, hence also
the operator
$$
 \tau:= \begin{bmatrix} I_\cH & 0 \\ 0 & I_{H^2} \otimes \Lambda \end{bmatrix} \colon \cK_S:= \begin{bmatrix} \cH \\ H^2(\cD_T) \end{bmatrix} \to 
 \begin{bmatrix} \cH \\ H^2(\cF) \end{bmatrix}  =: \bcK_S
$$
is unitary.  By definition we have
$$
 \tau \Pi_S =  \begin{bmatrix} I_\cH & 0 \\ 0 & I_{H^2} \otimes \Lambda \end{bmatrix}  \begin{bmatrix} I_\cH \\ 0 \end{bmatrix} 
 = \begin{bmatrix} I_\cH \\ 0 \end{bmatrix} = \bPi_S
 $$
 and 
 \begin{align*}
 \tau V_S & = \begin{bmatrix} I_\cH & 0 \\ 0 & I_{H^2} \otimes \Lambda \end{bmatrix} \begin{bmatrix}  T & 0 \\ \bev_{0, \cD_T}^* D_T & M_z^{\cD_T} \end{bmatrix}
=   \begin{bmatrix}  T & 0 \\ \bev_{0, \cF}^* \Lambda D_T & M_z^\cF \end{bmatrix} \begin{bmatrix} I_{H^2} & 0 \\ 0 & I_{H^2} \otimes \Lambda \end{bmatrix}  \\
& = \bV_S   \tau 
\end{align*}
 i.e., the Sz.-Nagy--Foias lift $(\bPi_S, \bV_S)$ of $T$ is unitarily equivalent (via $\tau$) to the Sch\"affer-model minimal Sz.-Nagy--Foias lift $(\Pi_S,  V_S)$ of $T$, and hence
 the lift $(\bPi_S, \bV_S)$ must itself also be minimal, meaning that the original And\^o lift $(\Pi, V_1, V_2)$ of $(T_1, T_2)$ is strongly minimal.
 This completes the proof of the theorem.
 \end{proof} 
 
 We have noted in Examples \ref{E:Dex} (Douglas version) and \ref{E:Sex} (Sch\"affer version) that commuting isometric operator-pairs with product operator equal to a shift
and commuting coisometric operator-pairs with product operator equal to the adjoint of a shift have strongly minimal Type I And\^o tuples (respectively,
strongly minimal strong Type II And\^o tuples).  The restriction that the product be a shift or adjoint of a shift is not essential:  the result still holds without this restriction.
On the other hand we have seen in item (3) of Theorem \ref{T:regfact}  that both $T_1 \cdot T_2$ and $T_2 \cdot T_1$ are regular factorizations if
$(T_1, T_2)$ is a commutative isometric operator-pair or a commutative co-isometric operator pair.  The next result shows that this confluence of observations is no accident.

\begin{theorem}\label{T:SMinReg}
A commuting contractive pair $(T_1, T_2)$ has a strongly minimal And\^o lift if and only if both factorizations $T_1 \cdot T_2$ and $T_2 \cdot T_1$ are regular. 
In this case, there is a unique strongly minimal canonical-form special And\^o tuple of $(T_1,T_2)$.
\end{theorem}

\begin{proof}
We note that the last statement in the theorem is simply a direct application of item (2) in Corollary \ref{C:UniqueMinAndoTuple}.

Suppose that $(T_1, T_2)$ is a commuting contractive pair such that both $T_1 \cdot T_2$ and $T_2 \cdot T_1$ are regular factorizations.
This means that both spaces $\cD_{U_0}$ and $\cR_{U_0}$ given by \eqref{U0dom-codom} are equal to the whole space $\cD_{T_1} \oplus \cD_{T_2}$.
Thus both the operator $\Lambda_\dagger$ and the operator $\Lambda_\dagger^\ff$ given on a dense set by
$$
  \Lambda_\dagger^\ff \colon D_T h \mapsto D_{T_1} h \oplus D_{T_2} T_1 h \text{ for } h \in \cH
 $$
define unitary operators from $\cD_T$ onto $\cD_{T_1} \oplus \cD_{T_2}$, and furthermore the operator $U_0$ defined densely by \eqref{U0}
extends to define a unitary operator from $\cF_\dagger:= \cD_{T_1} \oplus \cD_{T_2}$ onto itself. Define a projection $P_\dagger$
on $\cD_{T_1} \oplus \cD_{T_2}$ by
$$
   P_\dagger \colon f_1 \oplus f_2 = f_1 \oplus 0 \text{ for } f_1 \in \cD_{T_1}, \, f_2 \in \cD_{T_2}.
$$
Then the collection
$$
(\cF_\dagger = \cD_{T_1} \oplus \cD_{T_2}, \Lambda_\dagger, P_\dagger, U_\dagger = U_0)
$$
is a canonical-form special And\^o tuple for $(T_1, T_2)$  which is also strongly minimal as an And\^o tuple. By Theorem \ref{Thm:special-canonical}, 
this is a strong Type II And\^o tuple for $(T_1,T_2)$.  By Part (2) of Theorem \ref{T:sminLift=sminTuple}, $(T_1,T_2)$ has a strongly minimal And\^o lift.

Conversely, suppose that $(T_1,T_2)$ acting on $\cH$ has a strongly minimal And\^o lift which we take to be of Sch\"affer type $(\Pi = \iota_\cH,  V_1,V_2)$ with
$\iota_\cH \colon \cH \to \cK$ the inclusion map and with 
$(V_1, V_2)$ acting on the ambient space for the minimal isometric lift of $T = T_1 T_2$:
$$
\cK=\bigvee_{n\geq0}  V_1^nV_2^n \cH.
$$ 
We first note that with $V_0:=V=V_1V_2$,
\begin{equation}   \label{note1}
D_{V_j^*}\cK=D_{V_j^*}\cH \quad\mbox{for}\quad j=0,1,2.
\end{equation}
Indeed, if $n\geq1$, then since $D_{V_j^*}=(I_\cK - V_jV_j^*)$ we have
\begin{align*}
D_{V_1^*}V_1^nV_2^n=0=D_{V_2^*}V_2^nV_1^n  \text{ and }  D_{V^*}V^n=0.
\end{align*}
Furthermore note that the map
$\omega_j \colon \cD_{T_j^*} \to \cD_{V_j^*}$ given by
$$
 \omega_j \colon D_{T_j^*} h \mapsto D_{V_j^*} h \text{ for } h \in \cH
$$
for $j=0,1,2$ (here we set $T_0 = T_1 T_2$)
is isometric:  indeed, simply note that
\begin{align*}
  \| D_{V_j^*} h \|^2 & = \langle (I - V_j V_j^* ) h, h \rangle = \| h \|^2 - \| V_j^* h \|^2 \\
  &  = \| h \|^2 - \| T_j^* h \|^2 \text{ (by the lifting property) }  \\
 & =  \langle (I - T_j T_j^*) h, h \rangle = \| D_{T_j^*}^2 h\|^2.
 \end{align*}
Combining this with the observation \eqref{note1}, we see that each of the maps $\omega_j$ is unitary from $\cD_{T_j^*}$ onto $\cD_{V_j^*}$.
Focusing now just on the case $j=1,2$ and defining the map
$ \tau \colon \cD_{T_1^*} \oplus  \cD_{T_2^*} \to  \cD_{V_1^*} \oplus  \cD_{V_2^*}$ defined densely by
$$
 \tau \colon D_{T_1^*}h \oplus  D_{T_2^*}k \mapsto D_{V_1^*}h \oplus \cD_{V_2^*}k,
$$
we see that $\tau$ so defined is also unitary. 

On the other hand, since $(V_1^*, V_2^*)$ is a commuting coisometric pair, we know by item (3) in Theorem \ref{T:regfact} that $V_1^*\cdot V_2^*$ is a regular factorization. 
Hence the operator
$ \sigma: \cD_{V^*}\to  \cD_{V_1^*} \oplus  \cD_{V_2^*}$  defined densely by 
$$
 \sigma \colon D_{V^*}h \mapsto  D_{V_1^*}V_2^* h \oplus D_{V_2^*}  h
$$
is unitary as well. Now it just remains to read off from the definitions of the unitary operators $\omega_0, \tau$ and $\sigma$ that the isometry
$$
\Lambda_{\dag*}: D_{T^*}h  \mapsto D_{T_1^*}T_2^*  h \oplus  D_{T_2^*} h
$$
coincides with $\tau^* \circ \sigma \circ \omega_0$ on a dense set, viz., $\{D_{T^*}h: h\in\cH\}$.
Therefore 
$\Lambda_{\dag*}$ must be  unitary and hence by definition, the factorization $T_1^*\cdot T_2^*$ is regular, or equivalently $T_2 \cdot T_1$ is so. Now note that if 
$(T_1,T_2)$ has a strongly minimal And\^o lift, then so does the pair $(T_2,T_1)$. Thus proceeding as above for the pair $(T_2,T_1)$, one can conclude that 
the factorization $T_1\cdot T_2$ is regular. This completes the proof.
\end{proof}

It is not clear wheather the existence of a strongly minimal And\^o lift for the commutative contractive pair $(T_1, T_2)$ is eqivalent to the existence of a strongly minimal
And\^o lift for the adjoint pair $(T_1^*, T_2^*)$.  For the companion notion of existence  of a strongly minimal And\^o lift, this invariance-under-adjoint property
can be worked out via a systematic calculation as follows.

\begin{theorem}\label{T:AdjointT}
Let $(T_1,T_2)$ be a commuting contractive pair acting on $\cH$. Then $(T_1,T_2)$ has a strongly minimal And\^o lift if and only if $(T_1^*,T_2^*)$ has a 
strongly minimal And\^o lift.
\end{theorem}

\begin{proof}
Note that the two-way implications follow by symmetry once we establish one of them. We suppose that $(T_1,T_2)$ has a strongly minimal And\^o lift and prove that 
then so does $(T_1^*,T_2^*)$.  Let $(\Pi, V_1,V_2)$ be a strongly minimal And\^o lift of $(T_1,T_2)$, i.e., $V_1,V_2$ act on 
$\cK=\bigvee_{n\geq 0}\{V_1^nV_2^n\operatorname{Ran}\Pi \}$.  By Lemma \ref{L:special-ext}, $(V_1,V_2)$ has a unitary extension $
(W_1,W_2)$ acting on the space
$$
\widetilde\cK=\bigvee_{n\in\mathbb Z}\{W_1^nW_2^n\operatorname{Ran}\Pi\}.
$$ 
Let us set
\begin{align}\label{setup}
\cK':=\bigvee_{n\geq0} W_1^{*n}W_2^{*n}\operatorname{Ran}\Pi=\bigvee_{n\geq0} W^{*n}\operatorname{Ran}\Pi,
\end{align}
where we use the notation $W=W_1W_2$. We argue that $\cK'$ is a $(W_1^*,W_2^*)$-invariant subspace; the geometry of the minimal dilation space 
as discussed in Section 2.3 will be used here. 

Let us consider the spaces
\begin{align*}
\cK_+=\cK\ominus\operatorname{Ran}\Pi \quad\mbox{and}\quad \cK_-=\widetilde\cK \ominus \cK.
\end{align*}
This induces a three-fold orthogonal decomposition of the space $\cK'$:
\begin{equation}   \label{3decom}
\widetilde \cK  =  \cK_- \oplus \operatorname{Ran} \Pi \oplus \cK_+.
\end{equation}
Since $(W_1, W_2, W)$ is an extension of $(V_1, V_2, V = V_1 V_2)$ and $(V_1, V_2)$ is an isometric lift of $(T_1, T_2)$ we see that
$\cK =  \operatorname{Ran} \Pi \oplus  \cK_+$ is invariant for $(W_1, W_2, W)$.  From the facts that $(V_1, V_2)$ is the restriction of
$(W_1, W_2)$ to $\cK$ and that $(V_1, V_2)$ is a lift for $(T_1, T_2)$ on $\cK$ (with embedding operator $\Pi$),   we see that $\cK_+$
is also invariant for $(W_1, W_2, W)$ and furthermore, that, with respect to the
three-fold orthogonal decomposition \eqref{3decom} of the space $\cK'$, $(W_1, W_2, W)$ have block lower-triangular  $3 \times 3$ matrix decompositions of the form
\begin{equation}   \label{matrixWs}
W = \begin{bmatrix} * & 0 & 0 \\ * & \Pi T \Pi^* & 0 \\ * & * & * \end{bmatrix}, \quad
W_j =  \begin{bmatrix} * & 0 & 0 \\ * & \Pi T_j \Pi^* & 0 \\ * & * & * \end{bmatrix} \text{ for } j = 1,2.
\end{equation}
From these lower-triangular decompositions, we see that the space
$\cK_- \oplus \operatorname{Ran} \Pi$ is a $(W_1^*, W_2^*, W^*)$-invariant subspace.

We claim next that 
\begin{equation}   \label{claim}
\cK_- \oplus \operatorname{Ran}\Pi  =  \bigvee_{n\geq 0}W^{*n}\operatorname{Ran}\Pi  =: \cK'
\end{equation}
The containment
$$
\cK_- \oplus \operatorname{Ran}\Pi  \supset\bigvee_{n\geq 0}W^{*n}\operatorname{Ran}\Pi.
$$ 
is clear since we have seen that $\cK_- \oplus \operatorname{Ran} \Pi$ is invariant for $W^*$.
But from the definitions we have 
$$
\widetilde\cK=\bigvee_{n\in\mathbb Z}W^n\operatorname{Ran}\Pi,   \quad 
\operatorname{Ran} \Pi \oplus \cK_+ =\bigvee_{n\geq 0}W^n\operatorname{Ran}\Pi.
$$
Combining these forces us to the conclusion that in fact we recover $\cK_-  \oplus \operatorname{Ran} \Pi \subset \widetilde \cK$ as
$$
 \cK_-  \oplus \operatorname{Ran} \Pi = \bigvee_{n \le 0} W^n \operatorname{Ran} \Pi = \bigvee_{n \ge 0} W^{*n}  \operatorname{Ran} \Pi
$$
and the claim \eqref{claim} follows.    Therefore 
$$
(V_1',V_2'):=(W_1^*,W_2^*)|_{\cK'}
$$ 
is a pair of commuting isometries.  By taking adjoints of the block matrices in \eqref{matrixWs} we see that $(\Pi,  V_1',V_2')$ is an And\^o lift of 
$(T_1^*,T_2^*)$ acting on the space $\cK_- \oplus \operatorname{Ran} \Pi =  \cK'$ equal to the space for the minimal isometric lift for $T^*$,
i.e., $(\Pi, V_1',, V_2')$ is a strongly minimal And\^o lift of $(T_1^*, T_2^*)$ as wanted.
\end{proof}

Let us recall that the first part of Theorem \ref{T:sminLift=sminTuple} uses the Douglas model to obtain a criterion for $(T_1, T_2)$ to have a strongly minimal And\^o lift 
in terms of the existence of a particular kind of  And\^o tuple for  $(T_1^*, T_2^*)$,  while the second part uses the Sch\"affer model to obtain a criterion for such a 
strongly minimal And\^o lift,  but now in terms of the existence of a particular type of And\^o tuple for $(T_1, T_2)$. But by Theorem \ref{T:AdjointT} we know that the 
existence of a strongly minimal And\^o lift for $(T_1, T_2)$ is equivalent to the existence of such a lift
for $(T_1^*, T_2^*)$.  Hence we may interchange $(T_1, T_2)$ with $(T_1^*, T_2^*)$ in either part of Theorem \ref{T:sminLift=sminTuple} and still have a valid statement.
If we enhance Theorem \ref{T:sminLift=sminTuple} and succeeding theorems with these observations, we arrive at the following summary of all these results.

\begin{corollary}\label{C:s-minTuple/Lift} Let $(T_1,T_2)$ be a commuting contractive pair.
\begin{enumerate}
\item  Then the following are equivalent.
\begin{enumerate}
\item[(i)] $(T_1,T_2)$ has a strongly minimal And\^o lift.

\item[(ii)] There is a strongly minimal Type I And\^o tuple of $(T_1^*,T_2^*)$.

\item[(iii)] There is a strongly minimal strong Type II And\^o tuple of $(T_1,T_2)$.

\item[(iv)] Both  factorizations $T = T_1 \cdot T_2$ and $T = T_2 \cdot T_1$ are regular. In this case, there is a unique strongly minimal canonical special And\^o tuple of 
$(T_1,T_2)$.
\end{enumerate}

\smallskip

\item If $(T_1, T_2)$ is replaced by $(T_1^*, T_2^*)$ in the above statements,  then the corresponding statements are all mutually equivalent with 
each other and with any of the statements in part \rm{(1)} above.

\end{enumerate}
\end{corollary}

In the next Section we give yet another statement equivalent to any of the statements in Corollary \ref{C:s-minTuple/Lift} in terms of the Fundamental-Operator pair
for $(T_1, T_2)$ or $(T_1^*, T_2^*)$.

 \section[Fundamental-Operator pairs]{Strongly minimal And\^o lifts and Fundamental-Operator pairs}

Here we introduce the notion of {\em Fundamental-Operator pair} $(F_1, F_2)$ for a commuting contractive operator-pair $(T_1, T_2)$; sucj a notion 
has already been introduced and had an impact in the related theory of symmetrized-bidisk contractions \cite{B-P-SR} and tetrablock contractions \cite{Tirtha-tetrablock}.
We first need a preliminary lemma.

\begin{lemma} \label{L:FundOps}
Let $(T_1,T_2)$ be a commuting contractive pair acting on $\cH$ and $T=T_1T_2$. 
Then  a pair of operators $F_1, F_2$ on the defect space $\cD_{T}$ satisfies the pair of equations
\begin{equation}  \label{FundOps}
T_1-T_2^*T=D_T F_1D_T, \quad   T_2-T_1^*T=D_TF_2D_T.
\end{equation}
if and only if $F_1, F_2$ satisfies the pair of equations 
\begin{equation}\label{FundEqns}
D_T T_1=F_1 D_T+ F_2^* D_T T,  \quad  D_T T_2 = F_2 D_T + F_1^* D_T T.
\end{equation}
Furthermore, the solution of either pair \eqref{FundOps} or \eqref{FundEqns} is unique.
\end{lemma}

\begin{proof}
Let us suppose that the operator pair $(F_1, F_2)$ solves \eqref{FundOps}. We wish to prove that the same $(F_1, F_2)$ also
solves  \eqref{FundEqns},  Let us consider only the first equation in \eqref{FundEqns} for the moment.    Since $D_T$ is an injective, bounded operator on
$\cD_T$, the solution set of the first equation in \eqref{FundEqns} is unaffected if we multiply the equations \eqref{FundEqns} on the left by $D_T$.
In particular, multiplying  the first equation in \eqref{FundEqns} on the left by $D_T$ results in
$$
  (I - T^* T) T_1 = D_T F_1 D_T + D_T F_2^* D_T T.
$$
Now use that $(F_1, F_2)$ solves \eqref{FundOps} to eliminate $F_1$ and $F_2$ and rewrite this as
\begin{align*}
(I - T^* T) T_1 & =  (T_1 - T_2^* T) + (T_2^* - T^* T_1)T \\
& = T_1 - T^* T_1 T = (I - T^* T) T_1
\end{align*}
which is just an identity in $T_1$, $T_2$, and $T = T_1 T_2$, and  we conclude that the first equation in \eqref{FundEqns} holds.  A similar computation 
(which amounts to switching the roles of the indices $(1,2)$) verifies that second equation in \eqref{FundEqns} holds.
 We conclude that any solution of \eqref{FundOps} is also a solution of \eqref{FundEqns}.
 
 Conversely, suppose that $F_1, F_2 \in \cB(\cD_T)$ solves the system \eqref{FundEqns} and we wish to show that the same $F_1, F_2$ solves 
 the system \eqref{FundOps}.  
  Multiply both equations in \eqref{FundEqns} on the left by $D_T$ to get the pair of equations
 \begin{equation}   \label{FundEqns'}
 (I - T^* T) T_1 = D_T F_1 D_T + D_T F_2^* D_T T, \quad (I - T^* T) T_2 = D_T F_2 D_T + D_T F_1^* D_T T.
 \end{equation}
 Our goal is to solve this system for $D_T F_1 D_T$ and $D_T F_2 D_T$ and then arrive at the equations \eqref{FundOps} written in reverse order:
 \begin{equation}  \label{FundOps'}
  D_T F_1 D_T =  T_1 - T_2^* T, \quad  D_T F_2 D_T = T_2 - T_1^* T.
  \end{equation}
  We shall give the details only for the first equation as the verification of the second is completely similar.
  
  Let us take the adjoint of the second equation in \eqref{FundEqns'} and solve for $D_T F_2^* D_T$ to get
  $$
  D_T F_2^* D_T = T_2^* (I - T^* T) - T^* D_T F_1 D_T.
  $$
  Plugging this back into the first equation in  \eqref{FundEqns'} then gives us
  \begin{equation}   \label{FundEqns'1}
  (I - T^* T) T_1 = D_T F_1 D_T + T_2^* (I - T^* T) T - T^* D_T F_1 D_T T.
  \end{equation}
  We now have $D_T F_2 D_T$ eliminated and this becomes an equation for the single unknown $D_T F_1 D_T$:
 if we set 
 \begin{equation}  \label{substitutions}
 \Sigma_1 = D_T F_1 D_T,  \quad Y =   (I - T^* T) T_1 - T_2^* (I - T^* T) T,
 \end{equation}
 with $\Sigma_1$ now the unknown, then \eqref{FundEqns'1} has the form
 \begin{equation}  \label{FundEqn1}
 \Sigma_1 - T^* \Sigma_1 T = Y.
 \end{equation}
  Rewrite this as $\Sigma_1 = T^* \Sigma_1 T + Y$, plug in this expression for $\Sigma_1$ back into the right-hand side, and iterate to get 
   \begin{equation}   \label{Sigma1}
  \Sigma_1 =    T^{*N+1} \Sigma_1 T^{N+1} + \sum_{n=0}^N T^{*n} Y T^n \text{ for all } N = 0,1,2,\dots.
 \end{equation}
  If it is the case that $T^{N+1} \Sigma_1 T^{N+1} $ tends to $0$ (say in the weak operator topology), then  we can take limits on both sides of \eqref{Sigma1}
 to arrive at a formula for $\Sigma_1$:
 $$
   \Sigma_1 = \sum_{N=0}^\infty T^{*n} Y T^n.
 $$
 
 To analyze this further, let us recall the precise formulas \eqref{substitutions} for what $\Sigma_1$ and $Y$ are for our case here.
 Thus \eqref{Sigma1} specializes to
\begin{equation}   \label{specialize}
  D_T   F_1 D_T = T^{* N+1} D_T F_1 D_T T^{N+1} + \sum_{n=0}^N T^{*n} \big( (I - T^* T) T_1 - T_2^* (I - T^* T) T \big) T^n. 
 \end{equation}
 Let us note next that, for each $h \in \cH$, the following series is telescoping:
  \begin{align*}
&  \sum_{n=0}^N  \|D_T T^n h \|^2   = \sum_{n=0}^N \langle T^{*n} (I - T^* T) T^n h, h \rangle
 = \sum_{n=0}^N \big(  \| T^n h \|^2  - \| T^{n+1} h \|^2 \big) \\
 &  = \| h \|^2 - \| T^{N+1} h \|^2.
 \end{align*}
 As $T$ is a contraction,  $\| T^{N+1} h \|^2$ is decreasing and bounded below by zero and hence convergent.   
In particular we see that the series $\sum_{n=0}^N  \|D_T T^n h \|^2$ is convergent for each $h \in \cH$.  By the $n$-th term test it follows that
 $\lim_{n \to \infty} \| D_T T^n h \|^2   = 0$. It follows that the operator sequence  $T^{* N+1} D_T F_1 D_T T^{N+1}$ certainly converges to $0$
 in the weak operator topology. 
 We conclude from \eqref{specialize} that we have solved for $D_T F_1 D_T$:
 \begin{equation}   \label{F-sol}
 D_T F_1 D_T = \sum_{n=0}^\infty T^{*n}  \big( (I - T^* T) T_1 - T_2^* (I - T^* T) T \big) T^n
 \end{equation}
 with the infinite series converging in the weak operator topology.
 
 It remains only to show that the sum of this series is actually equal to $T_1 - T_2^* T$, i.e., that $F_1$ satisfies the first of equations \eqref{FundOps}.
 Let us recall from Section \ref{S:flats} that $\lim_{n \to \infty}T^{*n} T^n $ exists in the strong  operator topology with limit denoted as $Q_T^2$
 (with $Q_T$ then set equal to the unique positive semi-definite square root of $Q_T^2$) and moreover, given that $T = T_1 \cdot T_2$ is a commuting, contractive
 factorization of $T$, we then also have the three identities
  \begin{equation} \label{QT-Isoms}
   Q_T^2 = T^* Q_T^2 T, \quad Q_T^2 = T_1^* Q_T^2 T_1, \quad  Q_T^2 = T_2^* Q_T^2 T_2.
   \end{equation}
   Moreover, due to the telescoping property of the sequence of partial sums, we see that
   $$
       \sum_{n=0}^N T^{*n} (I - T^* T ) T^n = I - T^{*N+1} T^{N+1}
   $$
   and hence also 
   $$ 
   \sum_{n=0}^\infty T^{*n} (I - T^* T) T^n = I - Q_T^2
   $$
   with convergence in the strong operator topology.
  Let us rewrite \eqref{F-sol} as 
  \begin{equation} \label{specialize'}
   D_T   F_1 D_T = \sum_{n=0}^\infty T^{*n}  (I - T^* T) T_1T^n -  \sum_{n=0}^\infty T^{*n} T_2^* (I - T^* T) T  T^n.  
   \end{equation}
   The first term on the right-hand side of \eqref{specialize'}  is
   $$
    \sum_{n=0}^\infty T^{*n} (I - T^* T) T_1 T^n = \left( \sum_{n=0}^\infty T^{*n} (I - T^* T) T^n \right) T_1 = (I - Q_T^2)T_1
    $$
    while the second term (without the minus sign) on the right-hand side is
   \begin{align*}
    \sum_{n=0}^\infty T^{*n}T_2^* (I - T^* T) T  T^n & = 
    T_2^* \left( \sum_{n=0}^\infty T^{*n} (I - T^* T) T^n \right) T  \\
    & =  T_2^* \left( \sum_{n=0}^\infty T^{*n} ( I - T^* T) T^n \right) T  =  T_2^* (I - Q_T^2 ) T .
  \end{align*}
  Collecting terms and recalling \eqref{QT-Isoms} then gives
  \begin{align*}
 D_T F_1 D_T & = (I   - Q_T^2) T_1 - T_2^* (I - Q_T^2)T = T_1 - Q_T ^2T_1 - T_2^* T  + T_2^* Q_T^2T_2 T_1 \\
  & =  T_1 - Q_T^2 T_1 - T_2^* T + Q_T^2 T_1 = T_1 - T_2^* T
 \end{align*}
 as required.

It remains to show that solutions of \eqref{FundOps} (or equivalently of \eqref{FundEqns}) are unique whenever they exist.  
As for \eqref{FundEqns}, uniqueness is immediate from the fact that $F_1$ and $F_2$ are taken to be operators on $\cD_T$ and
the fact that $D_T|_{\cD_T}$ is an injective bounded operator on $\cD_T$ with dense range.  It is possible to give a direct proof of 
the uniqueness of solutions of \eqref{FundEqns} by showing that the only solution of the homogeneous equation is the zero solution; however,  this proof
is rather elaborate (much like the proof of the existence of a solution for \eqref{FundEqns}).  A much shorter proof is to note that this uniqueness
follows immediately from the equivalence between solutions of
\eqref{FundOps} and \eqref{FundEqns} together with the uniqueness of solutions of \eqref{FundOps} already observed.
\end{proof}

Solutions of \eqref{FundOps} are fundamental for later developments, so we now formally give them a name.  Existence of such solutions
for any commuting contractive pair will be shown shortly. 

\begin{definition}  \label{D:FundOps}\index{fundamental operator}
Given a commuting contractive operator-pair $(T_1,T_2)$,  the unique solution pair $(F_1,F_2)$ in $\cB(\cD_{T})$ (where $T = T_1 T_2$) of the pair of operator equations 
\eqref{FundOps} is called the {\em Fundamental-Operator pair} of $(T_1,T_2)$.
\end{definition}

We are now ready to prove the existence of a Fundamental-Operator pair for a commuting contractive operator-pair $(T_1, T_2)$.
 We note that this theorem is already proved in \cite[Theorem 3.2]{BS-Douglas}, where it was shown that the result is actually true for 
a tuple of any finite number of commuting contraction operators.  However, the proof there appeals to the parallel result for the case of $\Gamma$-contractions
(a commuting operator pair $(T_1, T_2)$ having the symmetrized bidisk as a spectral set) whereas here we give three direct proofs for the setting of
a commuting contractive pair.
 We take full advantage of the 2-dimensional dilation theory (i.e., the existence of an And\^o lift for a commuting contractive pair) to arrive at the new proofs.

\begin{theorem}\label{T:FundOps}
Let $(T_1,T_2)$ be a commuting contractive pair acting on $\cH$ and $T=T_1T_2$.   Then a Fundamental-Operator pair of $(T_1, T_2)$ in the sense of
Definition \ref{D:FundOps} exists, i.e., 
 there exists a unique pair of contraction operators $F_1,F_2\in\cB(\cD_T)$ satisfying \eqref{FundOps} (or equivalently, by Lemma \ref{L:FundOps}, \eqref{FundEqns}).
\end{theorem}

\begin{proof} We shall give three proofs of this result.

\smallskip

\noindent
{\sf First Proof via And\^o's Theorem:}
We first consider two illustrative special cases, and then use the second special case to prove the result for the general case.

\smallskip

\noindent
\textbf{Case 1. $(T_1, T_2) = (V_1, V_2)$ is a commuting isometric pair.}
Note that in this case \eqref{FundOps} is obviously true since both sides of \eqref{FundOps} are actually zero. 

\smallskip

\noindent
\textbf{Case 2. $(T_1, T_2) = (V_1^*, V_2^*)$ is a commuting co-isometric pair.}
To handle this case,  one can apply the Berger-Coburn-Lebow model for the isometric pair $(V_1, V_2)$ and then directly compute.  It suffices to
assume that the product $V = V_1 V_2$ is a shift, since the unitary part washes out when computing the operators $V_1 - V_2^*V$, $V_2 - V_1 V^*$
as well as the defect operators $D_V$, $D_{V_1}$, $D_{V_2}$ as seen from Case 1.   Recall the notation $\bev_{0, \cF}$ \eqref{bev0cF} for the operator
of evaluation-at-0 from $H^2(\cF)$ to $\cF$.

If we use the BCL1 model 
$$
(V_1, V_2) = (I_{H^2} \otimes P^\perp U + M_z \otimes P U, \, I_{H^2} \otimes U^* P  +  M_z \otimes U^* P^\perp) \text{ on } H^2(\cF),
$$
one sees that
$$
 V_1^* - V_2 V^* = (I_{H^2} - M_z M_z^*) \otimes U^* P^\perp    = D_{V^*}  \bev_{0, \cF}^* U^* P^\perp  \bev_{0, \cF} D_{V^*}.
$$
while 
$$
V_2^* - V_1 V^* =  (I_{H^2} - M_z M_z^*)\otimes PU  = D_{V^*}  \bev_{0,\cF}^*PU  \bev_{0, \cF}  D_{V^*}
$$
leading to the operators
\begin{equation}  \label{fund1}
   F_1 =  \bev_{0, \cF}^*U^* P^\perp  \bev_{0, \cF}|_{\cD_{V^*}},  \quad F_2 =  \bev_{0, \cF}^* PU \bev_{0, \cF}|_{\cD_{V^*}} \in \cB(\cD_{V*})
 \end{equation}
 being the unique solutions of the Fundamental-Operator equations \eqref{FundOps}.
  However, if we use the BCL2 model
  $$
  (V_1, V_2) = (I_{H^2}\otimes U^* P^\perp + M_z\otimes U^* P , \, I_{H^2}\otimes PU+ M_z\otimes P^\perp U) \text{ on } H^2(\cF),
 $$ 
 one gets
 $$
   V_1^* - V_2 V^* =  (I_{H^2} - M_z M_z^*) \otimes P^\perp U  = D_{V^*} \bev_{0,\cF}^* P^\perp U \bev_{0, \cF} D_{V^*}  
 $$
 while
 $$
   V_2^* - V_1 V^* = (I_{H^2} - M_z M_z^*)\otimes U^* P = D_{V^*}  \bev_{0, \cF}^* U^* P \bev_{0, \cF} D_{V^*}
   $$
 leading to unique solutions $F_1$, $F_2$ of the fundamental equations \eqref{FundOps} for this case being given by
 \begin{equation}  \label{fund2}
 F_1 =  \bev_{0, \cF}^* P^\perp U \bev_{0, \cF} |_{\cD_{V^*}}, \quad F_2 =  \bev_{0, \cF}^* U^*P \bev_{0, \cF} |_{\cD_{V^*}} \in \cB(\cD_{V^*}).
 \end{equation}
 Note that one can go from \eqref{fund1} to \eqref{fund2} by replacing $(P,U)$ in \eqref{fund1} by its flipped version $(P^\ff, U^\ff)
 = (U^*P U, U^*)$ to arrive at the version \eqref{fund2} for the fundamental operators.
 
 To avoid this phenomenon of the formula for the Fundamental-Operator pair $(F_1, F_2)$ depending on the choice of representation of the
 coisometric pair $(V_1^*, V_2^*)$,  we can apply the canonical version Theorem \ref{T:BCLcanonical} of the BCL model by expressing
 the fundamental operator pair $(F_1, F_2)$ directly in terms of $(V_1, V_2)$ as follows.  By part (4) of Theorem \ref{T:BCLcanonical}, a BCL1 tuple
 for $(V_1, V_2)$ is $(\cD_{V^*}, P, U)$ with
 \begin{equation}   \label{BCL-1PU}
    P = D_{V_1^*}|_{\cD_{V^*}}, \quad U =  (V_1 D_{V_2^*} + D_{V_1^*} V_2^*)|_{\cD_{V^*}}.
\end{equation}
Let us also observe here  that when we take $\cF = \cD_{V^*}$, then the operator $\bev_{0, \cD_{V^*}}$ acting from $H^2(\cD_{V^*})$ to $\cD_{V^*}$ 
when restricted to $\cD_{V^*}$ amounts to the identity operator:
$$
\bev_{0, \cD_{V^*}}|_{\cD_{V^*}} = I_{\cD_{V^*}}.
$$
Plugging the values for $(P, U)$ given by  \eqref{BCL-1PU} into the the expressions \eqref{fund1} for the corresponding Fundamental Operators and noting that 
$P^\perp = V_1 D_{V_2^*} V_1^*|_{\cD_{V^*}}$ then gives us expressions for $F_1$ and $F_2$ directly in terms of $(V_1, V_2)$:
$$
F_1 = \bev_{0, \cD_{V^*}}^* (D_{V_2^*} V_1^* + V_2 D_{V_1^*}) (V_1 D_{V_2^*} V_1^*) \bev_{0, \cD_{V^*}}|_{\cD_{V^*}}
=  D_{V_2^*} V_1^* |_{\cD_{V^*}}
$$
 while
 $$
 F_2 = \bev_{0, \cD_{V^*}}^* D_{V_1^*} (V_1 D_{V_2^*} + D_{V_1^*} V_2^*)  \bev_{0, \cF} |_{\cD_{V^*}} = D_{V_1^*} V_2^* |_{\cD_{V^*}}
 $$
 arriving at the formulas
 \begin{equation}  \label{fund1'}
 F_1 =  D_{V_2^*} V_1^* |_{\cD_{V^*}}, \quad F_2 = D_{V_1^*} V_2^*|_{\cD_{V^*}}
 \end{equation}
 We leave it to the reader to verify that,  if we instead use the BCL2 model for $(V_1, V_2)$ to get the expressions  \eqref{fund2}
 for the Fundamental Operators, and then plug into these expressions the canonical values $(P, U)$ in\eqref{BCL2canonical}  for the BCL2 model for
 the commuting isometric pair $(V_1, V_2)$
 \begin{equation}   \label{BCL-2PU}
   P =      V_2 D_{V_1^*} V_2|_{\cD_{V^*}},      \quad U = (D_{V_2^*} V_1^* + V_2 D_{V_1^*} V_2^*)|_{\cD_{V^*}},
 \end{equation}
 then the resulting expression for $(F_1, F_2)$ expressed directly in terms of $(V_1, V_2)$ turns out to be exactly the same as in \eqref{fund1'}.
 Alternatively, once one  identifies the candidate \eqref{fund1'}, one can compute directly that it works:
 \begin{equation}  \label{ToShow}
 V_1^*  - V_2 V^* = D_{V^*} D_{V_2^*} V_1^* D_{V^*}, \quad V_2^* - V_1 V^* = D_{V^*} D_{V_1^*} V_2^* D_{V^*}.
  \end{equation}
It suffices to verify the first equation as then the second follows by interchanging the roles of the indices $1,2$.
Note first that
$$
  V_1^* - V_2 V^* = V_1^* - V_2 V_2^* V_1^* = D_{V_2^*} V_1^*.
$$
Furthermore 
$$
V^* D_{V_2^*} V_1^* =  V_1^* V_2^* (I - V_2 V_2^*) V_1^* = 0
$$
and we conclude that 
$$
D_{V^*} D_{V_2^*} V_1^* = D_{V_2^*} V_1^*.
$$
Similarly  
$$
  D_{V_2^*} V_1^* V V^* = D_{V_2^*} V_1^* V_1 V_2V^* = (I - V_2 V_2^*) V_2 V^* = 0
$$
and hence we also have
$$
 D_{V_2^*} V_1^* D_V = D_{V_2^*} V_1^* (I - V^* V) = D_{V_2^*} V_1^*.
$$
Putting all the pieces together we get the first of equations  \eqref{ToShow} as expected.
\smallskip

\noindent
\textbf{Case 3:  The general case.} 
Now let $(T_1,T_2)$ be any commuting contractive pair acting on a Hilbert space $\cH$. By And\^o's theorem, we know that 
$(T_1^*, T_2^*)$ has an And\^o lift $(\Pi, V_1,V_2)$ with  $\Pi \colon \cH \to \cK$ and $V_1$, $V_2$ acting on $\cK$. For notational convenience, we may 
assume $\Pi=\sbm{I_\cH \\ 0}$, i.e. $\cH \subset \cK$.   Recall that the lifting property can be reformulated as 
\begin{equation}   \label{Id-0}
V_j^*|_\cH = T_j \text{ for } j=1,2,  \text{ and hence also } V^*|_\cH = T.
\end{equation}  
 where we set $V = V_1 V_2 = V_2 V_1$ and $T = T_1 T_2 = T_2 T_2$.   
 
 For $h \in \cH$ we then have
 $$
\|D_{V^*}h\|^2=\langle h,h \rangle - \langle V^* h, V^*h \rangle = \langle h,h \rangle - \langle T h,Th \rangle = \|D_Th\|^2.
$$
We conclude that the map $\Lambda \colon \cD_T \to \cD_{V^*}$ defined densely by
\begin{align}\label{Lambda}
  \Lambda \colon D_T h \mapsto D_{V^*} h
\end{align}
is an isometry.  From this definition, we can also write 
\begin{equation}  \label{Id-1}
D_{V^*}|_\cH = \Lambda D_T.
\end{equation}
Taking adjoints then gives
\begin{equation}  \label{Id-2}
P_\cH D_{V^*} = D_T \Lambda^*.
\end{equation}
From Case 2 above we know that there are fundamental operators  $(F^\bV_1$, $F^\bV_2)$ for the commuting coisometric pair $(V_1^*, V_2^*)$ for the
record by \eqref{fund1'} given by
$$  
F^\bV_1 = D_{V_2^*} V_1^*|_{\cD_{V^*}}, \quad F^\bV_2 = D_{V_1^*} V_2^*|_{\cD_{V^*}}
$$
which by definition satisfies the equations
$$
V_1^* - V_2 V^* = D_{V^*} F^\bV_1 D_{V^*}, \quad V_2^* - V_1 V^*  = D_{V^*} F^\bV_2 D_{V^*}.
$$
Let us compress both sides of each of these equations to the subspace $\cH \subset \cK$ to get
\begin{equation}   \label{id-prelim}
P_\cH (V_1^* - V_2 V^*) |_\cH = P_\cH D_{V^*} F^\bV_1 D_{V^*}|_\cH, \quad
   P_\cH( V_2^* - V_1 V^*)|_\cH   = P_\cH D_{V^*} F^\bV_2 D_{V*}|_\cH.
\end{equation}
Making use of the identities \eqref{Id-0} then leads us to 
$$
P_\cH (V_1^* - V_2 V^*) |_\cH = T_1 - T_2^* T, \quad  P_\cH( V_2^* - V_1 V^*)|_\cH  = T_2 - T_1^* T.
$$
On the other hand, the identities \eqref{Id-1} and \eqref{Id-2} gives us
$$
P_\cH D_{V^*} F^\bV_1 D_{V^*}|_\cH = D_T \Lambda^* F_1^\bV \Lambda D_T, \quad 
P_\cH D_{V^*} F^\bV_2 D_{V^*}|_\cH = D_T \Lambda^* F_2^\bV \Lambda D_T.
$$
Plugging these last two collections of identities back into \eqref{id-prelim} then gives us
$$
  T_1 - T_2^* T = D_T \Lambda^* F_1^\bV \Lambda D_T, \quad
  T_2 - T_1^* T = D_T \Lambda^* F_2^\bV \Lambda D_T
$$
and we conclude that the operator pair
$$
  F_1 = \Lambda^* F_1^\bV \Lambda, \quad F_2 = \Lambda^* F_2^\bV \Lambda
 $$
 serves as the fundamental-operator pair for the commuting, contractive pair $(T_1, T_2)$ as desired.  
 
\smallskip

\noindent
{\sf Second Proof via Type I And\^o Tuples:} Here we prove:
\begin{itemize}
\item {\em If $(\cF,\Lambda,P,U)$ is a Type I And\^o tuple for $(T_1,T_2)$, then the fundamental operator pair  $(F_1,F_2)$ for $(T_1, T_2)$ can be given by}
\begin{equation}  \label{choices}
F_1 =\Lambda^* P^\perp U\Lambda,  \quad F_2 =  \Lambda^* U^* P \Lambda.
\end{equation}
\end{itemize}  By Lemma \ref{L:FundOps},  to find the fundamental operator pair, it suffices to find a pair of operators
$F_1, F_2 \in \cB(\cD_T)$ so that $(F_1, F_2)$ solves \eqref{FundEqns} rather than \eqref{FundOps}. To see this, note that by Definition \ref{D:AndoTupleI}, $(\cF,\Lambda,P,U)$  being a Type I And\^o tuple for $(T_1,T_2)$ means that
\begin{equation} \label{TypeIAndo}
P^\perp U  \Lambda  D_{T} +   P U \Lambda D_{T} T  = \Lambda D_{T} T_1, \quad U^* P \Lambda D_{T} + U^* P^\perp \Lambda D_{T} T = \Lambda D_{T} T_2.
\end{equation}
Multiply each of these equations on the left by $\Lambda^*$  to get the equations \eqref{FundEqns} with $F_1,F_2$ as claimed.

\smallskip

\noindent
{\sf Third Proof via Type II And\^o Tuples:} As in the Second Proof, we actually prove the following assertion:
\begin{itemize}
\item {\em If $(\cF,\Lambda,P,U)$ is a strong Type II And\^o tuple for a commuting contractive pair $(T_1,T_2)$, then the fundamental operator pair  $(F_1,F_2)$ for $(T_1, T_2)$ is given as in \eqref{choices}.}
\end{itemize}
In the first step of the following computation we make use of the two expressions for $\Lambda D_T$ given by condition
(i$'$) in Definition \ref{AndoTuple}:
\begin{align*}
& D_T \Lambda^* P^\perp U \Lambda D_T= (\Lambda D_T)^* P^\perp U (\Lambda D_T)  \\
& \quad =
(PU\Lambda D_TT_2+P^\perp \Lambda D_T)^*P^\perp U(U^*P^\perp\Lambda D_TT_1+U^*PU\Lambda D_T)\\
& \quad =(T_2^*D_T\Lambda^*U^*P+D_T\Lambda^*P^\perp)P^\perp (P^\perp\Lambda D_TT_1+PU\Lambda D_T)\\
& \quad = D_T\Lambda^*P^\perp \Lambda D_TT_1\\
&\quad =D_{T_2}^2T_1 \text{(by the second equation in Definition \ref{AndoTuple} (ii))}\\
&\quad =T_1-T_2^*T.
\end{align*}
Similarly, in the first step of the next computation, we use the two expressions for $\Lambda D_T$  given by
condition (i$'$) in Definition \ref{AndoTuple} but in reverse order:
\begin{align*}
& D_T \Lambda^* U^* P \Lambda D_T = (\Lambda D_T)^* U^* P (\Lambda D_T) \\
& \quad =(U^*P^\perp\Lambda D_TT_1+U^*PU\Lambda D_T)^*U^*P(PU\Lambda D_TT_2+P^\perp \Lambda D_T)\\
& \quad =(T_1^*D_T\Lambda^*P^\perp+D_T\Lambda^* U^*P)P(PU\Lambda D_TT_2+P^\perp \Lambda D_T)\\
& \quad =D_T\Lambda^* U^*PU\Lambda D_TT_2\\
& \quad =D_{T_1}^2T_2 \text{ (by the first equation in Definition \ref{AndoTuple} (ii) )}\\
&\quad =T_2-T_1^*T.
\end{align*}
This establishes the claim and completes the proof of Theorem \ref{T:FundOps}.
\end{proof}

The following theorem gives a characterization of existence of a strongly minimal And\^o lift for a commuting contractive pair $(T_1, T_2)$ in terms of the
fundamental-operator pair $(F_1, F_2)$ for $(T_1, T_2)$.  The condition \eqref{strong-fund} on the Fundamental-Operator pair $(F_1, F_2)$ for $(T_1, T_2)$
is yet another equivalent condition that one can add to the list of equivalent conditions in Corollary \ref{C:s-minTuple/Lift}.

\begin{theorem}\label{T:StrongFund}
Let $(T_1,T_2)$ be a commuting contractive pair acting on $\cH$. Then $(T_1,T_2)$ has a strongly minimal And\^o lift if and only if its Fundamental-Operator pair
$(F_1,F_2)$ satisfies the system of equations
\begin{align}\label{strong-fund}
F_1 F_2=0=F_2 F_1,\quad F_1^*F_1+F_2F_2^*=I_{\cD_T} =F_1 F_1^*+ F_2^* F_2.
\end{align}\end{theorem}

\begin{proof} Suppose that $(T_1,T_2)$ has a strongly minimal And\^o lift. By Part (2) of Theorem \ref{T:sminLift=sminTuple}, 
$(T_1,T_2)$ has a strongly minimal strong Type II And\^o tuple, say $(\cF,\Lambda,P,U)$. By the Third Proof of Theorem \ref{T:FundOps},
$(F_1,F_2)=(\Lambda^*P^\perp U\Lambda,\Lambda^*U^*P\Lambda)$ is the fundamental operator pair for $(T_1,T_2)$. Since $\Lambda$ here is a unitary, it is now a 
matter of easy computation to check that $F_1,F_2$ satisfy equations \eqref{strong-fund}.

Conversely, suppose $F_1,F_2$ satisfy equations \eqref{strong-fund}. Apply Lemma \ref{relations-of-E-lem} to $(F_1^*,F_2^*)$ to get a projection $P$ 
and a unitary $U$ on $\cD_T$ such that $(F_1,F_2)=(P^\perp U, U^*P)$. Since $(F_1,F_2)$ is the fundamental operator pair for $(T_1,T_2)$, we have
\begin{align}
\label{Put2gether1}
T_1-T_2^*T=D_TP^\perp U D_T, &\quad T_2-T_1^*T=D_TU^*PD_T,  \\ 
 \label{Put2gether2} D_TT_1=P^\perp UD_T+PUD_TT, &\quad D_TT_2= U^*PD_T+U^*P^\perp D_TT.
\end{align}
We now unfold these equations to show that $(\cD_T, I_{\cD_T}, P, U)$ is a (strongly minimal) strong Type II And\^o tuple for $(T_1,T_2)$. Then by Part (2) of 
Theorem \ref{T:sminLift=sminTuple} we will be done. Multiply both sides of the second equation in \eqref{Put2gether2} on the left by $D_T PU$ to get
$$
D_T P U D_T T_2 = D_T P D_T.
$$
Combining this with the adjoint of the second equation in \eqref{Put2gether1} then gives us
\begin{align}
\notag
D_T P^\perp D_T & = D_T^2 -D_TPD_T = D_T^2 - D_T PU D_T T_2 \\
&= (I-T^*T)-(T_2^*-T^*T_1)T_2= D_{T_2}^2. \label{IsoCond1}
\end{align}
This is the second isometry condition in Definition \ref{AndoTuple} (here $\Lambda=I_{\cD_T}$).

 For the other isometry condition, we multiply both sides of the first equation in \eqref{Put2gether2} on the left by 
$D_TU^*P^\perp$  to get
$$
 D_T U^* P^\perp D_T T_1 = D_T U^* P^\perp U D_T.
 $$
 Combining this with the adjoint of the second equation in \eqref{Put2gether1} then gives us
 \begin{align}
\notag
D_TU^*PUD_T & = D_T^2-D_TU^*P^\perp UD_T =D_T^2 - D_TU^*P^\perp D_TT_1\\
&= I - T^* T - (T_1^*-T^*T_2)T_1=D_{T_1}^2.\label{IsoCond2}
\end{align}
By Definition \ref{AndoTuple}, it just remains to show that
$$
D_T=PU D_TT_2+P^\perp D_T=U^*P^\perp D_TT_1+U^*PU D_T.
$$
Since this is an operator equation form $\cD_T$ into $\cD_T$ and $D_T$ is injective on $\cD_T$, the above will hold if and only if
$$
D_T^2=D_TPU D_TT_2+D_TP^\perp D_T=D_TU^*P^\perp D_TT_1+D_TU^*PU D_T
$$holds. In view of equations \eqref{Put2gether1}, \eqref{IsoCond1} and \eqref{IsoCond2}, the above equations boil down to the operator equations
$$
D_T^2=T_2^*(I-T_1^*T_1)T_2+ (I-T_2^*T_2)=T_1^*(I-T_2^*T_2)T_1+ (I-T_1^*T_1),
$$
which is true as observed before in \eqref{id1}. Consequently, $(\cD_T, I_{\cD_T}, P, U)$ is a strongly minimal strong Type II And\^o tuple for $(T_1,T_2)$ and therefore by Part (2) of Theorem \ref{T:sminLift=sminTuple}, $(T_1,T_2)$ has a strongly minimal And\^o lift. This completes the proof.
\end{proof}

\begin{remark}  \label{R:s-min}  We have noted in Examples \ref{E:Dex} and \ref{E:Sex} that a commuting isometric pair $(V_1, V_2)$ as well as a commuting co-isometric
pair $(V_1^*, V_2^*)$ has a strongly minimal And\^o lift.  As a further exercise, we now verify how this can also be seen as an application of Theorem 
\ref{T:StrongFund}.

We have seen in the course of the First Proof of Theorem \ref{T:FundOps} that a commuting pair of isometries $(V_1, V_2)$
has a trivial fundamental operator pair $(F_1, F_2) = (0,0)$ acting on the zero space $\cD_{V^*}$ while a commuting pair of coisometries $(V_1^*, V_2^*)$
has fundamental-operator pair $(F_1, F_2)$ given explicitly by
\begin{equation}   \label{coisomFO}
  F_1 = D_{V_2^*} V_{1^*} |_{\cD_{V^*},} \quad F_2  = D_{V_2^*} V_1^*|_{\cD_{V^*}}.
\end{equation}

For the case of a commuting isometric pair $(V_1, V_2)$ we conclude that condition \eqref{strong-fund} holds trivially and hence (by Theorem \ref{T:StrongFund})
$(V_1, V_2)$ has a strongly minimal isometric lift (namely, itself).  

For the case of a commuting coisometric pair $(V_1^*, V_2^*)$, one can check directly that
$(F_1, F_2)$  given by \eqref{coisomFO} satisfies \eqref{strong-fund}: e.g.
\begin{align*}
F_1 F_2 & = D_{V_2^*} V_1^* D_{V_1^*} V_2^*|_{\cD_{V^*}}  = (I - V_2 V_2^*) V_1^* (I - V_1 V_1^*) V_2^*|_{\cD_{V^*}}  \\
& =  (I- V_2 V_2^*) V_1^* (I - V_1 V_1^*) V_2^*|_{\cD_{V^*}} = 0 \text{ since } V_1^* (I - V_1 V_1^*) = 0
\end{align*}
and similarly
$$
  F_2 F_1 = 0.
$$
Furthermore
\begin{align*}
& F_1^* F_1 + F_2 F_2^*  = 
\big( V_1 (I - V_2 V_2^*) V_1^* + (I - V_1 V_1^*) V_2^* V_2 (I - V_1 V_1^*) \big)|_{\cD_{V^*}} \\
&  \quad = \big( (V_1 V_1^* - V V^*) + (I - V_1 V_1^*) \big)|_{\cD_{V^*}} = (I - V V^*)|_{\cD_{V^*}} = I_{\cD_{V^*}}
\end{align*}
and similarly
$$
 F_1 F_1^* + F_2^* F_2 = I_{\cD_{V^*}}.
 $$
\end{remark}

As of this writing, we do not have an example of a Type I And\^o tuple which is not special or of a strong Type II And\^o tuple which is not special.  In the next
chapter, we produce an example of a Type II And\^o tuple which is not strong Type II (and hence not special):  see Proposition \ref{P:NonCanTuple}.
There we shall also see that whenever $(T_1, T_2)$ has at least one strongly minimal Type I or strongly minimal strong Type II And\^o tuple,
then any minimal Type I or  minimal strong Type II And\^o tuple is actually strongly minimal and coincides with the unique strongly minimal canonical-form 
special And\^o tuple (see Theorem \ref{T:strongly-minimal} and Corollary \ref{C:strongly-minimal}.
The next result shows that  within the category of strongly minimal And\^o tuples,  any Type I or strong Type II And\^o tuple is in fact special. 
We include this result here as it illustrates how use of the Fundamental-Operator pair leads to explicit formulas.

\begin{theorem} \label{T:TypeI/TypeII-special}
Every strongly minimal Type I or  strongly minimal strong Type II And\^o tuple of a commuting contractive pair coincides with a special And\^o tuple.
\end{theorem}

\begin{proof}
Let $(T_1,T_2)$ be a commuting contractive pair and $(\cF,\Lambda, P,U)$ be a strongly minimal Type I And\^o tuple of $(T_1,T_2)$. Then by Part (1) of 
Theorem \ref{T:sminLift=sminTuple} and Theorem \ref{T:AdjointT}, $(T_1,T_2)$ has a strongly minimal And\^o lift. By Theorem \ref{T:SMinReg}, there is a strongly minimal special 
And\^o tuple of $(T_1,T_2)$, call it $(\cF_\dag,\Lambda_\dag,P_\dag,U_\dag)$; note that $\cF_\dag$ here is just the space $\cD_{T_1} \oplus  \cD_{T_2}$. By 
Definition   \ref{D:preAndoTuple}, we will be done if we can find a unitary $\tau:\cF_\dag\to\cF$ such that 
\begin{align}\label{ToShow1}
\tau\cdot \Lambda_\dag = \Lambda \quad\mbox{and}\quad
(P,U)=(\tau P_\dag\tau^*,\tau U_\dag\tau^*).
\end{align} Set $\tau:=\Lambda\Lambda_\dag^*:\cF_\dag\to\cF$. This is a unitary because both $\Lambda_\dag$ and $\Lambda$ are unitary. First, note that $\Lambda = \Lambda\Lambda_\dag^* \cdot \Lambda_\dag$. Therefore the first equation in \eqref{ToShow1} is achieved. Second, since a special And\^o tuple is of Type I (see Theorem \ref{Thm:special}), by the Second Proof of Theorem \ref{T:FundOps}, both the pairs
\begin{align*}
(\Lambda^*P^\perp U\Lambda, \Lambda^* U^*P \Lambda)\quad \mbox{and}\quad (\Lambda_\dag^*P_\dag^\perp U_\dag\Lambda_\dag, \Lambda_\dag^* U_\dag^*P_\dag \Lambda_\dag)
\end{align*}are the fundamental operator pair for $(T_1,T_2)$. Since the fundamental operators are unique, we must have
\begin{align}\label{SameFundma}
\Lambda^*P^\perp U\Lambda=\Lambda_\dag^*P_\dag^\perp U_\dag\Lambda_\dag \quad \mbox{and}\quad \Lambda^* U^*P \Lambda=\Lambda_\dag^* U_\dag^*P_\dag \Lambda_\dag.
\end{align}
Adding the first of these two equations with the adjoint of the other, we get
\begin{align*}
\Lambda^* U \Lambda =\Lambda_\dag^* U_\dag \Lambda_\dag \quad\mbox{or, equivalently,}\quad U = \Lambda\Lambda_\dag^* \cdot U_\dag \cdot\Lambda_\dag\Lambda^*.
\end{align*}
Using this expression of $U$ in the second equation of \eqref{SameFundma} and simplifying we get
\begin{align*}
P = \Lambda\Lambda_\dag^*\cdot P_\dag\cdot \Lambda_\dag \Lambda^*.
\end{align*}
Therefore the second set of equations in \eqref{ToShow1} is also established. This shows that the strongly minimal Type I And\^o tuple of $(T_1,T_2)$ coincides with the special strongly minimal And\^o tuple $(\cF_\dag,\Lambda_\dag,P_\dag,U_\dag)$ via the unitary $\Lambda\Lambda_\dag^*: \cF_\dag\to \cF$. 

Via a similar analysis using Theorem \ref{Thm:special-canonical} and the Third Proof of Theorem \ref{T:FundOps}, one can show that every strongly minimal strong Type II And\^o tuple coincides with a special strongly minimal And\^o tuple.
\end{proof}

The following result shows how one can recover And\^o tuples (up to coincidence) for a commuting contractive pair $(T_1, T_2)$ from the Fundamental-Operator pair 
$(F_1, F_2)$ for $(T_1, T_2)$, at least for the case where $(T_1, T_2)$ has a strongly minimal And\^o dilation.

\begin{proposition}   \label{P:FO-tuple}
Let $(T_1,T_2)$ be a commuting contractive pair and $(F_1,F_2)$ be its fundamental-operator pair. Then every strongly minimal strong Type II And\^o tuple of 
$(T_1,T_2)$ coincides with $(\cD_T,I_{\cD_T}, F_2^*F_2, F_2^*+F_1)$. 
The same assertion holds for a strongly minimal Type I And\^o tuple of $(T_1^*,T_2^*)$ as well.
\end{proposition}

\begin{proof}
Let $(\cF,\Lambda,P,U)$ be a strongly minimal strong Type II And\^o tuple of $(T_1,T_2)$. Since $\Lambda$ is a unitary and we are interested in the 
coincidence envelope of strongly minimal And\^o tuples, without loss of generality we can suppose that the And\^o tuple is 
$(\Lambda^*\cF, \, \Lambda^*\Lambda,\, \Lambda^* P\Lambda, \, \Lambda^* U\Lambda)=(\cD_T, \, I_{\cD_T}, \, P, \, U)$. By the Third Proof of Theorem \ref{T:FundOps}, the fundamental operators for $(T_1,T_2)$ then are $(F_1,F_2)=(P^\perp U, U^*P)$. This readily implies that $F_1+F_2^*=U$ and hence $P=(F_1+F_2^*)F_2=F_2^*F_2$. Similarly, using the Second Proof of Theorem \ref{T:FundOps}, one can prove the assertion for a strongly minimal Type I And\^o tuple of $(T_1^*,T_2^*)$.
\end{proof}

\section{Appendix: examples}  \label{S:Appendix}

\smallskip

\noindent
\textbf{1. Example of a special And\^o tuple which is not a Type I$'$ And\^o tuple.}

\smallskip
Recall from Remark \ref{R:reasonsD} that a pre-And\^o tuple $(\cF, \Lambda, P, U)$ for the commuting contractive pair $(T_1^*, T_2^*)$ is said to be a Type I$^\prime$ Ando tuple for
$(T_1^*, T_2^*)$ if the system of equations \eqref{AndoTuple1'} - \eqref{AndoTuple2'} holds:
\begin{align}
 U^* P \Lambda D_{T^*} T^* + U^* P^\perp \Lambda D_{T^*} & = \Lambda D_{T^*} T_1^*,  \notag \\
 P^\perp U \Lambda D_{T^*} T^* + PU \Lambda D_{T^*} & = \Lambda D_{T^*} T_2^*.  \label{CheckNotTrue}
 \end{align}
We now complete the discussion in Remark \ref{R:reasonsD} by showing that it can happen that there is a special And\^o tuple which is not a Type I$^\prime$ And\^o tuple.

 To construct an example, proceed as follows.  Let $(T_1, T_2)$ be the BCL2-model commuting isometric pair associated with the BCL-tuple $(\cF, P, U)$:
  \begin{equation}  \label{T1T2}
 T_1 =( I_{H^2} \otimes U^* P^\perp) + (M_z \otimes U^*P), \quad T_2 = (I_{H^2} \otimes PU) + (M_z \otimes P^\perp U).
 \end{equation}
 Then the product isometry $T = T_1T_2 = T_2 T_1$ is $T = M_z^\cF$and the defect operator $D_{T^*}$ is the projection to the constant functions in $H^2(\cF)$:
 $D_{T^*} =  \bev_{0, \cF}^* \bev_{0, \cF}$. 
 
Let the map $\Lambda$ be given by
 $$
    \Lambda = \bev_{0, \cF}|_{\cD_{T^*}}\colon \cD_{T^*} \to \cF.
 $$
As was discussed in item (1) of Remark \ref{E:Dex}, 
 the collection $(\cF, \Lambda, P, U)$ is a Type I And\^o tuple for $(T_1^*, T_2^*)$ having the additional property that $\Lambda$ is unitary (rather than only isometry).
 For the ensuing discussion 
 \begin{equation}  \label{T1T2ATuple}
 \Xi_{T_1, T_2}= (\cF, \Lambda, P, U).
 \end{equation}
 refers to this specific choice of And\^o tuple constructed as above from $(T_1, T_2) = (V_1, V_2)$.

It is possible to find a unitary transformation $\tau \colon \cF \to \cD_{T_1^*} \oplus \cD_{T_2^*}$ which implements a coincidence between the Type I And\^o tuple
$(\cF, \Lambda, P, U)$ and a canonical-form special And\^o tuple $(\cF', \Lambda', P', U')$, but instead we present a higher-brow argument which uses some
general principles which are developed later in this exposition.  As we have already observed above,
the embedding operator $\Lambda \colon \cD_T \to \cF$ is actually unitary (i.e., a surjective isometry) which means in the terminology of Definition \ref{D:StrongMinTuple}  that $(\cF, \Lambda, P, U)$ is a {\em strongly minimal}
And\^o tuple for $(T_1^*, T_2^*)$ and that the Douglas-model And\^o lift  $(\bPi_D, \bV_{D,1}, \bV_{D,2})$ induced by the Type I And\^o tuple
$(\cF, \Lambda, P, U)$  for $(T_1^*, T_2^*)$ is actually {\em strongly minimal},  meaning that the pair $(\bPi_D, \bV = \bV_{D,1} \bV_{D,2})$ is a 
minimal Sz.-Nagy--Foias lift for the product contraction $T = T_1 T_2$ in Douglas-model form (in this case where $T$ is isometric, actually $V_D = T$ is the isometric lift of itself).
By the general result Theorem \ref{T:strongly-minimal} to come, it follows that all And\^o lifts of $(T_1, T_2)$ are unitarily equivalent, which in turn means (by Theorem \ref{Thm:DougCoin} to come) that all associated 
Type I And\^o tuples for $(T_1^*, T_2^*)$ coincide.
By combining Remark \ref{R:specATuple} and Theorem \ref{Thm:special} we see that special And\^o tuples for $(T_1^*, T_2^*)$ exist and each such And\^o tuple
is in fact a   Type I And\^o tuple for $(T_1^*, T_2^*)$.  Thus any Type I And\^o tuple for $(T_1^*, T_2^*)$  in fact coincides with a canonical-form special And\^o tuple, and
hence  (according to our terminology) is itself special.
In particular the specific And\^o tuple   $\Xi_{T_1, T_2}$ identified in  \eqref{T1T2ATuple} is a special And\^o tuple for $(T_1^*, T_2^*)$.

We next wish to check that $(\cF, \Lambda, P, U)$ is not a Type I$^\prime$ And\^o lift for $(T_1^*, T_2^*)$, i.e.~we wish
 to check the lack of general validity of the system of equations  \eqref{CheckNotTrue}.
 Applying the first equation to a general element $h \in H^2(\cF)$, we
 see that the first equation holds if and only if for all $h \in H^2(\cF)$ we have
 $$
 U^* P h'(0) + U^* P^\perp h(0) = (T_1^*h)(0) := P^\perp U h(0) + P U  h'(0).
 $$
 For this to hold, it must be the case that coefficients of $h(0)$ and of $h'(0)$ match:
 \begin{equation}  \label{UPcond}
U^*   P^\perp  = P^\perp U, \quad U^* P = P U.
\end{equation}
One can easily construct  counterexamples, even with $\cF = {\mathbb C}$, e.g.
 $$
  \cF = {\mathbb C}, \quad U^2 \ne 1, \quad P = 1.
 $$
 Thus the analogue of Theorem \ref{Thm:special} with Type I$'$ And\^o tuple in place of Type I And\^o tuple
 fails in general.
 A similar analysis holds for the second equation: applying the second equation to a general element $h \in H^2(\cF)$ leads to
 $$
 P^\perp U h'(0) + PU h(0) = (T_2^* h)(0):= U^*P h(0) + U^* P^\perp h'(0)
 $$
 which then leads to the same system of equations \eqref{UPcond}.
This completes the verification that the example is as desired. 
 \smallskip
 
 \noindent
 \textbf{2.  Example  of a special And\^o tuple which is not a strong Type II$^\prime$ And\^o tuple. }

 \smallskip
 
Recall from Remark \ref{R:reasonsS} that for a pre-And\^o tuple $(\cF, \Lambda, P, U)$ of $(T_1, T_2)$ to be a strong Type II$^\prime$ And\^o tuple, it must satisfy \begin{enumerate}
 \item[(a)] \textbf{Commutativity condition:}
 $$
 U^* P \Lambda D_T T_2 + U^* P^\perp U \Lambda D_T = P^\perp U \Lambda D_T T_1 + P \Lambda D_T.
 $$
\end{enumerate}
We show that there can be a special And\^o tuple of $(T_1,T_2)$ which fails to satisfy condition (a) above and hence is not a strong Type II$^\prime$ And\^o tuple.

Let $\cF$ be any coefficient Hilbert space, $P$ be any projection and $U$ be any unitary operator on $\cF$. We let $(T_1,T_2)$ be the commuting coisometric pair on $H^2(\cF)^\perp$ as in \eqref{SPair}:
\begin{align*}
& (T_1,T_2)=(\widetilde V_1^*, \widetilde V_2^*)\quad\mbox{where}\\
& (\widetilde V_1, \widetilde V_2) = (M_{U^*P^\perp + z^{-1} U^*P}, M_{PU+ z^{-1} P^\perp U}) \text{ on } H^2(\cF)^\perp.
\end{align*}
Let $\Lambda:\cD_T\to\cF$ be as in \eqref{SLambda}. We concluded in Example \ref{E:Sex} that $(\cF,\Lambda,P,U)$ is a strong Type II And\^o tuple for $(T_1,T_2)$ with $\Lambda$ actually a unitary. As in the discussion of part (1) above, the fact that $\Lambda$ is unitary implies that this $(\cF, \Lambda, P, U)$ is also special.  Thus it remains only to argue that
it can happen that this And\^o tuple $(\cF, \Lambda, P, U)$ is not a strong Type II$^\prime$ And\^o tuple for $(T_1, T_2)$.

Toward this goal, let us first compute, for $f(\zeta) = \sum_{n=-\infty}^{-1} f_n \zeta^{n} \in H^2(\cF)^\perp$,
\begin{align*}
& \Lambda D_T \colon f \mapsto f_{-1}, \\
& \Lambda D_T T_1 \colon f \mapsto \left[ M_{ P^\perp U  + \zeta P U} f \right]_{-1} = P^\perp U f_{-1} + PU f_{-2}, \\
& \Lambda D_T T_2 \colon f \mapsto \left[ M_{U^*P + \zeta U^* P^\perp} f \right]_{-1} =
U^* P f_{-1} + U^* P^\perp f_{-2},
\end{align*}
and hence we have
\begin{align*}
& U^* P \Lambda D_T T_2 + U^* P^\perp U \Lambda D_T \colon f \mapsto
(U^* P U^* P + U^* P^\perp U) f_{-1} + U^* P U^* P^\perp f_{-2},  \\
& P^\perp U \Lambda D_T T_1 + P \Lambda D_T \colon f \mapsto
(P^\perp U P^\perp U + P ) f_{-1} + P^\perp U P U f_{-2}, \\
&  P \Lambda D_T \colon f \mapsto P f_{-1}.
\end{align*}
Hence condition (a) requires that
\begin{equation}  \label{require-a}
U^* P U^* P + U^* P^\perp U = P^\perp U P^\perp U + P, \quad
U^* P U^* P^\perp = P^\perp U P U
\end{equation}
while condition (a$'$) requires in addition that the common value of the first expression is $I$ and the common value
of the second expression is $0$.
To get a counterexample to condition (a) (and hence also to (a$'$)),  it again suffices to take $\cF = {\mathbb C}$,
$P = 1$, $U^2 \ne 1$.

It turns out that the isometry condition (b) also fails to hold in general.  Indeed one can verify that
$$
D_{T_1}  \colon f \mapsto (U^* P U f_{-1}) \zeta^{-1}, \quad
D_{T_2}  \colon f \mapsto (P^\perp f_{-1}) \zeta^{-1}, \quad
D_T \colon f \mapsto f_{-1} \zeta^{-1}.
$$
and condition (b) requires
\begin{equation}  \label{require-b}
P = U^* P U, \quad U^* P^\perp U = P^\perp
\end{equation}
which is violated as soon as $P$ and $U$ do not commute, requiring $\dim  \cF \ge 2$.  One can verify that
$\cF = {\mathbb C}^2$, $U = \sbm{ 0 & 1 \\ 1 & 0}$, $P = \sbm{ 1 & 0 \\ 0 & 0}$ violates both \eqref{require-a} and
\eqref{require-b}.
Thus Theorem \ref{Thm:special-canonical} fails in general when {\em strong Type} II {\em And\^o tuple} is replaced with
{\em strong Type} II $^\prime$ {\em And\^o tuple}.

%
%
%
%

\chapter[Classification of And\^o  lifts]{Classification of Douglas/Sch\"affer-model lifts of commuting contractive operator-pairs}
\label{C:classification}

Recall that an And\^o lift $(\Pi, V_1,V_2)$ of a commuting contractive pair $(T_1,T_2)$ on $\cH$ is minimal, if the lift pair acts on the minimal joint-invariant subspace
for $(V_1, V_2)$  containing $\operatorname{Ran} \Pi$ (see \eqref{minAndoLift}). Unlike as in the classical case, a commuting contractive pair $(T_1, T_2)$ can have two minimal And\^o lifts that are not unitarily equivalent. For example, let 
$(T_1,T_2)=(0,0)$ on $\mathbb C$. 
Then both $(M_z,M_z)$ on $H^2$ and $(M_{z_1},M_{z_2})$ on $H^2_{{\mathbb D^2}}$ are minimal And\^o lifts of $(T_1,T_2)$ (note that the first is commuting but not doubly commuting
while the second is doubly commuting) but there is no unitary that intertwines these two pairs. 
In the previous chapter, we constructed And\^o lifts out of Type I and Type II And\^o tuples.  In this chapter, we show that both the Douglas and Sch\"affer model of an 
And\^o lift are uniquely associated to the Type I and Type II And\^o tuples from which they are constructed. We first deal with the Douglas model.

\section[Classification of model lifts]{Classification of Douglas/Sz.-Nagy--Foias models for And\^o lifts}\label{S:DougClass}
Suppose $(T_1,T_2)$ is a commuting contractive pair on a Hilbert space $\cH$ and $(\cF_*,\Lambda_*,P_*,U_*)$ is an And\^o tuple of Type I for $(T_1^*,T_2^*)$. 
Let us recall from \S \ref{S:DougArbMod} that the Douglas model of an And\^o  lift of $(T_1,T_2)$ is given by $(\Pi, V_1,V_2))$ on $\cK$, where the Hilbert space $\cK$, 
the pair of commuting isometries $(V_1,V_2)$ and
the embedding $\Pi$ are as given in \eqref{DougAndoMod}.

\begin{theorem}\label{Thm:DougCoin}
Let $(T_1,T_2)$  be a pair of commuting contractions acting on a Hilbert space $\cH$ and $(\cF_*,\Lambda_*,P_*,U_*)$, $(\cF_*',\Lambda_*',P_*',U_*')$ be  Type I 
And\^o tuples  for $(T_1^*,T_2^*)$. Let $(\bPi_D, \bV_{D,1}, \bV_{D,2})$, $(\bPi'_D, \bV'_{D,1},\bV'_{D,2} )$ be the Douglas-model And\^o  lifts of $(T_1,T_2)$ corresponding to $(\cF_*,\Lambda_*,P_*,U_*)$ and $(\cF_*',\Lambda_*',P_*',U_*')$, respectively as in Theorem \ref{T:Dmodel}. Then $(\bPi_D, \bV_{D,1}, \bV_{D,2})$ and $(\bPi_D', \bV_{D,1}, \bV_{D,2} )$ are unitarily equivalent if and only if $(\cF_*,\Lambda_*,P_*,U_*)$ and $(\cF_*',\Lambda_*',P_*',U_*')$ coincide.
\end{theorem}

\begin{proof}
For the `if' part, suppose two And\^o tuples $(\cF_*,\Lambda_*,P_*,U_*)$ and $(\cF_*',\Lambda_*',P_*',U_*')$ of Type I of $(T_1^*,T_2^*)$ coincide, i.e. by Definition \ref{D:preAndoTuple} there exists a unitary $u_*:\cF_*\to\cF_*'$ such that
\begin{align}\label{DAndoCoin}
 u_*\Lambda_*=\Lambda_*'\quad \text{ and } \quad u_*(P_*,U_*)=(P_*',U_*')u_*.
\end{align}
Define the unitary 
$$
\tilde{u}_*:\begin{bmatrix}
(I_{H^2}\otimes u_*)&0\\0& I_{\cQ_{T^*}}
\end{bmatrix}:\begin{bmatrix}
H^2(\cF_*)\\ \cQ_{T^*}
\end{bmatrix} \to \begin{bmatrix}
H^2(\cF'_*)\\ \cQ_{T^*}
\end{bmatrix} .
$$
Then it follows from (\ref{DAndoCoin}) that
\begin{align*}
\tilde{u}_*\bPi_D=\begin{bmatrix}
                    I_{H^2}\otimes u_* & 0 \\
                    0 & I_{\cQ_{T^*}}
                  \end{bmatrix}\begin{bmatrix} (I_{H^2}\otimes\Lambda_*) \cO_{D_{T^*}, T^*} \\ Q_{T^*} \end{bmatrix}
&=\begin{bmatrix} (I_{H^2}\otimes u_*\Lambda_*) \cO_{D_{T^*}, T^*} \\ Q_{T^*} \end{bmatrix}\\
&=\begin{bmatrix} (I_{H^2}\otimes\Lambda_*') \cO_{D_{T^*}, T^*} \\ Q_{T^*} \end{bmatrix}=\bPi'_D
\end{align*}
and
\begin{align*}
\tilde{u}_* \bV_{D,1} &=\begin{bmatrix} I_{H^2} \otimes u_* & 0 \\  0 & I_{\cQ_{T^*}} \end{bmatrix}  
\begin{bmatrix} M_{U_*^*P_*^\perp+z U_*^* P_*} & 0 \\ 0 &  W_{\flat1} \end{bmatrix}    \\
 &=
\begin{bmatrix} M_{u_*(U_*^* P_*^\perp+zU_*^* P_*)} & 0 \\ 0 &  W_{\flat1} \end{bmatrix}
=\begin{bmatrix} M_{(U'^* P_*'^\perp+z U'^* P_*')u_*} & 0 \\ 0 &  W_{\flat1}\end{bmatrix}=\bV'_{D,1} \tilde{u}_*.
\end{align*}  
 The intertwining $\tilde{u}_*\bV_{D,2}=\bV'_{D,2} \tilde{u}_*$     follows similarly.  This establishes the equivalence of $(\bPi_D, \bV_{D,1},\bV_{D,2})$
and $(\bPi'_D, \bV'_{D,1}, \bV'_{D,2})$.

Conversely, suppose that the two And\^o lifts $(\bPi_D, \bV_{D,1}, \bV_{D,2})$ and $(\bPi'_D, \bV'_{D,1},$ 
$\bV'_{D,2})$ of $(T_1,T_2)$ are unitarily equivalent (as lifts of $(T_1, T_2)$) and 
are in the model form coming from two Type I And\^o tuples
$(\cF_*,\Lambda_*,P_*,U_*)$ and $(\cF_*', \Lambda_*',P_*',U_*')$ for $(T_1, T_2)$, respectively.
This means that there exists a unitary $\tau_* \colon \bcK_D \to \bcK'_D$ such that
\begin{equation}   \label{DuniDIl}
\tau_* \bPi_D = \bPi'_D,  \quad    \tau_* (\bV_{D,1}, \bV_{D,2}) = (\bV'_{D,1}, \bV'_{D,2}) \tau_*.
\end{equation}
For more detailed calculations let us introduce that $2 \times 2$ matrix representation for the unitary $\tau_*$ and the column representations for the spaces
$\bcK_D $ and $\bcK'_D$:
$$
\tau_* = \begin{bmatrix} \tau' & \tau_{12} \\ \tau_{21} & \tau'' \end{bmatrix} \colon  \cK:=\begin{bmatrix} H^2(\cF_*) \\ \cQ_{T^*} \end{bmatrix}
\to \cK':= \begin{bmatrix} H^2(\cF_*') \\ \cQ_{T^*} \end{bmatrix}.
$$
From the second equality in (\ref{DuniDIl}) we see that  $\tau_* \bV_{D,1} \bV_{D,2}=\bV'_{D,1} \bV'_{D,2} \tau_*$ which in detail becomes
$$
\begin{bmatrix} \tau' & \tau_{12} \\ \tau_{21} & \tau'' \end{bmatrix}
\begin{bmatrix} M_z^{\cF_*} & 0 \\ 0 & W_D \end{bmatrix}
=\begin{bmatrix} M_z^{\cF_*'} & 0 \\ 0 & W_D \end{bmatrix}
\begin{bmatrix} \tau' & \tau_{12} \\ \tau_{21} & \tau''  \end{bmatrix}.
$$
As a consequence of part (2) of Lemma \ref{L:AuxLemma} we see that
$$
  \tau_{12}=0, \quad \tau_{21} = 0.
$$
and $\tau_*$ has the diagonal form
$$
   \tau_* = \begin{bmatrix} \tau' & 0 \\ 0 & \tau'' \end{bmatrix}.
$$
with $\tau'$ and $\tau''$ also unitary and satisfying the intertwinings
$$
  \tau' M_z^{\cF_*}  = M_z^{\cF_*'} \tau', \quad \tau'' W_D = W_D \tau''.
$$
The first equality forces $\tau'$ to have the form
$$
\tau' = I_{H^2} \otimes u_* \text{ for some unitary } u_* \colon \cF_* \to \cF'_*.
$$
Then
\begin{align*}
\tau_* \Pi & = \begin{bmatrix} I_{H^2} \otimes u_* & 0 \\ 0 & \tau'' \end{bmatrix} \begin{bmatrix} (I_{H^2}\otimes\Lambda_*) \cO_{D_{T^*}, T^*} \\ Q_{T^*} \end{bmatrix}
= \begin{bmatrix}  (I_{H^2} \otimes u_*\Lambda_*) \cO_{D_{T^*}, T^*} \\ \tau'' Q_{T^*} \end{bmatrix}
\end{align*}
while on the other hand
$$
\bPi'_D =  \begin{bmatrix} (I_{H^2}\otimes\Lambda'_*) \cO_{D_{T^*}, T^*} \\ Q_{T^*} \end{bmatrix}
$$
Thus the first equality in \eqref{DuniDIl} implies that
\begin{equation}  \label{DCoin1}
(I_{H^2}\otimes u_*) \Lambda_* =\Lambda_*', \quad \tau''Q_{T^*}=Q_{T^*}.
\end{equation}
The second equality in \eqref{DCoin2} together with the intertwining $\tau''W_D=W_D\tau''$ implies that $\tau''$ is equal to the identity on vectors of the form $W_D^nQ_{T^*}h$ 
with $h\in \cH$.  As $\cup_{n=0}^\infty W_D^nQ_{T^*} \cH$ is dense in $\cQ_{T^*}$, we conclude that
 $\tau''=I_{\cQ_{T^*}}$. 
 Since $\tau_*= \sbm{ I_{H^2}\otimes u_* & 0 \\ 0 &  I_{\cQ_{T^*}}}$ intertwines  $\bV_{D,1}$ with $\bV'_{D,1}$ 
 where
 $$
    \bV_{D,1}  = \begin{bmatrix} M_{U_*^* P_*^\perp+z U_*^* P_*} & 0 \\ 0 &  W_{\flat1} \end{bmatrix} , \quad
\bV'_{D,1} =  \begin{bmatrix} M_{U_*^{\prime *} P_*^{\prime \perp}+z U_*^{\prime *}P'_*} & 0 \\ 0 &  W_{\flat1} \end{bmatrix},
$$
we see that
\begin{align*}
  u_*(U_*^* P_*^\perp+z U_*^* P_*)= (U_*'^* P_*'^\perp+z U_*'^* P_*')u_*,
\end{align*}
or equivalently,
\begin{equation}    \label{DCoin2}
u_*U_*^*P_*^\perp=U_*'^*P_*'^\perp u_*, \quad u_*U_*^*P_*=U_*'^*P_*'u_*.
\end{equation}
These two equations together imply
\begin{equation}    \label{Dcoin3}
u_* U_*^*=u_*U_*^*(P_*^\perp+P_*)=U_*'^*(P_*'^\perp+P_*')u_*=U_*'^*u_*.
\end{equation}
This and the last equality in (\ref{DCoin2}) together establish the  intertwining
$$
 u_*P_*=P_*'u_*.
$$
The coincidence of the two And\^o tuples $(\cF_*,\Lambda_*,P_*,U_*)$ and $(\cF_*',\Lambda_*',P_*',U_*')$ now follows.
\end{proof}

We have seen in section \ref{S:NFmodel2} that the connection between the Douglas-model And\^o lift $(\bPi_D, \bV_{D,1}, \bV_{D,2})$ and the Sz.-Nagy--Foias model And\^o lift $(\bPi_{\rm NF}, \bV_{\rm NF,1}, \bV_{\rm NF,2})$ is rather straightforward, namely:
\begin{align*}
& \bPi_{\rm NF} = \begin{bmatrix} I_{H^2(\cF_*)} & 0 \\ 0 & \omega_{\rm NF, D} \end{bmatrix} \bPi_D, \\
& \bV_{\rm NF,j} = \begin{bmatrix} I_{H^2(\cF_*)} & 0 \\ 0 & \omega_{\rm NF,D} \end{bmatrix} V_{D,j} \begin{bmatrix} I_{H^2(\cF_*)} & 0 \\ 0 & \omega_{\rm NF,D}^* \end{bmatrix} 
\text{ for } j=1,2.
\end{align*}
Using this correspondence combined with the result of Theorem \ref{Thm:DougCoin} gives us the following immediate corollary.

\begin{corollary}  \label{Cor:NFCoin}
Let $(T_1,T_2)$  be a commuting contractive operator-pair on a Hilbert space $\cH$ and $(\cF_*,\Lambda_*,P_*,U_*)$, $(\cF_*',\Lambda_*',P_*',U_*')$ be two Type I 
And\^o tuples  for $(T_1^*,T_2^*)$. Let $(\bPi_{\rm NF}, \bV_{\rm NF,1}, \bV_{\rm NF,2})$ and $(\bPi'_{\rm NF}, \bV'_{\rm NF,1},\bV'_{\rm NF,2} )$ be the Sz.-Nagy--Foias model
And\^o  lifts of $(T_1,T_2)$ corresponding to $(\cF_*,\Lambda_*,P_*,U_*)$ and $(\cF_*',\Lambda_*',P_*',U_*')$ as in \eqref{NF-AndoGenForm}. Then $(\bPi, \bV_{\rm NF,1}, 
\bV_{\rm NF,2})$  and $(\bPi', \bV_{\rm NF,1}, \bV_{\rm NF,2} )$ are unitarily equivalent (as lifts of $(T_1, T_2)$) if and only if $(\cF_*,\Lambda_*,P_*,U_*)$ and $(\cF_*',
\Lambda_*',P_*',U_*')$ coincide (as pre-And\^o tuples).
\end{corollary}

To illustrate the ideas we here set down some And\^o lifts for a simple commuting pair of contractions $(T_1, T_2)$ and compute some
associated minimal And\^o tuples.  In particular the examples illustrates that a given commuting contractive pair can have many
minimal And\^o lifts which are not unitarily equivalent as lifts.

\begin{example}  \label{E:joint-evs}
Let $(T_1, T_2)$ be a commuting pair of contraction operators on a finite-dimensional Hilbert space $\cH$ of say dimension $N$.  For simplicity we
assume that $(T_1, T_2)$ has a basis of joint eigenvectors.  For convenience we work in detail with a basis of joint eigenvectors for the adjoint pair
$(T_1^*, T_2^*)$.  Let us denote by $\{v_1, \dots, v_N\}$ the basis of joint eigenvectors for $(T_1^*, T_2^*)$ with joint eigenvalues
$$
 ( \overline{\lambda}_1, \dots, \overline{\lambda}_N) =
  \bigg( (\overline{\lambda}_{1,1}, \overline{\lambda}_{1,2}), \dots,(\overline{\lambda}_{N,1}, \overline{\lambda}_{N,2})\bigg).
$$
Thus we have for $r=1,2$ and $j=1, \dots, N$ that
\begin{equation}   \label{evs}
  T_r^* v_j = \overline{\lambda}_{j,r} v_j.
\end{equation}
In particular we have
$$
  \langle (I - T_1 T_1^*) v_j, v_i \rangle_\cH = (1 - \lambda_{i,1} \overline{\lambda}_{j,1}) \langle v_j, v_i \rangle_\cH.
$$
As $T_1$ is a contraction, the matrix on the left (with rows indexed by $i$ and columns by $j$) is positive semidefinite, say of rank $d$.  Hence there are vectors
$y_1, \dots, y_N$ is a $d$-dimensional Hilbert space $\cY$ so that
\begin{equation}   \label{y-gram}
  \langle (I  - T_1 T_1^*) v_j, v_i \rangle_\cH = \langle y_j, y_i \rangle_\cY.
\end{equation}
Combining the last two displayed identities and using the assumption that each $\lambda_{j,1}$ is in the open unit disk, we can solve for $\langle v_j, v_i \rangle$
to get
\begin{equation}   \label{id1'}
  \langle v_j, v_i \rangle = \frac{ \langle y_j, y_i \rangle_\cY }{1 - \lambda_{i,1} \overline{\lambda_{j,1}}}.
\end{equation}

Let us now introduce the vectorial Hardy space $H^2(\cY)$ and the vectorial kernel functions $k_\lambda y$ (for $\lambda \in {\mathbb D}$ and $y \in \cY$) given by
$$
 ( k_\lambda y)(z) = \frac{1}{1 - z \overline{\lambda}} y
$$
having the reproducing kernel property
$$
  \langle f, k_\lambda y \rangle_{H^2(\cY)}    = \langle f(\lambda), y \rangle_\cY.
$$
The identity \eqref{id1'} implies that the map
\begin{equation}  \label{Lambda0}
\bPi \colon  v_j \mapsto k_{\lambda_{j,1}} y_j \text{ for } j = 1, \dots, N
\end{equation}
extends by linearity to a unitary map from $\cH$ onto the Hilbert space
\begin{equation}   \label{tildeH}
\widetilde \cH = \bigvee \{ k_{\lambda_{j,1}} y_j \colon j= 1, \dots, N\} \subset H^2(\cY)
\end{equation}
equal to the  span of the kernel
functions $k_{\lambda_{j,1}} y_j$ in the Hardy space $H^2(\cY)$.  
Furthermore the operators $T_1^*$ and $T_2^*$ are transformed via the unitary identification
$\bPi \colon \cH \to \widetilde \cH$ to the operators
\begin{align*}
&  \widetilde T_1^* \colon k_{\lambda_{j,1}} y_j  \mapsto \overline{\lambda}_{j,1} k_{\lambda_{j,1}} y_j,  \\
& \widetilde T_2^* \colon k_{\lambda_{j,1}} y_j \mapsto \overline{\lambda}_{j,2} k_{\lambda_{j,1}} y_j.
\end{align*}
Note next that the  contractivity of the operator $\widetilde T_2^*$  implies that
$$
 \bigg\| \sum_{j=1}^N c_j k_{\lambda_{j,1}} y_j \bigg\|^2 - \bigg \| \sum_{j=1}^N c_j \overline{\lambda}_{j,2} k_{\lambda_{j,1}} y_j \bigg\|^2 \ge 0
  \text{ for all } c_1, \dots, c_N \in {\mathbb C}.
$$
Spelling out this condition gives us the positive-semidefiniteness condition
$$
  \bigg[  \frac{ 1 - \lambda_{i,2} \overline{ \lambda}_{j,2}}{1 - \lambda_{i,1} \overline{\lambda}_{j,1}} \langle y_j, y_i \rangle_\cY \bigg]  \succeq 0.
$$
By the standard theory of matrix-valued Nevanlinna-Pick interpolation (see e.g.~\cite{BGR}), there is an inner function $\Theta$ with values in $\cB(\cY)$
so that
\begin{equation}  \label{NPint}
\Theta(\lambda_{j,1})^* y_j = \overline{\lambda}_{j,2} y_j \text{ for } j=1, \dots, N.
\end{equation}
Let us now view $\bPi$ as an isometric embedding operator of $\cH$ into $H^2(\cY)$.  The previous computations show that
$$
  (M_z^*, M_\Theta^*) \bPi v_j = \bPi (T_1^*, T_2^*) v_j \text{ for } j=1,\dots, N,
$$
Since $\cH$ is the span of $v_1, \dots, v_N$, we can rewrite this last identity in operator form
$$
 ( M_z^*, M_\Theta^*) \bPi = \bPi (T_1^*, T_2^*),
 $$
i.e., $( \bPi, \bV_1, \bV_2) := (\bPi, M_z, M_\Theta)$ is an And\^o lift for the commuting contractive pair $(T_1, T_2)$.

We argue next that the And\^o lift $(\bPi, \bV_1, \bV_2) = (\bPi, M_z, M_\Theta)$ is minimal.   Indeed, we shall prove the stronger statement
\begin{equation}   \label{1min}
\bigvee_{j=0}^\infty \bV_1^j \operatorname{Ran} \bPi = H^2(\cY).
\end{equation}
To see this observe that
$$
\big( (I - \overline{\lambda}_j \bV_1) k_{\lambda_{j,1}} y_j \big) (z)  = (1 - \overline{\lambda}_{j,1} z) \cdot \frac{ y_j}{1 - \overline{\lambda}_{j,1} z} = y_j
$$
and hence  
$$
\bigvee_{j=0}^\infty  \bV_1^j \operatorname{Ran} \bPi \supset \bigvee_{j=1}^N \{y_j  \colon j=1, \dots, N\}.
$$
Note that the Gramian matrix $\big[\langle y_i, y_j \rangle \big]_{i,j=1, \dots, N}$ of the vectors $y_1, \dots, y_N$ has rank equal to the rank $d$  of the defect operator
$D_{T_1^*}$ as a consequence of the identity \eqref{y-gram}.  But on the other hand we have chosen the space $\cY$ to have dimension equal to $d$ so we can conclude
that $\bigvee_{j=1}^d \{y_j \colon j = 1, \dots, N\} = \cY$ and the last displayed identity can be rewritten as
$$
\bigvee_{j=0}^\infty  \bV_1^j \operatorname{Ran} \bPi \supset \bigvee_{j=1}^N \{y_j  \colon j=1, \dots, N\} = \cY
$$
(where here we identify $\cY$ with the subspace of constant functions in $H^2(\cY)$).   It then follows that
$$
\bigvee_{j=0}^\infty  \bV_1^j \operatorname{Ran} \bPi \supset \bigvee_{j=1}^N M_z^j  \cY  = H^2(\cY)
$$
and \eqref{1min} follows, i.e., $(\bPi, \bV_1, \bV_2)$ in particular is a minimal And\^o  lift.

We are now at the starting point for the proof of the converse direction in Theorem \ref{T:Dmodel}.  The BCL model for the isometric pair $(\bV_1, \bV_2):=(M_z, M_\Theta)$
on $H^2(\cY)$ is computed in Example \ref{E:1nondoubly};  we see there that the coefficient space $\cF$ should be taken to be
$\cF:= \cY \oplus \fH(\Theta)$ with associated BCL tuple $(\cF, P, U)$ including projection $P$ and  unitary $U$ on $\cF$ given by \eqref{BCLdata-BDFmodel}, and with
implementing unitary identification map $\tau_{\rm BCL}$ here taking the form of $\tau_\Theta$ given by \eqref{tau-Theta}.  The next step is to observe that
$\Pi:= \tau_\Theta \bPi$ is an isometric embedding of $\cH$ into $H^2(\cY \oplus \fH(\Theta))$ and the collection
$$
(\Pi, M_1, M_2) := (\tau_\Theta \bPi, M_{U^*P^\perp + z U^* P}, M_{PU + z P^\perp U})
$$
is again a lift of $(T_1, T_2)$ which is unitarily equivalent (via $\tau_\Theta$) to the previously discussed lift $(\bPi, M_z, M_\Theta)$ on $H^2(\cY)$,
having the additional property that the commuting isometric pair $(M_1, M_2)$ giving the And\^o lift is in the BCL2-model form on $H^2(\cY \oplus \fH(\Theta))$.
Note also that here we are in the somewhat simpler case where the product isometry $M_1 \cdot M_2 = M_z$ on $H^2(\cY \oplus \fH(\Theta))$ is a shift, and hence
our model space involves only the top component of the block $2 \times 1$ column matrices appearing for the general case.
Specializing the explanation given in the proof of Theorem \ref{T:Dmodel} to the situation here, we see that there is an isometric embedding
$$
\Gamma \colon H^2(\cD_{T^*}) \to H^2( \cY \oplus \fH(\Theta))
$$
such that
\begin{align*}
& \Gamma M_z^{\cD_{T^*}} = M_z^{\cY \oplus \fH(\Theta)} \Gamma, \quad \Gamma \cO_{D_{T^*}, T^*} = \Pi \text{ (here $T = T_1 T_2$)}, \\
& \operatorname{Ran} \Gamma = \bigvee_{n \ge 0} (M_z^{\cY \oplus \fH(\Theta)})^n \operatorname{Ran} \Pi =: \cK_{\rm min} \subset H^2(\cY \oplus \fH(\Theta))
\end{align*}
and such that $(\Pi, M_z|_{\operatorname{Ran} \Gamma})$ is a version of the essentially unique minimal isometric lift for the single contraction operator $T$.
Furthermore $\Gamma$ has the form of a multiplication by a constant $\Gamma = I_{H^2} \otimes \Lambda$ for an isometry from $\Lambda \colon \cD_{T^*} \to
\cY \oplus \fH(\Theta)$.  The explicit formula \eqref{Lambda*explicit} for $\Lambda$ here can be given the form
\begin{equation}   \label{Lambda-jointevs}
  \Lambda \colon  D_{T^*} v_j \mapsto  \big( \tau_\Theta k_{\lambda_{j,1}} y_j \big)(0) \text{ for } j=1,2,\dots, N.
\end{equation}
It is then this $\Lambda$ which serves as the embedding operator for the Type I And\^o tuple 
\begin{equation}
\label{joint-evs-AndoTuple}
(\cY \oplus \fH(\Theta), \Lambda, P, U) \text{ with $U$ and $P$ as in \eqref{BCLdata-BDFmodel}}
\end{equation}
which is the parameter-set to build the Douglas-model And\^o lift for the original commuting contractive pair $(T_1, T_2)$ specified in terms of joint eigenvectors for
$(T_1^*, T_2^*)$ \eqref{evs}.  Since, as observed in the previous paragraph, the And\^o lift $(\Pi, \bV_1, \bV_2)$ is minimal, it follows that this
Type I And\^o tuple for $(T_1^*, T_2^*)$ is  minimal as an And\^o tuple as well.
\end{example}

\begin{example} \label{E:jointevs-bidisk}
We again let $(T_1, T_2)$ be a commuting contractive pair on a Hilbert space $\cH$ of finite dimension $N$ with a basis of joint eigenvectors \eqref{evs}.
With some additional hypotheses in place we shall construct a lift $(\Pi, M^\cY_{z_1}, M^\cY_{z_2})$ with the commuting isometric pair $(M^\cY_{z_1}, M^\cY_{z_2})$ equal to
the coordinate-function shift operators on the Hardy space over the bidisk $H^2_{{\mathbb D}^2}(\cY)$ for an appropriate coefficient Hilbert space $\cY$.
We shall then find a Type I And\^o tuple $(\cF, \Lambda, P, U)$ for $(T_1^*, T_2^*)$ which provides the set of parameters to build a  Douglas-model And\^o lift 
unitarily equivalent to bidisk And\^o lift $(\Pi, M^\cY_{z_1}, M^\cY_{z_2})$. 

We first introduce the required added hypotheses.  It is known that there is a couple of extra conditions required for a given commuting contractive pair $(T_1, T_2)$ to have
a lift $(\Pi, M^\cY_{z_1}, M^\cY_{z_2})$ to the bidisk shift tuple $(M^\cY_{z_1}, M^\cY_{z_2})$ acting on $H^2_{{\mathbb D}^2}(\cY)$, namely (see Theorem 3.16 in \cite{CV} for $n=2$
and with $T_j$ there replaced by $T_j^*$): 
\begin{enumerate}
\item The bidisk squared-defect operator 
\begin{equation}  \label{square-defect}
D^2_{T_1^*, T_2^*} : = I - T_1 T_1^* - T_2 T_2^* + T_1 T_2 T_2^* T_1^*
\end{equation}
 should be positive semi-definite:  
$$
D_{T_1^*, T_2^*} \succeq 0.
$$

\item Both $T_1$ and $T_2$ should be {\em pure} in the sense that\index{contraction!pure}
$$
   \lim_{N \to \infty}  \| T_j^{*n} h \|^2 = 0 \text{ for all } h \in \cH \text{ for } j=1,2.
 $$
\end{enumerate}

\begin{remark}
Let us note that the bidisk defect operator $D_{T_1^*, T_2^*}$ can be viewed as an application of a version of the Agler hereditary functional calculus
$$
f(\blam, \bmu) := \sum_{\boldn,\boldm \in {\mathbb Z}_+^2} a_{\boldn,\boldm} \blam^\boldn \overline{\bmu}^\boldm  \mapsto \sum_{\boldn,\boldm \in
{\mathbb Z}_+^2} a_{\boldn, \boldm} \bT^\boldn \bT^{* \boldm}
$$
(here $\boldn = (n_1, n_2)$, $\boldm = (m_1, m_2)$, $\blam = (\lambda_1, \lambda_2)$, $\bmu = (\mu_1, \mu_2)$,
$\blam^\boldn = \lambda_1^{n_1} \lambda_2^{n_2}$, $\overline{\bmu}^\boldn = \overline{\mu}^{n_1}  \overline{\mu}^{n_2}$
with similarly conventions for operators:   $\bT = (T_1, T_2)$, $\bT^* = (T_1^*, T_2^*$, $\bT^\boldn = T_1^{n_1} T_2^{n_2}$
and similarly for $\bT^* = (T_1^*, T_2^*)$ where here $(T_1, T_2)$ is a commuting operator pair) applied to the function
\begin{equation}  \label{sample-f}
f(\blam, \bmu) = (1 - \lambda_1 \overline{\mu}_1) (1 - \lambda_2 \overline{\mu}_2)  = 1 - \lambda_1 \overline{\mu}_1
  - \lambda_2 \overline{\mu}_2 + \lambda_1 \lambda_2 \overline{\mu}_1 \overline{\mu}_2.
\end{equation}
Here $T_1, T_2$ and $T_1^*, T_2^*$ commute but $T_j$ does not necessarily commute with $T_k^*$ for any pair of indices $j,k \in \{1, 2\}$.
The hereditary functional calculus gives an ad hoc rule (in this case adjoint powers of $T_j$ on the right) 
for plugging in non-commuting operator arguments into a function have commuting scalar arguments.  The operator calculus of Ambrozie-Engli\v s-M\"uller
\cite{AEM} gets around this by defining a functional calculus on operators:   define $L_{T_j}, R_{T_j} \in \cB(\cB(\cH))$ for $j=1,2$  by
$$
L_{T_j} \colon X \mapsto T_j X, \quad R_{T_j^*} \colon X \mapsto X T_j^* \text{ for } X \in \cB(\cH).
$$
Then the set of operators $L_{T_1}, L_{T_2}, R_{T^*_1}, R_{T^*_2}$ is a commuting set of operators in $\cB(\cB(\cH))$ (given that $(T_1, T_2)$ is a commuting
operator pair), and the function $f(L_{T_1}, L_{T_2}, R_{T_1}, R_{T_2})$ is well-defined (here we use the substitution $\overline{\mu}_j \mapsto L_{T_j^*}$).  Then the
desired operator $D^2_{T_1^*, T_2^*}$ resulting from the hereditary functional calculus using the function \eqref{sample-f} can be seen as applying the
function \eqref{sample-f} in the standard well-defined way to the commuting operator-tuple $(L_{T_1}, L_{T_2}, R_{T_1^*}, R_{T_2^*})$ and then evaluating the result
on the identity operator $I_\cH$:  
$$
 f(\blam, \bmu) \mapsto f(L_\bT, R_{\bT^*})(I_\cH) = D^2_{T_1^*, T_2^*} \text{ if } f \text{ is given by \eqref{sample-f}.}
$$
\end{remark}

\begin{remark}  It turns out that the same condition $D_{T_1^*, T_2^*} \succeq 0$ is necessary and sufficient for the commuting contractive pair $(T_1, T_2)$ to have
a regular unitary dilation, as originally discussed by Brehmer (see \cite{Nagy-Foias}).  This connection between existence of polydisk shift dilation and a regular
unitary dilation is also discussed in Curto-Vasilescu \cite{CV} and Timotin \cite{Timotin}.
 \end{remark}
 
 We now proceed as in Example \ref{E:joint-evs} but with an adaptation to get a bi-disk shift lift $(M^\cY_{z_1}, M^\cY_{z_2})$ on $H^2_{{\mathbb D}^2}(\cY)$
 rather than a Bercovici-Douglas-Foias model lift $(M^\cY_z, M_\Theta)$ on $H^2(\cY)$.  We are given $(T_1, T_2)$ on a finite-dimensional space with a basis
 $\{v_1, \dots, v_n\}$ of joint eigenvectors for $(T_1^*, T_2^*)$ with associated joint eigenvalues $(\overline{\lambda}_{j,1}, \overline{\lambda}_{j,2})$ for 
 $1 \le j \le N$ as in \eqref{evs}.    Then we see that
 \begin{align*}
  \langle   D^2_{T_1^*, T_2^*} v_j, v_i \rangle_\cH  & = 
 \langle (1 - \lambda_{i,1} \overline{\lambda}_{j,1}  - \lambda_{i,2} \overline{\lambda}_{j,2} + \lambda_{i,1} \lambda_{i,2} \overline{\lambda}_{j,1} \overline{\lambda}_ {j,2})
 v_j, v_i \rangle_\cH  \\
 & = (1 - \lambda_{i,1} \overline{\lambda}_{j,1}) (1 - \lambda_{i,2} \overline{\lambda}_{j,2}) \langle v_j, v_i \rangle.
 \end{align*}
 Let us set $d = \operatorname{rank} D^2_{T_1^*, T_2^*}$.  
 If we assume that $D_{T_1^*, T_2^*}^2 \succeq 0$ (as we know must be the case if $(T_1, T_2)$ is to have a lift to the bi-disk shift pair $(M_{z_1}, M_{z_2})$
 on $H^2_{{\mathbb D}^2}(\cY)$ for some coefficient Hilbert space $\cY$),  we see that the matrix on the left (rows indexed by $i$, columns indexed by $j$) is positive
 semi-definite.  Hence there are vectors $y_1, \dots, y_N$ in a $d$-dimensional Hilbert space $\cY$ so that
 $$
   \langle D_{T_1^*, T_2^*}^2 v_j, v_i \rangle_\cH = \langle y_j, y_i \rangle_\cY.
 $$
By combining various of the preceding displayed identities and using the assumption that each $\lambda_{j,1}$ and $\lambda_{j,2}$ is in the open unit disk, we see that
\begin{equation} \label{gramian=}
\langle v_j, v_i \rangle_\cH =  \frac{ \langle y_j, y_i \rangle_\cY}{ (1 - \lambda_{i,1} \overline{\lambda}_{j,1}) (1 - \lambda_{i,2} \overline{\lambda}_{j,2})}
\end{equation}

Let us now introduce the vectorial Hardy space over the bi-disk $H^2_{{\mathbb D}^2}(\cY)$ consisting of functions $f(z_1, z_2) = \sum_{n,m \ge 0}
\widehat f_{n,m} z_1^n z_2^m$ with Fourier coefficients $\widehat f_{n,m} \in \cY$  subject to $\| f \|^2 := \sum_{n,m \ge 0} \| \widehat f_{n,m} \|^2_\cY < \infty$
This is a reproducing kernel Hilbert space with vectorial kernel functions $k_{\blam}    y$ 
(for $\blam = (\lambda_1, \lambda_2) \in {\mathbb D}^2$ and $y \in \cY$)
given by 
$$
  ( k_\blam y) (\boldz) = \frac{y}{(1 - z_1 \overline{\lambda}_1) (1 - z_2 \overline{\lambda}_2)}\text{ where we set } \boldz = (z_1, z_2)
$$
having the reproducing kernel property:
$$
\langle f, k_\blam y  \rangle_{H^2_{{\mathbb D}^2}(\cY)} = \langle f(\blam), y \rangle_\cY.
$$
The identity \eqref{gramian=} shows that the map 
\begin{equation}  \label{bidisk-Pi}
\bPi \colon v_j \mapsto k_{\blam_j} y_j
\end{equation}
extends by linearity to a unitary map from $\cH$ to the Hilbert space
\begin{equation}   \label{bidisk-tildeH}
\widetilde \cH = \bigvee \{ k_{\blam_j} y_j \colon j = 1, \dots, N\}  \subset H^2_{{\mathbb D}^2}(\cY)
\end{equation}
Furthermore, the operators $T_1^*$ and $T_2^*$ are transformed via the unitary identification map $\bPi$ to the operators
$$
  \widetilde T_1^* \colon k_{\blam_j} y_j  \mapsto\overline{\lambda}_{j,1} k_{\blam_j} y_j,  
  \quad \widetilde T_2^* \colon k_{\blam_j}  y_j \mapsto \overline{\lambda}_{j,2} k_{\blam_j} y_j \text{ where } \blam_j = (\lambda_{j,1}, \lambda_{j,2}).
$$
But the operators $M^{\cY *}_{z_1}$ and $M^{\cY *}_{z_2}$ on $H^2_{{\mathbb D}^2}(\cY)$ have exactly the same action on kernel functions, and we conclude that
$$
  M^{\cY *}_{z_1}|_{\widetilde \cH} = \widetilde T_1^*, \quad M^{\cY *}_{z_2}|_{\widetilde \cH} = \widetilde T_2^*.
$$
We conclude that
$$
     (\bPi, \bV_1, \bV_2) := (\bPi, M^\cY_{z_1}, M^\cY_{z_2})
$$
(where $M^\cY_{z_1}, M^\cY_{z_2}$ are the coordinate-function shift operators on $H^2_{{\mathbb D}^2}(\cY)$)
is an And\^o lift for the commuting, contractive pair $(T_1, T_2)$. 

 Furthermore we can see that this $(\bPi, \bV_1, \bV_2)$  is a minimal lift for $T_1, T_2$ as follows. 
Note that 
$$ 
 (I - \overline{\lambda}_{j,1}  \bV_1) (I - \overline{\lambda}_{j,2} \bV_2) k_{\blam_j} y_j = y_j \in \bigvee_{n_1, n_2 \in {\mathbb Z}_+} \bV_1^{n_1} \bV_2^{n_2} \operatorname{Ran} \bPi
 =: \cK_{0}
$$
(where here we view each $y_j$ as a constant function in $H^2_{{\mathbb D}^2}(\cY)$, the ambient subspace for the minimal lift contained inside $(\bPi, \bV_1, \bV_2)$).
We conclude that 
\begin{equation} \label{contain1}
  \bigvee \{ y_j \colon 1 \le j \le N\} \subset \cK_{0}.
\end{equation}
From \eqref{gramian=} and the two displayed formulas preceding it, we see that 
$$
  \langle y_i, y_j \rangle_\cY = \langle D_{T_1^*, T_2^*}^2 v_i, v_j \rangle
$$
implying that the Gramian matrix for $y_1, \dots, y_d$ has the same rank as $\operatorname{rank} D_{T_1^*, T_2^*}^2 = d$.  As we chose $\cH$ to have
$\dim \cH = d$,  we see that the rank of the Gramian matrix  $  \langle y_i, y_j \rangle_\cY$ is the same as the dimension of the whole space $\cY$, 
implying in turn that the span of the vectors $y_1, \dots, y_N$ is equal to the whole space $\cY$.  Combining with \eqref{contain1} then gives us
$$
  \cY \subset  \cK_{0}.
$$
But then  also 
$$
H^2_{{\mathbb D}^2}(\cY) = \bigvee_{n_1, n_2} \bV_1^{n_1}   \bV_2^{n_2}  \cY \subset \cK_{0} \subset H^2_{{\mathbb D}^2}(\cY)
$$
forcing the equality 
 $$
    \cK_{0} = H^2_{{\mathbb D}^2}(\cY),
 $$
 i.e., the lift $(\bPi, V_1, V_2) = (\bPi, M^\cY_{z_1}, M^\cY_{z_2})$ is minimal as a lift of $(T_1, T_2)$.
  
  We are now at the starting point of the proof of the converse direction in Theorem \ref{T:Dmodel} to find the Douglas model for the commuting-isometric lift
  $(\bPi, \bV_1, \bV_2)$ of $(T_1, T_2)$.  As we saw in Example \ref{E:BCLbidisk}, a BCL2-tuple for $(M_{z_1}, M_{z_2})$ on $H^2_{{\mathbb D}^2}(\cY)$
  can be taken to be 
  $$
  (\cF, P, U) = (\ell^2_{\mathbb Z}(\cY), P_{\ell^2_{(1, \infty)}(\cY)}, {\mathbf S}^\cY)
  $$
  where ${\mathbf S}^\cY$ is the bilateral shift acting on  $\ell^2_{\mathbb Z}(\cY)$, with implementation operator 
  $$
  \tau_{{\rm bd}, \cY} \colon H^2_{{\mathbb D}^2}(\cY) \to H^2(\cF) = H^2(\ell^2_{\mathbb Z}(\cY))
  $$
  suggested by \eqref{taubd}:
  \begin{equation}  \label{taubdY}
  \tau_{{\rm bd}, \cY} \colon z_1^i z_2^j  \, y \mapsto 
  \begin{cases} \be_{j-i} y \,  z^j  &\text{ for } i \ge j, \\  \be_{j-i} y \,  z^i  &\text{ for } i \le j.  \end{cases}
  \end{equation}
   for $y \in \cY$.  It remains to identify the isometric embedding operator $\Lambda \colon \cD_{T^*} \to \cF = \ell^2_{\mathbb Z}(\cY)$ so that the resulting
   Type I And\^o tuple $(\cF, \Lambda, P, U)$ is the parameter set generating a Douglas-model lift \eqref{DougAndoMod} unitarily equivalent to our original
   lift $(\bPi, \bV_1, \bV_2) = (\bPi, M_{z_1}, M_{z_2})$ with the commuting, isometric pair $M_{z_1}, M_{z_2}$ acting on $H^2_{{\mathbb D}^2}(\cY)$.  
   A careful interpretation of formula \eqref{Lambda*explicit} gives us
   $$
   \Lambda  \colon D_{T^*} v_j \mapsto \big( \tau_{bd, \cY} k_{\blam_j} y_j \big) (0).
   $$
   \end{example}

\section{Classification of Sch\"affer-model  And\^o  lifts}\label{S:SchaffClass}
 To classify the unitary equivalence of two Sch\"affer models of an And\^o  lift in terms of tuple coincidence of the associated Type II And\^o
 tuples, it turns out to be essential to work only
 with strong Type II And\^o tuples, as in the following result.

\begin{theorem}\label{Thm:TupleLiftEquiv}
Let $(\cF,\Lambda,P,U)$ and $(\cF',\Lambda',P',U')$ be two strong Type II And\^o tuples  of a given commuting contractive pair $(T_1,T_2)$
on a Hilbert space $\cH$. Let $(\bV_{S,1},\bV_{S,2})$ and $(\bV'_{S,1},\bV'_{S,2})$ be the minimal
 And\^o  lifts of $(T_1,T_2)$ corresponding to the strong Type II And\^o tuples $(\cF,\Lambda,P,U)$ and
 $(\cF',\Lambda',P',U')$, respectively, as
 in \eqref{Conv-AndoGenForm'}--\eqref{Vcanonical'}. Then $(\bV_{S,1},\bV_{S,2})$ and $(\bV'_{S,1}, \bV'_{S,2})$ are unitarily equivalent if and only if
 $(\cF,\Lambda,P,U)$ and $(\cF',\Lambda',P',U')$ coincide.
\end{theorem}

\begin{proof}
We first prove the sufficiency (or ``if'')  direction. Suppose $u:\cF\to\cF'$ is a unitary such that
\begin{align}\label{AndoCoin}
u\Lambda=\Lambda'\quad \text{ and } \quad u(P,U)=(P',U')u.
\end{align}
Define the unitary
\begin{align}\label{uTilde}
 \tilde{u}:=
 \begin{bmatrix}
 I_{\cH}&0\\ 0& I_{H^2}\otimes u
\end{bmatrix}:\begin{bmatrix}
 \cH\\ H^2(\cF)
\end{bmatrix}\to
\begin{bmatrix}
\cH\\ H^2(\cF')
\end{bmatrix}.
\end{align}
Then keeping the equations in (\ref{AndoCoin}) in mind, we conclude from the computations
\begin{align*}
  \tilde{u} \bV_{S,1}&=\begin{bmatrix} I_{\cH} & 0 \\ 0 & I_{H^2}\otimes u\end{bmatrix}
 \begin{bmatrix} T_1 & 0 \\  \bev_{0,\cF}^*PU\Lambda D_T & M_{P^\perp U +zP U} \end{bmatrix}\\
 & = \begin{bmatrix} T_1 & 0 \\  \bev_{0,\cF'}^*uPU\Lambda D_T & M_{u(P^\perp U +zP U)} \end{bmatrix}
\end{align*}
and
\begin{align*}
\bV'_{S,1} \tilde{u} & = \begin{bmatrix}
            T_1 & 0 \\
           \bev_{0,\cF'}^* P'U'\Lambda' D_T & M_{(P'^\perp+zP')U'}
          \end{bmatrix} 
          \begin{bmatrix}
                 I_{\cH} & 0 \\
                 0 & I_{H^2}\otimes u
               \end{bmatrix}  \\
 &  =  \begin{bmatrix}
            T_1 & 0 \\
           \bev_{0,\cF'}^*P'U'\Lambda' D_T & M_{(P'^\perp U' +zP'U') u}
          \end{bmatrix}
\end{align*}
that $\tilde{u} \bV_{S,1}=\bV'_{S,1} \tilde{u}$. Similarly one can prove that $\tilde{u}\bV_{S,2}=\bV'_{S,2} \tilde{u}$.
Note that the proof of this direction works for any Type II And\^o tuples not necessarily strong.

Conversely, suppose two And\^o isometric lifts $(\bV_{S,1}, \bV_{S,2})$ and $(\bV'_{S,1}, \bV'_{S,2})$ of $(T_1,T_2)$ corresponding to
two strong Type II And\^o tuples $(\cF,\Lambda,P,U)$ and $(\cF',\Lambda',P',U')$, respectively, are unitarily equivalent.
This means that there exists a unitary
$$
\tau = \begin{bmatrix} \tau_{11}  & \tau_{12} \\ \tau_{21} & \tau_{22} \end{bmatrix} \colon \begin{bmatrix} \cH \\  H^2(\cF) \end{bmatrix}
\to \begin{bmatrix} \cH \\  H^2(\cF') \end{bmatrix}
$$
 such that
\begin{equation}\label{UniDil}
\tau(\bV_{S,1}, \bV_{S,2})=(\bV'_{S,1}, \bV'_{S,2})\tau, \quad \tau \begin{bmatrix} I_{\cH} \\ 0 \end{bmatrix} = \begin{bmatrix}  I_{\cH} \\ 0 \end{bmatrix}.
\end{equation}
The second equality in (\ref{UniDil}) implies that $\tau$ has the form
$$
   \tau = \begin{bmatrix} I_\cH   & \tau_{12}  \\ 0 &   \tau_{22} \end{bmatrix}.
 $$
 As $\tau$ is unitary, this in turn forces $\tau_{12} = 0$ and $\tau_{22} \colon H^2(\cF) \to H^2(\cF')$ to be unitary.
The first equality in  \eqref{UniDil}  implies in particular that
\begin{equation}   \label{first}
\tau \bV_{S,1} \bV_{S,2}=\bV'_{S,1} \bV'_{S,2} \tau
\end{equation}
where $\bV_{S,1} \bV_{S,2} = \bV_{S,2} \bV_{S,1} = \sbm{ T & 0 \\ \bev_{0,\cF}^* \Lambda D_T & M_z }$ and similarly for $\bV'_{S,1} \bV'_{S,2} =
\bV'_{S,2} \bV'_{S,1} = \sbm{ T& 0 \\ \bev_{0,\cF'}^* \Lambda' D_T &  M_z}$ by the assumption that $(\cF, \Lambda, P, U)$
and $(\cF', \Lambda', P', U')$ are both strong Type II And\^o tuples.
Hence
$$
\tau \bV_{S,1} \bV_{S,2}=\begin{bmatrix}  I_{\cH} & 0 \\ 0 & \tau_{22} \end{bmatrix}
\begin{bmatrix}  T & 0 \\  \bev_{0,\cF}^* \Lambda D_T & M^\cF_z \end{bmatrix} =
\begin{bmatrix} T & 0 \\  \tau_{22}  \bev_{0,\cF}^* \Lambda D_T & \tau_{22} M^\cF_z  \end{bmatrix}
$$
while
$$
\bV'_{S,1} \bV'_{S,2} \tau  = \begin{bmatrix}  T & 0 \\  \bev_{0,\cF'}^* \Lambda' D_T & M^{\cF'}_z \end{bmatrix}
 \begin{bmatrix} I_{\cH} & 0 \\ 0 & \tau_{22} \end{bmatrix}
=  \begin{bmatrix} T & 0 \\   \bev_{0,\cF'}^* \Lambda' D_T & M^{\cF'}_z \tau_{22} \end{bmatrix}.
$$
As a consequence of \eqref{first} we are led to the identity
\begin{equation}  \label{consequence}
\begin{bmatrix} T & 0 \\  \tau_{22}  \bev_{0,\cF}^* \Lambda D_T & \tau_{22} M_z  \end{bmatrix} =
\begin{bmatrix} T & 0 \\  \bev_{0,\cF'}^* \Lambda' D_T & M_z \tau_{22} \end{bmatrix}.
\end{equation}
Equality of the $(2,2)$-entries in \eqref{consequence}  combined with the fact that $\tau_{22}$ is unitary implies that $\tau_{22}$
has the form   $\tau_{22}=I_{H^2}\otimes u$ for some unitary $u \colon \cF\to\cF'$, from which it then follows that
$\tau_{22}  \bev_{0,\cF}^* = \bev_{0,\cF'}^* u$.
Comparison of the $(2,1)$-entries  in \eqref{consequence} then gives $u\Lambda=\Lambda'$.
A similar matrix computation and a comparison of the $(2,2)$-entries of $\tau(\bV_{S,1},\bV_{S,2})=(\bV'_{S,1}, \bV'_{S,2}) \tau$ implies
\begin{align*}
(M_{u(P^\perp U +zP  U) },M_{u( U^*P+z U^* P^\perp)})=(M_{(P'^\perp U' +zP' U' )u},M_{(U'^* P'+zU'^* P'^\perp)u}),
\end{align*}
which implies that
\begin{equation}   \label{intertwinings'}
  uP^\perp U=P'^\perp U'u,\;uPU=P'U'u, \quad uU^*P=U'^*P'u, \quad uU^*P^\perp=U'^*P'^\perp u.
\end{equation}
Adding the first two identities in \eqref{intertwinings'} gives $uU = U' u$. Use this identity
in the first equation in \eqref{intertwinings'} to get $uP^\perp = P'^\perp u$.  Apply a similar argument
starting with the second identity in \eqref{intertwinings'} instead, or alternatively plug in $P^\perp = I - P$,
$P'^\perp = I - P'$ into $u P^\perp = P'^\perp u$, to arrive at $u P = P u$ as well. We conclude that indeed
$(\cF, \Lambda, P, U)$ and $(\cF', \Lambda', P', U')$  coincide as strong Type II And\^o tuples.
\end{proof}

\section{Type II And\^o tuples versus strong Type II And\^o tuples}\label{S:TypeIIvsStrongTypeII}
The class of strong Type II And\^o tuples is strictly smaller than the class of Type II And\^o tuples as the following result demonstrates. 

\begin{proposition}\label{P:NonCanTuple}
Let $T_1$ be a contraction on a Hilbert space $ \cH$, $T=T_1^2$ and $\tau_1,\tau_2:\cD_{T_1}\to\cG$ be two isometries. Let $\Lambda_\dag$ be the isometry as in Definition \ref{D:SpcATuple}, i.e.,
\begin{align*}
 \Lambda_\dag:\cD_T\to \begin{bmatrix} \cD_{T_1}\\ \cD_{T_1}\end{bmatrix}\quad\mbox{and}\quad   \Lambda_\dag:D_T\mapsto \begin{bmatrix}
         D_{T_1}T_1 \\ D_{T_1}
    \end{bmatrix}.
\end{align*}Then the pre-And\^o tuple
\begin{align}\label{NonCanTuple}
    \left( \begin{bmatrix}
         \cG \\ \cG
    \end{bmatrix}, \begin{bmatrix}
         \tau_1 &0\\ 0 & \tau_2
    \end{bmatrix}\Lambda_\dag, \begin{bmatrix}
         I_\cG & 0\\ 0& 0
    \end{bmatrix}, \begin{bmatrix}
         0& I_\cG\\ I_\cG & 0
    \end{bmatrix} \right)
\end{align}is a Type II And\^o tuple for $(T_1,T_1)$. Furthermore:
\begin{enumerate}
    \item The tuple \eqref{NonCanTuple} is a strong Type II And\^o tuple if and only if 
$$
(\tau_1-\tau_2)D_{T_1}T_1 = 0.
$$
    
\item If $\tau_1=\tau_2$ is unitary, then \eqref{NonCanTuple} is a special Type II And\^o tuple.
\end{enumerate}
\end{proposition}
\begin{proof}
By simple matrix computation we have
\begin{align}\label{MatrixComp}
    PU=\begin{bmatrix}
         0& I \\ 0&0
    \end{bmatrix}= U^*P^\perp \quad \mbox{and}\quad U^*PU=\begin{bmatrix}
         0&0\\0&I
    \end{bmatrix}=P^\perp.
\end{align}This implies that the Commutativity condition for Type II And\^o tuples (condition (i) in Definition \ref{AndoTuple}), i.e.,
\begin{align*}
    PU\Lambda D_TT_2+P^\perp \Lambda D_T=  U^* P^\perp \Lambda D_T T_1 + U^* P U \Lambda D_T
\end{align*}is readily satisfied by the tuple in \eqref{NonCanTuple} because in this case $T_1=T_2$. Condition (ii) of Definition \ref{AndoTuple} is that
\begin{align*}
    D_T\Lambda^*U^*PU\Lambda D_T=D_{T_1}^2 \quad \mbox{and}\quad  D_T\Lambda^* P^\perp\Lambda D_T=D_{T_2}^2
\end{align*}
which, in view of \eqref{MatrixComp}, boils down to just
\begin{align}\label{CondiiHere}
    D_T\Lambda^*P^\perp \Lambda D_T= D_{T_1}^2.
\end{align}Since
$$
P^\perp \Lambda D_T=P^\perp \begin{bmatrix}
         \tau_1 &0\\ 0 & \tau_2
    \end{bmatrix}\Lambda_\dag D_T =\begin{bmatrix}
     0&0\\0&I
\end{bmatrix}\begin{bmatrix}
     \tau_1 D_{T_1}T_1 \\ \tau_2 D_{T_1}
\end{bmatrix}=\begin{bmatrix}
     0 \\ \tau_2 D_{T_1}
\end{bmatrix},
$$and $\tau_2$ is an isometry, we see that
\begin{align*}
    D_T\Lambda^*P^\perp \Lambda D_T= (P^\perp \Lambda D_T)^*(P^\perp \Lambda D_T)= \begin{bmatrix}
     0 & D_{T_1}\tau_2^*\end{bmatrix} \begin{bmatrix}
     0 \\ \tau_2 D_{T_1}\end{bmatrix}= D_{T_1}^2
\end{align*}and therefore \eqref{CondiiHere} holds. Consequently \eqref{NonCanTuple} is always a Type II And\^o tuple for $(T_1,T_1)$.

{\sf Proof of (1):} Note that for the tuple \eqref{NonCanTuple} to be a strong Type II And\^o tuple, it must, in addition, satisfy
$$
   PU\Lambda D_TT_1+P^\perp \Lambda D_T = \Lambda D_T \;(= U^* P^\perp \Lambda D_T T_1 + U^* P U \Lambda D_T).
$$So we compute
\begin{align*}
     PU\Lambda D_TT_1+P^\perp \Lambda D_T& = \begin{bmatrix}
          0& I\\ 0 &0
     \end{bmatrix}\begin{bmatrix}
          \tau_1 D_{T_1}T_1^2\\ \tau_2 D_{T_1}T_1
     \end{bmatrix}+ \begin{bmatrix}
          0&0\\0&I
     \end{bmatrix}\begin{bmatrix}
          \tau_1 D_{T_1}T_1 \\ \tau_2 D_{T_1}
     \end{bmatrix}\\
     &=\begin{bmatrix}
          \tau_2 D_{T_1}T_1 \\  \tau_2 D_{T_1}
     \end{bmatrix}.
\end{align*}Thus the tuple \eqref{NonCanTuple} will be strong if and only if 
\begin{align*}
\begin{bmatrix}
          \tau_2 D_{T_1}T_1 \\  \tau_2 D_{T_1}
     \end{bmatrix}=   PU\Lambda D_TT_1+P^\perp \Lambda D_T = \Lambda D_T= \begin{bmatrix}
          \tau_1 D_{T_1}T_1\\ \tau_2D_{T_1}
     \end{bmatrix}
\end{align*}which is true if and only if $(\tau_1-\tau_2)D_{T_1}T_1 \equiv 0$. This proves (1).

{\sf Proof of (2):} Let us denote $\tau_1=\tau_2=:\tau$ and the unitary
$$
\widehat\tau:=\begin{bmatrix}
     \tau&0\\0&\tau
\end{bmatrix}:\begin{bmatrix}
     \cD_{T_1} \\ \cD_{T_1}
\end{bmatrix}\to\begin{bmatrix}
     \cG \\ \cG
\end{bmatrix}.
$$Our goal is to show that the tuple \eqref{NonCanTuple} coincides (in the sense of Definition \ref{D:preAndoTuple}) with a special And\^o tuple of $(T_1,T_1)$ in its canonical form and therefore is special (see Definition \ref{D:SpcATuple}). Since $\tau:\cD_{T_1}\to\cG$ is a unitary, we make the following simple observations:
\begin{align}
&\notag \widehat\tau^* \begin{bmatrix}
         \cG \\ \cG
    \end{bmatrix}=\begin{bmatrix}
         \cD_{T_1} \\ \cD_{T_1}
    \end{bmatrix}=\cF_\dag,\quad \widehat\tau^*\Lambda = \Lambda_\dag,\quad
\widehat\tau^*\begin{bmatrix}
     I_\cG &0\\ 0&0
\end{bmatrix}\widehat\tau =\begin{bmatrix}
     I_{\cD_{T_1}}& 0\\ 0& 0
\end{bmatrix}=P_\dag\\ \label{Observe}
&\mbox{and lastly}\quad \widehat\tau^*U\widehat\tau=\widehat\tau^*\begin{bmatrix}
    0&I_\cG\\ I_\cG &0
\end{bmatrix}\widehat\tau = \begin{bmatrix}
     0& I_{\cD_{T_1}}\\ I_{\cD_{T_1}} &0
\end{bmatrix}.
\end{align}Note that the unitary $\sbm{0& I_{\cD_{T_1}}\\ I_{\cD_{T_1}} &0}$ satisfies 
$$
\begin{bmatrix}
     0& I_{\cD_{T_1}}\\ I_{\cD_{T_1}} &0
\end{bmatrix}:\begin{bmatrix}
     D_{T_1}T_1 \\ D_{T_1}
\end{bmatrix}\mapsto \begin{bmatrix}
     D_{T_1} \\ D_{T_1}T_1
\end{bmatrix}
$$and consequently the tuple
$$
\left( \begin{bmatrix}
         \cD_{T_1} \\ \cD_{T_1}
    \end{bmatrix}, \Lambda_\dag, \begin{bmatrix}
         I_{\cD_{T_1}} &0\\ 0& 0
    \end{bmatrix},  \begin{bmatrix}
         0&I_{\cD_{T_1}} \\ I_{\cD_{T_1}} &0
    \end{bmatrix} \right)
$$is a special And\^o tuple in its canonical form and it coincides with  \eqref{NonCanTuple} by the observations \eqref{Observe}.
\end{proof}

The next result shows how close general Type II And\^o tuples are to being strong Type II And\^o tuples.

\begin{proposition} \label{P:almost-canonical}
Suppose that $(\cF, \Lambda, P, U)$ is a Type II And\^o tuple for the commuting contractive operator pair $(T_1, T_2)$.
Then there is an isometry $\widetilde u$  from $\operatorname{Ran} \Lambda$ into $\cF$ so that
condition {\rm (i$^\prime$)} in Definition \ref{AndoTuple} holds in the somewhat weaker form
\begin{enumerate}
\item[(i$^{\prime \prime}$)]
 $PU\Lambda D_TT_2+P^\perp \Lambda D_T=  U^* P^\perp \Lambda D_T T_1 + U^* P U \Lambda D_T = \widetilde  u D_T$.
\end{enumerate}
 \end{proposition}

 \begin{proof}  We compute
 \begin{align*}
&  \left(PU\Lambda D_T T_2+P^\perp \Lambda D_T \right)^* (PU\Lambda D_T T_2+P^\perp \Lambda D_T) \\
& \quad \quad =  T_2^* D_T \Lambda^* U^* P U \Lambda D_T T_2 + D_T \Lambda^* P^\perp \Lambda D_T \\
& \quad \quad = T_2^* D_{T_1}^2 T_2 + D_{T_2}^2 \text{ (by condition (ii) in Definition \ref{AndoTuple})} \\
& \quad \quad = T_2^* (I - T_1^* T_1) T_2 + (I - T_2^* T_2) = I - T_2^* T_1^* T_1 T_2 = I - T^* T =
D_T \Lambda^* \Lambda D_T.
\end{align*}
From this it follows that there is an isometry $\widetilde u \colon \operatorname{Ran} \Lambda \to \cF$
so that $PU\Lambda D_T T_2+P^\perp \Lambda D_T = \widetilde u \Lambda D_T$ giving us equality of
the first and third term in (i$^{\prime \prime}$).
Equality of the first two terms in (i$^{\prime \prime}$) is a consequence of condition (i) (the Commutativity Condition)
in the definition of Type II And\^o tuple (Definition \ref{AndoTuple}) and
(i$^{\prime \prime}$) follows.
  \end{proof}

\chapter[Pseudo-commuting contractive lifts]{Pseudo-commuting contractive lifts of commuting contractive operator-pairs}
\label{S:pcc}

\section{Compressed And\^o lifts versus pseudo-commuting contractive lifts}

Given a commuting contractive operator-pair $(T_1, T_2)$ on $\cH$ and a And\^o lift $(\bPi, \bV_1, \bV_2)$ of $(T_1, T_2)$ on $\bcK$,  as we saw in the proof of
Theorem \ref{T:Dmodel},    it is always possible to restrict to the subspace
$$
\bcK_0 : = \bigvee_{n_1, n_2 \ge 0} \bV^{n_1} \bV^{n_2} \operatorname{Ran} \bPi  \subset \bcK
$$
to get a minimal And\^o lift $(\bPi_0, \bV_{0,1}, \bV_{0,2})$ of $(T_1, T_2)$, where we define $\bPi_0 \colon \cH \to \bcK_0$ and $\bV_{0,1}, \bV_{0,2}$ on
$\bcK_0$ via
\begin{equation}   \label{lift0}
\bPi_0 h = \bPi h \in \operatorname{Ran} \bPi \subset \bcK_0 \text{ for } h \in \cH, \quad
\bV_{0,1} = \bV_1|_{\bcK_0}, \quad \bV_{0,2} = \bV_2|_{\bcK_0}.
\end{equation}
If we are interested only in the product contraction  $T:= T_1 T_2$, by introducing the in principle even smaller subspace
\begin{equation}  \label{bcK00}
\bcK_{00} := \bigvee_{n \ge 0}  \bV_1^n \bV_2^n \operatorname{Ran} \bPi \subset \bcK_0 \subset \bcK,
\end{equation}
we can find a minimal Sz.-Nagy--Foias lift $(\bPi_{00}, \bV_{00})$ for the product contraction operator $T$ by setting $\bPi_{00} \colon \cH \to \bcK_{00}$
and $\bV_{00}$  on $\bcK_{00}$ equal to
\begin{equation}  \label{lift00}
\bPi_{00} h = \bPi h \in  \operatorname{Ran} \bPi \subset \bcK_{00} \text{ for } h \in \cH, \quad
\bV_{00} = \bV_1 \bV_2 |_{\bcK_{00}}.
\end{equation}
 Note that it is always the case that $\bcK_0$ is jointly invariant for $(\bV_1, \bV_2)$ and that  $\bcK_{00}$ is invariant for the product $\bV_1 \bV_2$.  However
the case where $\bcK_{00}$ is invariant for $\bV_1$ and $\bV_2$ individually is the special situation studied in Section \ref{S:StrongMinimality} where the minimal And\^o lift of $(T_1, T_2)$
given by \eqref{lift0} is actually {\em strongly minimal}  and $\bcK_0 = \bcK_{00}$.
Nevertheless we show here that in the general situation it is still of interest to consider the compressions
$\boldW_1:= P_{\bcK_{00} }\bV_1 |_{\bcK_{00}}, \boldW_2:= P_{\bcK_{00}} \bV_2 |_{\bcK_{00}}$ of $\bV_1, \bV_2$ to $\bcK_{00}$ even though when this is done
the compressed pair $(\boldW_1, \boldW_2)$ on $\bcK_{00}$ may not inherit the commuting and isometric properties of the original pair $(\bV_1, \bV_2)$ on $\bcK$.  
Before continuing this analysis, it is useful to have the
following more flexible definition of the {\em compression of an And\^o lift of $(T_1, T_2)$ to an embedded Sz.-Nagy--Foias lift for the product contraction $T = T_1 T_2$},  
which we shall refer to as simply a {\em minimal Sz.-Nagy--Foias compression of an And\^o lift of $(T_1, T_2)$} for short.

\begin{definition}  \label{D:embedded-lift}
Suppose that  $(\bPi, \bV_1, \bV_2)$ is an And\^o lift
of the commuting contractive operator-pair $(T_1, T_2)$ on $\bcK$ with embedded minimal isometric lift  $(\bPi_{00}, \bV_{00})$ of the product contraction operator $T = T_1 T_2$  
given by \eqref{lift00}.
Suppose that $\Pi \colon \cH \to \cK$ is an isometric embedding and $V$ is an isometry on another Hilbert space $\cK$ such that $(\Pi, V)$ is a minimal isometric lift   
of the product contraction  $T = T_1 T_2$.   
 By uniqueness of  Sz.-Nagy--Foias minimal isometric lift, there is a unitary operator $\tau \colon \cK \to \bcK_{00}$  so that
$$
\tau \Pi = \bPi_{00}, \quad \tau V = \bV_{00} \tau.
$$
Let us also view $\tau$ as an isometry from $\cK$ into $\bcK$ with final space equal to $\bcK_{00}$ depending on the context.
Define operators $\Pi \colon \cH \to \cK$ and $\bbW_1, \bbW_2\, \bW$ on $\cK$  by
$$
\Pi = \tau^* \bPi, \quad
\bbW_1 = \tau^* \bV_1 \tau, \quad \bbW_2 = \tau^* \bV_2 \tau, \quad \bbW =  \tau^* \bV_1 \bV_2 \tau =  \tau^* \bV_{00} \tau = V.
$$
Then we say that the collection {\em $(\Pi, \bbW_1, \bbW_2, V)$ is the compression of the And\^o lift $(\bPi, \bV_1, \bV_2, \bV_1 \bV_2)$ of $(T_1, T_2,  T = T_1 T_2)$ 
to the minimal Sz.-Nagy--Foias lift $(\Pi,  V)$ of $T$.}
\end{definition}

It  turns out that such compressed And\^o lifts to immersed minimal Sz.-Nagy--Foias lifts have an intrinsic characterization independent of any reference to having
a dilation to some And\^o lift. 
For further discussion, the following formal definitions will be useful.

\begin{definition}\label{D:pcc}   

\noindent
\textbf{1.}   Suppose that $(\bbW_1, \bbW_2, \bbW)$ is a triple of operators on the Hilbert space $\cK$.  We say that $(\bbW_1, \bbW_2, \bbW)$ is a 
{\em pseudo-commuting contractive operator-triple} \index{pseudo-commuting contractive operator-triple} if:
\begin{enumerate}
\item[(i)] $\bbW_1$, $\bbW_2$ are contractions while $\bbW$ is an isometry.
\item[(ii)] Both $\bbW_1$ and $\bbW_2$ commute with $\bbW$
(but not necessarily with each other),
\item[(iii)] $\bbW_1=\bbW_2^* \bbW$.
\end{enumerate}
We shall say that $(\bbW_1, \bbW_2, \bbW)$ is  a {\em pseudo-commuting algebraic triple} if $(\bbW_1, \bbW_2, \bbW)$ satisfies conditions (ii) and (iii) as above, but condition
$(i)$ is weakened to
\begin{enumerate} 
\item[(i$^\prime$)]  $\bbW$ is an isometry,
\end{enumerate}
i.e., if  $\bbW_1$ and $\bbW_2$ are now only required to be bounded operators on $\cK$ rather than contractions.

\smallskip

\noindent
\textbf{2.}  Suppose that $(T_1, T_2)$ is a commuting contractive operator-pair on a Hilbert space $\cH$, $\Pi \colon \cH \to \cK$ is an isometric embedding of $\cH$ into $\cK$
and that $(\bbW_1, \bbW_2, \bbW)$ is a pseudo-commuting contractive operator-triple on $\cK$.  
We shall say that $(\Pi, \bbW_1, \bbW_2, \bbW)$ is a {\em pseudo-commuting contractive lift} \index{pseudo-commuting contractive lift}
of $(T_1, T_2)$ if $(W_1, W_2, W)$ is a pseudo-commuting contractive operator-triple on
$\cK$ and in addition:
\begin{enumerate}
\item[(iv)] $(\Pi, \bbW_1, \bbW_2, \bbW)$ is a lift of $(T_1, T_2, T:= T_1 T_2)$ in the sense that 
$$
(\bbW_1^*,\bbW_2^*,\bbW^*)\Pi=\Pi(T_1^*,T_2^*,T_1^*T_2^*)
$$
(in particular, $(\Pi, \bbW)$ is an isometric lift of $T = T_1 T_2$), and in addition
\item[(v)] $(\Pi, \bbW)$ is a minimal isometric lift for $T = T_1 T_2$, i.e.
$$
  \cK= \bigvee_{n \ge 0} \bbW^n \operatorname{Ran} \Pi.
$$
\end{enumerate}
\end{definition}

\begin{remark}  \label{R:iii'}
Let us observe that, whenever $(\bbW_1, \bbW_2, \bbW)$ is a pseudo-commuting contractive operator-triple, 
in addition to condition (ii) in the definition one also has
\begin{enumerate}
\item[(iii$^\prime$)] $\bbW_2 = \bbW_1^* \bbW$.
\end{enumerate}
Indeed, from the identity
$\bbW_1 = \bbW_2^* \bbW$ we get
$$
  \bbW^* \bbW_1 = \bbW^* (\bbW_2^* \bbW) = (\bbW^* \bbW_2^*) \bbW = (\bbW_2^* \bbW^*) \bbW = \bbW_2^* (\bbW^* \bbW) = \bbW_2^*.
$$
\end{remark}

The next result gives  the promised intrinsic characterization of And\^o lifts of $(T_1, T_2)$ compressed to a minimal Sz.-Nagy--Foias lift of $T = T_1 T_2$, namely, they are the same as pseudo-commuting contractive
lifts  of $(T_1, T_2)$ defined as above with no reference to any And\^o lift of  $(T_1, T_2)$.

\begin{theorem} \label{T:comp=pcc} Suppose that $\Pi \colon \cH \to \cK$ is an isometry, $(\bbW_1, \bbW_2, \bbW)$ is a triple of operators on $\cK$, and
$(T_1, T_2)$ is a commuting contractive pair on $\cH$.
Then $(\Pi, \bbW_1, \bbW_2, \bW)$ is the compression of an And\^o lift of $(T_1, T_2)$ to an embedded minimal Sz.-Nagy--Foias lift of $T$ as in Definition \ref{D:embedded-lift}
 if and only if
$(\Pi, W_1, W_2, W)$ is a pseudo-commuting contractive lift of $(T_1, T_2)$ as in Definition \ref{D:pcc}.
\end{theorem}

\begin{proof}   We suppose first that  $(\Pi, \bbW_1, \bbW_2, \bbW)$ is the compression  of the And\^o lift of  $(T_1, T_2)$ to some minimal Sz.-Nagy--Foias lift $(\Pi, V)$
of $T$.  In detail, 
this means that  there is a And\^o lift $(\bPi, \bV_1, \bV_2)$ (say $\bPi \colon \cH \to \bcK$ and $\bV_1, \bV_2$ are commuting isometries on $\bcK$)
and a minimal isometric lift $(\Pi, V)$ of the product contraction $T = T_1 T_2$ (say $\Pi \colon \cH \to \cK$ and $V$ is an isometry on $\cK$)
and an isometry $\tau \colon \cK \to \bcK$ with range equal to $\bcK_{00}$ as in \eqref{bcK00} so that 
$$
\Pi = \tau^* \bPi, \bbW_1 = \tau^* \bV_1 \tau, \quad \bbW_2 = \tau^* \bV_2 \tau, \quad \bbW = \tau^* \bV_1 \bV_2 \tau = V.
$$
Since $\bV_1, \bV_2, \bV= \bV_1 \bV_2$ are all isometries  and furthermore $\bcK_{00}$ is invariant for $\bV:= \bV_1 \bV_2$, (i) follows.

Condition (ii) follows from the fact that $\bV_1$ and $\bV_2$ commute with $\bV$ and again $\bcK_{00}$ is invariant for $\bV$.

Since $\bV = \bV_1 \bV_2 = \bV_2 \bV_1$ and $\bV_1$ is isometric, we see that we can solve for $\bV_1$ as $\bV_1 = \bV_2^* \bV$.  The formulas
for $\bbW, \bbW_1, \bbW_2$ combined with the fact that $\bcK_{00}$ is invariant for $\bV$ then leads us to condition (iii).

Since $(\bPi, \bV_1, \bV_2)$ is a lift of $(T_1, T_2)$, we know that
$$ 
(\bV_1^*, \bV_2^*, \bV^*) \bPi = \bPi (T_1^*, T_2^*, T^*). 
$$
which is actually the same as
$$
(\bV_1^*, \bV_2^*, \bV^*) \bPi_{00} = \bPi_{00} (T_1^*, T_2^*, T^*).
$$
Recalling now that $\bPi_{00} = \tau \Pi$ and that $\tau \colon \cK \to \bcK_{00}$ is unitary, this last expression becomes
$$
 \tau^* (\bV_1^*, \bV_2^*, \bV^*) \tau \Pi = \Pi (T_1^*, T_2^*, T^*)
 $$
 and (iv) follows.
 
Finally,  by construction $\bbW = \tau^* \bV_{00} \tau = V$ where by definition of compressed And\^o tuple $(\Pi, V)$ is a minimal lift of $T$, from
which we see that (v) holds. 
  This completes the proof of {\em compressed And\^o lift $\Rightarrow$ pseudo-commuting contractive lift}.

We postpone the proof of the converse ({\em pseudo-commuting contractive lift $\Rightarrow$ compressed And\^o lift}) until after
we develop the Douglas-model for compressed And\^o lifts in the next section (see Corollary \ref{C:FundOp} below).
\end{proof}

\section[Douglas-model pseudo-commuting contractive lifts]{Douglas-model pseudo-commuting contractive lifts}

We know by Theorem \ref{T:Dmodel} that minimal And\^o lifts for a commuting contractive operator pair $(T_1, T_2)$ can be given up to unitary equivalence in the 
Douglas-model form \eqref{DougAndoMod} specified by a Type I And\^o tuple $(\cF_*, \Lambda_*, P_*,  U_*)$ for $(T_1^*, T_2^*)$.  Identifying the embedded minimal
Sz.-Nagy--Foias lift space $\bcK_{D,00}$ for the product contraction inside $\bcK_D$ and then identifying this with the Douglas-model isometric lift $(\Pi_D, V_D)$
for $T$ then leads to a Douglas model for a compressed And\^o lift of $(T_1, T_2)$ as follows.

\begin{theorem}\label{Thm:DPseudo}
Given a commuting contractive operator-pair $(T_1, T_2)$ on $\cH$,  let $\cK_D = \sbm{ H^2(\cD_{T^*}) \\  \cQ_{T^*} }$ be the Douglas isometric-lift model space for $T$,
let $\Pi_D \colon \cH \to \cK_D$ be the Douglas isometric embedding operator $\Pi_D = \sbm{ \cO_{D_{T^*}, T^*} \\ Q_{T^*}}$, let $(G_1, G_2)$ be the 
Fundamental-Operator pair for $(T_1^*, T_2^*)$, and define operators $W_{\flat 1}, W_{\flat 2}, W_D$ on $\cQ_{T^*}$ as in Theorem \ref{T:flats}.
Finally define operators $W_{D,1}, W_{D,2}, V_D$ on $\cK_D$  according to the formulas
\begin{equation}   \label{Dmodel-PCC}
(\bbW_{D,1}, \bbW_{D,2}, V_D) = \left( \begin{bmatrix} M_{G_1^* + z G_2} & 0 \\ 0 & W_{\flat 1} \end{bmatrix}, 
\begin{bmatrix} M_{G_2^* + z G_1} & 0 \\ 0 & W_{\flat 2} \end{bmatrix},   \begin{bmatrix} M^{\cD_{T^*}}_z & 0 \\ 0 & W_D \end{bmatrix} \right).
\end{equation}
Then $( \Pi_D, \bbW_{D,1}, \bbW_{D,2}, V_D)$ is the compression of the Douglas-model And\^o lift of $(T_1, T_2)$ to the embedded Douglas-model Sz.-Nagy--Foias 
lift $(\Pi, V_D)$ and hence also is a pseudo-commuting contractive lift of $(T_1, T_2)$.

Conversely,  suppose that  $(\Pi_D, \bbW_1, \bbW_2, V_D)$ a pseudo-commuting contractive lift of $(T_1, T_2)$ such that 
$$
(\Pi_D, V_D) = \left( \begin{bmatrix} \cO_{D_{T^*}, T^*} \\ \cQ_{T^*} \end{bmatrix} \colon \cH \to \cK_D, \begin{bmatrix}  M^{\cD_{T^*}}_z & 0 \\ 0 & W_D \end{bmatrix}  
\text{ on } \cK_D \right)
$$
is the Douglas-model minimal isometric lift of $T$ on $\cK_D = \sbm{ H^2(\cD_{T^*}) \\ \cQ_{T^*}}$.
Then necessarily $(\bbW_1, \bbW_2) = (\bbW_{D,1}, \bbW_{D,2})$ is given as  in formula \eqref{Dmodel-PCC}.  
\end{theorem}

\begin{proof}
Let $(\cF_*,\Lambda_*,P_*,U_*)$ be a Type I And\^o tuple for $(T_1^*,T_2^*)$.  Then the Douglas-model And\^o lift $(\bPi_D, \bV_{D,1}, \bV_{D,2})$ on $\bcK_D$
 associated with this And\^o tuple is defined as in \eqref{DougAndoMod}:
\begin{align*}
 & \bcK_D = \begin{bmatrix} H^2(\cF_*) \\ \cQ_{T^*} \end{bmatrix}, \notag  \\
 & (\bV_{D,1}, \bV_{D,2}) =  \left( \begin{bmatrix} M_{U_*^*P_*^\perp + z U_*^*P_*} & 0 \\ 0 & W_{\flat 1} \end{bmatrix},
 \begin{bmatrix} M_{P_* U_* + z P_*^\perp U_*} & 0 \\ 0 & W_{\flat 2} \end{bmatrix} \right) \text{ acting on } \bcK_D,  \notag \\
 &   \bPi_D= \begin{bmatrix} (I_{H^2} \otimes \Lambda_*)  \cO_{D_{T^*}, T^*} \\ Q_{T^*} \end{bmatrix} \colon \cH \to \bcK_D.
 \end{align*}
 Then, as seen in the proof of Theorem \ref{T:Dmodel}, the Sz.-Nagy--Foias isometric lift of $T$ embedded in the And\^o lift $(\bPi, \bV_1, \bV_2)$ of $(T_1, T_2)$
 is $(\bPi_{D,00}, \bV_{D,00})$ where $\bPi_{D,00}$ is the same as $\bPi$ but with codomain taken to be $\bcK_{D,00} = \sbm{ H^2(\operatorname{Ran} \Lambda_*)
 \\ \cQ_{T^*} }$ and where $\bV_{D,00} = \bV_{D1} \bV_{D2}|_{\bcK_{D,00}}$, and furthermore, the unique unitary operator $\tau \colon \cK_D \to \bcK_{D, 00}$ implementing the unitary equivalence
 between the two minimal isometric lifts $(\Pi_D, V_D)$ and $(\bPi_{D,00}, \bV_{D,00})$  of $T =T_1 T_2$ is
 $$
  \tau = \begin{bmatrix} I_{H^2} \otimes \Lambda_* & 0 \\ 0 & I_{\cQ_{T^*}} \end{bmatrix} \colon \cK_D \to \bcK_{D,00}.
 $$
 Hence the associated Douglas-model compressed And\^o lift (obtained by using the Douglas model for the the And\^o lift $(T_1, T_2)$
 as well as Douglas model for the Sz.-Nagy--Foias lift of the product contraction $T = T_1 T_2$)  is given by
  \begin{align}
& (\bbW_{D,1},\bbW_{D,2},V_D):=\tau^*(V_1,V_2,V_1V_2) \tau:=  \notag \\
& \left( \begin{bmatrix} M_{\Lambda_*^*(U_*^*P_*^\perp + z U_*^*P_*)\Lambda_*} & 0 \\ 0 & W_{\flat 1} \end{bmatrix},
 \begin{bmatrix} M_{\Lambda_*^*(P_* U_*^* + z P_*^\perp U_* ) \Lambda_*} & 0 \\ 0 & W_{\flat 2} \end{bmatrix},\begin{bmatrix} M_z & 0 \\ 0 & W_D \end{bmatrix} \right) \notag \\ 
 &= \left( \begin{bmatrix} M_{G_1^* + z G_2} & 0 \\ 0 & W_{\flat 1} \end{bmatrix},
 \begin{bmatrix} M_{G_2^*+ z G_1} & 0 \\ 0 & W_{\flat 2} \end{bmatrix},\begin{bmatrix} M_z & 0 \\ 0 & W_D \end{bmatrix} \right),
 \label{Dpcc}
 \end{align}
 where here we make use of the connection between the Fundamental-Operator pair of $(T_1^*, T_2^*)$ and a Type I And\^o tuple $(\cF_*, \Lambda_*, P_*, U_*)$ for 
 $(T_1^*, T_2^*)$ (see \eqref{choices}) coming out of the Second Proof of Theorem \ref{T:FundOps}:
 $$
 (G_1,G_2)=(\Lambda_*^*P_*^\perp U_*\Lambda_*,\Lambda_*^*U_*^*P_*\Lambda_*)
 $$
 Then by definition the model triple \eqref{Dmodel-PCC} is a compressed And\^o lift, and hence also, by the part of Theorem \ref{T:comp=pcc} already proved,
 is also a pseudo-commuting contractive lift of $(T_1, T_2)$. 
 
Conversely, suppose that  the operator triple $\left(\Pi_D, \bW_1, \bW_2, \sbm{M_z^{\cD_{T^*}} & 0 \\ 0 & W_D} \right)$ is a pseudo-commuting contractive lift of $(T_1, T_2)$.
 We break the proof into two steps:
 
 \smallskip
 
 \noindent
 \textbf{Step 1. Show:} {\sl If $\left(\bbW_1, \bbW_2, \sbm{ M_z^{\cD_{T^*}} & 0 \\ 0 & W_D}\right)$ is a pseudo-commuting contractive triple, then
 \begin{equation} \label{Step1Ws}
 (\bbW_1, \bbW_2) = \left( \begin{bmatrix}M_{G_1^* + z G_2} & 0 \\ 0 & \widetilde W_2^*W_D  \end{bmatrix}, 
 \begin{bmatrix} M_{G_2^* + z G_1} & 0 \\ 0 & \widetilde W_2 \end{bmatrix} \right)
 \end{equation}
 for some operators $G_1, G_2 \in \cB(\cD_{T^*})$ such that $\varphi_1(z) := G_1^* + z G_2$ and $\varphi_2(z): = G_2^* + z G_1$ are contractive analytic functions on ${\mathbb D}$
 and $\widetilde W_2$ is some contraction operator on $\cQ_{T^*}$ commuting with $W_D$.}

\smallskip

\noindent
\textbf{Proof of Step 1.}  Assume that $\left(\bbW_1, \bbW_2, \sbm{ M_z & 0 \\ 0 & W_D}\right)$ is a pseudo-commuting contractive triple on 
$\cK_D:= \sbm{ H^2(\cD_{T^*}) \\ \cQ_{T^*} }$. As a first step we write out $\bbW_1, \bbW_2$ as block $2 \times 2$ matrices with respect to the decomposition of $\cK_D$ as 
$\sbm{ H^2(\cD_{T^*}) \\ \cQ_{T^*}}$:
$$
   \bbW_j = \begin{bmatrix} \bbW_{j,11} & \bbW_{j,12}  \\ \bbW_{j, 21} & \bbW_{j, 22} \end{bmatrix} \text{ for } j = 1,2.
$$
By Axiom (ii) in Definition \ref{D:pcc} combined with Lemma \ref{L:AuxLemma}, we see immediately that $\bbW_{j,12} = 0$ for $j=1,2$ and the commutativity of $\bbW_j$ with
$\sbm{ M_z^{\cD_{T^*}} & 0 \\ 0 & W_D}$ comes down to
\begin{equation}  \label{WjWcom}
 \begin{bmatrix} \bbW_{j,11} M_z^{\cD_{T^*}} & 0 \\ \bbW_{j,21} M_z^{\cD_{T^*}} & \bbW_{j,22} W_D \end{bmatrix} = 
 \begin{bmatrix} M_z^{\cD_{T^*}}  \bbW_{j,11} & 0 \\ W_D \bbW_{j,21} & W_D \bbW_{j,22} \end{bmatrix}.
 \end{equation}
 By Axiom (iii) in Definition \ref{D:pcc} we know that $\bbW_1 = \bbW_2^* \sbm{ M_z^{\cD_{T^*}} & 0 \\ 0 & W_D}$;  writing this out in detail gives
 \begin{equation}  \label{W1=W2*W}
 \begin{bmatrix} \bbW_{1,11} & 0 \\ \bbW_{1,21} & \bbW_{1,22} \end{bmatrix} = 
 \begin{bmatrix} \bbW_{2,11}^* M_z^{\cD_{T^*}} & \bbW_{2, 21}^* W_D \\ 0 & \bbW_{2,22}^* W_D \end{bmatrix}.
 \end{equation}
 From the $(2,1)$ entry we see that $\bbW_{1,21} = 0$ and from the $(1,2)$ entry we see that $\bbW_{2,21} = 0$ since $W_D$ is unitary.   Thus  both $\bbW_1$ an $\bbW_2$ are diagonal
 $$
   \bbW_j = \begin{bmatrix} \bbW_{j,11} & 0 \\ 0 &  \bbW_{j,22} \end{bmatrix} \text{ for } j = 1,2.
 $$
 and we see from \eqref{WjWcom} that
 \begin{equation}  \label{diag-com}
  \bbW_{j,11} M_z^{\cD_{T^*}} = M_z^{\cD_{T^*}} \bbW_{j,11}, \quad \bbW_{j,22} W_D = W_D \bbW_{j,22} \text{ for } j = 1,2.
 \end{equation}
 From the first relation in \eqref{diag-com}, by standard Hardy-space theory we conclude that $W_{j,11}$ must be a multiplication operator $M_{\varphi_j} \colon h(z) \mapsto
 \varphi_j(z) h(z)$ for a contractive analytic function
 $\varphi_j(z) = \sum_{k = 0}^\infty \varphi_{j,k}   z^\ell$ holomorphic on the unit disk ${\mathbb D}$,
  where the Taylor coefficients $\varphi_{j,k}$ ($j=1,2$, $k = 0,1,2,\dots$) are operators on $\cD_{T^*}$.  From \eqref{W1=W2*W}
 we see that $M_{\varphi_1}  = M_{\varphi_2}^* M_z^{\cD_{T^*}}$.  A Taylor-series argument then shows that the pair $(\varphi_1(z), \varphi_2(z))$ must have the coupled pencil form
 $$
 \varphi_1(z) = G_1^* + z G_2, \quad \varphi_2(z) = G_2^* + z G_1.
 $$
 for some operators $G_1, G_2$ on $\cD_{T^*}$.
 
 Finally, from \eqref{WjWcom} we see that both $\widetilde W_1 := W_{1,22}$ and $\widetilde W_2:= W_{2,22}$ commute with $W_D$.  By the Fuglede-Putnam theorem,
 it follows that each of $\widetilde W_1$ and $\widetilde W_2$ also commutes with $W_D^*$.   If we let $\widetilde W_2$ be any contractive operator commuting with the
 unitary operator $W_D$ and then set $\widetilde W_1 =  \widetilde W_2^* W_D$, then $\widetilde W_1$ automatically commutes with $W_D$ and this is the general form
 for a pseudo-commuting contractive triple $(W_1, W_2, W_D)$ with last component equal to the unitary operator $W_D$.
This completes the proof of Step 1.
\smallskip

\noindent
\textbf{Step 2.  Show:}  {\sl If $G_1, G_2, \widetilde W_1 = W_D^* \widetilde W_2, \widetilde W_2$ are as in Step 1 and 
$$
\left(\Pi_D, \begin{bmatrix}  M_{G_1^* + z G_2} & 0 \\ 0 & \widetilde W_1 \end{bmatrix},  \begin{bmatrix} M_{G_2^* + z G_1} & 0 \\ 0 & \widetilde W_2 \end{bmatrix}  \right)
$$
is a lift of $(T_1, T_2)$, then 
  $(G_1, G_2)$ is the Fundamental-Operator pair for $(T_1, T_2)$ and $(\widetilde W_1, \widetilde W_2) = (W_{\flat 1}, W_{\flat 2})$ is the canonical pair of
  unitaries on the space $\cQ_{T^*}$ associated with the contractive operator pair $(T_1, T_2)$ as in Theorem \ref{T:flats}.}
  
  \smallskip

\noindent
\textbf{Proof of Step 2.}    The hypothesis that $\left(\Pi_D, \sbm{ M_{G_1^* + z G_2} & 0 \\ 0 & \widetilde W_1}, \sbm{M_{G_2^* + z G_1} & 0 \\ 0 & \widetilde W_2} \right)$
is a lift of $(T_1, T_2)$ means that
$$
\left( \sbm{ M_{G_1^* + z G_2} & 0 \\ 0 & \widetilde W_1}^*, \sbm{M_{G_2^* + z G_1} & 0 \\ 0 & \widetilde W_2}^*\right) \sbm{ \cO_{D_{T^*}, T^*} \\ Q_{T^*} }
 =   \sbm{ \cO_{D_{T^*}, T^*} \\ Q_{T^*} }\left( T_1^*, T_2^* \right)
$$
which breaks apart into the set of conditions
\begin{align}
&M_{G_1^* + z G_2}^* \cO_{D_{T^*}, T^*} = \cO_{D_{T^*}, T^*} T_1^*, \quad M_{G_2^* + z G_1}^* \cO_{D_{T^*}, T^*} \cO_{D_{T^*}, T^*} = \cO_{D_{T^*}, T^*} T_2^*  \label{Lift1} \\
&\widetilde W_1^* Q_{T^*} = Q_{T^*} T_1^*, \quad \widetilde W_2^* Q_{T^*} = Q_{T^*} T_2^*.  \label{Lift2}
\end{align}
By Theorem \ref{T:flats} it is immediate from \eqref{Lift2} that 
$$
   (\widetilde W_1, \widetilde W_2) = (W_{\flat 1}, W_{2, \flat})
$$
as claimed.  As for the first equation in \eqref{Lift1}, note that each side is an operator from $\cH$ into $H^2(\cD_{T^*})$.  Applying each side to a fixed vector in $\cH$ gives
$$
\sum_{k=0}^\infty \left( G_1^* D_{T^*} T^{*k}  + G_2 D_{T^*} T^{* k+1} \right) h z^k = \sum_{k=0}^\infty D_{T^*} T^{*k} T_1^* h z^k.
$$
Equating Taylor coefficients and cancelling off the vector $h$ gives us the system of operator equations
$$
G_1^* D_{T^*} T^{*k}  + G_2 D_{T^*} T^{* k+1} = D_{T^*} T^{*k} T_1^*. \text{ for all } k=0,1,2,\dots.
$$
As $T_1^*$ commutes with $T^*$, we can rewrite this with a common right factor of $T^{*k}$:
$$
(G_1^* D_{T^*}    + G_2 D_{T^*} T^*)  T^{* k} = (D_{T^*} T_1^*) T^{*k} \text{ for } k=0,1,2,\dots.
$$
For this to hold for all $k=0,1,2,\dots$, it is now clear that it suffices that it hold for $k=0$:
$$
G_1^* D_{T^*}    + G_2 D_{T^*} T^* =  D_{T^*} T_1^*,
$$
i.e., $(G_1, G_2)$ is a solution of the first of equations \eqref{FundEqns} (with $(G_1, G_2)$ in place of $(F_1, F_2)$ and with $(T_1^*, T_2^*)$ in place of $(T_1, T_2)$.
A similar analysis starting with the second of equations \eqref{Lift1} leads to the equation
$$
G_2^* D_{T^*}  + G_1 D_{T^*} T^* =  D_{T^*} T_2^*,
$$
i.e., $(G_1, G_2)$ also solves the second equation in \eqref{FundEqns} (again with $(G_1, G_2)$ in place of $(F_1, F_2)$ and $(T_1^*, T_2^*)$ in place of $(T_1, T_2)$).
By Definition \ref{D:FundOps} and the uniqueness result Theorem \ref{T:FundOps}, it follows that $(G_1, G_2)$ turns out to be the Fundamental-Operator pair for
$(T_1^*, T_2^*)$ as claimed.
\end{proof}

\begin{remark} \label{R:Ddirect} The proof of Theorem \ref{Thm:DPseudo} made reference to the construction of the triple $(\bbW_{D,1}, \bbW_{D,2}, V_D)$ 
as the compression (in the sense of Definition 
\ref{D:embedded-lift} of a Douglas-model And\^o lift of $(T_1, T_2)$ to the embedding of a Douglas-model minimal isometric lift of the product contraction $T = T_1 T_2$
 to conclude that the Douglas-model triple $(\bbW_{D,1}, \bbW_{D,2}, V_D)$ is a pseudo-commuting contractive lift.  However it is also of interest to see
if it is possible to check this directly from the formula \eqref{Dmodel-PCC} which a priori has no reference to the existence of an And\^o lift of $(T_1, T_2)$.  
To get the lifting property, the proof of the direct statement uses the connection of the fundamental operator pair $(G_1, G_2)$ for $(T_1^*, T_2^*)$
with the existence of a Type I And\^o tuple $(\cF_*, \Lambda_*, P_*, U_*)$ defining a Douglas-model And\^o lift of $(T_1, T_2)$.  However in the proof of the  converse
it is shown how to get a more direct statement:  the lifting property for $(\Pi_D, \bbW_{D,1}, \bbW_{D,2}, V_D)$ is associated with the Fundamental-Operator system of equations
\eqref{FundEqns} (with $(T_1^*, T_2^*)$ in place of $(T_1, T_2)$ and with $(G_1, G_2)$ in place of $(F_1, F_2)$).  Conditions (ii), (iii) in Definition \ref{D:pcc} follow by a direct check
from the formulas \eqref{Dmodel-PCC} and we conclude that one can show directly that $(\Pi_D, \bbW_{D,1}, \bbW_{D,2}, V_D)$ is at least a {\em pseudo-commuting algebraic lift}
(as defined in Definition \ref{D:pcc}) of $(T_1, T_2, T_1T_2)$.  To show that $\bbW_{D,1}$ and $\bbW_{D,2}$ are contraction operators;  
to our knowledge the only argument for showing this
goes through the fact that  $\bbW_{D,1}$ and $\bbW_{D,2}$ are compressions of the commuting isometries $(\bV_1, \bV_2)$ in a Douglas-model And\^o lift of $(T_1, T_2)$:
\begin{align*}
& \bbW_{D,1}:= M_{G_1^* + z G_2} =( I_{H^2} \otimes \Lambda_*^*) M_{U_*^* P_*^\perp + z U^*_* P_*}(I_{H^2} \otimes \Lambda_*), \\
& \bbW_{D,2}:= M_{G_2^* + z G_1} = (I_{H^2} \otimes \Lambda_*^*) M_{P_* U_* + z P_*^\perp U_*} (I_{H^2} \otimes \Lambda_*).
\end{align*}
\end{remark}

We are now ready to complete the proof of  \ref{T:comp=pcc}.

\begin{corollary}  \label{C:FundOp}  Suppose that $(T_1, T_2)$ is a commuting contractive pair on $\cH$ and suppose that
$\Pi \colon \cH \to \cK$ is an embedding of $\cH$ into $\cK$ and that $(\bbW_1, \bbW_2, \bbW)$ is a triple of operators on $\cK$. Then
the following are equivalent:
\begin{enumerate}
\item $(\Pi, \bbW_1, \bbW_2, \bbW)$ is the compression of an And\^o lift of $(T_1, T_2, T_1 T_2)$ to an immersed minimal Sz.-Nagy--Foias lift of $T$, 
i.e., $(\Pi, \bbW)$ is a minimal Sz.-Nagy--Foias (isometric) lift of $T = T_1 T_2$ on $\cK$ and there is an And\^o lift
$(\bPi, \bV_1, \bV_2)$ of $(T_1, T_2)$  on $\bcK$ together with a unitary embedding 
$$
\tau \colon \cK \to \bcK_{00}:= \bigvee_{n \ge 0} \bV_1^n \bV_2^n \operatorname{Ran} \bPi \subset \bcK
$$
so that
$$
\tau \Pi = \bPi, \quad \tau^* (\bV_1, \bV_2, \bV_1 \bV_2) \tau = (\bbW_1, \bbW_2, \bbW).
$$

\item $(\Pi, \bbW_1, \bbW_2, \bbW)$ is a pseudo-commuting contractive lift of $(T_1, T_2, T_1T_2)$. 

\item $(\Pi, \bbW_1, \bbW_2, \bbW)$ is unitarily equivalent to the Douglas-model pseudo-commuta\-tive contractive lift  $\left(\Pi_D, \sbm{M_{G_1^* + z G_2} & 0 \\ 0 & W_{\flat 1}},
\sbm{M_{G_2^* + z G_1} & 0 \\ 0 & W_{\flat 2}}, \sbm{M_z & 0 \\ 0 & W_D}\right)$ given by Theorem \ref{Thm:DPseudo}.
\end{enumerate}

Furthermore, if $(\Pi, \bbW_1, \bbW_2, \bbW)$ are $(\Pi', \bbW_1', \bbW_2', \bbW')$ are two pseudo-commuting contractive lifts of $(T_1, T_2, T_1T_2)$ on $\cK$ and $\cK'$
respectively such that there is a unitary operator $\tau' \colon \cK \to \cK'$ such that
\begin{equation}  \label{partial-intertwine}
  \Pi' = \tau' \Pi, \quad \bbW' \tau' = \tau' \bbW,
\end{equation}
then it follows that also
\begin{equation}  \label{claim1}
 \bbW_1' \tau' = \tau' \bbW_1, \quad \bbW_2' \tau'= \tau' \bbW_2.
 \end{equation}
 \end{corollary}

\begin{proof}  We first show that (1) $\Rightarrow$ (2) $\Rightarrow$ (3) $\Rightarrow$ (1).

\smallskip

\noindent
\textbf{(1) $\Rightarrow$ (2):}  This follows from the part of Theorem \ref{T:comp=pcc} already proved above.

\smallskip

\noindent
\textbf{(2) $\Rightarrow$ (3):}  Assume (2). Then $(\Pi, \bbW)$ is a minimal isometric lift of $T = T_1T_2$ on $\cK$. By uniqueness of the minimal isometric lift for $T$,  there is a unitary $\tau \colon \cK \to \cK_D$ which brings the Sz.-Nagy--Foias lift $(\Pi, V)$ to the Douglas-model Sz.-Nagy--Foias lift:
$$
\tau \Pi = \Pi_D, \quad \tau \bbW = V_D \tau.
$$
Since $(\bbW_1, \bbW_2, \bbW)$ is a pseudo-commuting contractive triple and $\tau$ is unitary, it is easily checked that 
$$
 \tau ( \bbW_1, \bbW_2, \bbW) \tau^* = ( \tau \bbW_1 \tau^*, \tau \bbW_2 \tau^*, V_D)
$$
is also a pseudo-commuting contractive triple.   Similarly, since $(\Pi, \bbW_1, \bbW_2, \bbW)$ is a lift of $(T_1, T_2, T_1 T_2)$, it follows that
$(\tau \Pi = \Pi_D, \tau \bbW_1 \tau^*, \tau \bbW_2 \tau^*, V_D)$ is also a lift of $(T_1, T_2, T_1 T_2)$.
But as a consequence of Theorem \ref{Thm:DPseudo} we see that this then forces
$$  \tau \bbW_1 \tau^* = \sbm{M_{G_1^* + z G_2}  & 0 \\ 0 & W_{\flat 1}}, \quad
\tau \bbW_2 \tau^* = \sbm{M_{G_2^* + z G_1} & 0 \\ 0 & W_{\flat 2} },
$$
and hence $\tau$ implements a unitary equivalence of the pseudo-commuting contractive lift $(\Pi, \bbW_1, \bbW_2, \bbW)$ with $(\Pi_D, W_1^D, W_2^D, V_D)$ (notation as in \eqref{Dmodel-PCC}), and (3) follows.

\smallskip

\noindent
\textbf{(3) $\Rightarrow$ (1):}  This is part of the content of the first part of Theorem \ref{Thm:DPseudo}.

\smallskip
We verify the last part of the corollary as follows.  Suppose that $(\Pi, \bbW_1, \bbW_2, \bbW)$ are $(\Pi', \bbW_1', \bbW_2', \bbW')$ are two pseudo-commuting contractive lifts of 
$(T_1, T_2, T_1T_2)$ on $\cK$ and $\cK'$ such that \eqref{partial-intertwine} holds.  Let $\tau_D \colon \cK \to \cK_D$ be a unitary identification map bringing
the pseudo-commuting contractive lift $(\Pi, \bbW_1, \bbW_2, \bbW)$ to the Douglas model form:
\begin{align*}
& \tau_D \Pi = \Pi_D, \quad \tau_D \bbW_1 = \bbW_{D,1}  \tau_D, \quad \tau_D \bbW_2 = \bbW_{D,2} \tau_D, \quad \tau_D \bbW_D = V_D \tau_D.
\end{align*}
Set $\tau = \tau_D \tau^{\prime *}  \colon \cK' \to \cK_D$.  Then note that
\begin{align*}
& \tau \Pi' = \tau_D \Pi = \Pi_D, \\
& \tau \bbW' = \tau_D \tau^{\prime *} \bbW' = \tau_D \bbW \tau^{\prime *} = V_D \tau_D \tau^{\prime *} = V_D \tau.
\end{align*}
Thus 
$$
( \tau \Pi', \tau W_1' \tau^*, \tau W'_2 \tau^*, \tau W' \tau^*) =
( \Pi_D, \tau \bbW'_1 \tau^*, \tau \bbW'_2 \tau^*, V_D)
$$
is a pseudo commuting contractive lift of $(T_1, T_2, T_1 T_2)$.  By the converse statement in Theorem \ref{Thm:DPseudo}, we see that we must have
$$
   \tau \bbW'_1 \tau^* = \bbW_{D,1}, \quad \tau \bbW'_2 \tau^* = \bbW_{D,2}.
 $$
Thus from the definitions we have
$$
\bbW'_1 =  \tau^* \bbW_{D,1} \tau = \tau'  \tau_D^* \bbW_{D,1} \tau_D \tau^{\prime *} = \tau' \bbW_1 \tau^{\prime *}
$$
and similarly
$$
\bbW'_2 = \tau' \bbW_2 \tau^{\prime *}
$$
and \eqref{claim1} follows as wanted.
\end{proof}

\section{Sz.-Nagy--Foias-model pseudo-commuting contractive lifts}  \label{S:pccNF}
We have defined the Sz.-Nagy--Foias model minimal lift of a contraction operator $T$, as well as the Sz.-Nagy--Foias model And\^o lift of a commuting contractive
operator-pair $(T_1, T_2)$, as a simple transformation, using the unitary operator $U_{\rm NF, D} \colon \cK_D \to \cK_{\Theta_T}$ or $\bU_{\rm  NF, D} \colon \bcK_{\rm D} \to
\bcK_{\rm NF}$, of the corresponding Douglas model.  The analogous procedure applies also to the construction of  a Sz.-Nagy--Foias model pseudo commuting contractive lift of a
commuting contractive operator-pair $(T_1, T_2)$ as follows.

\begin{theorem}  \label{Thm:NFPseudo}
Let $(T_1,T_2)$ be a commuting contractive operator-pair on a Hilbert space $\cH$ and set  $T=T_1T_2$. Let 
$$
U_{\rm NF,D} =\begin{bmatrix} I_{H^2(\cD_{T^*})} & 0 \\ 0 &  \omega_{\rm NF,D} \end{bmatrix} \colon
\begin{bmatrix} H^2(\cD_{T^*}) \\ \cQ_{T^*} \end{bmatrix} \to \begin{bmatrix} H^2(\cD_{T^*}) \\ \overline{\Delta_{\Theta_T}(L^2(\cD_T))} \end{bmatrix}
$$ 
be the unitary as in \eqref{DNFintwin}. Let us define operators
\begin{align}   
& \underline{\bbW}_{\rm NF}  = (\bbW_{{\rm NF}, 1}, \bbW_{{\rm NF},2}, V_{\rm NF}) := U_{{\rm NF},D}(\bbW_{D,1},\bbW_{D,2},V_D)U_{{\rm NF}, D}^* \notag \\
&= \left( \begin{bmatrix} M_{G_1^* + z G_2} & 0 \\ 0 & W_{\sharp 1} \end{bmatrix},
 \begin{bmatrix} M_{G_2^*+ z G_1} & 0 \\ 0 & W_{\sharp 2} \end{bmatrix},
 \begin{bmatrix} M_z^{\cD_{T^*}} & 0 \\ 0 & M_\zeta|_{\overline{\Delta_{\Theta_T} L^2(\cD_T)}} \end{bmatrix} \right) \text{ on } \cK_{\Theta_T}, \notag  \\
 &   \Pi_{\rm NF} =  U_{\rm NF, D} \Pi_D:= \begin{bmatrix} \cO_{D_{T^*}, T^*} \\ \omega_{\rm NF,D} Q_{T^*} \end{bmatrix} \colon \cH \to \cK_{\Theta_T}
 \label{pccNF}
\end{align}
where $\cK_{\Theta_T}=  \sbm{ H^2(\cD_{T^*}) \\ \overline{\Delta_{\Theta_T}(L^2(\cD_T))}}$ and  
where $(\bbW_{D,1},\bbW_{D,2},V_D)$ is the Douglas-model pseudo-commuting
contractive triple as defined in (\ref{Dpcc})
(so $(G_1,$ $G_2)$ is the Fundamental-Operator pair for $(T_1^*, T_2^*)$) and the pair 
$(W_{\sharp1},W_{\sharp2})$ is the commuting unitary operator-pair as in (\ref{WandVNFs}).   Then $(\Pi_{\rm NF},  W_{\rm NF,1}, W_{\rm NF2}, V_{\rm NF})$
is a pseudo-commuting contractive lift of $(T_1, T_2, T)$, called the {\em Sz.-Nagy--Foias model pseudo-commuting contractive lift}.
\end{theorem}

\begin{proof}
We showed in the proof of Theorem \ref{Thm:DPseudo} that 
$(\Pi_D, \bbW_{D,1},\bbW_{D,2},V_D )$ is a pseudo commuting contractive lift of $(T_1,T_2)$.
Since the isometry $\Pi_{\rm NF}:\cH\to\cK_{\Theta_T}$ is given by $\Pi_{\rm NF}=U_{{\rm NF, D}}\Pi_D$, 
it then follows  that the collection
$(\Pi_{\rm NF}, \underline{\bbW}_{\rm NF})$ is unitarily equivalent to  the lift $(\Pi_{\rm D}, \underline{\bbW}_D)$ (where we set
$\underline{\bbW}_D = (\bbW_{D,1}, \bbW_{D,2}, V_D)$), and hence itself must also be a lift, of $(T_1, T_2, T_1T_2)$; in detail we have
\begin{align*}
\underline{\bbW}_{\rm NF}^* \Pi_{\rm NF} & = U_{\rm NF,D}\,  \underline{\bbW}_D^* \, U_{\rm NF,D}^* \cdot U_{\rm NF,D} \Pi_D \\
& = U_{\rm NFD} \,  \underline{\bbW}_D^* \,  \Pi_D = U_{\rm NF,D} \, \Pi_D \,  \underline{T}^*  = \Pi_{\rm NF} \,  \underline{T}^*
\end{align*}
(where here we set $\underline{T} = (T_1, T_2, T = T_1 T_2)$) thereby verifying that $( \Pi_{\rm NF}, \underline{\bbW}_{\rm NF})$ is a lift of
$\underline{T}$.  Since the pseudo-commuting contractive property is invariant under unitary equivalence, it also follows that
$\underline{\bbW}_D$ being a pseudo-commuting contractive triple implies that $\underline{\bbW}_{\rm NF}$ is also a pseudo-commuting contractive triple.
\end{proof}

\begin{remark} \label{R:DtoNF}
Let us observe that, since $(\Pi_{\rm NF}, \underline{\bbW}_{\rm NF})$  is related to $(\Pi_{\rm D}, \underline{\bbW}_D)$ via the innocuous change of coordinates in the
second coordinate $U_{\rm NF,D}$,  it is a routine exercise to see that all the results concerning $(\Pi_D, \underline{\bbW}_D)$ in Theorem \ref{Thm:DPseudo},
Remark \ref{R:Ddirect}, Corollary \ref{C:FundOp} hold equally well for $(\Pi_{\rm NF}, \underline{\bbW}_{\rm NF})$.   In particular, the Sz.-Nagy--Foias model 
pseudo-commuting contractive lift $(\Pi_{\rm NF}, \underline{\bbW}_{\rm NF})$ can also be viewed as the compression of a Sz.-Nagy--Foias model
And\^o lift to an embedded Sz.-Nagy--Foias model for a minimal Sz.-Nagy--Foias lift
$(\Pi_{\rm NF}, V_{\rm NF})$ of $T = T_1 T_2$,
thereby proving that $(\Pi_{\rm NF}, \underline{\bbW}_{\rm NF})$ is a pseudo-commuting contractive lift of
$\underline{T}$.
\end{remark}

\section[Sch\"affer-model pseudo-commuting contractive lifts]{Sch\"affer-model pseudo-commuting contractive lifts}  \label{S:pccSchaffer}

Unlike the case for the Douglas model, the Sch\"affer model for a general And\^o  lift arising from a Type II And\^o tuple does not appear to be sufficiently tractable for the identification of a pseudo-commuting contractive lift.
We therefore restrict ourselves to strong Type II And\^o tuples (see Definition \ref{AndoTuple}). 


We have seen in Theorem \ref{Thm:LiftFromTuple} that any minimal And\^o lift for a commuting contractive pair $(T_1, T_2)$ is unitarily equivalent to the
Sch\"affer-model And\^o lift associated with some strong Type II And\^o tuple $(\cF, \Lambda, P, U)$ for $(T_1, T_2)$ and conversely, the Sch\"affer-model
And\^o lift associated with a strong Type II And\^o tuple for $(T_1, T_2)$ is an And\^o lift for $(T_1, T_2)$.  By Theorem \ref{T:comp=pcc} the compression
of such a Sch\"affer-model And\^o lift in the sense of Definition \ref{D:embedded-lift} yields a pseudo-commuting contractive lift $(\bbW_1, \bbW_2, \bbW)$
for $(T_1, T_2, T_1 T_2)$.  The next result computes such a Sch\"affer-model compressed And\^o lift for a given commuting contractive pair $(T_1, T_2)$
using also the Sch\"affer model for the minimal isometric lift of the product contraction operator $T = T_1T_2$.

\begin{theorem}\label{Thm:Pseudo}  Given a commuting contractive operator-pair $(T_1, T_2)$ on $\cH$, let $\cK_S = \sbm{ \cH \\ H^2(\cD_T)}$ be the
Sch\"affer-model isometric-lift space for  the product contraction $T:= T_1 T_2$, let $\Pi_S \colon \cH \to \cK_S$ be the Sch\"affer-model embedding operator 
$\Pi_S = \sbm{ I_\cH \\ 0 }$, let $(F_1, F_2)$ be the Fundamental-Operator pair for $(T_1, T_2)$ (see Theorem \ref{T:FundOps}), and define operators 
$\bbW_{S,1}, \bbW_{S,2}, \bbW_S$ on $\cK_S$ according to the formula
\begin{align}    
\bbW_{S,1} & =  \begin{bmatrix} T_1 & 0 \\ \bev_{0, \cD_T}^* F_2^*D_T & M_{F_1 + z F_2^*} \end{bmatrix},
\quad \bbW_{S,2} =  \begin{bmatrix} T_2 & 0 \\ \bev_{0, \cD_T}^* F_1^*D_T & M_{F_2 + z F_1^*} \end{bmatrix}  \notag \\
\bbW_S & = V_S = \begin{bmatrix} T & 0 \\  \bev_{0, \cD_T}^* D_T & M^{\cD_T}_z \end{bmatrix} \text{ on } \cK_S:= \begin{bmatrix} \cH \\ H^2(\cD_T) \end{bmatrix}.
\label{Smodel-PCC}
\end{align}
Then $(\Pi_S, \bbW_{S,1}, \bbW_{S,2}, V_S)$ is a the compression of an And\^o lift of $(T_1, T_2, T_1T_2))$
to an embedded minimal Sz.-Nagy--Foias lift of $T= T_1 T_2$ and hence also a pseudo-commuting contractive lift of $(T_1, T_2, T_1T_2)$.

Conversely, if $(\Pi_S, \bbW_1, \bbW_2, V_S)$ is any pseudo-commuting contractive lift of $(T_1, T_2)$ such that
\begin{equation}  \label{Schaffer1}
 (\Pi_S, V_S) = \left( \begin{bmatrix} I_\cH \\ 0 \end{bmatrix}  \colon \cH \to \cK_S,  \begin{bmatrix} T & 0 \\ \bev_{0, \cD_T}^*  D_T  & M_z^{\cD_T} \end{bmatrix} \text{ on } \cK_S \right)
 \end{equation}
 is the Sch\"affer-model minimal isometric lift of $T$ on $\cK_S  = \sbm{ \cH \\ H^2(\cD_T)}$, then necessarily also $(\bbW_1, \bbW_2) = (\bbW_{S,1}, \bbW_{S,2})$
 are given as in formula \eqref{Smodel-PCC}.
\end{theorem}

\begin{proof}
Let $(\cF, \Lambda, P, U)$ be a strong Type II And\^o tuple for the commuting contractive operator-pair $(T_1, T_2)$.  Then the associated Sch\"affer-model And\^o lift of $(T_1, T_2)$ 
is given by $(\bPi_S, \bV_{S,1}, \bV_{S,2})$ where the isometric embedding operator $\bPi_S \colon \cH \to \bcK_S$ and the isometries $\bV_{S,1}, \bV_{S,2}$ on 
$\bcK_S$ are given as in Theorem \ref{Thm:LiftFromTuple}: 
\begin{align*}
&\bcK_S= \sbm{ \cH \\ H^2(\cF) }, \quad \bPi_S = \sbm{ I_\cH \\ 0  } \colon \cH \to \bcK_S, \\
 & (\bV_{S,1},\bV_{S,2}) =   \left( \begin{bmatrix} T_1 & 0   \\  \bev_{0,\cF}^* P U \Lambda D_T & M_{P^\perp U + z PU}   \end{bmatrix},
  \begin{bmatrix} T_2 & 0  \\ \bev_{0,\cF}^* U^* P^\perp  \Lambda D_T & M_{U^*P + z U^* P^\perp }  \end{bmatrix} \right),  \\
& \bV_S= \bV_{S,1} \bV_{S,2} = \bV_{S,2} \bV_{S,1} = \begin{bmatrix} T & 0   \\   \bev_{0,\cF}^* \Lambda D_T & M_z   \end{bmatrix}
 \end{align*}
 where  $T : = T_1 T_2 = T_2 T_1$ is the product contraction operator on $\cH$.
Let us next compute the space
$$
  \bcK_{S, 00}: = \bigvee_{n \ge 0} \bV_S^n \operatorname{Ran} \bPi_S.
$$
By an induction argument one can see that
$$ \bV_S^n \bPi_S = \begin{bmatrix} T^n \\ \sum_{j=0}^{n-1} M_z^j \bev_{0, \cF}^* \Lambda D_T T^{n-1-j} \end{bmatrix}
$$
where the bottom entry should be interpreted to be $0$ for the case $n=0$.  By taking $n=0$ we see that $\sbm{ \cH \\ 0 } \subset \bcK_{S, 00}$.
By next taking $n=1$ we see that 
$$
  \bigvee_{n=0,1}  \operatorname{Ran} \bV_S^n \bPi_S =  \begin{bmatrix} \cH \\   \bev_{0, \cF}^* \Lambda \cD_T   \end{bmatrix} \subset \bcK_{S, 00}.
$$
Inductively assume that
$$
 \bigvee_{n=0,1, \dots, K}  \operatorname{Ran} \bV_S^n \bPi_S = \begin{bmatrix} \cH \\  \oplus_{j=0}^{K-1} (M_z^{\cD_T})^j \bev_{0,\cF}^*  \Lambda \cD_T \end{bmatrix}.
$$
It then follows that
\begin{align*}
 \bigvee_{n=0,1,\dots, K, K+1}  \operatorname{Ran} \bV_S^n \bPi_S & = \begin{bmatrix} \cH \\ ( \oplus_{j=0}^{K-1}  (M_z^{\cD_T})^j \bev_{0 \cF}^* \Lambda \cD_T)
\oplus (M_z^{D_T})^K \bev_{0, \cF}^* \Lambda \cD_T \end{bmatrix}  \\
&  =  \begin{bmatrix} \cH \\  \oplus_{j=0}^K  (M_z^{\cD_T})^j \bev_{0 \cF}^* \Lambda \cD_T \end{bmatrix}.
\end{align*}
Hence 
\begin{equation}  \label{cKSmin}
\bcK_{S, 00} =  \operatorname{closure} \left( \bigcup_{K\ge 0} \begin{bmatrix} \cH \\ ( \oplus_{j=0}^{K-1}  (M_z^{\cD_T})^j \bev_{0 \cF}^* \Lambda \cD_T \end{bmatrix} \right)
= \begin{bmatrix} \cH \\ H^2( \operatorname{Ran} \Lambda)  \end{bmatrix}.
\end{equation}

Let us now introduce the Sch\"affer model  $(\Pi_S, V_S)$ for the minimal isometric lift of the product contraction operator $T:= T_1 T_2$.  By uniqueness of 
minimal isometric lifts for a single contraction operator, there exists an isometry $\tau$ from $\cK_S$ onto $\bcK_{S,00}$ so that
$$
  \tau \Pi_S = \bPi_{S, 00}, \quad \tau V_S  = \bV_{S, 00} \tau 
$$
where we set $\bV_{S,00} := \bV_S |_{\bcK_{00}}$.
It is easy to check that
\begin{equation}  \label{Ch6tau}
  \tau:= \begin{bmatrix} I_\cH & 0 \\ 0 & I_{H^2} \otimes \Lambda \end{bmatrix}
\end{equation}
does the job.

Putting all the pieces together, it follows by definition that the compressed And\^o lift of $(T_1, T_2)$ associated with 
\begin{enumerate}
\item[(i)] the Sch\"affer-model  lift $(\bPi_S, \bV_{S,1}, \bV_{S,2})$ determined by $(\cF, \Lambda, P, U)$ for $(T_1, T_2)$ and 
\item[(ii)]  the Sch\"affer-model minimal Sz.-Nagy--Foias lift $(\Pi_S, V_S)$ for $T = T_1 T_2$
\end{enumerate}
 is given by
\begin{equation}  \label{Schaffer-comp-lift}
  (\bbW_{S,1}, \bbW_{S,2}, \bbW_S) = (\tau^* \bV_{S,1} \tau, \, \tau^* \bV_{S,2} \tau, \, V_S) \text{ on } \cK_S.
 \end{equation}
 Thus
 \begin{align}  
 \bbW_{S,1} & =   \begin{bmatrix} I_\cH & 0 \\ 0 & I_{H^2} \otimes \Lambda^* \end{bmatrix}  
 \begin{bmatrix} T_1 & 0 \\ \bev^*_{0, \cF} PU \Lambda D_T & M_{P^\perp U + z P U}  \end{bmatrix}  
 \begin{bmatrix} I_\cH & 0 \\ 0 & I_{H^2} \otimes \Lambda \end{bmatrix} \notag \\
 & =  \begin{bmatrix} T_1 & 0 \\ \bev_{0, \cD_T}^*  \Lambda^* PU\Lambda D_T & M_{\Lambda^* P^\perp U \Lambda + z \Lambda^* P U \Lambda} \end{bmatrix} \notag \\
 & = \begin{bmatrix} T_1 & 0 \\ \bev_{0, \cD_T}^* F_2^*D_T & M_{F_1 + z F_2^*} \end{bmatrix}  \label{SchafferW1}
  \end{align}
  where  $(F_1, F_2)$ is the Fundamental-Operator pair for the commuting contractive pair $(T_1, T_2)$.  Here in the last step we used the characterization \eqref{choices}
 of the Fundamental-Operator pair in terms of a strong Type II And\^o tuple for $(T_1, T_2)$ 
 $$
   F_1 = \Lambda^* P^\perp U \Lambda, \quad F_2 = \Lambda^* U^* P \Lambda
 $$
 coming out of the Third Proof of Theorem \ref{T:FundOps}.
 
 A similar computation  gives
 \begin{align*}
 \bbW_{S,2} & = \begin{bmatrix} I_\cH & 0 \\ 0 & I_{H^2} \otimes \Lambda^* \end{bmatrix} 
 \begin{bmatrix} T_2 & 0 \\ 0 & \bev_{0, \cF}^* U^* P^\perp \Lambda D_T & M_{U^*P + z U^* P^\perp} \end{bmatrix}  
 \begin{bmatrix}  I_\cH & 0 \\ 0 & I_{H^2} \otimes \Lambda \end{bmatrix}  \notag \\
  & = \begin{bmatrix} T_2 & 0 \\ \bev_{0, \cD_T}^* \Lambda^* U^* P^\perp \Lambda D_T & M_{\Lambda^*U^*P\Lambda + z \Lambda^*U^* P^\perp\Lambda} \end{bmatrix}  \notag \\
 & = \begin{bmatrix} T_2 & 0 \\ \bev_{0, \cD_T}^* F_1^* D_T  & M_{F_2 + z F_1^*} \end{bmatrix}.
 \end{align*}
We have now verified that the formula \eqref{Smodel-PCC} gives a compressed And\^o lift for the commuting contractive pair $(T_1, T_2)$.
  
 The fact that then $(\Pi_S, W_{S,1}, W_{S,2}, W_S = V_S)$ is also a pseudo-commuting contractive lift of $(T_1, T_2, T=T_1T_2)$ follows as a consequence of
 the general principle {\em compressed And\^o lift $\Rightarrow$ pseudo-commuting contractive lift} of $(T_1, T_2)$ verified in Theorem \ref{T:comp=pcc}.
   
 The converse follows from the model-independent result in Corollary \ref{C:FundOp}. More precisely, apply the ``furthermore" part of Corollary \ref{C:FundOp} to the two pseudo-commuting contractive lifts $(\bbW_{S1},\bbW_{S2},V_S)$ as in \eqref{Smodel-PCC} and $(\bbW_1,\bbW_2,V_S)$ as in the converse part of Theorem \ref{Thm:Pseudo}. Observe that if $\tau':\cK_S\to\cK_S$ is a unitary such that $\tau'V_S=V_S\tau'$ and $\tau'|_\cH=I_\cH$, then by minimality of the lift $V_S$, we must have $\tau'=I_{\cK_S}$. Thus the unitary $\tau'$ as in \eqref{partial-intertwine} must be the identity operator, and consequently, the converse here follows. Alternatively, one can prove it directly by following the steps in the proof of the converse part of Theorem \ref{Thm:DPseudo} but with substitution of Sch\"affer models for the Douglas models.
 
\end{proof}

Recall that two minimal isometric lifts of a commuting contractive pair need not be unitary equivalent. What if minimality is replaced by strong minimality? We end this chapter with 
the following result that shows that the existence of one strongly minimal And\^o lift is a sufficiently strong condition to force uniqueness of any two minimal And\^o lifts.

\begin{theorem} \label{T:strongly-minimal}
Let $(T_1,T_2)$ be a commuting contractive pair such that it has a strongly minimal And\^o lift. Then any minimal isometric And\^o lift of $(T_1,T_2)$ is strongly minimal and consequently, any two minimal isometric lifts of $(T_1,T_2)$ are unitarily equivalent.
\end{theorem}

\begin{proof}
Let $(V_1,V_2)$ acting on $\cK$ be a minimal isometric lift of $(T_1,T_2)$ via the isometric embedding $\Pi:\cH\to\cK$. Consider the space 
$\cK_{00}=\bigvee_{n\geq0}V_1^nV_2^n\Pi\cH$. Let $\iota:\cK_{00}\to\cK$ be the embedding of $\cK_{00}$ into $\cK$. Consider the compression of the And\^o lift $(V_1,V_2)$ 
to $\cK_{00}$:
$$
(V_{1,00},V_{2,00},V_{00})=\iota^*(V_1,V_2,V_1V_2)\iota.
$$
By Theorem \ref{T:comp=pcc}, $(V_{1,00},V_{2,00},V_{00})$ is a pseudo-commuting contractive lift of $(T_1,T_2)$. 
Let $(W_1,W_2)$ be a strongly minimal And\^o lift of $(T_1,T_2)$ acting on the space $\cW=$ $\bigvee_{n\geq0}W_1^nW_2^n\cH$. For simplicity, we take the embedding 
to be the inclusion map. It is routine to see that the triple $(W_1,W_2,W_1W_2)$ satisfies all the conditions for a pseudo-commuting contractive lift. By Corollary \ref{C:FundOp}, 
the compression $(V_{1,00},V_{2,00})$ is unitarily equivalent to the And\^o lift $(W_1,W_2)$, which in turn implies that $(V_{1,00},V_{2,00})$ must also be And\^o lift of $(T_1,T_2)$. 
By minimality of the And\^o lift $(V_1,V_2)$, we must have $\cK=\cK_{00}$ showing that $(V_1,V_2)$ is strongly minimal. Since strongly minimal And\^o lifts are all 
pseudo-commuting contractive lifts, and since the latter class are all unique up to unitary equivalence, any two minimal isometric lifts of $(T_1,T_2)$ must be unitarily equivalent.
\end{proof}

This result leads to the following immediate corollary.

\begin{corollary} \label{C:strongly-minimal}  Suppose that $(T_1, T_2)$ is a commuitng pair of isometries such that both $T = T_1 \cdot T_2$ an $T = T_2 \cdot T_1$
are regular factorizations  (see Definition \ref{D:RegFact}).  The $(T_1, T_2)$ has a strongly minimal And\^o lift $(V_1, V_2)$ and all minimal And\^o lifts are stronlgy minimal and mutually unitarily equiavalent as lifts.
\end{corollary}

\chapter[Models and invariants for commuting contractive pairs]{Characteristic/admissible triples and functional model for a commuting pair of contractions}
\label{C:NFmodel}

As seen in the preceding chapters that there is a lack of uniqueness (up to unitary equivalence of lifts) in general for And\^o lifts of a 
given contractive commuting operator-pair $(T_1, T_2)$ but there is uniqueness for a pseudo-commuting contractive lift $(\Pi, \bbW_1, \bbW_2, \bbW)$ of
$(T_1, T_2, T = T_1T_2)$ which has embedded in it a minimal Sz.-Nagy--Foias (isometric) lift $(\Pi, V = \bbW)$ of the single contraction operator $T$.
When we use the Sz.-Nagy--Foias model $(\Pi_{\rm NF}, V_{\rm NF})$  of the minimal isometric lift for $T$ (expressed in terms of the characteristic function $\Theta_T$) for 
$T$, together with  the 
Sz.-Nagy--Foias model $(\Pi_{\rm NF}, \bbW_{\rm NF,1}, \bbW_{\rm NF,2}, V_{\rm NF})$ for the pseudo-commuting contractive lift of $(T_1, T_2, T)$
(expressed in terms of a larger characteristic triple $\Xi_T = ((G_1, G_2),  (W_{\sharp 1}, W_{\sharp 2}), \Theta_T)$  for $(T_1, T_2, T)$
(where $(G_1, G_2)$ and $(W_{\sharp 1}, W_{\sharp 2})$ are as in \eqref{pccNF}),  we arrive at a functional model
$$
(T_1, T_2, T) \underset{u}\cong  P_{\cH_{\Theta_T}}( \bbW_{\rm NF,1}, \bbW_{\rm NF,2}, V_{\rm NF})|_{\cH_{\Theta_T}}
$$
for the commuting triple $(T_1, T_2, T)$ itself.
Conversely, there is a notion of {\em admissible triple} $\Xi$ for the case where no commuting contractive pair  $(T_1, T_2)$ is initially specified, 
from which one can build a commuting contractive operator-pair $(T_{\Xi,1}, T_{\Xi,2})$ on a functional model space  $\cH_{\Theta}$ which in turn has a
characteristic triple $\Xi_{T_{\Xi,1}, T_{\Xi,2}}$ which can be shown to coincide with the original admissible triple, giving a complete parallel with the
Sz.-Nagy--Foias theory  outlined in Remark \ref{R:NFmodel} (where one now has {\em characteristic triple} in place of {\em characteristic function} and {\em admissible triple} in place of
{\em purely contractive analytic function}.  Perhaps as is to be expected, however, there are some compatibility conditions in the definition of admissible triple
which may be difficult to check in practice. In the succeeding sections we spell out the details.  The first order of business is to understand precisely the notion of
{\em completely nonunitary} for the commuting contractive operator-pair case.

\section{Characteristic triples and functional models}

  The following definition makes precise the notion of {\em characteristic triple} for a commuting contractive operator-pair.
  
\begin{definition} \label{D:char-triple} \index{characteristic triple}
Let $(T_1, T_2)$ be a commuting contractive operator-pair on $\cH$.  Let us introduce the following objects:
\begin{enumerate}
\item[(i)] $(G_1, G_2) =$ the {\em Fundamental-Operator pair} for $(T_1^*, T_2^*)$ as in Definition \ref{D:FundOps}.

\item[(ii)]   $(W_{\sharp 1}, W_{\sharp 2})  = $  the commuting unitary operator-pair canonically associated with $(T_1, T_2)$ as in \eqref{WandVNFs}.

\item[(iii)] $\Theta_T = $  {\em  the characteristic operator function}  \eqref{char-func} for the product contraction operator $T = T_1 T_2$.
Here we set $\underline{T}$ to the operator triple $\underline{T} = (T_1, T_2, T = T_1 T_2)$.
\end{enumerate}
Then the triple $\Xi_{\underline{T}} := ( (G_1, G_2), (W_{\sharp 1}, W_{\sharp 2}), \Theta_T)$ is called 
the {\em characteristic triple} for $(T_1,T_2)$.
\end{definition}

Note that the components $((G_1, G_2), (W_{\sharp 1}, W_{\sharp 2}), \Theta_T)$ is all that is needed to write down $(\Pi_{\rm NF}, W_{\rm NF,1}, W_{\rm NF,2},  V_{\rm NF})$, the Sz.-Nagy--Foias model pseudo-commuting contractive lift of $(T_1, T_2, T = T_1 T_2)$ acting on the space $\cK_{\Theta_T}$ as in Theorem \ref{Thm:NFPseudo}.

We now present the bivariate analogue of the result discussed in Remark  \ref{R:NFmodel}  and reviewed in the preceding paragraphs.

\begin{theorem}\label{Thm:SNFmodelPair}
Let $(T_1,T_2)$ be a commuting contractive operator-pair and let its characteristic triple be 
$$
\Xi_{\underline{T}} = ((G_1,G_2),(W_{\sharp1},W_{\sharp2}),\Theta_T).
$$ 
Then the Sz.-Nagy--Foias model space
\begin{equation}  \label{HNF'}   
\cH_{\Theta_T}= \begin{bmatrix} H^2(\cD_{T^*})   \\  \overline{\Delta_{\Theta_T} \cdot L^2(\cD_T)} \end{bmatrix} \ominus 
\begin{bmatrix} \Theta_T \\ \Delta_{\Theta_T} \end{bmatrix}  \cdot H^2(\cD_T)
\end{equation}
is coinvariant under
$$
\left( \begin{bmatrix} M_{G_1^*+zG_2} & 0 \\ 0 &  W_{\sharp1} \end{bmatrix},
\begin{bmatrix} M_{G_2^*+zG_1} & 0 \\ 0 &  W_{\sharp2} \end{bmatrix},
\begin{bmatrix} M_z & 0 \\ 0 &  \left. M_{\zeta} \right|_{\overline{\Delta_{\Theta_T}(L^2(\cD_T))}} \end{bmatrix} \right)
$$
and $\underline{T} = (T_1,T_2,T_1T_2)$ is unitarily equivalent to
\begin{align}\label{NFmodelPair}
P_{\cH_{\Theta_T}} \left.  \left(  \begin{bmatrix} M_{G_1^*+zG_2} & 0 \\ 0 &  W_{\sharp1} \end{bmatrix},
\begin{bmatrix} M_{G_2^*+zG_1} & 0 \\ 0 &  W_{\sharp2}  \end{bmatrix},
\begin{bmatrix}  M_z & 0 \\ 0 &  M_{\zeta}|_{\overline{\Delta_{\Theta_T}(L^2(\cD_T))}} \end{bmatrix} \right) \right|_{\cH_{\Theta_T}}
\end{align}
via the unitary operator $ \Pi_{\rm NF, 0} \colon \cH \to \cH_{\Theta_T}$ given by
\begin{equation}  \label{PiNF0}
  \Pi_{\rm NF,0} \colon h \mapsto \Pi_{\rm NF} h \in \cH_{\Theta_T} \text{ for } h \in \cH.
\end{equation}
\end{theorem}

\begin{proof}
Let us denote by $\underline{\bbW}_{\rm NF}$ the operator triple
\begin{align}
\underline{\bbW}_{\rm NF}:=&( \bbW_{{\rm NF},1}, \bbW_{{\rm NF},2}, V_{\rm NF})  \notag  \\ 
 := & \left(\sbm{ M_{G_{\sharp1}^*+zG_{\sharp2}} & 0 \\ 0 &  W_{\sharp1}}, \sbm{ M_{G_{\sharp2}^*+zG_{\sharp1}} & 0 \\ 0 &  W_{\sharp2}}, 
  \sbm{M_z & 0 \\ 0 &   \left. M_{\zeta} \right|_{\overline{\Delta_{\Theta_T}(L^2(\cD_T))}}} \right)  \label{Vs}
\end{align}
acting on $\cK_{\Theta_T}:= \sbm{ H^2(\cD_{T^*}) \\ \overline{ \Delta_{\Theta_T} L^2(\cD_{T})} }$ and set $\Pi_{\rm NF} = \sbm{ \cO_{D_{T^*}, T^*} \\ \omega_{\rm NF,D} Q_{T^*}}$
equal to the Sz.-Nagy--Foias embedding operator from $\cH$ into $\cK_{\Theta_T}$ with range equal to $\cH_{\Theta_T}$.  By Theorem \ref{Thm:NFPseudo} we know that
$(\Pi_{\rm NF}, \bbW_{\rm NF,1},\bbW_{\rm NF,2}, V_{\rm NF})$ is a lift of $(T_1, T_2, T = T_1 T_2)$, i.e.,
\begin{equation}  \label{WNFlift-intertwine}
(\bbW_{\rm NF,1}^*, \bbW_{\rm NF,2}^*, V_{\rm NF}^*) \Pi_{\rm NF}  = \Pi_{\rm NF} (T_1^*, T_2^*, T^* = T_1^* T_2^*).
\end{equation}
where (as noted in Remark \ref{R:NFmodel}) $\operatorname{Ran} \Pi_{\rm NF} = \cH_{\Theta_T}$.  This shows immediately that $\cH_{\Theta_T}$ is coinvariant under
$\underline{\bbW}_{\rm NF}$.   Apply $\Pi_{\rm NF}^*$ on the left to both sides of \eqref{WNFlift-intertwine}, use that $\Pi_{\rm NF}$ is an isometry
($\Pi_{\rm NF}^* \Pi_{\rm NF} = I_\cH$), and then take adjoints to arrive at
$$
\Pi_{\rm NF}^* (W_{\rm NF,1}, W_{\rm NF,2}, V_{\rm NF}) \Pi_{\rm NF} = (T_1, T_2, T_1T_2).
$$
Then use the connection  \eqref{PiNF0} between $\Pi_{\rm NF}$ and $\Pi_{\rm NF,0}$ to reinterpret this last identity as
$$
\Pi_{\rm NF,0}^*  \left(   P_{\cH_{\Theta_T}}  (W_{\rm NF,1}, W_{\rm NF,2}, V_{\rm NF}) |_{\cH_{\Theta_T}} \right) \Pi_{\rm NF,0} = (T_1, T_2, T_1T_2).
$$
This last equality is the statement that the model operator-triple \eqref{NFmodelPair}
$$
P_{\cH_{\Theta_T}} (\bbW_{\rm NF,1}, \bbW_{\rm NF,2}, V_{\rm NF})|_{\cH_{\Theta_T}}
$$
is unitarily equivalent via $\Pi_{\rm NF, 0}$ to the original contractive operator-triple $(T_1, T_2, T_1 T_2)$ as claimed, and the theorem follows.
\end{proof}

\section[Canonical decomposition]{Canonical decomposition for pairs of commuting contractions}
\label{S:can-decom}

Our eventual goal is to prove that characteristic triples form a complete unitary invariant for commuting contractive pairs $(T_1,T_2)$
with the condition that $T=T_1T_2$ is a completely nonunitary (c.n.u.) contraction.  The goal of this section is to argue that this c.n.u.~ assumption on $T = T_1T_2$
can be discarded without any substantive loss of generality due to the existence of a
{\em canonical decomposition} for any commuting contractive pair $(T_1, T_2)$.  This is analogous to the single-variable phenomenon that every single contraction 
has a decomposition as the direct sum of a unitary operator and a c.n.u.~contraction operator.  This result actually follows as a special case of the canonical
decomposition for tetrablock contractions recently obtained by Pal \cite{PalJMAA}.  Here we present a more
elementary direct proof for the special setting of commuting pairs of contractions. We shall need the following lemma.

\begin{lemma}\label{L:ZeroOp}
Let $A$ be a bounded operator on a Hilbert space such that $\omega A$ has negative semidefinite real part for all $\omega$ on the unit circle:
$$
 \operatorname{Re} (\omega A) := \omega A + (\omega A)^* \preceq 0 \text{ for all }\omega \in {\mathbb T}.
$$
Then $A=0$.
\end{lemma}

\begin{proof}
The hypothesis means that the operator-valued function $R(\omega):= \omega A+\overline{\omega} A^*$ satisfies  $R(\omega) \preceq 0$ for every $\omega \in {\mathbb T}$. 
Note that $R(-\omega)=-R(\omega)$ for every $\omega \in{\mathbb T}$ and hence
$$
R(\omega):= \omega A+\overline{\omega} A^*= 0 \text{ for all } \omega\in\mathbb{T},
$$
which readily implies that $A=\frac{1}{2}(R(1)- i R(i))=0$.
\end{proof}

\begin{theorem}\label{Thm:CanDecPair}
For every pair $(T_1,T_2)$ of commuting contractions on a Hilbert space $\cH$ there corresponds a decomposition of $\cH$ into the orthogonal sum of two subspaces reducing for both $T_1$ and $T_2$, say
$\cH=\cH_{u}\oplus\cH_{c}$, such that, with notation
\begin{align}\label{CanDecPair}
(T_{1u},T_{2u})=(T_1,T_2)|_{\cH_u} \text{ and }(T_{1c},T_{2c})=(T_1,T_2)|_{\cH_c},
\end{align}
we have that $T_u=T_{1u}T_{2u}$ is a unitary and $T_{c}=T_{1c}T_{2c}$ is a c.n.u.~contraction.
Moreover, then $T_u \oplus T_c$ with respect to $\cH=\cH_{u}\oplus\cH_{c}$ is the Sz.-Nagy--Foias canonical decomposition for the contraction operator $T=T_1 T_2$.
\end{theorem}

\begin{proof}
Let $(T_1,T_2)$ be a pair of commuting contractive operator-pair on a Hilbert space $\cH$ such that  $(F_1, F_2)$ is  the Fundamental-Operator pair for $(T_1^*, T_2^*)$. 
By Definition  \ref{D:FundOps} (combined with Lemma \ref{L:FundOps}), on the one hand  $(F_1, F_2)$  is characterized as the unique solution 
of the pair of operator equations
\begin{equation}   \label{CanFundEqn}
T_i - T_j^* T = D_{T} F_{\sharp i } D_{T} \text{ where } (i,j) = (1,2) \text{ or } (2,1)
\end{equation}
but on the other hand, as a consequence of the Second Proof of Theorem \ref{T:FundOps}, can also be expressed directly in terms of a  Type I And\^o tuple 
$(\cF,\Lambda,P,U)$  for $(T_1^*,T^*_2)$ as
\begin{equation}\label{FundAndoRel}
 (F_1,F_2)=\Lambda^*(P^\perp U,U^*P)\Lambda
\end{equation}
Since both of $F_1$ and $F_2$ are contractions, we have for every $\omega$ and $\zeta$ in $\mathbb{T}$
\begin{equation}     \label{CrucIneq1}
I_{\cD_T}-\operatorname{Re}(\omega F_1) \succeq 0 \text{ and }I_{\cD_T}-\operatorname{Re}(\zeta F_2) \succeq 0.
\end{equation}
Adding these two inequalities then gives
\begin{align}\label{CrucIneq}
2I_{\cD_T}-\operatorname{Re}(\omega F_1+\zeta F_2) \succeq 0, \text{ for all }\omega,\zeta\in\mathbb{T}.
\end{align}
Note that inequality (\ref{CrucIneq}) is equivalent to
$$
2D_T^2-\operatorname{Re}(\omega D_TF_1D_T+\zeta D_TF_2D_T) \succeq 0 \text{ for all }\omega,\zeta \in\mathbb{T}.
$$
By (\ref{CanFundEqn}) this is same as
\begin{equation}   \label{FinalIneq}
2D_T^2-\operatorname{Re}(\omega (T_1-T_2^*T))-\operatorname{Re}(\zeta (T_2-T_1^*T)) \succeq 0, \text{ for all }\omega,\zeta \in\mathbb{T}.
\end{equation}
Let
\begin{align}\label{CanT}
  T=\begin{bmatrix}
      T_u & 0 \\
      0 & T_{c}
    \end{bmatrix}:\cH_u\oplus\cH_c\to\cH_u\oplus\cH_{c}
\end{align}
be the canonical decomposition of $T$ into unitary piece $T_u$ and completely nonunitary piece $T_{c}$.   
It remains to show that $T_1$ and $T_2$ are also block-diagonal with respect to this decomposition.
To get started, we consider the $2 \times 2$-matrix representation of  each $T_j$ with respect to the decomposition $\cH = \cH_u \oplus \cH_c$:
\begin{equation}\label{Tjs}
T_j=\begin{bmatrix}
      A_j & B_j \\
      C_j & D_j
    \end{bmatrix}:\cH_u\oplus\cH_c\to\cH_u\oplus\cH_c  \text{ for } j=1,2.
\end{equation}
Next apply (\ref{FinalIneq}) to obtain that
\begin{align}\label{CrucIneq2}
 \nonumber \begin{bmatrix} 0 & 0 \\ 0 & 2D_{T_{c}}^2  \end{bmatrix} &
 - \operatorname{Re}\left(\omega \begin{bmatrix} A_1-A_2^*T_u & B_1-C_2^*T_{c} \\ C_1-B_2^*T_u & D_1-D_2^*T_{c}  \end{bmatrix}\right) \\
    &  -\operatorname{Re}\left(\zeta \begin{bmatrix} A_2-A_1^*T_u & B_2-C_1^*T_{c} \\ C_2-B_1^*T_u & D_2-D_1^*T_{c} \end{bmatrix}\right)
    \succeq  0 \text{ for all }\omega,\zeta \in\mathbb{T}.
    \end{align}
    In particular, the $(1,1)$-entry in this inequality works out to be
    \begin{align}\label{(11)}
      {\mathbb P}_{11}(\omega,\zeta):= \operatorname{Re}(\omega(A_1-A_2^*T_u ))+\operatorname{Re}(\zeta(A_2-A_1^*T_u)) \preceq 0, \text{ for all } \omega,\zeta \in\mathbb{T}.
    \end{align}
   This in turn implies that
    \begin{align*}
    {\mathbb P}_{11}(\omega,1)+ {\mathbb P}_{11}(\omega,-1)&=2\operatorname{Re}(\omega(A_1-A_2^*T_u )) \preceq 0 \text{ and}\\
    {\mathbb P}_{11}(1,\zeta)+{\mathbb P}_{11}(-1,\zeta)&=2\operatorname{Re}(\zeta(A_2-A_1^*T_u)) \preceq 0.
    \end{align*}
   Now we apply Lemma \ref{L:ZeroOp} to conclude that
    \begin{align}\label{CrucIneq3}
     A_1=A_2^*T_u, \quad  A_2=A_1^*T_u.
    \end{align}
    This shows that the $(1,1)$-entry of the matrix on the left-hand side of (\ref{CrucIneq2}) is zero. Since the matrix is positive semi-definite, the $(1,2)$-entry (and hence also the $(2,1)$-entry) is also zero, i.e., for all $\omega,\zeta \in\mathbb{T}$
    \begin{align*}
    {\mathbb P}_{12}(\omega,\zeta):= \omega(B_1-C_2^*T_{c})+\bar{\omega}(C_1^*-T_u^*B_2)+\zeta(B_2-C_1^*T_{c})+\bar{\zeta}( C_2^*-T_u^*B_1)=0.
    \end{align*}
 In particular we then get that
    \begin{align*}
      {\mathbb P}(\omega):={\mathbb P}_{12}(\omega,1)+ {\mathbb P}_{12}(\omega,-1)=
      2\omega(B_1-C_2^*T_{c})+2\bar{\omega}(C_1^*-T_u^*B_2)=0
    \end{align*}
    for every $\omega\in\mathbb{T}$. This implies the first two of the following equations while the last two are obtained similarly:
    \begin{align}\label{Cruceq}
      B_1=C_2^*T_{c},\quad C_1^*=T_u^*B_2,\quad B_2=C_1^*T_{c}\quad\text{and}\quad C_2^*=T_u^*B_1.
    \end{align}
 Now from the commutativity of $T_j$ with $T$ we have the following for $j=1,2$:
 \begin{align}\label{Cruceq2}
A_jT_u=T_uA_j,\quad B_jT_{c}=T_uB_j,\quad C_jT_u=T_{c}C_j\quad \text{and}\quad T_{c}D_j=D_jT_{c}.
 \end{align}Let us note that commutativity of $(T_1,T_2)$ has been used in the beginning of the proof, viz., $F_1$ and $F_2$ are contractions because of the commutativity of $T_1$ and $T_2$.

 For $(i,j)=(1,2)$ or $(2,1)$ we obtain using the second and third equation in (\ref{Cruceq}) and the third equation in (\ref{Cruceq2}) that
 \begin{align}\label{Cruceq3}
 B_j^*T_u^2=C_iT_u=T_{c}C_i=T_{c}B_j^*T_u,
 \end{align}
 which implies that $B_j^*T_u=T_{c}B_j^*$, for $j=1,2$. Using this and the second equality in (\ref{Cruceq2}) we obtain
 \begin{align*}
T_{c}T_{c}^*B_j^*=T_{c}B_j^*T_u^*=B_j^*=B_j^*T_u^*T_u=T_{c}^*B_j^*T_{u}=T_{c}^*T_{c}B_j^*,
 \end{align*}
 which implies that $T_{c}$ is unitary on $\overline{\operatorname{Ran}}\, B_j^*$, for every $j=1,2$. Since $T_{c}$ is completely nonunitary, $B_j=0$ for each $j=1,2$. Similarly one can show that $C_j=0$, for each $j=1,2$. This completes the proof.
 \end{proof}

\section[Characteristic triple]{Characteristic triple as a complete unitary invariant}
As already discussed in Remark \ref{R:NFmodel}, 
it was proved by Sz.-Nagy--Foias  (see \cite[Chapter VI]{Nagy-Foias} that the characteristic function $\Theta_T$
for a c.n.u.\ contraction $T$ is a complete unitary invariant. This means that two c.n.u.\ contractions $T$ and $T'$
are unitarily equivalent if and only if their characteristic functions {\em{coincide}} in the sense that there exist
unitary operators $u: \mathcal{D}_T \to \mathcal{D}_{T'}$ and $u_{*}: \mathcal{D}_{T^*} \to \mathcal{D}_{{T'}^*}$ such that the following diagram commutes for every $z\in\mathbb D$:
\index{characteristic function!coincidence of}
\begin{align}\label{coindiagram}
\begin{CD}
\mathcal{D}_T @>\Theta_T(z)>> \mathcal{D}_{T^*}\\
@Vu VV @VVu_{*} V\\
\mathcal{D}_{T'} @>>\Theta_{T'}(z)> \mathcal{D}_{{T'}^*}
\end{CD}.
\end{align}
Theorem \ref{UnitaryInv} below shows that such a result holds for characteristic triples of pairs of
commuting contractions also. First we define a notion of coincidence for such a triple.

Let us recall (see Remark \ref{R:NFmodel} and also \cite{Nagy-Foias} for complete details) that a contractive analytic 
function $(\cD,\cD_*,\Theta)$\index{$(\cD,\cD_*,\Theta)$} is a $\cB(\cD,\cD_*)$-valued analytic function on 
$\mathbb{D}$ such that
$$
\|\Theta(z)\|\leq 1 \text{ for all } z\in\mathbb{D}.
$$
Such a function is called {\it purely contractive}\index{contractive analytic function!purely}
if $\Theta(0)$ does not preserve the norm of any nonzero
vector, i.e.,
\begin{equation}   \label{pureCAF}
\|\Theta(0)\xi\|_{\cD_*}<\|\xi\|_{\cD} \text{ for all nonzero }\xi\in\cD.
\end{equation}
We note that a Sz.-Nagy--Foias characteristic function $\Theta_T$ is always purely contractive
(see \cite[Section VI.1]{Nagy-Foias}),  and that it is
always the case that a general contractive analytic function $(\cD, \cD_*, \Theta)$ has a block diagonal
decomposition  $\Theta = \Theta' \oplus \Theta^0$ where $(\cD', \cD_*', \Theta')$ is a unitary constant function
and $(\cD^0, \cD_*^0, \Theta^0)$ is purely contractive.  A key easily checked property of this decomposition
is the following:

\begin{obs}   \label{O:reduction}
 The model space associated with a contractive analytic function $\Theta$ is defined to be
$$
\cH_{\Theta}:= \begin{bmatrix} H^2(\cD_*) \\ \overline{ \Delta_\Theta L^2(\cD)} \end{bmatrix}
 \ominus \begin{bmatrix} \Theta \\ \Delta_\Theta \end{bmatrix} H^2(\cD)
$$
where here we set 
$$
 \Delta_\Theta(\zeta) = ( I_\cD - \Theta(\zeta)^* \Theta(\zeta))^{\frac{1}{2}}.
$$
We define the  model operator\index{$T_\Theta$} associated with the contractive analytic function $\Theta$ to be
$$
T_{\Theta} = P_{\cH_{\Theta}} \left.  \begin{bmatrix}  M_z & 0 \\ 0 & M_\zeta \end{bmatrix} \right|_{\cH_{\Theta}}.
$$
The observation here is that the model operator $T_\Theta$ 
remains exactly the same (after some natural identification of respective coefficient spaces) when $\Theta$ is replaced by its purely contractive part $\Theta^0$.
Thus only purely contractive analytic functions are relevant when discussing Sz.-Nagy--Foias functional models.
\end{obs}

In view of Observation \ref{O:reduction}, for the moment we consider only purely contractive analytic functions.

\begin{definition}\label{coincidence}
\index{contractive analytic function!coincidence of}
Let $(\mathcal{D},\mathcal{D}_*,\Theta)$, $(\mathcal{D'},\mathcal{D'_*},\Theta')$ be two purely contractive
analytic functions, let $(G_1,G_2)$ on $\mathcal{D}_*$ and  $(G_1',G_2')$ on $\mathcal{D'_*}$
be two pairs of contractions, and let $(W_1,W_2)$ on $\overline{\Delta_\Theta L^2(\mathcal{D})}$
$(W_1',W_2')$ on $\overline{\Delta_{\Theta'} L^2(\mathcal{D'})}$ be two pairs of commuting
unitaries having product equal to $M_{\zeta}$ on the respective spaces. We say that the two triples $((G_1,G_2),(W_1,W_2),\Theta)$ 
and $((G_1',G_2'),(W_1',W_2'),\Theta')$ {\em coincide} if:
\begin{itemize}
  \item[(i)] $(\mathcal{D},\mathcal{D_*},\Theta)$ and $(\mathcal{D'},\mathcal{D'_*},\Theta')$ coincide, i.e.,
there exist unitary operators $u: \mathcal{D} \to \mathcal{D'}$ and $u_{*}: \mathcal{D}_{*} \to \mathcal{D'}_{*}$
such that 
$$ 
u_* \Theta(z) = \Theta'(z) u \text{ for all } z \in {\mathbb D},
$$
i.e., the diagram (\ref{coindiagram}) commutes with $\Theta$ and $\Theta'$ in place of $\Theta_T$
and $\Theta_{T'}$, respectively.

\item[(ii)] the unitary operators $u$, $u_*$ also have the  intertwining properties
\begin{eqnarray}
\begin{cases}
(G_1',G_2')=u_*(G_1,G_2)u_*^*=(u_*G_1u_*^*,u_*G_2u_*^*) \text{ and }\\
(W_1',W_2')=\omega_u(W_1,W_2)\omega_u^*=(\omega_uW_1\omega_u^*,\omega_uW_2\omega_u^*),
\end{cases}
\end{eqnarray}
where $\omega_u:\overline{\Delta_{\Theta} L^2(\mathcal{D})}\to\overline{\Delta_{\Theta'} L^2(\mathcal{D'})}$
is the unitary map induced by $u$ defined by
\begin{eqnarray}\label{omega-u}\index{$\omega_u$}
\omega_u:=(I_{L^2}\otimes u)|_{\overline{\Delta_{\Theta} L^2(\mathcal{D})}}.
\end{eqnarray}
\end{itemize}
\end{definition}

\begin{theorem}\label{UnitaryInv}
Let $(T_1,T_2)$ and $(T_1',T_2')$ be two pairs of commuting contractions such that their products $T=T_1T_2$ and $T'=T_1'T_2'$ are c.n.u.\ contractions. 
Then $(T_1,T_2)$ and $(T_1',T_2')$ are unitarily equivalent if and only if their characteristic triples coincide.
\end{theorem}

\begin{proof}
Let $(T_1,T_2)$ on $\mathcal{H}$ and $(T_1',T_2')$ on $\mathcal{H}'$ be unitarily equivalent via the unitary operator $U:\mathcal{H}\to\mathcal{H}'$. 
Let 
$$((G_1,G_2),(W_{\sharp1},W_{\sharp2}),\Theta_T), \quad ((G'_1,G'_2),(W_{\sharp1}',W_{\sharp2}'),\Theta_{T'})
$$
be their respective characteristic triples.  It is easy to see that $UD_T=D_{T'}U$ and $UD_{T^*}=D_{T'^*}U$ and that the unitaries
\begin{eqnarray}
u:=U|_{\mathcal{D}_T}: \mathcal{D}_{T} \to \mathcal{D}_{T'}\text{ and }u_{*}:=U|_{\mathcal{D}_{T^*}}: \mathcal{D}_{T^*} \to \mathcal{D}_{T'^{*}}
\end{eqnarray} have the following property:
\begin{eqnarray}\label{coincd}
u_*\Theta_T(z)=\Theta_{T'}(z) u \text{ for all } z \in {\mathbb D}.
\end{eqnarray} 
Hence $\Theta_T$ and $\Theta_{T'}$ coincide. 

We next show that the unitary $u_*$ also implements a unitary equivalence of $(G_1,G_2)$ 
with $(G_1',G_2')$ as follows.
Note first that since by definition $(G_1, G_2)$ and $(G'_1, G'_2)$ are the Fundamental-Operator pairs for
$(T_1^*, T_2^*)$ and $(T_1^{\prime *}, T_2^{\prime *})$ respectively, we have:
$$
 T_i^* - T_j T^* = D_{T^*} G_i D_{T^*},   \quad
 T_i^{\prime*} - T'_j T^{\prime *} = D_{T^{\prime *}} G'_i D_{T^{\prime *}} \text{ for }
 (i,j) = (1,2) \text{ or } (2,1).
 $$
It then follows that
\begin{eqnarray}\label{fundequiv}
u_*(G_1,G_2)=(G_1',G_2')u_*.
\end{eqnarray} 
We have seen in Theorem \ref{Thm:NFPseudo} that for a pair $(T_1,T_2)$ of commuting contractions, $(\Pi_{\rm NF}, \underline{\bbW}_{\rm NF})$ 
is a pseudo-commuting contractive lift of $(\underline{T} = (T_1,T_2, T_1T_2)$ acting on the space $\cK_{\Theta_T}$.
Let $(\Pi'_{\rm NF}, \underline{\bbW'}_{\rm NF})$ be the corresponding pseudo-commuting contractive lift of $\underline{T'} = (T_1',T_2', T_1'T_2')$ acting on the space 
$\mathcal{K}_{\Theta_{T'}}$.
Let  $\Pi''$ denote the  isometry
\begin{equation}\label{Pi''}
\Pi'':= \sbm{ I_{H^2}\otimes u_*^* & 0 \\  0 & \omega_u^* } \Pi_{\rm NF}'U \colon \mathcal{H}  \to 
\sbm{ H^2(\mathcal{D}_{T^*})\\ \overline{\Delta_{\Theta_T}L^2(\mathcal{D}_{T})} }
\end{equation}
where $\omega_u:\overline{\Delta_{\Theta_T}L^2(\mathcal{D}_T)}\to\overline{\Delta_{\Theta_{T'}}L^2(\mathcal{D}_{T'})}$ is the unitary
$\omega_u=(I_{L^2}\otimes u)|_{\overline{\Delta_{\Theta_T}L^2(\mathcal{D}_T)}}$. 
We observe that
$(\Pi'',\mathcal{K}_{\rm NF},\underline{\bbW''})$ is also a pseudo-commuting contractive lift of $(T_1,T_2, T)$, where
$(W_1'',W_2'')=\omega_u^*(W_{\sharp1}',W_{\sharp2}')\omega_u$ and
$$
\underline{\bbW''}:=\left( \sbm{ M_{G_1^*+zG_2} & 0 \\ 0 &  W_1''}, \sbm{M_{G_2^*+zG_1} & 0 \\ 0 &  W_2''},
\sbm{ M_z & 0 \\ 0 &  M_{\zeta}|_{\overline{\Delta_{\Theta_T}(L^2(\cD_T))}}} \right).
$$
By making use of  (\ref{fundequiv}), we see that, for $(i,j)=(1,2)$ or $(2,1)$, we have
\begin{eqnarray*}
\Pi''T_i^*&=& \sbm{ I_{H^2}\otimes u_*^* & 0 \\ 0 &  \omega_u^*} \Pi_{\rm NF}'{T'_i}^*U \text{ (by \eqref{Pi''})} \\
&=& \sbm{I_{H^2}\otimes u_*^* & 0 \\0 & \omega_u^* } \sbm{ M_{{G'_i}^*+z{G'_j}}^*  & 0 \\ 0 & {W'_{\sharp i}}^*}
\Pi_{\rm NF}'U  \\
&=& \sbm{ M_{G_i^*+zG_j} & 0 \\ 0 &  {W''_i}^*}  \sbm{ I_{H^2}\otimes u_*^* & 0 \\ 0 &  \omega_u^*} \Pi_{\rm NF}'U
\text{ (by \eqref{fundequiv}).}
\end{eqnarray*}
Now since the last entry of $\underline{\bbW''}$ is the same as that of $\underline{\bbW}_{\rm NF}$, applying
Corollary \ref{C:FundOp}, we get $\underline{\bbW''}=\underline{\bbW}_{\rm NF}$ and we conclude that
$$
(W_1'',W_2'')=\omega_u^*(W_{\sharp1}',W_{\sharp2}')\omega_u=(W_{\sharp1},W_{\sharp2}),
$$
This together with equations (\ref{coincd}) and (\ref{fundequiv}) establishes the first part of the theorem.

 Conversely, let $((G_1,G_2),(W_{\sharp1},W_{\sharp2}),\Theta_T)$ and $((G'_1, G'_2),(W_{\sharp1}',W_{\sharp2}'),\Theta_{T'})$ be the characteristic triples of $(T_1,T_2)$ and 
 $(T_1',T_2')$ respectively, and suppose the respective characteristic triples coincide. Thus there exist unitaries $u:\mathcal{D}_T\to\mathcal{D}_{T'}$ and $u_*:\mathcal{D}_{T^*}\to\mathcal{D}_{T'^*}$ such that part $(i)$ and part $(ii)$ in Definition \ref{coincidence} hold. Let $\omega_u$ be the unitary induced by $u$ as defined in $(\ref{omega-u})$. Then it is easy to see that the unitary
\begin{equation}\label{unitary-coin}
\begin{bmatrix} I_{H^2}\otimes u_*  & 0 \\  0 & \omega_u \end{bmatrix}  \colon
      \begin{bmatrix} H^2(\mathcal{D}_{T^*}) \\ \overline{\Delta_{\Theta_T}L^2(\mathcal{D}_{T})} \end{bmatrix}
                                               \to \begin{bmatrix} H^2(\mathcal{D}_{T'^*} )\\
                                                    \overline{\Delta_{\Theta_{T'}}L^2(\mathcal{D}_{T'})} \end{bmatrix}
\end{equation}
intertwines
\begin{align*}
    &\underline{\bbW}=\left(\begin{bmatrix} M_{G_1^*+zG_2} & 0 \\ 0 &  W_{\sharp1} \end{bmatrix},
    \begin{bmatrix} M_{G_2^*+zG_1} & 0 \\ 0 &  W_{\sharp 2}\end{bmatrix}, 
    \begin{bmatrix} M_z^{\cD_{T^*}} & 0 \\ 0 &  M_{\zeta}^{\cD_T}|_{\overline{\Delta_{\Theta_T}(L^2(\cD_T)})} \end{bmatrix} \right) 
\end{align*}
with 
$$
 \underline{\bbW'}=\left( \begin{bmatrix} M_{G_1^{\prime *}+zG'_2} & 0 \\ 0 & W_{\sharp1}' \end{bmatrix},
 \begin{bmatrix} M_{G_2^{\prime *} +zG'_1} & 0 \\ 0 &  W_{\sharp2}' \end{bmatrix},
 \begin{bmatrix} M_z^{\cD_{T^{\prime *}}} & 0 \\ 0 &  M_{\zeta}|_{\overline{\Delta_{\Theta_{T'}}(L^2(\cD_{T'}))}}  \end{bmatrix}    \right).
 $$
Also, the unitary in $(\ref{unitary-coin})$ clearly takes $\sbm{ \Theta_T \\  \Delta_{\Theta_T} } H^2(\cD_T)$ onto
$\sbm{ \Theta_{T'} \\  \Delta_{\Theta_{T'}} }  H^2(\mathcal{D}_{T'})\}$ and hence
\begin{align*}
\begin{bmatrix} H^2(\cD_{T^*}) \\ \overline{ \Delta_{\Theta_T} L^2(\cD_T)} \end{bmatrix}
 \ominus \begin{bmatrix} \Theta_T \\ \Delta_{\Theta_T} \end{bmatrix} H^2(\cD_T)
\text{ onto }\begin{bmatrix} H^2(\cD_{T'^*}) \\ \overline{ \Delta_{\Theta_{T'}} L^2(\cD_{T'})} \end{bmatrix}
 \ominus \begin{bmatrix} \Theta_{T'} \\ \Delta_{\Theta_{T'}} \end{bmatrix} H^2(\cD_{T'}).
\end{align*}
Thus that the functional models for $(T_1,T_2)$ and $(T_1',T_2')$ as in (\ref{NFmodelPair}) are unitarily equivalent and hence by (\ref{NFmodelPair}) the pairs $(T_1,T_2)$ and $(T_1',T_2')$ are unitarily equivalent also.
\end{proof}

\section{Admissible triples}
In this section, we consider general contractive analytic  functions $(\cD, \cD_*, \Theta)$ and do not insist
that $\Theta$ be also pure.
We start with a  contractive analytic function $(\cD,\cD_*,\Theta)$, a commuting unitary operator-pair $(W_1,W_2)$, and a pair of contraction operators $(G_1,G_2)$ 
and investigate when the triple $((G_1,G_2),$ $(W_1,W_2),$ $\Theta)$ gives rise to a commuting contractive operator-pair $(T_1,T_2)$  such that $T=T_1T_2$ is 
completely non-unitary (c.n.u.).

If $((G_1, G_2), (W_1, W_2), \Theta)$ is equal to the characteristic triple $(G_1, G_2),$ $(W_{\sharp 1},$ $W_{\sharp 2}),$ $\Theta_T)$ for a commuting contractive
operator-pair $(T_1, T_2)$,  it is easy to check from the fact that $(\Pi_{\rm NF}, \bbW_{\rm NF,1}, \bbW_{\rm NF,2}, V_{\rm NF})$ is a pseudo-commuting contractive lift of 
$(T_1, T_2)$  that the characteristic triple $((G_1, G_2), (W_1, W_2), \Theta)$ in particular satisfies the set of {\em admissibility
conditions} listed in following definition of {\em admissibility conditions} for such a triple.

\begin{definition}  \label{D:admis-cond} {\rm  \textbf{Admissibility conditions:}\index{admissible conditions}
For $(i,j) = (1,2)$ or $(2,1)$ we have:
 \begin{enumerate}
\item $M_{G_{i}^* + z G_{j}} \oplus W_{ i}$ is a contraction.
\item $W_{1}W_{2}=W_{2}W_{1}=M_{\zeta}|_{\overline{\Delta_\Theta L^2(\cD)}}$.
\item The space $\cQ_{\Theta}:=\{\Theta f\oplus\Delta_\Theta f:f\in H^2(\cD)\}$ is
jointly invariant under
$(M_{G_{1}^* + z G_{2}} \oplus W_{1},M_{G_{2}^* + z G_{1}} \oplus W_{2},M_z\oplus
M_{\zeta}|_{\overline{\Delta_\Theta L^2(\cD)}})$.
\item With $\cK_\Theta:=H^2(\cD_{*})\oplus\overline{\Delta_\Theta L^2(\cD)}$ and $\cH_{\Theta}:=\cK_\Theta\ominus\cQ_{\Theta}$ we have
$$(M_{G_{i}^* + z G_{ j}}^*\oplus {W_{i}}^*)(M_{G_{ j}^* + z G_{i}}^*\oplus {W_{j}}^*)|_{\cH_{\Theta}}=(M_z^*\oplus M_{\zeta}^*|_{\overline{\Delta_\Theta L^2(\cD)}})|_{\cH_{\Theta}}.$$
\end{enumerate}
In particular, since condition (4) holds for both $(i,j) = (1,2)$ and $(i,j) = (2,1)$, we see that
{\em $T_1^*: = (M_{G_{1}^* + z G_{ 2}}^*\oplus {W_{1}}^*)|_{\cH_\Theta}$ commutes with
$T_2^*: = (M_{G_{2}^* + z G_{ 1}}^*\oplus {W_{2}}^*)|_{\cH_\Theta}$.}
}\end{definition}

This motivates the following definition.

\begin{definition}\label{D:AdmissTriple}\index{admissible triple}
Let $(\mathcal{D},\mathcal{D}_*,\Theta)$ be a  contractive analytic function and $(G_1,G_2)$ on $\mathcal{D}_*$
be a pair of contractions. Let $(W_1,W_2)$ be a pair of commuting unitaries on $\overline{\Delta_\Theta L^2(\mathcal{D})}$. We say that the triple $\Xi=((G_1,G_2),(W_1,W_2),\Theta)$ is {\em admissible} if it satisfies
the admissibility conditions $(1)$--$(4)$  in Definition \ref{D:admis-cond}.
We then say that the triple\index{functional model for admissible triples}
\begin{align}\label{AdmisFuncModel}
\nonumber\underline{\bf T}_{\Xi}&:=({\bf T}_1,{\bf T}_2,{\bf T}_1{\bf T}_2)_{\Xi}\\
&:=P_{\mathcal{H}_\Theta}(M_{G_1^*+zG_2}\oplus W_1,M_{G_2^*+zG_1}\oplus W_2,M_z\oplus M_{\zeta}|_{\overline{\Delta_\Theta(L^2(\cD))}})|_{\mathcal{H}_\Theta}
\end{align}
is the functional model associated with the admissible triple $((G_1,G_2),(W_1,W_2),\Theta)$.

Let us also say that the admissible triple $((G_1,G_2),(W_1,W_2),\Theta)$ is {\em pure}
\index{admissible triple!pure} 
if its
last component $\Theta$ is a purely contractive analytic function.  
\end{definition}

Then we have the following analogue
of Observation \ref{O:reduction} for the Sz.-Nagy--Foias model.

\begin{proposition}   \label{P:pure-adm-triple}
Suppose $((G_1, G_2), (W_1, W_2), \Theta)$ is an admissible triple and that $\Theta$ has a
(possibly nontrivial) decomposition $\Theta = \Theta' \oplus \Theta^0$ with $(\cD', \cD'_*, \Theta')$ a unitary constant and $\Theta^0$ a purely contractive analytic function. Then there is an admissible triple of the form $((G_1^0, G_2^0),
(W_1^0, W_2^0), \Theta^0)$ so that the functional model for $((G_1, G_2), (W_1, W_2), \Theta)$
is unitarily equivalent to the functional model for $((G_1^0, G_2^0), (W_1^0, W_2^0), \Theta^0)$.
\end{proposition}

We shall refer to  $((G_1^0, G_2^0), (W_1^0, W_2^0), \Theta^0)$ as the {\em pure part} of the admissible triple
$((G_1, G_2), (W_1, W_2), \Theta)$. 
\index{admissible triple!pure part of}

\begin{proof}  We suppose that $((G_1, G_2), (W_1, W_2), \Theta)$ is an admissible triple and that $\Theta$
has a (possibly nontrivial) decomposition $\Theta = \Theta' \oplus \Theta^0$ with
$(\cD', \cD'_*, \Theta')$ a unitary constant function and $(\cD^0, \cD_*^0, \Theta^0)$ a purely contractive
analytic function.  Let
$$
\cH_\Theta =  H^2(\cD_*) \oplus \overline{ \Delta_\Theta L^2(\cD)}  \ominus \{ \Theta f \oplus \Delta_\Theta f \colon f \in  H^2(\cD)\}
$$
be the Sz.-Nagy--Foias functional model space associated with $\Theta$
(and hence also the functional model space associated with the admissible triple
$((G_1, G_2), (W_1, W_2), \Theta)$), and let
$$
({\mathbf T}_1^*, {\mathbf T}_2^*, {\mathbf T}^*) =
\left((M_{G_1^* + z G_2} \oplus W_1)^*, (M_{G_2^* + z G_1} \oplus W_2)^*,
(M_z \oplus M_\zeta|_{\overline{\Delta_\Theta L^2(\cD)}})^* \right) \big|_{\cH(\Theta)}
$$
be the associated functional-model triple of contraction operators. (with ${\mathbf T} = {\mathbf T}_1 {\mathbf T}_2$).
As a result of \cite[Theorem VI.3.1]{Nagy-Foias}, we know that ${\mathbf T}$ is c.n.u.\ with
characteristic function $\Theta_{\mathbf T}$ coinciding with $\Theta^0$.  Thus the characteristic triple for
$({\mathbf T}_1^*, {\mathbf T}_2^*, {\mathbf T}^*)$ has the form
$$
\widetilde \Xi : = ((\widetilde G_1, \widetilde G_2), (\widetilde W_1, \widetilde W_2), \Theta_{\mathbf T})
$$
and by Theorem \ref{Thm:SNFmodelPair} it follows that $({\mathbf T}_1, {\mathbf T}_2, {\mathbf T})$
is unitarily equivalent to the model operators
associated with $\widetilde \Xi$.  As already noted, $\Theta_{\mathbf T}$ coincides with $\Theta^0$;
hence there are unitary operators $u \colon \cD_T \to \cD^0$, $u_* \colon \cD_{T^*} \to \cD_*^0$ so that
$$
 \Theta^0(z) u = u_* \Theta_T(z) \text{ for all } z \in {\mathbb D}.
$$
Define operators $G_1^0$, $G_2^0$ on $\cD_*^0$ and $W_1^0$, $W_2^0$ on
$\overline{\Delta_{\Theta^0} L^2(\cD^0)}$ by
$$
  G_i^0 = u_* \widetilde G_i u_*^*, \quad W_i^0 = (u \otimes I_{L^2}) \widetilde W_i (u^* \otimes I_{L^2})|_{\overline{\Delta_{\Theta^0} L^2(\cD^0)}}
$$
for $i = 1,2$.  Then by construction the triple
$$
  \Xi^0 = \left( (G_1^0, G_2^0), (W_1^0, W_2^0), \Theta^0 \right)
 $$
 coincides with $\widetilde \Xi$ and hence is also admissible.  Then by Theorem \ref{UnitaryInv}
 the commuting contractive pair $({\mathbf T}_1, {\mathbf T}_2)$ is also unitarily equivalent
 to the functional-model commuting contractive pair associated with the admissible triple $\Xi^0$.
This completes the proof of Proposition \ref{P:pure-adm-triple}.
\end{proof}

\begin{proposition}\label{P:admissPCC}
Let $\Xi=((G_1,G_2),(W_1,W_2),\Theta)$ be a pure admissible triple
and let  $\underline{T}_\Xi$ be the functional model associated with $\Xi$. Then the model triple of operators
\begin{align*}
\underline{\bbW}=(M_{G_1^*+zG_2}\oplus W_1, \quad M_{G_2^*+zG_1}\oplus W_2, \quad M_z\oplus M_{\zeta}|_{\overline{\Delta_\Theta(L^2(\cD))}})
\end{align*} 
on $\cK_\Theta$ is a pseudo-commuting contractive lift of $\underline{\mathbf T}_\Xi$ with the inclusion map $i:\cH_\Theta\to\cK_\Theta$ as the 
associated isometric embedding operator.
\end{proposition}

\begin{proof}
By the analysis of Case 1 in the proof of Theorem \ref{Thm:DPseudo} we see that the content of the admissibility conditions (1), (2) in Definition \ref{D:admis-cond} is that the triple
$\underline{\bbW} = (\bbW_1, \bbW_2, \bbW)$  given by
$$
\bbW_1 = M_{G_1^* + z G_2} \oplus W_1, \quad \bbW_2 = M_{G_2^* + z G_1} \oplus W_2, \quad \bbW = M_z^{\cD_*} \oplus M_\zeta |_{ \overline{\Delta_\Theta L^2(\cD)}},
$$
all acting on $H^2(\cD_*) \oplus \overline{ \Delta_\Theta L^2(\cD)}$,  is the general form for a pseudo-commuting contractive triple with third component specified to be
$\bbW = M_z^{\cD_*} \oplus M_\zeta|_{ \overline{  \Delta_\Theta L^2(\cD)}}$.  Thus it follows that $\underline{\bbW}$ is a pseudo-commuting contractive triple.
Condition (3) in Definition \ref{D:admis-cond}  is equivalent to saying that $\cH_\Theta:=  \cK_\Theta \ominus \cQ_\Theta$ is jointly invariant for the adjoint
$\underline{\bbW}^* = (\bbW_1^*, \bbW_2^*, \bbW^*)$ and we can define $\underline{\mathbf T}^*$ on $\cH(\Theta)$ by
\begin{equation}  \label{i-lift}
\underline{\mathbf T}^* =\underline{\bbW}^*|_{\cH(\Theta)} = :({\mathbf T}_1^*, {\mathbf T}^*_2, {\mathbf T}^*).
\end{equation}
As a consequence of admissibility condition (4) in Definition \ref{D:admis-cond} we see that ${\mathbf T}_1^*$ and ${\mathbf T}_2^*$ commute, and the
product operator ${\mathbf T}_1^* {\mathbf T}_2^*$  is equal to $M_z^* \oplus M_\zeta^*|_{\overline{\Delta_\Theta L^2(\cD)}} = {\mathbf T}^*$, and we conclude that
$\underline{\mathbf T}$ is equal to the functional model operator-triple \eqref{AdmisFuncModel} associated with the admissible triple $\Xi$.
Next note that equation \eqref{i-lift} is just the statement that $(i, \underline{\bbW})$ is a lift of ${\mathbf T}_\Xi$ on $\cH(\Theta)$, where the isometric embedding operator
$i \colon \cH_\Theta \to \cK_\Theta$ is just the inclusion map.
Finally,  the fact that $(i, {\mathbb W})$ is a minimal lift of ${\mathbf T}$ is part of the assertion of Theorem VI.3.1 in \cite{Nagy-Foias}.  We can now conclude that 
$(i, {\mathbb W}_1, {\mathbb W}_2, {\mathbb W})$ is a pseudo-commuting lift of $({\mathbf T}_1, {\mathbf T}_2, {\mathbf T} = {\mathbf T}_1 {\mathbf T}_2)$
in the sense of Definition  \ref{D:pcc}. This completes the proof.
\end{proof}

For $\Theta$ a purely contractive analytic function, we have the following result.

\begin{theorem}\label{AdmisCharc}
Let $(\mathcal{D},\mathcal{D}_*,\Theta)$ be a purely contractive analytic function, let $(G_1,$ $G_2)$ on $\mathcal{D}_*$ be a pair of contractions, and let $(W_1,W_2)$ on 
$\overline{\Delta_\Theta L^2(\mathcal{D})}$ be a pair of
commuting unitaries such that their product $W_1 W_2$ is equal to $M_{\zeta}|_{\overline{\Delta_\Theta L^2(\cD)}}$.  Then
$((G_1,G_2),(W_1,W_2),\Theta)$ is admissible if and only if it coincides with the characteristic triple of
some commuting contractive pair $(T_1, T_2)$ with  product operator $T = T_1 T_2$ equal to a c.n.u.\ contraction. In fact, the triple
$((G_1,G_2),(W_1,W_2),\Theta)$ coincides with the characteristic triple of its functional model as
defined in (\ref{AdmisFuncModel}).
\end{theorem}

\begin{proof}
We have already observed that the characteristic triple of a pair $(T_1,T_2)$ of commuting contractions with
$T=T_1T_2$ being a c.n.u.\ contraction is indeed a pure admissible triple (since characteristic functions
$\Theta_T$ are necessarily purely contractive analytic functions).

Conversely suppose that $\Xi=((G_1,G_2),(W_1,W_2),\Theta)$ is a pure admissible triple. This means
that the pair $({\bf T}_1,{\bf T}_2)$ defined on
$$
\mathcal{H}_\Theta:=\big(H^2(\mathcal{D_*})\oplus \overline{\Delta_\Theta L^2(\mathcal{D})}\big)\ominus \{\Theta f\oplus \Delta_\Theta f:f\in H^2(\mathcal{D})\}
$$
by
$$
({\bf T}_1,{\bf T}_2):=P_{\mathcal{H}_\Theta}(M_{G_1^*+zG_2}\oplus W_1,M_{G_2^*+zG_1}
\oplus W_2)|_{\mathcal{H}_\Theta}
$$
is a commuting pair of contractions with product operator given by
\begin{align}\label{producT}
{\bf T}:={\bf T}_1{\bf T}_2=
P_{\mathcal{H}_{\Theta}}(M_z\oplus M_{\zeta}|_{\overline{\Delta_\Theta(L^2(\cD))}})|_{\mathcal{H}_\Theta}.
\end{align}
By the Sz.-Nagy--Foias model theory for a single contraction operator $T$ (see \cite[Theorem VI.3.1]{Nagy-Foias}),
we conclude that ${\bf T}$ is a c.n.u.\ contraction.
We claim that the triple $((G_1,G_2),(W_1,W_2),\Theta)$
coincides with the characteristic triple for $({\bf T}_1,{\bf T}_2)$, which we assume to be
$((G'_1,G'_2),(W'_1,W'_2),\Theta_{{\bf T}})$. Since $\Theta$ is a purely contractive analytic function,
by (\ref{producT}) and Theorem VI.3.1 in  \cite{Nagy-Foias}, we conclude that $\Theta$ coincides
with $\Theta_{{\bf T}}$.  By definition this  means
that there exist unitaries $u:\mathcal{D}\to\mathcal{D}_{\bf T}$ and $u_*:\mathcal{D}_*\to\mathcal{D}_{{\bf T}^*}$
 such that $\Theta_{{\bf T}}u=u_*\Theta$. Then the unitary operator $u_*\oplus\omega_u$ takes
the space $\cK_\Theta:= H^2(\mathcal{D_*})\oplus\overline{\Delta_\Theta L^2(\mathcal{D})}$ onto the space
 $\cK_{\Theta_{\bT}}:= H^2(\mathcal{D}_{\bT^*})\oplus\overline{\Delta_{\Theta_{\bf T}} L^2(\cD_{\bf T})}$ and also 
 $$  u_* \oplus \omega_u \colon  \{ \Theta f \oplus \Delta_\Theta f \colon f \in H^2(\cD) \} \underset{\rm onto}\to 
  \{\Theta_\bT g \oplus \Delta_{\Theta_\bT} g \colon g \in H^2(\cD_\bT) \}.
  $$
  We can therefore conclude that furthermore $u_*\oplus \omega_u$ maps $\cH_\Theta$ onto $\cH_{\Theta_{\bT}}$ where
  $$
  \cH_\Theta := \cK_\Theta \ominus \{ \Theta f \oplus \Delta_\Theta f \colon f \in H^2(\cD) \}, \quad
  \cH_{\Theta_\bT}:= \cK_{\Theta_\bT} \ominus \{ \Theta_\bT g + \Delta_{\Theta_\bT} h \colon g \in H^2(\cD_\bT \}.
  $$
 Denote by $\tau$ the restriction of $u_*\oplus\omega_u$ to $\mathcal{H}_\Theta$.
 Then we have the following commuting diagram, where $i$ and $i'$ are the inclusion maps:
 $$
\begin{CD}
\mathcal{H}_\Theta @> i>> \mathcal K_\Theta\\
@V\tau VV @VVu_*\oplus\omega_u V\\
\mathcal{H}_{\rm NF}@>>i'> \mathcal{K}_{\rm NF}
\end{CD}
$$
By Proposition \ref{P:admissPCC}, we know that $(i, \underline{\bbW})$ is a pseudo-commuting contractive lift of $\underline{\bf T}_\Xi$ on $\cK_\Theta$, where
$$
\underline{\bbW}=(M_{G_1^*+zG_2}\oplus W_1,M_{G_2^*+zG_1}\oplus W_2,
M_z\oplus M_{\zeta}|_{\overline{\Delta_\Theta(L^2(\cD))}}).
$$
Theorem \ref{Thm:NFPseudo} and the diagram above shows that $(i'\circ\tau, \underline{\bbW'})$ is also a pseudo-commuting contractive lift of $\underline{\bf T}_\Xi$,
on $\cK_{\Theta_{\bT}}$
where
$$
\underline{\bbW'}=(M_{G_1'^*+zG_2'}\oplus W_1',M_{G_2'^*+zG_1'}\oplus W_2',M_z\oplus
M_{\zeta}|_{\overline{\Delta_{{\bf T}}(L^2(\cD_{\bf T}))}}).
$$
Now by the uniqueness result in Corollary \ref{C:FundOp}, there exists a unitary $U:\mathcal{K}_\Theta\to\mathcal{K}_{\Theta_\bT}$ such that 
$U\underline{\bbW}=\underline{\bbW'}U$ and $U \circ i=i'\circ\tau$. Since the last entries of $\underline{\bbW}$ and $\underline{\bbW'}$ are the minimal
isometric dilations of $T=T_1T_2$, such a unitary is in fact unique. By the above commuting diagram we see that $u_*\oplus\omega_u$ is one such unitary. Therefore we get
$$
(u_*\oplus\omega_u)\underline{W}=\underline{W'}(u_*\oplus\omega_u).
$$
Consequently $((G_1,G_2),(W_1,W_2),\Theta)$ coincides with $((G'_1,G'_2),(W'_1,W'_2),\Theta_{{\bf T}})$
and the theorem follows.
\end{proof}

Let us mention   that the results of this and the previous section can be stated more succinctly in the language of Category Theory
as follows.

\begin{proposition}   \label{P:catlan}
Define the following categories:
\begin{enumerate}
\item[(i)]
Let ${\mathfrak C}_1$ be the category of all commuting pairs of contraction operators ${\mathbf T} = (T_1, T_2)$
where we set $T = T_1 \cdot T_2 = T_2 \cdot T_1$ and we assume that $T$ is c.n.u.
\item[(ii)] Let ${\mathfrak C}_2$ be the category of all purely contractive admissible triples
$\Xi=((G_1,G_2),$  $(W_1,W_2), \Theta)$.
\end{enumerate}
Define functors ${\mathfrak f} \colon {\mathfrak C}_1 \to {\mathfrak C}_2$ and
${\mathfrak g} \colon {\mathfrak C}_2 \to {\mathfrak C}_1$  by
\begin{align*}
& {\mathfrak f} \colon {\mathbf T} \mapsto \Xi_{\mathbf T} =
\text{ characteristic triple for } {\mathbf T}, \\
& {\mathfrak g}  \colon \Xi \mapsto {\mathbf T}_\Xi
= \text{functional-model commuting contractive pair} \\
& \quad \quad \quad \quad \text{ associated with } \Xi
\text{ as in } \eqref{AdmisFuncModel}.
\end{align*}
Then, for ${\mathbf T}, {\mathbf T}' \in {\mathfrak C}_1$ and $\Xi,
\Xi'  \in {\mathfrak C}_2$, we have
 \begin{enumerate}
\item  ${\mathbf T} \underset{u}\cong {\mathbf T'} \Leftrightarrow
{\mathfrak f}({\mathbf T}) \underset{c}\cong {\mathfrak f}({\mathbf T}')$,

\smallskip

\item  $\Xi \underset{c} \cong \Xi'\Leftrightarrow
{\mathfrak g}(\Xi) \underset{u} \cong
{\mathfrak g}(\Xi')$,

\smallskip

\item ${\mathfrak g} \circ {\mathfrak f}({\mathbf T}) \underset{u}\cong {\mathbf T}$,

\smallskip

\item ${\mathfrak f} \circ {\mathfrak g}(\Xi) \underset{c}\cong
(\Xi)$
\end{enumerate}
where $\underset{u} \cong$ denotes {\em unitary equivalence of operator tuples} and
$\underset{c} \cong$ denotes {\em coincidence of admissible triples}.
\end{proposition}

Finally,  we now show how it is possible to apply the theory of admissible triples to the following factorization problem:

\smallskip

\noindent
\textbf{Commuting Contractive Factorization Problem:}  {\sl Given a c.n.u.\ contraction operator $T$, find all commuting contractive operator-pairs $(T_1, T_2)$ which generate a  commuting contractive factorization of $T$:  $T = T_1 T_2 = T_2 T_1$.}

\begin{corollary}  \label{C:contr-fact} Solutions $(T_1, T_2)$ of the commuting contractive factorization problem are in one-to-one correspondence with pairs $\left((G_1, G_2),
(W_1, W_2) \right)$ which complete the characteristic function $\Theta_T$ to an admissible triple, i.e., with such pairs such that
$$
  \left( (G_1, G_2), (W_1, W_2), \Theta_T \right)
$$
is an admissible triple.
\end{corollary}

\begin{proof}  If $(T_1, T_2)$ is a solution of the commuting contractive factorization problem for $T$, then from the definitions we see that the characteristic triple of $(T_1, T_2)$
solves the admissible-triple completion problem for $\Theta_T$.  

Conversely, suppose that $\left( (G_1, G_2), (W_1, W_2) \right)$ solves the
admissible-triple completion problem for $T$.  Then the model operator-triple $\underline{\bT}_\Xi$ for $\Xi = \left( (G_1, G_2), (W_1,\right. $ $\left. W_2), \Theta_T \right)$
(see \eqref{AdmisFuncModel})  generates a commuting contractive operator-pair $(\bT_1, \bT_2)$ with product $\bT_1 \bT_2 = \bT_2 \bT_1$
equal to 
$$
\bT:= P_{\cH(\Theta_T)} ( M_z^{\cD_{T^*}} \oplus  M_\zeta|_{\overline{\Delta_{\Theta_T} L^2(\cD_T)}}) |_{\cH(\Theta_T)}
$$ 
which is the Sz.-Nagy--Foias functional-model operator for $T$ and hence is unitarily equivalent to $T$.
\end{proof}

\begin{example}  \label{E:ComContrFact}  As an illustrative example of the previous result,  let us suppose that $(\cD, \cD_*, \Theta)$ is a finite Blaschke-Potapov-product 
matrix inner function (so $\dim \cD_* < \infty$). Then the model space $\cH_\Theta$ collapses to $\cH_\Theta = H^2(\cD_*) \ominus \Theta H^2(\cD)$ and is finite-dimensional.
As a mildly simplifying assumption, let us suppose that $\cH_\Theta$ has a basis consisting of vector-valued kernel functions 
$$
 {\mathfrak B} = \{ d_{m, n} k_{w_m} \colon 1 \le m \le M, \, 1 \le n \le n_m \}
$$
where $w_1, \dots, w_M$ are distinct points in ${\mathbb D}$, where for each $m$ ($1 \le m \le M$), the set $\{ d_{m, 1}, \dots, d_{m, n_m}\}$
is a linearly independent set of vectors in $\cD_*$ (so $n_m \le \dim \cD_*$ for each $1 \le m \le M$), and
where in general, for  $w \in {\mathbb D}$ and $d \in \cD_*$ the function  $(d k_{w})(z) = \frac{d}{1 - z \overline{w}}$ is the $H^2(\cD_*)$-kernel function 
for evaluation of $h\in H^2(\cD_*)$ at the point $w \in {\mathbb D}$ in direction $d \in \cD_*$:
$$
   \langle h, \,  d k_{w} \rangle_{H^2(\cD_*)} = \langle h(w), d \rangle_{\cD_*},
 $$
Since $\Theta$ is inner, the second component $(W_1, W_2)$ of any admissible triple solving the admissible-triple completion problem for $\Theta$ is vacuous 
so any  solution of the admissible-triple completion problem consists 
simply of two matrices $G_1, G_2$ considered as operators on $\cD_*$.
Given such a pair of matrices, 
let us define two matrix pencils.
$$
\varphi_1(z) = G_1^* + z G_2, \quad \varphi_2(z) = G_2^* + z G_1.
$$
Then $(G_1, G_2)$ solves the admissible-triple completion problem for $\Theta$ if and only if:
\begin{itemize}
\item[(i)] the operators $M_{\varphi_1}$ and $M_{\varphi_2}$ are contractions on $H^2(\cD_*)$, i.e. $\| \varphi_i(z) \| \le 1$ for all $z \in {\mathbb D}$ and $i=1,2$.
\item[(ii)] the space $\cH_\Theta$ is jointly invariant for $(M_{\varphi_1}^*, M_{\varphi_2}^*)$.
\item[(iii)] $M_{\varphi_1}^*M_{\varphi_2}^*|_{\cH_\Theta}=M_{\varphi_2}^*M_{\varphi_1}^*|_{\cH_\Theta} =(M_z^{\cD_*})^*|_{\cH_\Theta}$.
\end{itemize}
If we can find a solution $(G_1, G_2)$ of conditions (ii) and (iii), we get a solution of condition (i) simply by a rescaling $(G_1, G_2) \mapsto (\mu G_1, \mu G_2)$ for a scalar $\mu >0$
sufficiently small, so it suffices to consider conditions (ii) and (iii).
Let us note the action of $M_{\varphi_i}^*$ on a  general kernel function:
$$
 M_{\varphi_i}^* \colon  d \, k_w \mapsto (G_i  + \overline{w} G_j^*)d \, k_w.
$$
Let us introduce for $1 \le m \le M$ the subspace
$$
   \cS_m = \bigvee \{ d_{m, n}  \colon  1 \le n  \le n_m  \}.
$$
Then condition (ii) amounts to the collection of invariant subspace conditions:
\begin{itemize}
\item[(ii$^\prime$)]  
  $G_i + \overline{w}_m G_j^* \colon \cS_m \to \cS_m$  for  $1 \le m \le N$ for $(i,j) = (1,2)$ and $(2,1)$.
\end{itemize}

Let us note next that for a general kernel function $d\,k_w$ we have, for $(i,j) = (1,2)$ or $(2,1)$,
\begin{align*}
M_{\varphi_i}^* M_{\varphi_j}^* (d\, k_w) & = M_{\varphi_i}^* (G_j + \overline{w} G_i^*) d k_w  \\
& = G_i (G_j + \overline{w} G_i^*) d k_w + \overline{w} G_j^* (G_j + \overline{w} G_i^*) d  k_w  \\
& = \left(G_i G_j + \overline{w} (G_i G_i^* + G_j^* G_j) + \overline{w}^2 G_j^* G_i^* \right)  d k_w
\end{align*}
and hence condition (iii) boils down to the condition:  
\begin{itemize}
\item[(iii$^\prime$)]  For $1 \le m \le M$, for all  $d \in \cS_m$ we have
\begin{align*}
& \left(G_1G_2 + \overline{w}_m (G_1 G_1^* + G_2^* G_2) + \overline{w}_m^2 G_2^* G_1^* \right) d \\
& \quad = \left(G_2 G_1 + \overline{w}_m (G_2 G_2^* + G_1^* G_1) + \overline{w}_m^2 G_1^* G_2^* \right) d\\
& \quad = \overline{w}_m d
\end{align*}
\end{itemize}

It appears to be difficult to sort out how to solve these conditions (ii$^\prime$) and (iii$\prime$) in general.  Therefore we resort to discussing a couple of more tractable special
cases where a complete solution can be found.

\smallskip

\noindent
\textbf{Special Case 1:  $\cD_* = {\mathbb C}$, $M > 1$ and $w_1 = 0$.}  In this case condition (ii$^\prime$) is automatic as $\cS_m$ is the whole space ${\mathbb C}$
and any solution $(G_1, G_2)$ of the admissible-triple completion problem has the form $(G_1, G_2) = (g_1, g_2)$ where $g_1, g_2$ are complex numbers.  
We shall show: 

\smallskip

\noindent
\textbf{Claim:}
{\em The only solution of the admissible-triple completion problem for this special case is that either $(g_1, g_2)$ or $(g_2, g_2)$ is in the set
${\mathbb T} \times \{0\}$.}

\smallskip
Indeed, note that 
condition (ii$^\prime$) now assumes the form
\begin{equation}   \label{scalar-form}
 g_1 g_2 + \overline{w}_m ( |g_1|^2 + | g_2 |^2) + \overline{w}_m^2  \overline{g}_2 \overline{g}_1
 = g_2 g_1 + \overline{w}_ (|g_2|^2 + | g_1|^2) + \overline{w}_m^2 \overline{g_1} \overline{g}_2 = \overline{w}_m
 \end{equation}
 for $m= 1, \dots, M$.  In particular, if we apply criterion \eqref{scalar-form} to the case $m = 1$ where by hypothesis $w_1 = 0$ we get
 $$
   g_1 g_2 = 0
 $$
 and hence at least one of $g_1$ and $g_2$ is equal to $0$.  If both $g_1 = g_2 = 0$ and we apply \eqref{scalar-form} to the case  where $w_m \ne 0$,
 we get $0 = \overline{w}_m$ leading to a contradiction.  Hence exactly one of $g_1$, $g_2$ is zero and the other is non-zero;  say that $g_1 \ne 0$ and $g_2 = 0$.  Apply 
 \eqref{scalar-form} again to the case of any $\alpha$ where $w_m \ne 0$ to get
 $$  \overline{w}_m |g_1|^2 = \overline{w}_m
 $$
 leading to the conclusion that $|g_1 | = 1$.   Similarly, if we assume that $g_1 = 0$, we are forced to the solution $(g_1, g_2) = (0, \omega)$ for some
 $\omega \in {\mathbb T}$.
  Conversely, one can easily check that  $(g_1, g_2) = (\omega, 0)$ or $(g_1, g_2) = (0, \omega)$  for some unimodular  $\omega \in {\mathbb T}$
 leads to a solution of \eqref{scalar-form} for all $m = 1, \dots, M$, and the Claim follows.
 
 As a corollary we get a solution of the associated  commuting contractive factorization problem:
 {\em Suppose that $T$ is a contraction operator on a finite-dimensional Hilbert space $\cH$ with $\dim \cH \ge 2$ having defect rank $\operatorname{rank} (I - T T^*)$
 equal to  $1$ and having distinct simple eigenvalues all in the open unit disk,  one of which is $0$, and suppose that $T_1$ and $T_2$ are commuting contractive operators
 on $\cH$ such that $T = T_1 T_2$.  Then there is unimodular $\omega \in {\mathbb T}$ so that}
 $$
 \text{ either } (T_1, T_2) = (\omega T, \overline{\omega} I)  \text{ or } (T_1, T_2) = (\overline{\omega} I, \omega T).
 $$

 \smallskip
 
 \noindent
 \textbf{Special Case 2:   $\Theta = ( \{0\}, \cD_*, 0)$.}
 We consider next the limiting special case where we replace the finite point-set $(w_1, \dots, w_M)$ with the whole unit disk ${\mathbb D}$
 and the associated coefficient space $\cS_m$ (now indexed as $\cS_{w_m}$) with the whole coefficient space $\cD_*$.  
 Then the span of the kernels $\{ d k_w \colon d \in \cD_*, w \in {\mathbb D} \}$
 is the whole space $H^2(\cD_*)$ which can be viewed as $\cH_\Theta$ where $\Theta$ is the zero function from $\{0\}$ into $\cD_*$ (and hence can be viewed now as a
 non-square inner function). Then we seek a solution of the admissible-triple completion problem for the case where we have the non-square inner function
 $(\{0\}, \cD_*, \Theta = 0)$.  Note again that the third component $(W_1, W_2)$ of an admissible triple solving the admissible-triple completion problem for the case
 $\Theta = (\{0\}, \cD_*, 0)$ is again vacuous since in this case the input space is trivial:  $\cD = 0$.  
 Thus we wish to find an operator-pair $(G_1, G_2)$ on the coefficient Hilbert space $\cD_*$ so that
 conditions (i), (ii), (iii) hold on $\cH_\Theta$ which now is equal to all of $H^2(\cD_*)$.   As the kernel functions $\{ d k_w \colon d \in \cD_*, w \in {\mathbb D} \}$
 has span equal to all of $H^2(\cD_*) = \cH_\Theta$, we can again formulate the problem just as was done above, but now with the point-set $\{ w_1, w_2, \dots \}$ equal to all of ${\mathbb D}$ and with the associated vector spaces $\cS_w$ equal to the whole space $\cD_*$ for each $w \in {\mathbb D}$.  This leads us to
 conditions (i), (ii), (iii) having to hold.
 Again condition (i) can be handled by a rescaling of $(G_1, G_2)$ and conditions (ii) and (iii) become conditions (ii$^\prime$) and (iii$^\prime$).
 Condition (ii$^\prime$) is again automatic since $\cS_w = \cD_*$ for all $w \in {\mathbb D}$.   Then condition (iii$^\prime$) must hold for all $w_m \in {\mathbb D}$.
 This leads to the  system of quadratic operator polynomial equations in the variable $\overline{w} \in {\mathbb D}$:
 \begin{align*}
 & G_1 G_2 + \overline{w} (G_1 G_1^* + G_2 ^* G_2) + \overline{w}^2 G_2^* G_1^* \\
 & \quad = G_2 G_1 + \overline{w} (G_2 G_2^* + G_1^* G_1) + \overline{w}^2 G_1^* G_2^*
   = \overline{w} I_{\cD_*}.
 \end{align*}
 As this must hold for all $w \in {\mathbb D}$, we are led to the system of operator equations
 \begin{align*}
 & G_1 G_2 = G_2 G_1 = 0, \\
 & G_1 G_1^* + G_2^* G_2 = G_2 G_2^* + G_1^* G_1 = I_{\cD_*}, \\
 & G_2^* G_1^* = G_1^* G_2^* = 0.
 \end{align*}
 Note that the third line of equations is just the adjoint version of the first line and hence can be dropped.  Next observe that this slimmed-down system of operator equations
 is exactly the same as the system of equations \eqref{PIconds} appearing in Lemma \ref{relations-of-E-lem} (with $(G_1, G_2)$ replacing the $(E_1, E_2)$
 appearing there).  By Lemma \ref{relations-of-E-lem} we know that any solution has the form
 $$
   (G_1, G_2) = (U^* P^\perp, PU)
 $$
 where $U$ is a unitary operator and $P$ is a projection operator on $\cD_*$.  
 As a corollary we get  what we can call {\em the Berger-Coburn-Lebow solution of the associated commuting contractive factorization problem:}
 {\em if $(T_1, T_2)$ is a commuting contractive operator-pair on $H^2(\cD_*)$ solving the factorization problem
 $$
    M_z^{\cD_*} = T_1 T_2 = T_2 T_1,
 $$
 then there is a projection $P$ and unitary $U$ on $\cD_*$ so that}
 $$
   T_1 = M_{ P^\perp U + z P U}, \quad T_2 = M_{U^*P + z U^* P^\perp}.
 $$ 
 In the special case where $\cD_* = {\mathbb C}$, the possibilities for projection operators $P$  and unitary operators $U$ are limited to
 $$
   P = 0 \text{ or } P = 1, \quad U= \omega \in {\mathbb T}.
 $$
 Then solutions of the associated commuting contractive factorization problem {\em find a commuting contractive operator-pair $(T_1, T_2)$ on $H^2$ so that
 $M_z  = T_1 T_2 = T_2 T_2$} has the same trivial form as was the case for Special Case 1 discussed above:  {\em there is an $\omega \in {\mathbb T}$ so that}
 $$
   (T_1, T_2)  = (\omega I_{H^2}, \overline{\omega} M_z) \text{ or } (T_1, T_2) = (\omega M_z, \overline{\omega} I_{H^2}).
 $$
 \end{example}

\numberwithin{theorem}{chapter}
\numberwithin{equation}{chapter}

\chapter[Joint invariant subspaces]{Characterization of joint invariant subspaces for pairs of commuting contractions}
\label{C:invsub}

In this chapter we characterize invariant subspaces for pairs $(T_1,T_2)$ of commuting contractions
such that $T=T_1T_2$ is a c.n.u.\ contraction. Sz.-Nagy and Foias characterized how invariant subspaces for
c.n.u.~contractions arise in the functional model (see \cite[Chapter VII]{Nagy-Foias}). They showed that invariant subspaces of a c.n.u.~contraction $T$
are in one-to-one correspondence with regular factorizations of the characteristic function of $T$.
Recall that, given contractive analytic functions $(\cD, \cD_*, \Theta)$ $(\cD, \cF, \Theta')$, $(\cF, \cD_*, \Theta'')$
such that there is a factorization, $\Theta = \Theta'' \Theta'$,  the factorization is said to be {\em regular}
\index{regular factorization for!contractive analytic functions} 
if the 
contractive operator factorization $\Theta(\zeta) = \Theta''(\zeta) \Theta'(\zeta)$ is regular in the sense discussed in Remark \ref{R:specATuple}. 
This means that the map $Z \colon
\overline{ \Delta_\Theta L^2(\cD) } \to \overline{ \Delta_{\Theta''} L^2(\cF)} \oplus  \overline{ \Delta_{\Theta'} L^2(\cD)}$ defined densely by
\begin{equation}   \label{Z'}
 Z \colon \Delta_\Theta(\zeta) h(\zeta)  \mapsto \Delta_{\Theta''} (\zeta) \Theta'(\zeta) h(\zeta) \oplus \Delta_{\Theta'}(\zeta) h(\zeta),
\end{equation}
which is necessarily an isometry (the pointwise version of the map \eqref{Z'} given in \S \ref{S:SpcATuple}), is actually surjective and hence unitary.   
A minor complication in the theory is that the factors in a regular factorization of a purely contractive analytic
function need not again be purely contractive.  We now recall
their result  as we shall have use of it later in this section.

\begin{theorem}[Sz.-Nagy--Foias]\label{NFcnu}
Let $(\cD,\cD_*,\Theta)$ be a purely contractive analytic function and $\bf T$ be the contraction on
$$
\bH := \cH_\Theta=H^2(\cD_*)\oplus\overline{\Delta_{\Theta} L^2(\cD)}\ominus\{\Theta f\oplus\Delta_{\Theta} f: f\in H^2(\cD)\}
$$
defined by
$$
  {\bf T}=P_{\bH}\big(M_z\oplus M_{\zeta}\big)|_{\cH_\bH}.
$$
A subspace $\bH'$ of $\bH$ is invariant under ${\bf T}$ if and only if there exist  contractive analytic functions
$(\mathcal{D},\mathcal{F},\Theta')$, $(\mathcal{F},\mathcal{D}_{*},\Theta'')$ such that
$$
  \Theta=\Theta'' \Theta'
$$
is a regular factorization, and with the unitary $Z$ as in \eqref{Z'} we have
\begin{align} \label{cnuH'}
\bH'=\{\Theta''f\oplus Z^{-1}(\Delta_{\Theta''}f\oplus g):\;f\in H^2(\mathcal{F}),g\in\overline{\Delta_{\Theta'}L^2(\mathcal{D})}\}\\\nonumber
\ominus\;\{\Theta h\oplus\Delta_{\Theta}h:h\in H^2(\mathcal{D})\}
\end{align}
and
\begin{align} \label{cnuH''}
\bH'' := \bH\ominus\bH'=H^2(\cD_*)\oplus &Z^{-1}(\overline{\Delta_{\Theta''}L^2(\mathcal{F})}\oplus\{0\})\\\nonumber
&\ominus\{\Theta''f\oplus Z^{-1}(\Delta_{\Theta''}f\oplus 0):f\in H^2(\mathcal{F})\}.
\end{align}
Moreover, the characteristic function of ${\mathbf T}|_{{\mathbb H}'}$ coincides with the purely contractive part
of $\Theta'$, and the characteristic function of $P_{{\mathbb H}''} {\mathbf T}|_{{\mathbb H}''}$
coincides with the purely contractive part of $\Theta''$.
\end{theorem}

\begin{remark}  \label{R:BL}  For future reference let us mention the special case of Theorem \ref{NFcnu} where the input coefficient space
$\cD$ is taken to be the zero space and the contractive analytic function $\Theta$ is just the zero function $\Theta(z) = 0 \colon \{0\} \to \cD_*$.  In this
case the only contractive analytic factorizations $\Theta = \Theta'' \Theta'$ have $\Theta''$ equal to an arbitrary contractive analytic function
$(\cF, \cD_*, \Theta'')$ and $\Theta'$ equal to a zero function $(\{0\}, \cF, \Theta')$ ($\Theta'(z) = 0 \colon \{0\} \to \cF$).  The only such factorizations which are
regular are those for which $\Theta''$ is an inner function.    In this case, in the context of Theorem \ref{NFcnu}, $\bH = H^2(\cD_*)$, ${\bf{T}} = M_z$ on
$H^2(\cD_*)$, and the invariant subspace corresponding to the regular factorization $0 = \Theta'' \cdot 0$ is $\bH' = \Theta'' H^2(\cF)$.
Thus Theorem \ref{NFcnu} can be thought of as a generalization of the Beurling-Lax Theorem\index{Beurling-Lax Theorem}.
\end{remark}

Let $T$ be a c.n.u.\ contraction such that $T=T_1T_2$ for a pair $(T_1,T_2)$ of commuting contractions.
It is natural that one would need more conditions than (\ref{cnuH'}) and (\ref{cnuH''}) for an invariant
subspace of $T$ to be jointly invariant under $(T_1,T_2)$.

\begin{theorem} \label{T:inv-cnu}
Let $(\cD,\cD_*,\Theta)$ be a purely contractive analytic function and let  the triple $((G_1,G_2),(W_1,W_2),\Theta)$
be admissible. Define the pair $({\bf T_1},{\bf T_2})$ of commuting contractions on
\begin{align}\label{cnuMspace}
\bH=H^2(\cD_*)\oplus\overline{\Delta_{\Theta} L^2(\cD)}\ominus\{\Theta f\oplus\Delta_{\Theta} f: f\in H^2(\cD)\}
\end{align}
by
\begin{align}\label{Mop}
 ({\bf T_1},{\bf T_2})=P_{\bH}\big(M_{G_1^*+zG_2}\oplus W_1,M_{G_2^*+zG_1}\oplus W_2\big)|_{\bH}.
\end{align}
A subspace $\bH'$ of $\bH$ is jointly invariant under $({\bf T_1},{\bf T_2})$ if and only if there exist a Hilbert space $\cF$, two contractions
$G'_1$, $G'_2$ in $\mathcal{B}(\mathcal{F})$, two contractive analytic functions $(\mathcal{D},\mathcal{F},\Theta')$, $(\mathcal{F},\mathcal{D}_{*},\Theta'')$
such that
\begin{eqnarray*}
\Theta=\Theta''  \Theta'
\end{eqnarray*}is a regular factorization, a pair $(W'_1,W'_2)$ of unitary operators on $\overline{\Delta_{\Theta'}L^2(\cD)}$ with the property
\begin{align}\label{propW'}
W'_1W'_2=W'_2W'_1=M_{\zeta}|_{\overline{\Delta_{\Theta'}L^2(\cD)}},
\end{align}and also, with $Z$ the pointwise unitary operator  as in (\ref{Z'}),
\begin{align}\label{H'}
\bH'=\{\Theta''f\oplus Z^{-1}(\Delta_{\Theta''}f\oplus g):\;f\in H^2(\mathcal{F}),g\in\overline{\Delta_{\Theta'}L^2(\cD)}\}\\\nonumber
\ominus\;\{\Theta h\oplus\Delta_{\Theta}h:h\in H^2(\cD)\},
\end{align}
\begin{align}\label{H''}
\bH'':=\bH\ominus\bH'=&\;H^2(\cD_*)\oplus Z^{-1}(\overline{\Delta_{\Theta''}L^2(\mathcal{F})}\oplus\{0\})\\
&\nonumber\ominus\{\Theta''f\oplus Z^{-1}(\Delta_{\Theta''}f\oplus 0):f\in H^2(\mathcal{F})\},
\end{align}
and for every $f\in H^2(\mathcal{F})$ and $g\in\overline{\Delta_{\Theta'}L^2(\cD)}$
\begin{align}\label{ExtraCond}
\begin{bmatrix}
      M_{G_i^*+zG_j} & 0 \\
      0 & W_i \\
     \end{bmatrix}
\begin{bmatrix}
\Theta''f\\Z^{-1}(\Delta_{\Theta''}f\oplus g)
\end{bmatrix}
=\begin{bmatrix}
\Theta''M_{G'_i+zG'_j}f\\Z^{-1}(\Delta_{\Theta''}M_{G'_i+zG'_j}f\oplus W'_i g)
\end{bmatrix},
\end{align}where $(i,j)=(1,2),(2,1)$.
\end{theorem}

\begin{proof}
We first prove the easier part---the proof of sufficiency.   Suppose that
$$
((G_1,G_2),(W_1,W_2),\Theta)
$$
is a purely contractive admissible triple (i.e., $( \cD, \cD_*, \Theta)$ is a purely contractive analytic function)
such that $\Theta$ has a regular factorization $\Theta = \Theta'' \Theta'$ with $(\cD, \cF, \Theta')$ and
$(\cF, \cD_*, \Theta'')$ contractive analytic functions.  We suppose also that $G_1'$ and $G_2'$ are
contraction operators on $\cF$, $W_1'$, $W_2'$ are unitary operators on $\overline{\Delta_{\Theta'} L^2(\cD)}$
so that \eqref{propW'} and \eqref{ExtraCond} hold.  Then we have all the ingredients to define ${\mathbb H}'$
and ${\mathbb H}''$ as in \eqref{H'} and \eqref{H''}.  Note next that ${\mathbb H}'$ is indeed a subspace of
${\mathbb H}$.  We wish to show that the space $\bH'$ given in (\ref{H'}) is
jointly invariant under the pair $({\bf T_1},\bf{T_2})$ defined in (\ref{Mop}). Firstly, it is easy to see that $\bH'$
is a subspace of $\bH$. Since the operator
\begin{align}\label{iso-cnu}
{\mathcal I} \colon H^2(\mathcal{F})\oplus\overline{\Delta_{\Theta'}L^2(\cD)} &\to
H^2(\cD_*)\oplus\overline{\Delta_{\Theta}L^2(\cD)}\\\nonumber
f\oplus g &\mapsto \Theta''f\oplus Z^{-1}(\Delta_{\Theta''}f\oplus g)
\end{align}
is an isometry, the space
$$
\{\Theta''f\oplus Z^{-1}(\Delta_{\Theta''}f\oplus g):f\in H^2(\mathcal{F})
\text{ and } g \in \overline{\Delta_{\Theta'}L^2(\cD)}\}
$$
is closed and by (\ref{ExtraCond}) we see that it is jointly invariant under
$$
(M_{G_1^*+zG_2}\oplus W_1,M_{G_2^*+zG_1}\oplus W_2,M_z\oplus M_{\zeta}).$$
We also see that
$$
\operatorname{Ran}\, \mathcal{I}=\big(H^2(\cD_*)\oplus \overline{\Delta_{\Theta} L^2(\cD)} \big)
\ominus(\bH\ominus\mathbb{H}').
$$
Now the sufficiency follows from the definition of $({\bf T_1},\bf{T_2})$ and from the general fact that if $V$
is an operator on $\mathcal{K}$ containing $\mathcal{H}$, $V(\cK\ominus\cH)\subset\cK\ominus\cH$,
and $V^*|_{\mathcal{H}}=T^*$, then for a subspace $\mathcal{H}'$ of $\mathcal{H}$,
$$
V(\mathcal{K}\ominus(\mathcal{H}\ominus\mathcal{H}'))\subseteq (\mathcal{K}\ominus(\mathcal{H}\ominus\mathcal{H}'))
\text{ if and only if } T(\mathcal{H}')\subseteq\mathcal{H}'.
$$

Now we show that the conditions are necessary. The first step of the proof is an application of
Theorem \ref{NFcnu}. Indeed, if $\bH'\subset\bH$ is jointly invariant under $({\bf T_1},\bf{T_2})$,
then it is also invariant under the product ${\bf T_1T_2}$ and by definition of admissibility
$$
{\bf T}={\bf T_1T_2}={\bf T_2T_1}=P_{\bH}\big(M_z\oplus M_{\zeta}\big)|_{\bH}.
$$
Hence by Theorem \ref{NFcnu}, there exist two contractive analytic functions
$$
(\mathcal{D},\mathcal{F},\Theta'), \quad (\mathcal{F},\mathcal{D}_{*},\Theta'')
$$
such that $\Theta=\Theta''\Theta'$ is a regular factorization and the
spaces $\bH'$ and $\bH''$ are realized as in (\ref{H'}) and (\ref{H''}), respectively. It only remains to
produce contraction operators $G_1'$, $G_2'$ on $\cF$ and unitary operators $W_1'$, $W_2'$ on
$\overline{\Delta_{\Theta'} L^2(\cD)}$ so that conditions \eqref{propW'} and \eqref{ExtraCond} hold.
Note that, once we have found
$G_1'$, $G_2'$, $W_1'$, $W_2'$, verification of \eqref{ExtraCond} breaks up into three linear pieces,
where $(i,j) = (1,2)$ or $(2,1)$:
\begin{align}
& M_{G_i^* + z G_j} \Theta'' f = \Theta'' M_{G_i' + z G'_j} f \text{ for all } f \in H^2(\cF), \label{verify1} \\
& W_i (Z^{-1} (\Delta_{\Theta''} f \oplus 0)= Z^{-1} (\Delta_{\Theta''} M_{G_i' + z G_j'} f \oplus 0)
\text{ for all } f \in H^2(\cF),  \label{verify2} \\
& W_i Z^{-1} (0 \oplus g) = Z^{-1}(0 \oplus  W_i' g) \text{ for all } g \in \overline{\Delta_{\Theta'} L^2(\cD)}.
\label{verify3}
\end{align}

As a first step, we define operators $X_i$ on $H^2(\mathcal{F})$ and $W_i'$ on
$\overline{\Delta_{\Theta'}L^2(\cD)}$, for $i=1,2$, such that for every $f\in H^2(\mathcal{F})$ and
$g\in\overline{\Delta_{\Theta'}L^2(\cD)}$,
\begin{align}\label{Intertwining2}
\nonumber\begin{bmatrix}
      M_{G_i^*+zG_j} & 0 \\
      0 & W_i \\
     \end{bmatrix}\mathcal{I}
\begin{bmatrix}
f\\g
\end{bmatrix}&=
\begin{bmatrix}
M_{G_i^*+zG_j} & 0 \\
      0 & W_i \\
\end{bmatrix}
\begin{bmatrix}
\Theta''f\\Z^{-1}(\Delta_{\Theta''}f\oplus g)
\end{bmatrix}\\
&=\begin{bmatrix}
\Theta''X_if\\Z^{-1}(\Delta_{\Theta''}X_if\oplus W_i'g)
\end{bmatrix}=
\mathcal{I}
\begin{bmatrix}
     X_i & 0 \\
      0 & W_i' \\
\end{bmatrix}
\begin{bmatrix}
f\\g
\end{bmatrix},
\end{align}
where $\cI$ is the isometry as defined in (\ref{iso-cnu}). The operators $X_1,X_2$ and $W_1',W_2'$ are
well-defined because the operator $\mathcal{I}$ is an isometry. Indeed, it follows that $X_1,X_2$ and
$W_1',W_2'$ are contractions. Since the unitary $Z$ commutes with $M_{\zeta}$, it is easy to see from the
definition of $\mathcal I$ that it has the following intertwining property
\begin{equation} \label{Intertwining1}
  \mathcal{I} (M_z \oplus M_{\zeta}|_{\overline{\Delta_{\Theta'}L^2(\cD)}}) =
  (M_z\oplus M_{\zeta}|_{\overline{\Delta_{\Theta}L^2(\cD)}})\mathcal{I}.
\end{equation}

From the intertwining properties (\ref{Intertwining1}) and (\ref{Intertwining2}) of $\mathcal{I}$, we get for $i=1,2$
$$
(X_i\oplus W_i')(M_z\oplus M_{\zeta}|_{\overline{\Delta_{\Theta'}L^2(\cD)}})=
(M_z\oplus M_{\zeta}|_{\overline{\Delta_{\Theta'}L^2(\cD)}})(X_i\oplus W_i'),
$$
which implies that $(X_1,X_2)=(M_{\varphi_1},M_{\varphi_2})$, for some $\varphi_1$ and $\varphi_2$ in
$L^\infty(\mathcal{B}(\mathcal{F}))$. We next show that $\varphi_1$ and $\varphi_2$ are actually linear pencils.
 Toward this end,  notice from (\ref{Intertwining2}) that
\begin{align}\label{subeqn1}
  \begin{bmatrix}
     M_{\varphi_1} & 0 \\
      0 & W_1' \\
\end{bmatrix}&= \mathcal{I}^*
\begin{bmatrix}
      M_{G_1^*+zG_2} & 0 \\
      0 & W_1 \\
\end{bmatrix}\mathcal{I} \\\label{subeqn}
\begin{bmatrix}
     M_{\varphi_2} & 0 \\
      0 & W_2' \\
\end{bmatrix}&= \mathcal{I}^*
\begin{bmatrix}
      M_{G_2^*+zG_1} & 0 \\
      0 & W_2 \\
\end{bmatrix}\mathcal{I}
\end{align}
Now multiplying (\ref{subeqn1}) on the left by $M_z^*\oplus M_{\zeta}^*|_{\overline{\Delta_{\Theta'}L^2(\cD)}}$,
then using the intertwining property (\ref{Intertwining1}) of $\cI$ and then remembering that $(W_1,W_2)$ is a commuting pair of unitaries such that $W_1W_2=M_{\zeta}^*|_{\overline{\Delta_{\Theta}L^2(\cD)}}$, we get
$$
\begin{bmatrix}
     M_{z}^* & 0 \\
      0 & M^*_{\zeta}|_{\overline{\Delta_{\Theta'}L^2(\cD)}} \\
\end{bmatrix}
\begin{bmatrix}
     M_{\varphi_1} & 0 \\
      0 & W_1' \\
\end{bmatrix}
=\mathcal{I}^*
\begin{bmatrix}
      M^*_{G_2^*+zG_1} & 0 \\
      0 & W_2^* \\
\end{bmatrix}\mathcal{I}=
\begin{bmatrix}
     M_{\varphi_2}^* & 0 \\
      0 & W'^*_2 \\
\end{bmatrix}.
$$
Consequently, $M_{\varphi_2}=M_{\varphi_1}^*M_z$. A similar argument as above yields
$M_{\varphi_1}=M_{\varphi_2}^*M_z$. Considering these two relations and the power series
expansions of $\varphi_1$ and $\varphi_2$, we get
\begin{equation}   \label{varphi}
\varphi_1(z)=G_1'^*+zG_2' \text{ and }\varphi_2(z)=G_2'^*+zG_1',
\end{equation}
for some $G_1',G_2'\in\mathcal{B}(\mathcal{F})$. The fact that $M_{\varphi_1}$ (and $M_{\varphi_2}$)
is a contraction implies that $G_1'$ and $G_2'$ are contractions too.  Recalling \eqref{Intertwining2}
and the substitution $(X_1, X_2) = (M_{\varphi_1}, M_{\varphi_2})$ where $\varphi_1$ and $\varphi_2$
are given by \eqref{varphi},  we see that we have established \eqref{verify1} with the choice of $G_1'$, $G_2'$
as in \eqref{varphi}.

Next note that the bottom component of  \eqref{Intertwining2} gives us
\begin{equation}          \label{ConseqInt2}
W_iZ^{-1}(\Delta_{\Theta''}f\oplus g)=Z^{-1}(\Delta_{\Theta''}X_if\oplus W_i'g).
\end{equation}
for all $f \in H^2(\cF)$, $g \in \overline{\Delta_{\Theta'} L^2(\cD)}$, and $i=1,2$.
In particular, setting $g=0$  and recalling that $X_i = M_{\varphi_i} = M_{G_i^{\prime *} + z G_j'}$, we get
\begin{equation}   \label{verify2'}
W_i  Z^{-1}  (\Delta_{\Theta''}   f \oplus 0)       = Z^{-1} (\Delta_{\Theta''} M_{G_i^{\prime *} + z G_j'} f \oplus 0),
\end{equation}
thereby verifying \eqref{verify2}. We next consider \eqref{ConseqInt2} with $f=0$ and $g$ equal to a general element
of $\overline{\Delta_{\Theta'} L^2(\cD)}$ to get
\begin{equation}   \label{verify3'}
W_i  Z^{-1}(0\oplus g)=Z^{-1}(0\oplus W'_i g),
\end{equation}
thereby verifying \eqref{verify3} and hence also completing the proof of \eqref{ExtraCond}.

It remains to show that $(W_1', W_2')$ is a commuting pair of unitary operators satisfying condition
\eqref{propW'}.  Toward this goal, let us rewrite \eqref{verify3'} in the form
\begin{equation}  \label{verify3''}
 Z W_i Z^{-1} (0 \oplus g) = 0 \oplus W_i' g.
\end{equation}
which implies that, for $i=1,2$,
$$
ZW_iZ^{-1}(\{0\}\oplus\overline{\Delta_{\Theta'}L^2(\cD)})\subseteq(\{0\}\oplus\overline{\Delta_{\Theta'}L^2(\cD)}).
$$
On the other hand, using \eqref{verify2'} and noting that $M_{\zeta}|_{\overline{\Delta_{\Theta}L^2(\cD)}}$ commutes with $Z,W_1,W_2$ and $\Delta_{\Theta''}$,
we get for every $f\in H^2(\mathcal{F})$ and $n\geq 0$
$$
ZW_iZ^{-1}(\Delta_{\Theta''}e^{-int}f\oplus 0)=(\Delta_{\Theta''}e^{-int}X_if\oplus 0),
$$
which implies that $ZW_iZ^{-1}(\overline{\Delta_{\Theta}L^2(\cD)}\oplus \{0\})\subseteq
(\overline{\Delta_{\Theta'}L^2(\cD)}\oplus\{0\})$, for $i=1,2$. We conclude that $Z^{-1} ( \{0\} \oplus
\overline{\Delta_{\Theta'} L^2(\cD)})$ is a reducing subspace for the pair of unitaries $(W_1, W_2)$
and hence $(W_1, W_2) |_{Z^{-1} (\{0\} \oplus \overline{\Delta_{\Theta'} L^2(\cD)}}$ is a pair of
commuting unitary operators.  The intertwining \eqref{verify3''} shows that the pair $(W_1', W_2')$ on
$\overline{\Delta_{\Delta'} L^2(\cD)}$  is jointly unitarily equivalent to the commuting unitary pair
$(W_1, W_2)|_{Z^{-1}(\{0\} \oplus \overline{\Delta_\Theta'} L^2(\cD)}$ and hence is itself a commuting
unitary pair.  Furthermore, since $W_1 W_2 = M_\zeta$ in particular on $Z^{-1} ( \{0\} \oplus \overline{\Delta_{\Theta'}
L^2(\cS)})$ and $M_\zeta$ commutes past $Z$ and $Z^{-1}$, we conclude that
condition \eqref{propW'} holds as well.
This completes the proof of the necessary part.
\end{proof}

As we see from the last part of the statement of Theorem \ref{T:inv-cnu}, Sz.-Nagy and Foias went on to
prove that, under the conditions of Theorem \ref{NFcnu}, the characteristic functions
of $\bf T|_{\bH'}$ and $P_{\bH\ominus\bH'} {\bf T}|_{\bH\ominus\bH'}$ coincide with the purely contractive parts of
$\Theta'$ and $\Theta''$, respectively. We next find an analogous result (at least for the first part of this statement)
for pairs of commuting contractions.
The strategy of the proof is the same as that of Sz.-Nagy--Foias, namely: application of model theory.

\begin{theorem}  \label{T:inv-char}
Under the conditions of Theorem \ref{T:inv-cnu}, let $\bH'$ be a joint invariant subspace of $\bH$ induced by the
regular factorization $\Theta=\Theta''\Theta'$. Then with the notations as in Theorem \ref{T:inv-cnu},
the triple $((G_1',G_2'),(W_1',W_2'),\Theta')$ is admissible and its purely contractive part coincides
with the characteristic triple for $({\bf T_1},{\bf T_2})|_{\bH '}$.
\end{theorem}

\begin{proof}
With the isometry $\cI$ as in (\ref{iso-cnu}), define a unitary $U:=\cI^*|_{\operatorname{Ran }\cI}$. Therefore
\begin{align}\label{iso-cnu^*}
\nonumber
U:\{\Theta''f\oplus Z^{-1}(\Delta_{\Theta''}f\oplus g):f\in H^2(\mathcal{F}),\;g\in\overline{\Delta_{\Theta'}L^2(\cD)}\}
&\to H^2(\mathcal{F})\oplus\overline{\Delta_{\Theta'}L^2(\cD)} \\
U \colon \Theta''f\oplus Z^{-1}(\Delta_{\Theta''}f\oplus g) \mapsto  f\oplus g.
\end{align}
For every $g\in H^2(\cD)$,
\begin{eqnarray}\label{ortho}
U(\Theta g\oplus\Delta_{\Theta}g)=U(\Theta''\Theta'g\oplus Z^{-1}(\Delta_{\Theta''}\Theta'g\oplus\Delta_{\Theta'}g))
=\Theta'g\oplus\Delta_{\Theta'}g,
\end{eqnarray}which implies that $U$ takes $\bH'$ as given in (\ref{H'}) onto the Hilbert space
\begin{align}\label{fracH'}
\mathfrak{H}':=H^2(\mathcal{F})\oplus\overline{\Delta_{\Theta'}L^2(\cD)}\ominus\{\Theta'g\oplus
\Delta_{\Theta'}g:g\in H^2(\cD)\}
\end{align}
The basis of the proof is the following unitary equivalences:
\begin{align}
 \label{goalcnu}  U(M_{G_i^*+zG_j}\oplus W_i) U^*&=M_{G_i'^*+zG_j'}\oplus W_i' \text{ for } (i,j)=(1,2),(2,1), \\
 \label{goalcnu1} U(M_z\oplus M_{\zeta})U^*&=(M_z\oplus M_{\zeta}).
\end{align}
To verify \eqref{goalcnu}--\eqref{goalcnu1}, proceed as follows.
Since $M_{\zeta}$ commutes with $Z$ and $\Delta_{\Theta''}$, (\ref{goalcnu1}) follows easily.
We establish equation (\ref{goalcnu}) only for $(i,j)=(1,2)$ and omit the proof for the other case because it is similar.
For $f\in H^2(\mathcal{F})$ and $g\in\overline{\Delta_{\Theta'}L^2(\cD)}$,
\begin{align*}
&U(M_{G_1^*+zG_2}\oplus W_1)U^*(f\oplus g)=U(M_{G_1^*+zG_2}\oplus W_1)(\Theta''f\oplus Z^{-1}(\Delta_{\Theta''}f\oplus g))\\
=&\; U(\Theta''M_{G_1'^*+zG_2'}f\oplus Z^{-1}(\Delta_{\Theta''}M_{G_1'^*+zG_2'}f\oplus W_1'g))\;[\text{by }(\ref{ExtraCond})]\\
=&\; M_{G_1'^*+zG_2'}f\oplus W_1'g
\end{align*}
and \eqref{goalcnu} also follows.

We now show that the triple $((G_1',G_2'),(W_1',W_2'),\Theta')$ is admissible. Recall that in the course of the
proof of Theorem \ref{T:inv-cnu}, we saw that both $G_1'$ and $G_2'$ are contractions and that $(W_1',W_2')$
is a pair of commuting unitaries satisfying (\ref{propW'}). From (\ref{goalcnu}) we see that for every $f\in H^2(\cF)$
and $(i,j)=(1,2),(2,1)$,
\begin{align*}
(M_{G_i'^*+zG_j'}\oplus W_i') (\Theta'f\oplus\Delta_{\Theta'}f)
=&\;U(M_{G_i^*+zG_j}\oplus W_i) U^*(\Theta'f\oplus\Delta_{\Theta'}f)\\
=&\;U(M_{G_i^*+zG_j}\oplus W_i)\big(\Theta''\Theta'f\oplus Z^{-1}(\Delta_{\Theta''}\Theta'f\oplus \Delta_{\Theta'}f\big)\\
=&\;U(M_{G_i^*+zG_j}\oplus W_i)\big(\Theta f\oplus \Delta_{\Theta}f\big).
\end{align*}
From the admissibility of $((G_1,G_2),(W_1,W_2),\Theta)$, we know that each of the contraction operators
$(M_{G_i^*+zG_j}\oplus W_i)$ takes the space
$\{\Theta f\oplus \Delta_{\Theta}f:f\in H^2(\cD)\}$ into itself. Therefore from the last term of the above computation and (\ref{ortho}), we see that for each $(i,j)=(1,2),(2,1)$,
$$
(M_{G_i'^*+zG_j'}\oplus W_i')\big(\{\Theta' f\oplus \Delta_{\Theta'}f:f\in H^2(\cD)\}\big)\subset\{\Theta' f\oplus \Delta_{\Theta'}f:f\in H^2(\cD)\}.
$$
From (\ref{goalcnu}) it is also clear that for each $(i,j)=(1,2),(2,1)$, the operators $(M_{G_i'^*+zG_j'}\oplus W_i')$ are contractions and that with
$\mathfrak{H}'$ as in (\ref{fracH'})
\begin{align*}
&(M_{G_i'^*+zG_j'}\oplus W_i')^*(M_{G_j'^*+zG_i'}\oplus W_j')^*|_{\mathfrak{H}'}=U(M_{G_i^*+zG_j}
\oplus W_i)^*(M_{G_j^*+zG_i}\oplus W_j)^*|_{\bH}\\
=&\;U(M_z\oplus M_{\zeta})|_{\bH}=U(M_z\oplus M_{\zeta})U^*|_{\mathfrak{H}'}=
(M_z\oplus M_{\zeta})|_{\mathfrak{H}'} \quad [\text{by } (\ref{goalcnu1})].
\end{align*}
This completes the proof of admissibility of $((G_1',G_2'),(W_1',W_2'),\Theta')$.

And finally to prove the last part we first observe that
\begin{align*}
({\bf T_1},{\bf T_2})=P_{\bH'}(M_{G_1^*+zG_2}\oplus W_1,M_{G_2^*+zG_1}\oplus W_2)|_{\bH'}.
\end{align*}
Now from equations (\ref{goalcnu}) and (\ref{goalcnu1}) again and from the fact that
$U(\bH')=\mathfrak{H}'$ (hence $UP_{\bH'}=P_{\mathfrak{H}'}U$), we conclude that
\begin{align*}
({\bf T_1},{\bf T_2},{\bf T_1T_2})|_{\bH'}=P_{\bH'}(M_{G_1^*+zG_2}\oplus W_1,M_{G_2^*+zG_1}\oplus W_2,M_z\oplus M_{\zeta})|_{\bH'}
\end{align*}
is unitarily equivalent to the functional model associated to $((G_1',G_2'),(W_1',W_2'),\Theta')$, i.e.,
\begin{align*}
P_{\mathfrak H'}(M_{G_1^*+zG_2}\oplus W_1,M_{G_2^*+zG_1}\oplus W_2,M_z\oplus M_{\zeta})|_{\mathfrak H'}
\end{align*}
via the unitary $U|_{\bH'}:\bH'\to\mathfrak{H}'$. Therefore appeal to Theorem \ref{UnitaryInv},
Theorem \ref{AdmisCharc} and Proposition \ref{P:pure-adm-triple} completes the proof.
\end{proof}

In case the purely contractive analytic function $(\cD,\cD_{*},\Theta)$ is inner, the results above are much
simpler, as in the following statement.

\begin{theorem}  \label{T:inv-inner}
Let $(\cD,\cD_*,\Theta)$ be an inner function and $((G_1,G_2),\Theta)$ be an admissible pair.
Define the pair $({\bf T_1},{\bf T_2})$ of commuting contractions on
\begin{align}\label{cnuMspace-inner}
\bH=H^2(\cD_*)\ominus\{\Theta f: f\in H^2(\cD)\}
\end{align}by
\begin{align}\label{Mop-inner}
 ({\bf T_1},{\bf T_2})=P_{\bH}\big(M_{G_1^*+zG_2},M_{G_2^*+zG_1}\big)|_{\bH}.
\end{align}
A subspace $\bH'$ of $\bH$ is jointly invariant under $({\bf T_1},{\bf T_2})$ if and only if there exist two inner functions
$(\mathcal{D},\mathcal{F},\Theta')$, $(\mathcal{F},\mathcal{D}_{*},\Theta'')$ such that
\begin{eqnarray*}
\Theta=\Theta''\Theta'
\end{eqnarray*}
is a regular factorization,
\begin{align}
&  \bH'=\{\Theta''f:f\in H^2(\mathcal{F})\}\ominus\{\Theta h:h\in H^2(\cD)\}, \notag  \\
& \bH'':=\bH\ominus\bH'=\;  H^2(\cD_*)\ominus\{\Theta''f:f\in H^2(\mathcal{F})\},
\label{Hprimes-inner}
\end{align}
and two contractions $G'_1,G'_2$ in $\mathcal{B}(\mathcal{F})$ such that
\begin{equation}   \label{ExtraCond-inner}
M_{G_i^*+zG_j}M_{\Theta''}=M_{\Theta''}M_{G_i'^*+zG'_j}.
\end{equation}
Moreover, the pair $((G_1',G_2'),\Theta')$ coincides with the characteristic pair for $(T_1,T_2)|_{\mathcal{H}'}$.
\end{theorem}

Another interesting simplification is the case where the input space $\cD$ is the zero space and the purely contractive analytic functions $(\{0\} , \cD_*, \Theta)$ is
necessarily the zero function $\Theta(z) = 0 \colon \{ 0\} \to \cD_*$ for all $z \in {\mathbb D}$.  If $\Theta = \Theta'' \Theta'$
is any factorization
into contractive analytic functions $(\cF, \cD_*, \Theta'')$,  $(\{0\}, \cF, \Theta')$, then $\Theta'$ is forced to be the zero
function.  One can then show that the
factorization $ 0 = \Theta'' \cdot 0$ is regular exactly when $\Theta''$ is inner, so $I - \Theta''(\zeta)^* \Theta''(\zeta) = I_\cF$
for a.e. $\zeta \in {\mathbb T}$.
Then Theorem \ref{T:inv-cnu} simplifies to the following form; in view of Remark \ref{R:BL},
this result can be viewed as a bivariate version of the
Beurling-Lax Theorem.
\index{Beurling-Lax Theorem!bivariate version of}  
We note that an analogue of this result appears in the context of model theory for a commuting 
pair of operators $(S,P)$ having
the symmetrized bidisk as a spectral set (i.e., a $\Gamma$-contraction)---see Sarkar \cite[Theorem 3.3]{Sarkar2015}, and for a triple of commuting operators $(A,B,P)$ having the tetrablock as a spectral set (i.e., a tetrablock contraction)---see \cite[Theorem 3.1]{SauNYJM}.

\begin{theorem}  \label{T:inv-0}
Suppose that $(\cD_*, P, U)$ is a  BCL tuple.  Define a pair of commuting isometries $(\bf T_1, \bf T_2)$ on
$\bH = H^2(\cD_*)$ as the associated BCL2 model pair of commuting isometries
$$
  ( \bT_1, \bT_2) =  (M_{ U^* ( P^\perp + z P)},  M_{(P + zP^\perp) U}).
$$
A subspace $\bH'$ of $\bH$ is jointly invariant under $(\bT_1, \bT_2)$ if and only if there exist an inner function
$(\cF, \cD_*, \Theta'')$ and another BCL tuple $(\cF, P', U')$ so that
\begin{equation}  \label{Hprimes-0}
\bH' = \Theta'' H^2(\cF), \quad \bH'' = H^2(\cD_*) \ominus \bH'  = H^2(\cD_*) \ominus \Theta'' H^2(\cF)
\end{equation}
and lastly
$$
M_{ U^* ( P^\perp + z P)} M_{\Theta''} = M_{\Theta''} M_{ U'^* ( P'^\perp + z P')}, \quad
M_{(P + z P^\perp) U} M_{\Theta''}  = M_{\Theta''} M_{(P' + z P'^\perp) U'}.
$$
\end{theorem}

A particular choice of BCL tuple in Theorem \ref{T:inv-0} is $(\cD_*, P, U) = (\ell^2_{\mathbb Z}, P_{[1, \infty)}, \bS)$ as in Example \ref{E:BCLbidisk}, so that
the resulting $(\bT_1, \bT_2)$ is just the BCL1 model for the bidisk shift-pair $(M_{z_1}, M_{z_2})$ on $H^2_{{\mathbb D}^2}$.  In this model the operator $M_{z_1 z_2}$
on $H^2_{{\mathbb D}^2}$ becomes the operator $M_z$ on $H^2(\ell^2_{\mathbb Z})$.  The standard single-variable Beurling-Lax Theorem tells us that
invariant subspaces for $M_{z_1 z_2}$ on the Hardy space of the bidisk are then in one-to-one correspondence with inner functions of the form
$(\cF, \ell^2_{\mathbb Z}, \Theta'')$; the latter object is not easy to classify since the target coefficient space $\ell^2_{\mathbb Z}$ is infinite-dimensional.  Joint invariant subspaces
for $(M_{z_1}, M_{z_2})$ on $H^2_{{\mathbb D}^2}$ then correspond to such inner $\Theta''$ such that in addition
$\Theta'' \cdot H^2(\ell^2_{\mathbb Z})$ is jointly invariant for $(M_{ \bS^* ( P^\perp + z P)}, M_{(P + z P^\perp) \bS})$
(where $P = P_{\ell^2_{[1, \infty)}}$).  The tradeoff between the bidisk setting versus the BCL-setting is:  in the bidisk setting one has scalar-valued functions
at the cost of the functions being of two variables, while in the BCL setting one has single-variable functions but with values in the infinite-dimensional space
($\ell^2_{\mathbb Z}$).  Characterizing joint invariant subspaces in either setting appears to be rather intractable.  There has been by now much work on the problem
leading to deeper appreciation as to how complicated the structure is in the bidisk setting:  a small sample of such work is \cite{Burdak, Douglas-Yan, SY, Yang-survey}.

\bibliographystyle{amsalpha}

\printindex

\end{document}